\documentclass[11pt]{amsart}

\usepackage{geometry}
\usepackage[utf8]{inputenc}
\usepackage{amsmath, amssymb,amsthm}
\usepackage{amsfonts}
\usepackage{tikz-cd}
\usepackage{tikz}
\usepackage{todonotes}
\renewcommand{\todo}[1]{}
\usepackage{mathrsfs}
\usepackage{extarrows}
\usepackage{graphicx}
\usepackage{thmtools}
\usepackage{mathtools}
\usepackage{hyperref}
\usepackage[all]{xy}

\newdir^{ (}{{}*!/-3pt/\dir^{(}}    
\newdir^{  }{{}*!/-3pt/\dir^{}}    
\newdir_{ (}{{}*!/-3pt/\dir_{(}}    
\newdir_{  }{{}*!/-3pt/\dir_{}}    

\usepackage{cleveref}
\usepackage{geometry}[margin=1 in]

\usepackage[english]{babel}

\usepackage[toc,page]{appendix}

\makeatletter
\newcommand*{\da@rightarrow}{\mathchar"0\hexnumber@\symAMSa 4B }
\newcommand*{\da@leftarrow}{\mathchar"0\hexnumber@\symAMSa 4C }
\newcommand*{\xdashrightarrow}[2][]{
  \mathrel{
    \mathpalette{\da@xarrow{#1}{#2}{}\da@rightarrow{\,}{}}{}
  }
}
\newcommand{\xdashleftarrow}[2][]{
  \mathrel{
    \mathpalette{\da@xarrow{#1}{#2}\da@leftarrow{}{}{\,}}{}%
  }
}
\newcommand*{\da@xarrow}[7]{
  \sbox0{$\ifx#7\scriptstyle\scriptscriptstyle\else\scriptstyle\fi#5#1#6\m@th$}
  \sbox2{$\ifx#7\scriptstyle\scriptscriptstyle\else\scriptstyle\fi#5#2#6\m@th$}
  \sbox4{$#7\dabar@\m@th$}
  \dimen@=\wd0 
  \ifdim\wd2 >\dimen@
    \dimen@=\wd2 %
  \fi
  \count@=2 %
  \def\da@bars{\dabar@\dabar@}%
  \@whiledim\count@\wd4<\dimen@\do{%
    \advance\count@\@ne
    \expandafter\def\expandafter\da@bars\expandafter{%
      \da@bars
      \dabar@ 
    }%
  }%
  \mathrel{#3}%
  \mathrel{%
    \mathop{\da@bars}\limits
    \ifx\\#1\\%
    \else
      _{\copy0}%
    \fi
    \ifx\\#2\\%
    \else
      ^{\copy2}%
    \fi
  }%
  \mathrel{#4}%
}
\makeatother

\def\theenumi{(\alph{enumi})}

\def\p@enumii{\theenumi} 

\long\def\forget#1{}
\long\def\exclude#1{} 

\newcommand{\N}{\mathbb{N}}
\newcommand{\Z}{\mathbb{Z}}

\newcommand{\C}{\mathbb{C}}
\newcommand{\Q}{\mathbb{Q}}
\newcommand{\F}{\mathbb{F}}
\newcommand{\A}{\mathbb{A}}
\newcommand{\mcA}{\mathcal{A}}

\newcommand{\mcE}{\mathcal{E}}
\newcommand{\mcF}{\mathcal{F}}
\newcommand{\mcG}{\mathcal{G}}

\newcommand{\mcL}{\mathcal{L}}
\newcommand{\mcM}{\mathcal{M}}
\newcommand{\mcZ}{\mathcal{Z}}

\newcommand{\Fl}{\mathcal{F}\!\ell}

\newcommand{\G}{\mathbb{G}}

\newcommand{\Oo}{\mathcal{O}}
\newcommand{\alg}{{\rm alg}}

\newcommand{\et}{{\rm \acute{e}t\/}}
\newcommand{\Defo}{\operatorname{Defo}}
\newcommand{\End}{\operatorname{End}}

\newcommand{\Image}{\operatorname{im}}
\newcommand{\Int}{\operatorname{int}}

\newcommand{\Id}{\operatorname{Id}}
\newcommand{\id}{\operatorname{id}}

\newcommand{\Frob}{\operatorname{Frob}}
\newcommand{\Nilp}{{\mathcal{N}\!{\it ilp}}}
\newcommand{\QIsog}{\operatorname{QIsog}}
\newcommand{\Rep}{\operatorname{Rep}}
\newcommand{\Var}{\operatorname{V}}

\newcommand{\pr}{\operatorname{pr}}

\newcommand{\red}{{\rm red}}

\DeclareMathOperator{\widehattimes}{\mathchoice
            {\widehat{\raisebox{0ex}[0ex]{$\displaystyle\times$}}}
            {\widehat{\raisebox{0ex}[0ex]{$\textstyle\times$}}}
            {\widehat{\raisebox{0ex}[0ex]{$\scriptstyle\times$}}}
            {\widehat{\raisebox{0ex}[0ex]{$\scriptscriptstyle\times$}}}}

\newcommand{\fppf}{{\it fppf\/}}
\newcommand{\fpqc}{{\it fpqc\/}}

\newcommand{\Aut}{\operatorname{Aut}}

\newcommand{\Spec}{\operatorname{Spec}}
\newcommand{\Spf}{\operatorname{Spf}}

\newcommand{\Frac}{\operatorname{Frac}}

\newcommand{\Hom}{\operatorname{Hom}}
\newcommand{\Isom}{\operatorname{Isom}}

\newcommand{\Res}{\operatorname{Res}}

\newcommand{\ext}{\mathrm{ext}}

\newcommand{\Gal}{\operatorname{Gal}}
\newcommand{\sep}{{\operatorname{sep}}}
\newcommand{\GL}{\operatorname{GL}}
\newcommand{\SL}{\operatorname{SL}}
\newcommand{\PGL}{\operatorname{PGL}}

\newcommand{\Bun}{\operatorname{Bun}}
\newcommand{\CBun}{\mathcal{CB}un}
\newcommand{\Hecke}{\operatorname{Hecke}}
\newcommand{\PHecke}{{{}'\!\operatorname{Hecke}}}

\newcommand{\Gr}{\operatorname{Gr}}
\newcommand{\PGr}{{{}'\!\operatorname{Gr}}}

\newcommand{\PmcZ}{{{}'\!\mcZ}}
\newcommand{\Sht}{\operatorname{Sht}}

\newcommand{\LocSht}{\operatorname{LocSht}}

\renewcommand{\Nilp}{\mathcal{N}ilp}

\newcommand{\RZ}{\mathcal{R}\mcZ}

\newcommand{\dbl}{{\mathchoice{\mbox{\rm [\hspace{-0.15em}[}}
                              {\mbox{\rm [\hspace{-0.15em}[}}
                              {\mbox{\scriptsize\rm [\hspace{-0.15em}[}}
                              {\mbox{\tiny\rm [\hspace{-0.15em}[}}}}
\newcommand{\dbr}{{\mathchoice{\mbox{\rm ]\hspace{-0.15em}]}}
                              {\mbox{\rm ]\hspace{-0.15em}]}}
                              {\mbox{\scriptsize\rm ]\hspace{-0.15em}]}}
                              {\mbox{\tiny\rm ]\hspace{-0.15em}]}}}}
\newcommand{\dpl}{{\mathchoice{\mbox{\rm (\hspace{-0.15em}(}}
                              {\mbox{\rm (\hspace{-0.15em}(}}
                              {\mbox{\scriptsize\rm (\hspace{-0.15em}(}}
                              {\mbox{\tiny\rm (\hspace{-0.15em}(}}}}
\newcommand{\dpr}{{\mathchoice{\mbox{\rm )\hspace{-0.15em})}}
                              {\mbox{\rm )\hspace{-0.15em})}}
                              {\mbox{\scriptsize\rm )\hspace{-0.15em})}}
                              {\mbox{\tiny\rm )\hspace{-0.15em})}}}}

\usepackage{ifthen}

\newcommand{\invlim}[1][]{\ifthenelse{\equal{#1}{}}
{\displaystyle \lim_{\longleftarrow}}
{\displaystyle \lim_{\underset{#1}{\longleftarrow}}}
}

\newcommand{\dirlim}[1][]{\ifthenelse{\equal{#1}{}}
{\displaystyle \lim_{\longrightarrow}}
{\displaystyle \lim_{\underset{#1}{\longrightarrow}}}
}

\def\longto{\longrightarrow}
\def\into{\hookrightarrow}
\def\onto{\rsurj}
\def\isoto{\stackrel{}{\mbox{\hspace{1mm}\raisebox{+1.4mm}{$\scriptstyle\sim$}\hspace{-4.2mm}$\longrightarrow$}}}
\def\onto{\mbox{$\kern2pt\to\kern-8pt\to\kern2pt$}}
\def\longinto{\lhook\joinrel\longrightarrow}

\hypersetup{
    colorlinks,
    citecolor=blue,
    filecolor=blue,
    linkcolor=blue,
    urlcolor=blue
}

\renewcommand{\setminus}{\smallsetminus}

\renewcommand{\mod}{\operatorname{mod}}

\numberwithin{equation}{subsection}
\newtheorem{thm}[equation]{Theorem}

\newtheorem{lem}[equation]{Lemma}

\newtheorem{cor}[equation]{Corollary}

\numberwithin{equation}{subsection}
\newtheorem{prop}[equation]{Proposition}

\theoremstyle{definition}
\numberwithin{equation}{subsection}
\newtheorem{remark}[equation]{Remark}

\newtheorem{numberedparagraph}[equation]{}

\newtheorem{defn}[equation]{Definition}
\newtheorem{construction}[equation]{Construction}
\newtheorem{example}[equation]{Example}

\usepackage{color}

\title{Uniformizing the moduli stacks of global $G$-shtukas II}

\author{Urs Hartl}
\address{Uni Muenster, Germany}
\email{urs.hartl@uni-muenster.de}

\author{Yujie Xu}
\address{Columbia University, New York City, USA}
\email{yujiexu@mit.edu}

\newcommand{\gpSch}{\mathcal{G}}
\newcommand{\gengpSch}{G}
\newcommand{\FcnFld}{Q} 
\newcommand{\diagphi}{\varphi^\lhd}
\newcommand{\FramingObject}{{\underline{\mathbb{E}}}}
\newcommand{\FramingObjectComp}{{\mathbb{E}}}
\newcommand{\LocalFramingObject}{{\mathbb{L}}}
\newcommand{\ulLocalFramingObject}{{\underline{\LocalFramingObject}}}

\newcommand{\BaseFldInSectUnif}{{\overline\F_\infty}}
\newcommand{\OReflZMuBeta}{{\mathcal{O}_{\mu,\beta}}}
\newcommand{\BreveOReflZMuBeta}{{\Breve{\mathcal{O}}_{\mu,\beta}}}
\newcommand{\KappaReflZMuBeta}{{\kappa_{\mu,\beta}}}

\begin{document}

\maketitle

\begin{abstract}
We show that the moduli spaces of bounded global $\gpSch$-Shtukas with pairwise colliding legs 
admit $p$-adic uniformization isomorphisms by Rapoport-Zink spaces. Here $\gpSch$ is a smooth affine group scheme with connected fibers and reductive generic fiber, i.e.~we do not assume it to be parahoric, or even hyperspecial.
Moreover, we deduce the Langlands-Rapoport Conjecture over function fields in the case of colliding legs using our uniformization theorem. 
\end{abstract}

\color{black}

\section{Introduction}
\subsection{Main Results}

Shimura varieties play an important role in arithmetic geometry. 
Their structure, especially their reduction at bad primes, has been intensively studied. One way to investigate their reduction is through $p$-adic uniformization (see for example \cite{RZ}). In this paper, 
we consider the function field analogues of Shimura varieties and their $p$-adic uniformization. The first such analogues are the moduli spaces of Drinfeld modules~\cite{Drinfeld-elliptic-modules}. In order to prove Langlands reciprocity \cite{Drinfeld-GL2,Drinfeld-moduli-Fsheaves} for $\GL_2$ over the function field $\FcnFld$ of a smooth projective curve $X$ over a finite field $\F_q$, Drinfeld~\cite{Drinfeld-moduli-Fsheaves} defined \emph{global $\gpSch$-shtukas} (which he called ``$F$-sheaves'') and constructed their moduli spaces. These were later generalized by Varshavsky~\cite{Varshavsky04} and V.~Lafforgue~\cite{Lafforgue12} to the case of arbitrary constant split reductive groups $\gpSch$, and by Ng\^o and Ng\^o Dac~\cite{NgoNgo,NgoDac13} to the case of certain non-constant groups $\gpSch$. For general flat affine group schemes $\gpSch$ of finite type over $X$, the moduli spaces of bounded global $\gpSch$-shtukas with $n$ legs were constructed by Arasteh Rad and the first author \cite{AH_Unif} as separated Deligne-Mumford stacks locally of finite type over $\F_q$.

More precisely, an (\emph{iterated}) \emph{global $\gpSch$-shtuka $\underline{\mcE}=(\underline x,\mcE^{(i)},\varphi^{(i)}\colon i=0,\ldots,n-1)$ with $n$ legs} over a base scheme $S$ consists of a tuple $\underline x=(x_1,\ldots,x_n)\in X^n(S)$ called \emph{legs}, and $\gpSch$-bundles $\mcE^{(i)}$ over $X_S:=X\times_{\F_q} S$ together with isomorphisms (called \emph{modifications}) $\varphi^{(i-1)}\colon \mcE^{(i-1)}|_{X_S\smallsetminus\Gamma_{x_i}}\isoto \mcE^{(i)}|_{X_S\smallsetminus\Gamma_{x_i}}$ of $\gpSch$-bundles outside the graph $\Gamma_{x_i}$ of $x_i$, with $\mcE^{(n)}:={}^{\tau\!}\mcE^{(0)}$, where $\tau:=\id_X \times \Frob_{q,S}$ (see Definition~\ref{Def_Sht} for the general setup). 
Drinfeld considered the case where $\gpSch:=\GL_r$ and $n=2$ (see Example~\ref{ExDrinfeld}). In addition to \cite{AH_Unif}, various other special cases have been treated in literature: Hilbert-Blumenthal Shtukas by Stuhler~\cite{Stuhler} for $\gpSch=\Res_{X'|X}\SL_r$ where $X'$ is a smooth curve, finite flat over $X$ on which $\infty$ splits completely (``totally real case''); $\mathscr{D}$-elliptic sheaves by Laumon, Rapoport and Stuhler~\cite{Laumon-Rapoport-Stuhler} where $\gpSch$ is the unit group of (a maximal order in) a central division algebra over $\FcnFld$ (see Example~\ref{ExLRS}); generalizations of $\mathscr{D}$-elliptic sheaves by L.~Lafforgue~\cite{Lafforgue-Ramanujan}, Lau~\cite{Lau07}, Ng\^o~\cite{Ngo06} and Spiess~\cite{Spiess10}; global $\gpSch$-shtukas for $\gpSch=\PGL_2$ in~\cite{YunZhang,YunZhang2}, and for unitary groups $\gpSch$ in~\cite{FYZ,FYZ2}. In all these cases, the modifications $\varphi^{(i)}$ were suitably bounded.

In the current article, we study uniformization of the moduli stacks of $\gpSch$-shtukas with \emph{colliding legs}. We generalize the notion of boundedness from \cite{AH_Unif,Bieker}. We define \emph{bounds} $\mcZ$ in Section~\ref{subsec:Bounds} as closed subschemes of the Beilinson-Drinfeld Grassmannian $\Gr_{\gpSch,X^n,I_\bullet}$ which are defined over their \emph{reflex scheme}. The latter generalizes the notion of a reflex field, which is familiar from the theory of Shimura varieties. Here we allow more general reflex schemes than in \cite{AH_Unif,Bieker}. In particular, the modifications $\varphi^{(i)}$ can be bounded by cocharacters of $\gpSch$. But we also allow the bound to be defined in the fiber of $\Gr_{\gpSch,X^n,I_\bullet}$ at a closed point of $X$ when the leg is constant at that point. This is needed to recover and generalize the bounds used in \cite{Stuhler,Laumon-Rapoport-Stuhler}. We define the stack $\Sht_{\gpSch,X^n,I_\bullet}^\mcZ$ of global $\gpSch$-shtukas with $n$ legs bounded by $\mcZ$. In Theorems~\ref{Thm_ShtBounded} and \ref{ThmLocMod}, we prove the following result (in more detail and generality).

\begin{thm}\label{Thm_ShtBounded-intro}
The stack $\Sht_{\gpSch,X^n,I_\bullet}^\mcZ$ is a Deligne-Mumford stack locally of finite type and separated over $X^n$. The bound $\mcZ$ is a local model for $\Sht_{\gpSch,D,X^n,I_\bullet}^{\mcZ}$, i.e.~both are isomorphic locally for the \'etale topology. 
\end{thm}

Our main result in this article is the uniformization of these stacks. In Section~\ref{subsect:history}, we give a summary of the history of uniformization for Shimura varieties and moduli stacks of shtukas. For the latter, uniformization was proven in full generality for smooth affine group schemes $\gpSch$ over $X$ with connected fibers and reductive generic fiber by Arasteh Rad and the first author \cite{AH_Unif} under the assumption that the legs $x_i$ stay \emph{disjoint}. For \emph{colliding legs}, uniformization of shtuka stacks is only known in the very special case of $\gpSch=\GL_r$ or $\Res_{X'|X}\SL_r$ as above, with two legs, one fixed at a closed point $\infty\in X$ with basic bound and the other leg moving into $\infty$ bounded by the cocharacter $\mu=(d,0,\ldots,0)$ for a positive integer $d$ (see~\cite{Drinfeld-elliptic-modules, Drinfeld-commutative-subrings, Stuhler, Blum-Stuhler, HartlAbSh}). In all of the above articles, the bound at $\infty$ was imposed by using chains of $\gpSch$-bundles. Our results recover the case of two legs, one fixed at $\infty$ with minimal basic bound and the other leg moving into $\infty$ (see $\S$\ref{subsec:Chains}), but we generalize to arbitrary smooth, affine group schemes $\gpSch$ over $X$ with connected fibers and reductive generic fiber $\gengpSch$ and arbitrary cocharacter $\mu$ as bound. 

Our \emph{first innovation} is to replace the chains from \cite{Drinfeld-elliptic-modules, Drinfeld-commutative-subrings, Stuhler, Blum-Stuhler, HartlAbSh} by a suitably chosen bound at $\infty$. 
Let $\FcnFld_\infty$ denote the completion of $\FcnFld$ at $\infty$, and let $\Breve{\FcnFld}_\infty$ be the completion of the maximal unramified extension of $\FcnFld_\infty$. Let $\gpSch_\infty:=\gpSch\times_X \Spec\Oo_\infty$ be the base change to the valuation ring $\Oo_\infty$ of $\FcnFld_\infty$. Take $\beta\in \gengpSch(\Breve{\FcnFld}_\infty)$ such that $\beta \gpSch_\infty\beta^{-1}=\gpSch_\infty$. Thus in a sense, $\beta$ is as small as possible (see \S\,\ref{Def_beta}). We then consider the stack $\Sht_{\gpSch,X\times\infty}$ of global $\gpSch$-shtukas with two legs, of which one is allowed to vary over all of $X$ and the other is fixed at $\infty$. For a conjugacy class of cocharacters $\mu\colon\mathbb{G}_m \to \gengpSch$ and a compact open subgroup $H\subset\gengpSch(\mathbb{A}^\infty)$, we define the stack $\Sht_{\gpSch,H,\widehat{\infty}\times\infty}^{\mcZ(\mu,\beta)}$ of such global $\gpSch$-shtukas with $H$-level structure, whose varying leg $x\colon S\to X$ factors through $\widehat{\infty}:=\Spf\Oo_\infty$ and whose modification $\varphi^{(0)}$ (respectively $\varphi^{(1)}$) at $x$ (respectively $\infty$) is bounded by $\mu$ (respectively by $\beta$) (see Definitions~\ref{DefZmubeta} and \ref{Def_ShtBounded}).

Our \emph{second innovation} is to construct the $\beta^{-1}$-twisted global-local functor $L^+_{\infty,\mcM}$ in Definition~\ref{DefGlobLocM}, which relates such global $\gpSch$-shtukas to local $\mcM$-shtukas (see \S\ref{subsec-LocSht-and-RZ}), where $\mcM$ is the inner form of $\gpSch_\infty$ given by $\beta^{-1}$. 
The cocharacter $\mu$ gives rise to a cocharacter $\mu$ of $\mcM$ by which we can bound local $\mcM$-shtukas (see Definition~\ref{DefLocShtBounded} for more details). Let $\Oo_\mu$ be the ring of integers in the reflex field extension of $\FcnFld_\infty$ of $\mu$ and let $\Breve{\Oo}_\mu$ be the completion of the maximal unramified extension of $\Oo_\mu$. For any global $\gpSch$-shtuka $\underline{\mcE}\in \Sht_{\gpSch,H,\widehat{\infty}\times\infty}^{\mcZ(\mu,\beta)}(S)$, the local $\mcM$-shtuka $L^+_{\infty,\mcM}(\underline{\mcE})$ is bounded by $\mu$ (Corollary~\ref{CorGlobalLocalFunctorWithChains}). Via this twisted global-local functor, we obtain a Serre-Tate Theorem for Shtukas. 
\begin{thm}\label{Serre-Tate-thm-intro}
There is an equivalence between deformations of $\underline{\mcE}$ and deformations of $L^+_{\infty,\mcM}(\underline{\mcE})$.    
\end{thm}
Theorem \ref{Serre-Tate-thm-intro} allows us to prove the following Theorem \ref{Uniformization1Intro}. More precisely, we fix a global framing object $\gpSch$-shtuka $\FramingObject\in \Sht_{\gpSch,\widehat{\infty}\times\infty}^{\mcZ(\mu,\beta)}(\overline{\F}_\infty)$ over $\overline{\F}_\infty$. Let $\ulLocalFramingObject$ be its associated local $\mcM$-shtuka under the $\beta$-twisted global-to-local functor (see Definition \ref{DefGlobLocM}). Let $\RZ_{\mcM,\ulLocalFramingObject}^{\leq \mu}$ be the \emph{Rapoport-Zink space} of $\ulLocalFramingObject$ bounded by $\mu$. 
It was shown by Arasteh Rad and the first author in \cite[Theorem~4.18]{AH_Local} that the Rapoport-Zink space is representable by a formal scheme locally formally of finite type over $\Spf\Breve{\Oo}_\mu$. Finally, let $I_\FramingObject(\FcnFld)$ be the quasi-isogeny group of $\FramingObject$ (see Remark~\ref{RemIsogGlobalSht}). With this setup, we generalize the uniformization theorem of \cite{Drinfeld-elliptic-modules, Drinfeld-commutative-subrings, Stuhler, Blum-Stuhler, HartlAbSh} to the following. 

\begin{thm} \label{Uniformization1Intro}
(a) There is a canonical morphism 
\begin{equation}\label{EqUnifIntro}
\Theta_{\FramingObject}\colon  I_{\FramingObject}(\FcnFld) \big{\backslash}\bigl(\RZ_{\mcM,\ulLocalFramingObject}^{\leq \mu}\times \gengpSch(\mathbb{A}^\infty)/H\bigr) \;\longto\; \Sht_{\gpSch,H,\widehat{\infty}\times\infty}^{\mcZ(\mu,\beta)}\widehattimes \Spf\Breve{\Oo}_\mu
\end{equation}
of formal algebraic Deligne-Mumford stacks over $\Spf\Breve{\Oo}_\mu$, which is an ind-proper and formally \'etale monomorphism.

(b) Let $\mathcal{X}:=\mathcal{X}_\FramingObject$ be the image of $\Theta_{\FramingObject}$ as in Lemma~\ref{LemmaImageOfTheta}\ref{LemmaImageOfTheta_B}. It is the isogeny class of $\FramingObject$. Let $\Sht_{\gpSch,H,\widehat{\infty}\times\infty}^{\mcZ(\mu,\beta)}{}{}_{/\mathcal{X}}$ be the formal completion of $\Sht_{\gpSch,H,\widehat{\infty}\times\infty}^{\mcZ(\mu,\beta)}\widehattimes \Spf\Breve{\Oo}_\mu$ along the set $\mathcal{X}_\FramingObject$. Then $\Theta_{\FramingObject}$ induces an isomorphism of locally noetherian, adic formal algebraic Deligne-Mumford stacks locally formally of finite type over $\Spf\Breve{\Oo}_\mu$
$$
\Theta_{\FramingObject,\mathcal{X}}\colon  I_{\FramingObject}(\FcnFld) \big{\backslash}\bigl(\RZ_{\mcM,\ulLocalFramingObject}^{\leq \mu}\times \gengpSch(\mathbb{A}^\infty)/H\bigr)\;\isoto\; \Sht_{\gpSch,H,\widehat{\infty}\times\infty}^{\mcZ(\mu,\beta)}{}_{/\mathcal{X}}\,.
$$ 
\end{thm}

Our theorem recovers all known uniformization results for moduli spaces of shtukas with colliding legs. For the moduli space of $\mathscr{D}$-elliptic sheaves, we hereby prove the result expected by Laumon, Rapoport and Stuhler~\cite[(14.19)]{Laumon-Rapoport-Stuhler}. We prove Theorem~\ref{Uniformization1Intro} slightly more generally in Theorems~\ref{Uniformization1} and \ref{Uniformization2} and discuss the compatibility of the morphisms with the actions through Hecke correspondences and other actions. Combining our techniques with the ones from \cite{AH_Unif}, one can even extend our result to the case of several disjoint pairs of two colliding legs, such that in each pair one leg varies and the other leg is fixed at a place $\infty_i$ and bounded by some element $\beta_i$. Then the uniformizing space will be a product of Rapoport-Zink spaces for inner forms of $\gpSch$; see Remark~\ref{RemManyPairsOfLegs}.

As an application of our Theorem \ref{Uniformization1Intro}, we prove the function field analogue of the Langlands-Rapoport conjecture for moduli spaces of global $\gpSch$-Shtukas with colliding legs. The case of disjoint legs was solved by Arasteh Rad and the first author in \cite{AH_LRConj}. 
The Langlands-Rapoport conjecture was first proposed for Shimura varieties in \cite{Langlands-Jugendtraum}, which describes some of the ideas flowing from Kronecker's \textit{Jugendtraum}. The conjecture, following the works of Ihara \cite{Ihara1,Ihara2-Vol1, Ihara2-Vol2}, is an essential part of the Langlands program \cite{Langlands2, Langlands3, Langlands4} to express the zeta function of a Shimura variety as a product of automorphic $L$-functions. This conjecture was made more precise by Kottwitz \cite{Kottwitz-lambda-adic} and then further refined by Langlands-Rapoport \cite{Langlands-Rapoport-gerbes} using the formalism of motives. In the Shimura varieties setting, there has been progress towards this conjecture in the case of abelian type Shimura varieties at hyperspecial level at $p$ \cite{Kisin-mod-p-points, Kisin-Shin-Zhu}. 

In the function field setting, we consider the $\FcnFld$-linear semi-simple Tannakian category $\mathcal{M}ot_X^{\infty
}$ of ``$X$-motives'' away from $\infty$ (see Definition \ref{defn-category-Xmotives}), which generalizes Anderson's $t$-motives \cite{Anderson-tmotives}. It is equipped with a fiber functor $\underline{\omega}$ to vector spaces over $\FcnFld\otimes_{\F_q}\overline{\F}_q$, whose Tannakian fundamental group $\mathfrak{P}:=\mathrm{Aut}^{\otimes}(\underline{\omega}|\mathcal{M}ot_X^{\infty
})$ is the \emph{motivic groupoid}. One can alternatively view $\mathfrak{P}$ as a motivic Galois gerbe via $1\to \mathfrak{P}^{\Delta}\to\mathfrak{P}\to \Gal(\FcnFld\otimes_{\F_q}{\overline{\F}_q}/\FcnFld)\to 1$. Given any global $\gpSch$-Shtuka $\underline{\mcE}\in\Sht_{\gpSch,\varnothing,\widehat{\infty}\times\infty}(\overline{\F}_\infty)$, one can associate a corresponding ``$G$-motive'', which is a tensor functor $h_{\underline{\mcE}}$ from $\Rep_\FcnFld G$ to the category $\mathcal{M}ot_X^{\infty
}$ of ``$X$-motives''; equivalently, $h_{\underline{\mcE}}$ gives a homomorphism (of Galois gerbes) from $\mathfrak{P}$ to the the neutral Galois gerbe $\mathfrak{G}_\gengpSch:=\gengpSch(\Breve{\FcnFld})\rtimes \Gal(\Breve{\FcnFld}/\FcnFld)$ of $\gengpSch$. To each such homomorphism $h$, one can attach a set $X^{\infty}(h)$, which corresponds to ``prime-to-$\infty$'' quasi-isogenies (i.e.~they are isomorphisms at $\infty$), and a set $X_{\infty}(h)$, which corresponds to ``at-$\infty$'' quasi-isogenies (i.e.~they are isomorphisms away from $\infty$). Let $I_h$ be the ``isogeny group'' of $h$ (see Remark \ref{RemIsogGlobalSht}). The Langlands-Rapoport conjecture gives a precise description of the action of $I_h(\FcnFld)$ on $X_{\infty}(h)\times X^{\infty}(h)$. 

\begin{thm}
    The $\overline{\F}_{\infty}$-points of the Shtuka space $\Sht_{\gpSch,H,\widehat{\infty}\times\infty}^{\mathcal{Z}(\mu,\beta)}$ has the form predicted by the Langlands-Rapoport conjecture, i.e. \begin{equation}
    \Sht_{\gpSch,H,\widehat{\infty}\times\infty}^{\mathcal{Z}(\mu,\beta)}(\overline{\F}_{\infty})=\coprod\limits_h I_h(\FcnFld)\backslash X_{\infty}(h)\times X^{\infty}(h)/H,
    \end{equation}
    compatible with Hecke correspondences, Frobenius, and the action of the center. 
\end{thm}
Note that for the level structure at $\infty$, we take the group scheme $\gpSch_{\infty}$ which is only assumed to be smooth, affine, with connected fibers and reductive generic fiber. In particular, our level structure at $\infty$ does not need to be parahoric (or even hyperspecial). 

\subsection{Historical overview of uniformization} \label{subsect:history}
For convenience of the reader, let us summarize the history of uniformization for Shimura varieties and shtuka stacks. 

\smallskip\noindent
{\bfseries A. Uniformization varieties at infinity.} 
The history begins in the 19th century with (1) elliptic modular curves over the complex numbers $\C$, which can be written as quotients of Poincar\'e's upper halfplane by congruence subgroups. It was generalized by Baily, Borel
, and Shimura who showed that (2) certain quotients of Hermitian symmetric domains by discrete arithmetic groups are algebraic varieties and defined over number fields. Deligne~\cite{Deligne-varietes-de-Shimura} systematically developed the theory of these varieties, which today are called \emph{Shimura varieties} and wrote them as a double quotient of a Hermitian symmetric domain times the ad\`ele points of the corresponding reductive group. All the cases (1), (2) of uniformization are for number fields and over $\C$, i.e.~``at infinity''. Function fields first came into play with (3) Drinfeld modular varieties \cite{Drinfeld-elliptic-modules} which parameterize Drinfeld $A$-modules of rank $r$ where $A=\Gamma(X\smallsetminus\{\infty\},\Oo_X)$ for a fixed closed point $\infty\in X$. Drinfeld $A$-modules have one leg $x\colon S\to\Spec A\subset X$. Let $\C_\infty$ be the completion of an algebraic closure of $\FcnFld_\infty$. Then Drinfeld showed that the points of the Drinfeld modular varieties with values in $S=\Spec\C_\infty$ are the quotient of his $(r-1)$-dimensional upper halfspace $\Omega^r_{\FcnFld_\infty}$ by a congruence subgroup. Deligne~\cite{Deligne-Husemoller} explained that this can be rewritten as a double coset $\GL_r(\FcnFld)\backslash\Omega^r_{\FcnFld_\infty}\times \GL_r(\mathbb{A}^\infty)/H$, where $\mathbb{A}^\infty$ are the ad\`eles of $\FcnFld$ outside $\infty$ and $H\subset\GL_r(\mathbb{A}^\infty)$ is a compact open subgroup. Also the uniformization of the Drinfeld modular varieties is ``at infinity'', because $\infty$ is the place forbidden for the leg $x\colon S\to\Spec A = X\smallsetminus\{\infty\}$, which still moves ``close'' to $\infty$ on $S=\Spec\C_\infty$.

\smallskip\noindent
{\bfseries B. Uniformization away from infinity.} 
At the same time, uniformization at a place $p$ different from infinity arose in the work (4) of \v{C}erednik~\cite{Cerednik76}, who proved that certain Shimura curves of EL-type have $p$-adic uniformization by Deligne's formal model $\widehat{\Omega}^2_{\Q_p}$ of Drinfeld's upper halfplane $\Omega^2_{\Q_p}$. Drinfeld~\cite{Drinfeld-padic-covering} explained that for any $r$ the formal model $\widehat{\Omega}^r_{\Q_p}$ is a Rapoport-Zink space for an inner form of $\GL_r$, i.e.~a moduli space for $p$-divisible groups with extra structure that are isogenous to a fixed supersingular $p$-divisible group. See Boutot--Carayol~\cite{Boutot-Carayol} and Genestier~\cite{Genestier-Asterisque} for a detailed account. This was vastly generalized (5) by Rapoport and Zink~\cite{RZ} to (partial) $p$-adic uniformization of integral models of higher dimensional Shimura varieties by more general moduli spaces for $p$-divisible groups. These integral models have a morphism to $\Spec\Z_p$, which can be called the ``leg'' of the data parameterized by the integral model. This leg stays disjoint from $\infty$, which is a kind of ``fixed second leg'' for all Shimura varieties, see Remark~\ref{RemSecondLegShiVar}. In contrast, for the uniformizations (1), (2), (3) mentioned in the previous paragraph the varying leg moved towards $\infty$. 

\smallskip\noindent
{\bfseries C. Uniformization of moduli spaces of shtukas.} 
Generalizing the uniformization (3) of the Drinfeld modular varieties, Stuhler~\cite{Stuhler} proved (6) uniformization at $\infty$ of his moduli spaces of Hilbert-Blumenthal shtukas. In \cite{Blum-Stuhler}, Blum and Stuhler reinterpreted and reproved the uniformization (3) in terms of (7) Drinfeld's ``elliptic sheaves'' \cite{Drinfeld-commutative-subrings}. The latter was generalized by the first author in \cite{HartlAbSh} where (8) partial uniformization at $\infty$ of moduli stacks of ``abelian $\tau$-sheaves'' was proven. Laumon, Rapoport and Stuhler mention (9) the uniformization at $\infty$ of their moduli spaces of $\mathscr{D}$-elliptic sheaves in \cite[p.~493]{Stuhler} and \cite[(14.19)]{Laumon-Rapoport-Stuhler}, but do not prove it. Uniformization (10) for these spaces at a place $v$ different from $\infty$, i.e. for the two legs $v$ and $\infty$ staying disjoint, was proven by Hausberger~\cite{Hausberger}. When all the legs stay \emph{disjoint}, the uniformization of moduli spaces of $\gpSch$-shtukas with $n$ legs was proven in full generality (11) for smooth affine $\gpSch$ with connected fibers and reductive generic fiber by Arasteh Rad and the first author~\cite{AH_Unif}. In all these cases (7), (8), (9), (10), (11) the uniformizing spaces are Rapoport-Zink spaces for local shtukas as above. For \emph{colliding legs} uniformization of shtuka stacks (3), (6), (7), (8) is only known in the very special case for $\gpSch=\GL_r$ or $\Res_{X'|X}\SL_r$ and two legs, one fixed at $\infty$ with basic bound and the other leg moving into $\infty$. The condition on the legs is analogous to (1), (2), see Remark~\ref{RemSecondLegShiVar}. In (6), (7), (8), (9) the boundedness condition at $\infty$ on the shtukas was imposed by using chains. We give a detailed explanation of this in Section~\ref{subsec:Chains}.

\begin{remark} \label{RemSecondLegShiVar}
We want to explain, why we think that Shimura varieties have two legs, i.e.~a hidden leg at infinity in addition to the obvious leg, which is the structure morphism of the Shimura variety over $\Spec\Z$. For simplicity, we restrict to the case of Shimura varieties of PEL-type parameterizing abelian varieties with extra structures. Consider an abelian variety $\mcA$ over a finite field $\F_q$ of characteristic $p$. The (varying) leg of $\mcA$ is the morphism $\Spec \F_q\to \Spec\Z$ given by the natural homomorphism $\Z\to \Z/(p)\to \F_q$. This leg and the hidden leg at infinity of $\mcA$ can be seen by looking at the absolute values of the $q$-Frobenius endomorphism $\pi\colon \mcA\to\mcA$ of $\mcA$. Since the endomorphism algebra $\End^\circ(\mcA)$ of $\mcA$ is a finite dimensional $\Q$-algebra, $\pi$ is the root of its minimal polynomial $m_\pi\in\Q[T]$. Fix an absolute value $|\,.\,|$ on an algebraic closure $\Q^\alg$ of $\Q$ and a root $\alpha\in\Q^\alg$ of $m_\pi$. If $|\,.\,|$ extends an $\ell$-adic absolute value on $\Q$ for $\ell\ne p$, then $|\alpha|=1$. On the other hand, if $\ell=p$ then $|\alpha|$ can be different from $1$ depending on the slopes of (the $p$-divisible group of) $\mcA$. Responsible for both cases $\ell\ne p$ and $\ell=p$ is the leg at $p$ which implies that at $\ell\ne p$ all slopes are zero. Finally, if $|\,.\,|$ is obtained from the archimedean absolute value on $\C$ by an inclusion $\FcnFld^\alg\hookrightarrow\C$, we have $|\alpha|=q^{1/2} \ne 1$. This hints at the presence of another leg at infinity, where all slopes of the Frobenius endomorphism $\pi$ are equal, that is $\pi$ and $\mcA$ and its $p$-divisible group could be called ``basic''.    
\end{remark}

\bigskip

\noindent{\it Acknowledgments.} The authors would like to thank Eva Viehmann for helpful conversations. U.H.~acknowledges support of the DFG (German Research Foundation) in form of Project-ID 427320536 -- SFB 1442, and Germany's Excellence Strategy EXC 2044--390685587 ``Mathematics M\"unster: Dynamics--Geometry--Structure''. Y.X.~was supported by the National Science Foundation under Award No.~2202677.

\section{Preliminaries}
\subsection{Notations}\label{subsec-notations}
Let $\F_q$ denote a finite field with $q$ elements. Let $X$ denote a smooth, projective, geometrically connected curve over $\F_q$. Let $\FcnFld:=\F_q(X)$ be its function field. 
Let $\FcnFld^\alg$ and $\FcnFld^{\sep}$ denote an algebraic and separable closure of $\FcnFld$, respectively. For an $\F_q$-scheme $S$ and an open or closed subscheme $U\subset X$, denote $U_S:=U\times_{\F_q}S$. For a morphism $x\colon S\to X$ of $\F_q$-schemes, we denote by $\Gamma_x\subset X_S$ the graph of $x$. Closed points of $X$ are also called \emph{places} of $\FcnFld$ or of $X$.

For a closed point $v\in X$, we denote by $\F_v$ its residue field, $\Oo_v:=\widehat{\Oo}_{X,v}$ its complete local ring, and $\FcnFld_v=\Frac(\Oo_v)$ the fraction field of $\Oo_v$. 
Let $\overline{\F}_v$ be a separable closure of $\F_v$. 
Let $\Breve{\Oo}_v$ and $\Breve{\FcnFld}_v$ denote the completions of maximal unramified extensions of $\Oo_v$ and $\FcnFld_v$, respectively.~Upon fixing a uniformizer $z_v$ at $v$, one has canonical identifications $\Oo_v=\F_v\dbl z_v \dbr, \FcnFld_v=\F_v\dpl z_v \dpr, \Breve{\Oo}_v=\overline{\F}_v\dbl z_v \dbr$ and $\Breve{\FcnFld}_v=\overline{\F}_v\dpl z_v \dpr$. 
Let $\Nilp_{\Oo_v}$ (resp.~$\Nilp_{\Breve{\Oo}_v}$) denote the category of all $\Oo_v$-schemes (resp.~$\Breve{\Oo}_v$-schemes) on which $z_v$ is locally nilpotent (in the structure sheaf). 

In some parts of this article we will consider a point $\infty\in X(\F_q)$, which is assumed to be $\F_q$-rational for simplicity. Then $\F_\infty=\F_q$, but we will still use the notation $\F_\infty$ to emphasize that it comes with the morphism $\Spec\F_\infty\isoto\{\infty\}\subset X$.

Let $\gpSch$ be a smooth affine group scheme over $X$ with connected fibers and reductive generic fiber $\gengpSch:=\gpSch\times_X \Spec \FcnFld$. Set $\gpSch_v:=\gpSch\times_X \Spec\Oo_v$ and $\gengpSch_v:=\gpSch\times_X \Spec \FcnFld_v$.  
As usual, for elements $g,h\in\gpSch(S)$ for some $S\to X$ we write $\Int_h\colon g\mapsto h\,g\,h^{-1}$ for the conjugation action (``interior automorphism''). 
For a sheaf $\mathscr{H}$ of groups (in the \fppf-topology) on a scheme $Y$, an \emph{$\mathscr{H}$-bundle} 
(also called a \emph{right} \emph{$\mathscr{H}$-torsor}) on $Y$ is a sheaf $\mcE$ for the \fppf-topology on $Y$, together with a right action of the sheaf $\mathscr{H}$ such that $\mcE$ is isomorphic to $\mathscr{H}$ on an \fppf-covering of $Y$. Here $\mathscr{H}$ is viewed as an $\mathscr{H}$-torsor via right multiplication. 

We denote by $\tau:=\Frob_{q,S}$ the absolute $q$-Frobenius of an $\F_q$-scheme $S$, which is the identity on the topological space, and the $q$-power map on the structure sheaf. For a place $v\in X$ we let $q_v:=\#\F_v=q^{[\F_v:\F_q]}$ and $\hat{\tau}_v:=\tau^{[\F_v:\F_q]}=\Frob_{q_v,S}$. For data defined over $S$ (e.g. $\gpSch$-bundles $\mcE$ on $X_S$), we denote the pullback under $\tau$ by a left superscript $\tau$ (e.g. ${}^{\tau\!}\mcE$). We use a similar notation with $\hat{\tau}_v$ or more generally with $\tau^n$ for $n\in \N$.
For a linear algebraic group $M$ over $\FcnFld_v$, the Frobenius 
\begin{equation}\label{EqTau_G}
\tau_{M}\colon L_vM(S)\to L_vM(S),
\end{equation}
for an $\F_v$-scheme $S$, is defined by sending $g\colon S\to L_vM$ to $\tau_{M}(g):=g\circ\Frob_{q_v,S}$. In particular, this applies to $M=\gengpSch_v$.

\subsection{Loop groups} \label{subsec:LoopGp}

We will recall the definition of the (positive) loop groups $L_\Delta \gpSch$ and $L^+_\Delta \gpSch$ in the general setting for an effective relative Cartier divisor $\Delta\subset X_R$ over $\Spec R$. We then modify the notation in the important special cases discussed in Example~\ref{ExDivisors}. 

Let $S=\Spec R$ be affine and write $X_R:=X_S$. Let $\Delta\subset X_R$ be an effective relative Cartier divisor over $S$, i.e.~$\Delta$ is an effective Cartier divisor on $X_R$ and is finite flat over $S$. In particular, $\Delta$ is an affine scheme. Its ideal sheaf $\mathscr{I}_\Delta\subseteq\Oo_{X_R}$ is invertible. Thus Zariski-locally on $X_R$, the sheaf $\mathscr{I}_\Delta=z_\Delta\cdot \Oo_{X_R}$ is principal, for some $z_\Delta\in \Oo_{X_R}$. In particular, $\Delta=\Spec\Oo_{X_R}/\mathscr{I}_\Delta$ is locally of the form $\Spec \Oo_{X_R}/(z_\Delta)$. 

Let $\widehat{\Delta}$ be the formal completion of $X_R$ at $\Delta$. It is an affine formal scheme of the form $\Spf \widehat{\Oo}_{X_R,\Delta}$, where $\widehat{\Oo}_{X_R,\Delta}:=\underset{n}{\varprojlim}\,\Oo_{X_R}/\mathscr{I}_\Delta^n$. Looking at an open neighborhood where $\mathscr{I}_\Delta=z_\Delta\cdot \Oo_{X_R}$ is principal, we see that Zariski-locally on $\Spec R$, the formal scheme $\widehat{\Delta}$ is of the form $\Spf R\dbl z_\Delta \dbr$. 

\begin{defn}\label{DefLoopGpAtDelta}
Let $\Delta\subset X_R$ be an effective relative Cartier divisor over $S=\Spec R$.
\begin{enumerate}
\item 
The \emph{positive loop group of $\gpSch$ at $\Delta$} is the $\fpqc$-sheaf of groups $L^+_\Delta\gpSch$ over $\Spec R$ whose $R'$-points, for an $R$-algebra $R'$ are given by
\begin{align}
    \begin{split}
        L^+_\Delta\gpSch(R')&:=\gpSch(\widehat{\Oo}_{X_{R'},\Delta'})=\Hom_X(\Spf \widehat{\Oo}_{X_{R'},\Delta'},\gpSch)=\Hom_{\Oo_X}(\Oo_{\gpSch},\widehat{\Oo}_{X_{R'},\Delta'})\\
        &=\underset{n}{\varprojlim} \Hom_{\Oo_X}(\Oo_{\gpSch},\Oo_{X_{R'}}/\mathscr{I}_{\Delta'}^n)=\underset{n}{\varprojlim}\,\gpSch(\Spec\Oo_{X_{R'}}/\mathscr{I}_{\Delta'}^n),
    \end{split}
\end{align}
where $\Delta'\subset X_{R'}$ denotes the pullback of $\Delta$ to $X_{R'}$.
\item 
The \emph{loop group of $\gpSch$ at $\Delta$} is the $\fpqc$-sheaf of groups $L_\Delta\gpSch$ over $\Spec R$ whose $R'$-points, for an $R$-algebra $R'$, are given by 
\begin{equation}\label{Defn-LDeltaG}
L_\Delta\gpSch(R'):=\gpSch(\widehat{\Oo}_{X_{R'},\Delta'}[z_\Delta^{-1}]).
\end{equation}
\end{enumerate}
Zariski-locally on $R$ the group $L_\Delta^+\gpSch(R')$ is of the form $\gpSch(R'\dbl z_\Delta \dbr)$ and the group $L_\Delta\gpSch(R')$ is of the form $\gpSch(R'\dpl z_\Delta \dpr)$ for $R'\dpl z_\Delta \dpr:=R'\dbl z_\Delta \dbr[z_\Delta^{-1}]$. Note that \eqref{Defn-LDeltaG} is independent of the choice of $z_\Delta$, because any other $z_\Delta$ is of the form $\widetilde{z}_\Delta=u\cdot z_\Delta$ for some $u\in \Oo_{X_R}^{\times}\subseteq\widehat{\Oo}_{X_R,\Delta}^\times$. 

Thus, the functor $L_\Delta^+\gpSch$ is representable by an infinite-dimensional affine group scheme over $R$. Moreover, the functor $L_\Delta\gpSch$ is representable by an ind-affine ind-scheme of ind-finite type over $R$ by \cite[Lemma~2.11]{Richarz16}. 
\end{defn}

Let $[\Spec R/L_\Delta^+\gpSch]$ (resp.~$[\Spec R/L_\Delta\gpSch]$) denote the classifying space of $L_\Delta^+\gpSch$-bundles (resp.~$L_\Delta\gpSch$-bundles). It is a stack fibered in groupoids over the category of $R$-schemes $S'$ whose category $[\Spec R/L_\Delta^+\gpSch](S')$ (resp.~$[\Spec R/L_\Delta\gpSch](S')$) consists of all $L_\Delta^+\gpSch$-bundles (resp.~$L_\Delta\gpSch$-bundles) on $S'$. The inclusion of sheaves $L_\Delta^+\gpSch\subset L_\Delta\gpSch$ gives rise to the natural 1-morphism 
\begin{equation}\label{EqLoopTorsorDelta}
L_\Delta\colon[\Spec R/L_\Delta^+\gpSch]\longrightarrow [\Spec R/L_\Delta\gpSch],\quad \mcL \longmapsto L_\Delta\mcL.
\end{equation}

\begin{defn}\label{DefAffineFlagVarAtx}
The (\emph{local}) \emph{affine flag variety} of $\gpSch$ at a divisor $\Delta\subset X_R$ is the \fpqc-sheaf $\Fl_{\gpSch,\Delta}:=L_\Delta \gpSch/L^+_\Delta \gpSch$ on $\Spec R$.
\end{defn}

\begin{lem}\label{LemmaAffineFlagVar}
The affine flag variety $\Fl_{\gpSch,\Delta}$ represents the functor on $R$-schemes that sends an $R$-scheme $S$ to the set of isomorphism classes of pairs $(\mcL,\hat{\delta})$, where $\mcL$ is an $L^+_\Delta \gpSch$-bundle over $S$ and $\hat{\delta}\colon L_\Delta \mcL\isoto (L_\Delta \gpSch)_S$ is an isomorphism of $L_\Delta \gpSch$-bundles over $S$.
\end{lem}

\begin{proof}
This was proven in \cite[Lemma~2.12]{Richarz16} (or \cite{Pappas-Rapoport-twisted-loop-groups} in the special case where $\Delta= \{v\}\times_{\F_q} S \subset X_S$ for a closed point $v$ in $X$). 
\end{proof}

\begin{example}\label{ExDivisors}
(a) For any $x\in X(R)$, the graph $\Delta:=\Gamma_x\subset X_R$ of $x$ is an effective relative Cartier divisor over $\Spec R$ by \cite[\S\,8.2, Lemma~6]{Neron-models-book}. In this case we write $\widehat{\Gamma}_x:=\widehat{\Delta}$ for the formal completion of $X_R$ along $\Gamma_x$. We also write $L^+_x\gpSch:=L^+_\Delta\gpSch$ and $L_x\gpSch:=L_\Delta\gpSch$ for the (positive) loop group of $\gpSch$ at $x$, and $\Fl_{\gpSch,x}:=\Fl_{\gpSch,\Delta}$ for the affine flag variety, and $L_x$ for the functor from \eqref{EqLoopTorsorDelta}.

\medskip\noindent
(b) As a special case of (a) consider a closed point $v\in X$ and view it as a point $x:=v\in X(\F_v)$ for $R=\F_v$. Let $z_v$ be a uniformizing parameter at $v$. In this case we write $\Fl_{\gpSch,v}:=\Fl_{\gpSch,x}$ for the affine flag variety, $L_v$ for the functor from \eqref{EqLoopTorsorDelta}, and $L^+_v\gpSch:=L^+_x\gpSch$ and $L_v\gpSch:=L_x\gpSch$ for the (positive) loop group of $\gpSch$ at $v$. They are $\fpqc$-sheaves of groups on $\Spec \F_v$. Their $R'$-valued points for an $\F_v$-algebra $R'$ are
\begin{equation*}
L_v^+\gpSch(R')=\gpSch(R'\dbl z_v\dbr)=\gpSch_v(R'\dbl z_v\dbr) \qquad \text{and} \qquad L_v\gpSch(R')=\gpSch(R'\dpl z_v\dpr)=\gpSch_v(R'\dpl z_v\dpr).
\end{equation*}
This definition extends to arbitrary smooth affine group schemes $\mcM$ over $\Oo_v$ instead of $\gpSch_v$. The group $L_v^+\mcM$ is also called the \emph{positive loop group associated with $\mcM$}, and $L_v\mcM$ is called the \emph{loop group associated with $\mcM$}. The latter only depends on the generic fiber $\mcM\times_{\Oo_v} \FcnFld_v$. The fact that $L_v\mcM$ is represented by an ind-scheme was proven earlier in \cite[\S\,1.a]{Pappas-Rapoport-twisted-loop-groups}, or when $\mcM$ is constant in \cite[\S4.5]{Beilinson-Drinfeld}, and \cite{Ngo-Polo}, and \cite{Faltings03}.

\medskip\noindent
(c) More generally than (a) let $n\in\N_{>0}$ and let $\underline x=(x_1,\ldots,x_n)\in X^n(R)$. Then $\Delta:=\Gamma_{\underline x}:=\Gamma_{x_1}+\ldots + \Gamma_{x_n}$ is an effective relative Cartier divisor over $\Spec R$. In this case we write $\widehat{\Gamma}_{\underline x}:=\widehat{\Delta}$ for the formal completion of $X_R$ along $\Gamma_{\underline x}$. We also write $L^+_{\underline x}\gpSch:=L^+_\Delta\gpSch$ and $L_{\underline x}\gpSch:=L_\Delta\gpSch$ for the (positive) loop group of $\gpSch$ at $\underline x$, and $\Fl_{\gpSch,\underline x}:=\Fl_{\gpSch,\Delta}$ for the affine flag variety, and $L_{\underline x}$ for the functor from \eqref{EqLoopTorsorDelta}.
\end{example}

\begin{defn}\label{DefGlobalLoopGp}
Let $n\in\N_{>0}$. The \emph{global positive loop group} is defined as the $\fpqc$-sheaf on $\Spec\F_q$ whose $R$-valued points for an $\F_q$-algebra $R$ are given by
\begin{equation}\label{EqGlobalPosLoopGp}
\mcL^+_{X^n}\gpSch(R):=\{(\underline x,g) \colon \underline x\in X^n(R),g\in L_{\underline x}^+\gpSch(R)\}.
\end{equation}
The \emph{global loop group} is defined as the $\fpqc$-sheaf on $\Spec\F_q$ whose $R$-valued points for an $\F_q$-algebra $R$ are given by
\begin{equation}\label{EqGlobalLoopGp}
\mcL_{X^n}\gpSch(R):=\{(\underline x,g) \colon \underline x\in X^n(R),g\in L_{\underline x}\gpSch(R)\}.
\end{equation}
For an $n$-tuple of non-negative integers $(c_i)_i$ we also consider the \emph{truncated global positive loop group} defined as the $\fpqc$-sheaf on $\Spec\F_q$ whose $R$-valued points for an $\F_q$-algebra $R$ are given by
\begin{equation}\label{EqGlobalTruncLoopGp}
\mcL^{+,(c_i)_i}_{X^n}\gpSch(R):=\{(\underline x,g) \colon \underline x\in X^n(R),g\in \gpSch(\Delta)\},
\end{equation}
for the divisor $\Delta:=\sum_i c_i\cdot \Gamma_{x_i}\subset X_R$ considered as a scheme over $X$.

Clearly, $\mcL^+_{X^n}\gpSch$ is a subsheaf of $\mcL_{X^n}\gpSch$, and $\mcL^{+,(c_i)_i}_{X^n}\gpSch$ is a quotient of $\mcL^+_{X^n}\gpSch$. The projection onto $\underline x$ defines morphisms $\mcL^+_{X^n}\gpSch \to X^n$ and $\mcL_{X^n}\gpSch \to X^n$ and $\mcL^{+,(c_i)_i}_{X^n}\gpSch \to X^n$.
\end{defn}

\begin{lem}\label{LemmaGlobalLoopIsqc}
\begin{enumerate}
\item \label{LemmaGlobalLoopIsqc_A} The global loop group $\mcL_{X^n}\gpSch$ is representable by an ind-group scheme which is ind-affine over $X^n$.
\item \label{LemmaGlobalLoopIsqc_B} The global positive loop group $\mcL^+_{X^n}\gpSch$ is representable by a quasi-compact, reduced, infinite dimensional group scheme, which is affine and flat over $X^n$ with geometrically connected fibers. 
\item \label{LemmaGlobalLoopIsqc_C} The truncated global positive loop group $\mcL^{+,(c_i)_i}_{X^n}\gpSch$ is representable by a smooth, affine group scheme over $X^n$ of relative dimension equal to $(\sum_i c_i)\cdot \dim\gengpSch$ with geometrically connected fibers.
\end{enumerate} 
\end{lem}

\begin{proof}
The statement is local on $X^n$. Thus for \ref{LemmaGlobalLoopIsqc_A} and \ref{LemmaGlobalLoopIsqc_B} we can work on an affine open subscheme $U\subset X^n$, and assume that the divisor $\Gamma_{\underline x}\subset X_U$ is principal and the zero locus of an element $z_{\underline x}\in \Oo_{X_U}$. Then $L^+_{\underline x}\gpSch(R)=\gpSch(R\dbl z_{\underline x}\dbr)$ and $L_{\underline x}\gpSch(R)=\gpSch(R\dpl z_{\underline x}\dpr)$. After this reformulation, \ref{LemmaGlobalLoopIsqc_A} was proven by Heinloth~\cite[Proposition~2]{Heinloth} and \ref{LemmaGlobalLoopIsqc_B} was proven by Richarz~\cite[Lemma~2.11]{Richarz16}. Since $\mcL^+_{X^n}\gpSch$ is affine over the quasi-compact $X^n$, also $\mcL^+_{X^n}\gpSch$ is quasi-compact.

\medskip \noindent
\ref{LemmaGlobalLoopIsqc_C} We consider the universal situation over $X^n$ in which the section $x_i\colon X^n\to X$ is the projection onto the $i$-th component. Then the divisor $\Delta:=\sum_i c_i\cdot \Gamma_{x_i}$ is a closed subscheme in $X\times_{\F_q} X^n$. Let $\pr_1\colon\Delta \to X$ and $\pr_2\colon\Delta\to X^n$ be the projections. The morphism $\pr_2$ is finite and flat of degree $\sum_i c_i$. Then the group $\mcL^{+,(c_i)_i}_{X^n}\gpSch$ over $X^n$ is the Weil restriction $\Res_{\pr_2}(\pr_1^*\gpSch)$ of $\pr_1^*\gpSch$ under $\pr_2$. Our assertions now follow from \cite[Propositions~A.5.2 and A.5.9]{CGP}.
\end{proof}

\subsection{The stacks of $\gengpSch$-bundles} \label{subsec:Bun}

\begin{defn}\label{DefD-LevelStr}
Let $\Bun_\gpSch:=\Bun^X_\gpSch$ denote the category fiberd in groupoids over the category of $\F_q$-schemes, which assigns to an $\F_q$-scheme $S$ the category whose objects $\Bun_\gpSch(S)$ are $\gpSch$-bundles over $X_S$ and morphisms are isomorphisms of $\gpSch$-bundles. 

Let $D\subset X$ be a proper closed subscheme. A \emph{$D$-level structure} on a $\gpSch$-bundle $\mcE$ on $X_S$ is a trivialization $\psi\colon \mcE\times_{X_S}{D_S}\isoto \gpSch\times_X D_S$ along $D_S$. Let $\Bun_{\gpSch,D}$ denote the stack classifying $\gpSch$-bundles with $D$-level structures, i.e.~for an $\F_q$-scheme $S$ the objects of the category $\Bun_{\gpSch,D}(S)$ are
\begin{equation}
\Bun_{\gpSch,D}(S):=\left\lbrace (\mcE,\psi)\colon \mcE\in \Bun_\gpSch(S),\, \psi\colon \mcE\times_{X_S}{D_S}\isoto \gpSch\times_X D_S \right\rbrace,
\end{equation}
and the morphisms are those isomorphisms of $\gpSch$-bundles that preserve the $D$-level structure $\psi$.
\end{defn}

The following theorem is well known, see for example \cite[\S\,4.4]{Beh}, \cite[Proposition~1]{Heinloth} and \cite[Theorem~2.5]{AH_Unif}.

\begin{thm}\label{Bun_G}
The stack $\Bun_{\gpSch,D}$ is a smooth Artin stack, locally of finite type over $\F_q$. It admits a covering by connected open substacks (given by bounds on the Harder-Narasinham filtration) of finite type over $\F_q$. 
\end{thm}

\begin{numberedparagraph} 
Generalizing \cite[\S\,5.1]{AH_Local} we define the global-local-functor $L_\Delta$ for $\Bun_\gpSch$. Let $S=\Spec R$ be affine and let $\Delta\subset X_R$ be an effective relative Cartier divisor. Let $[X/\gpSch](X_R\smallsetminus\Delta)$ be the category of $\gpSch$-bundles $\overset{\circ}{\mcE}$ on $X_R\smallsetminus\Delta$ and let $[X/\gpSch](X_R\smallsetminus\Delta)^\ext$ be the full subcategory of $[X/\gpSch](X_R\smallsetminus\Delta)$ consisting of those $\gpSch$-bundles $\overset{\circ}{\mcE}$ over $X_R\smallsetminus\Delta$ that can be extended to some $\gpSch$-bundle $\mcE$ over the whole relative curve $X_R$. The restriction functor 
\begin{equation}
{}_\bullet\,|_{X_R\smallsetminus\Delta}\colon\Bun_\gpSch(R)\longrightarrow [X/\gpSch](X_R\smallsetminus\Delta)^\ext,\quad
\mcE\longmapsto \mcE|_{X_R\smallsetminus\Delta}
\end{equation} 
assigns to a $\gpSch$-bundle $\mcE$ over $X_R$ the $\gpSch$-bundle $\mcE|_{X_R\smallsetminus\Delta}:=\mcE\times_{X_R}(X_R\smallsetminus\Delta)$ over $X_R\smallsetminus\Delta$. 

For $\mcE\in\Bun_\gpSch(R)$, we also consider its formal completion $\mcE|_{\widehat{\Delta}} :=\mcE\times_{X_R} \widehat{\Delta}$ along $\Delta$. By \cite[Proposition~2.4]{AH_Local}, the formal completion $\mcE|_{\widehat{\Delta}}$ corresponds to an $L^+_\Delta\gpSch$-bundle over $\Spec R$ which we denote as $L^+_\Delta(\mcE)$. This gives a functor
\begin{equation}\label{EqL^+_v}
L^+_\Delta\colon \Bun_\gpSch(R)\longrightarrow [\Spec R/L^+_\Delta\gpSch](R)\,,\quad\mcE\mapsto L^+_\Delta(\mcE)\,.
\end{equation}
Moreover, we have a functor
\begin{equation}\label{EqL_v}
L_\Delta\colon [X/\gpSch](X_R\smallsetminus\Delta)^\ext \longrightarrow [\Spec R/L_\Delta\gengpSch](R)\,,\quad\overset{\circ}{\mcE}\mapsto L_\Delta(\overset{\circ}{\mcE}):=L_\Delta L^+_\Delta(\mcE)
\end{equation}
which sends the $\gpSch$-bundle $\overset{\circ}{\mcE}$ over $X_R\smallsetminus\Delta$ (equipped with some extension $\mcE$ over $X_R$) to the $L_\Delta\gengpSch$-bundle $L_\Delta(\overset{\circ}{\mcE})$ associated with $L^+_\Delta(\mcE)$ under the functor $L_\Delta$ from \eqref{EqLoopTorsorDelta}. As in \cite[\S\,5.1]{AH_Local} one can show that $L_\Delta(\overset{\circ}{\mcE})$ is independent of the choice of the extension $\mcE$. The functors from \eqref{EqL^+_v} and \eqref{EqL_v} are called the \emph{global-local-functors at $\Delta$}.

\begin{example}\label{ExGlobLocFunctor}
We continue with Example~\ref{ExDivisors} and introduce the following notation.

\medskip\noindent
(a) For a point $x\in X(R)$ we write $L^+_x$ and $L_x$ for the global-local functors from \eqref{EqL^+_v} and \eqref{EqL_v}.

\medskip\noindent
(b) When $v\in X$ is a closed point we write $L^+_v$ and $L_v$ for the global-local functors from \eqref{EqL^+_v} and \eqref{EqL_v}. Later on we will apply these functors for different groups. For clarity, we will then include the group in the subscript and write $L^+_{v,\gpSch}$ (resp.~$L_{v,\gpSch}$) instead.
\end{example}

Let $\mathrm{Tri}_\gpSch(X_R,\Delta)$ denote the category whose objects are triples $(\overset{\circ}{\mcE},\mcL,\gamma)$, where $\overset{\circ}{\mcE}\in [X/\gpSch](X_R\smallsetminus\Delta)^\ext$, and $\mcL\in [\Spec R/L^+_\Delta\gpSch](R)$, and $\gamma\colon L_\Delta(\overset{\circ}{\mcE})\isoto L_\Delta(\mcL)$ is an isomorphism of $L_\Delta\gengpSch$-bundles on $\Spec R$. We obtain a functor
\begin{align}\label{BunG-to-triples}
    \begin{split}
        \Bun_\gpSch(R)&\longrightarrow\mathrm{Tri}_\gpSch(X_R,\Delta)\\ 
        \mcE &\longmapsto \bigl(\mcE|_{X_R\smallsetminus\Delta},L^+_\Delta(\mcE),\gamma),
    \end{split}
\end{align}
where $\gamma$ is the identity morphism of the $L_\Delta\gengpSch$-bundle $L_\Delta(\mcE|_{X_R\smallsetminus\Delta}):=L_\Delta L^+_\Delta(\mcE)$. The following lemma generalizes \cite[Lemma~5.1]{AH_Local}.
\end{numberedparagraph}

\begin{lem}\label{LemmaBL}
The functor \eqref{BunG-to-triples} is an equivalence of categories $\Bun_\gpSch(R)\cong \mathrm{Tri}_\gpSch(X_R,\Delta)$.
\end{lem}

\begin{proof}
This follows from the glueing lemma of Beauville and Lazlo~\cite{Beauville-Laszlo} as in \cite[Lemma~5.1]{AH_Local}.
\end{proof}

\subsection{The Hecke stack}\label{subsec:Hecke-stack}

We recall the definition of the $\Hecke$ stack with $n$ legs from \cite[Definition~1.2]{Lafforgue12}. 
Let $n\in\N_0$, let $I=\{1,\ldots,n\}$, and let $I_{\bullet}=(I_1,\ldots,I_k)$ be an ordered partition of $I$, i.e.~$I=I_1\sqcup\ldots\sqcup I_k$. Let $D\subset X$ be a proper closed subset. 
\begin{defn}\label{DefHecke_nlegs}
The \emph{Hecke stack} $\Hecke_{\gpSch,D,X^n,I_\bullet}$ \emph{with $n$ legs and partition $I_\bullet$} is the stack over $\F_q$, whose $S$-valued points, for an $\F_q$-scheme $S$, are tuples $\bigl(\underline x,\,(\mcE^{(j)},\psi^{(j)})_{j=0\ldots k},\,(\varphi^{(j-1)})_{j=1\ldots k}\,\bigr)$ where\footnote{Here we use superscript for indexing inside the Hecke stack and subscript for indexing inside $\CBun$; see Definition~\ref{DefCBun}. Both super- and sub- scripts will be used in Section~\ref{subsec:Chains}, where we explain that our $\gpSch$-shtukas generalize the $\mathscr{D}$-elliptic sheaves from \cite{Laumon-Rapoport-Stuhler}.}
\begin{itemize}
\item $x_i \in (X\smallsetminus D)(S)$ for $i=1,\ldots, n$ are sections, called \emph{legs}, and $\underline x:=(x_i)_{i=1\ldots n}\in (X\smallsetminus D)^n(S)$
\item $(\mcE^{(j)},\psi^{(j)})$ for $j=0,\ldots,k$ are objects in $\Bun_{\gpSch,D}(S)$, and
\item the \emph{modifications} $\varphi^{(j-1)}\colon \mcE^{(j-1)}|_{{X_S}\smallsetminus\cup_{i\in I_j}\Gamma_{x_i}}\isoto \mcE^{(j)}|_{{X_S}\smallsetminus\cup_{i\in I_j}\Gamma_{x_i}}$ for $j=1,\ldots,k$ are isomorphisms preserving the $D$-level structures, i.e.~$\psi^{(j)}\circ\varphi^{(j-1)}|_{D_S}=\psi^{(j-1)}$.
\end{itemize}
Morphisms $\bigl(\underline x,\,(\mcE^{(j)},\psi^{(j)})_{j=0\ldots k},\,(\varphi^{(j-1)})_{j=1\ldots k}\,\bigr)\to \bigl(\underline x,\,(\widetilde{\mcE}^{(j)},\widetilde{\psi}^{(j)})_{j=0\ldots k},\,(\widetilde{\varphi}^{(j-1)})_{j=1\ldots k}\,\bigr)$ are tuples of isomorphisms $f^{(j)}\colon (\mcE^{(j)},\psi^{(j)})\isoto (\widetilde{\mcE}^{(j)},\widetilde{\psi}^{(j)})$ in $\Bun_{\gpSch,D}(S)$ for all $j$ which are compatible with the $\varphi^{(j-1)}$ and $\widetilde{\varphi}^{(j-1)}$.
We can visualize the above data as
\begin{equation}\label{Hecke-GDXn-visual}
\xymatrix @C+1pc {
(\mcE^{(0)},\psi^{(0)}) \ar@{-->}[r]^-{\varphi^{(0)}}_-{x_i\colon i\in I_1} & (\mcE^{(1)},\psi^{(1)}) \ar@{-->}[r]^-{\varphi^{(1)}}_-{x_i\colon i\in I_2} & \ldots \ar@{-->}[r]^-{\varphi^{(k-1)}}_-{x_i\colon i\in I_k} & (\mcE^{(k)},\psi^{(k)}) \,.
}
\end{equation}

The projection map of \eqref{Hecke-GDXn-visual} onto $(x_i)_{i=1\ldots n}$ defines a morphism
\begin{equation}
\Hecke_{\gpSch,D,X^n,I_\bullet}\to (X\smallsetminus D)^I.
\end{equation} 
When $D=\varnothing$, we will drop it (and the $\psi^{(j)}$) from the notation. For $n=1$, the set $I=\{1\}$ only has the trivial partition $I_1:=I$. Thus we drop $I_\bullet$ from the notation and simply write $\Hecke_{\gpSch,D,X}$.

\end{defn}

\begin{remark}\label{RemHeckeChangeI}
Let $\widetilde{I}_\bullet=(\widetilde{I}_1,\ldots,\widetilde{I}_{\tilde k})$ be a partition of $I$ and $I_\bullet=(I_1,\ldots,I_k)$ a coarsening of $\widetilde{I}_\bullet$ obtained by uniting certain $\widetilde{I}_{\tilde j}$ with neighboring indices. More precisely, we require that there are integers $0=\ell_0< \ell_1< \ldots <\ell_k=\tilde k$ and $I_j=\bigcup_{\ell_{j-1}<\tilde{j} \le\ell_j} \widetilde{I}_{\tilde j}$. Then there is an $X^n$-morphism 
\begin{align}\label{EqHeckeChangeI}
\Hecke_{\gpSch,D,X^n,\widetilde{I}_\bullet} & \longrightarrow \Hecke_{\gpSch,D,X^n,I_\bullet} \\
\bigl(\underline x,\,(\widetilde{\mcE}^{(\tilde j)},\widetilde{\psi}^{(\tilde j)})_{\tilde j=0\ldots \tilde k},\,(\widetilde{\varphi}^{(\tilde j-1)})_{\tilde j=1\ldots \tilde k}\,\bigr) & \longmapsto \bigl(\underline x,\,(\mcE^{(j)},\psi^{(j)})_{j=0\ldots k},\,(\varphi^{(j-1)})_{j=1\ldots k}\,\bigr) \nonumber
\end{align}
given by forgetting the $(\widetilde{\mcE}^{(\tilde{j})},\widetilde{\psi}^{(\tilde{j})})$ for $\tilde{j}\notin\{\ell_0,\ldots,\ell_k\}$, reindexing by $(\mcE^{(j)},\psi^{(j)}) :=(\widetilde{\mcE}^{(\ell_j)},\widetilde{\psi}^{(\ell_j)})$, and composing the corresponding $\widetilde{\varphi}^{(\tilde{j})}$ to $\varphi^{(j)}:=\widetilde{\varphi}^{(-1+\ell_{j+1})}\circ\ldots\circ\widetilde{\varphi}^{(\ell_j)}$.
\end{remark}

\begin{prop}\label{PropHeckeArtin}
$\Hecke_{\gpSch,D,X^n,I_\bullet}$ is an ind-Artin stack locally of ind-finite type over $X$. The morphism $\Hecke_{\gpSch,D,X^n,I_\bullet}\to X^n\times_{\F_q} \Bun_{\gpSch,D}$ sending $\bigl(\underline x,\,(\mcE^{(j)},\psi^{(j)})_{j=0\ldots k},\,(\varphi^{(j-1)})_{j=1\ldots k}\,\bigr)$ to $\bigl(\underline x,(\mcE^{(k)},\psi^{(k)})\bigr)$ is relatively representable by a morphism of ind-schemes which is of ind-finite type and ind-quasi-projective. It is even ind-projective if and only if the group scheme $\gpSch$ is parahoric as in \cite[Appendix, Definition~1]{Pappas-Rapoport-twisted-loop-groups}.
\end{prop}
\begin{proof}
This was proven in \cite[Propositions~3.9 and 3.12]{AH_Unif}, where instead of the $\varphi^{(j)}$, their inverses are considered (and called $\tau_{k-j}$, while also the numbering of the $\gpSch$-bundles is reversed). 
\end{proof}

\begin{prop} \label{PropHeckeChangeI}
Let $U:=\{(x_i)_i\in X^n \colon x_i\neq x_j\text{ for }i\neq j\}\subseteq X^n$ be the complement of all diagonals. Let $I_\bullet$ and $\widetilde{I}_\bullet$ be partitions of $I$, such that $I_\bullet$ is a coarsening of $\widetilde{I}_\bullet$ as in Remark~\ref{RemHeckeChangeI}. Then over the open set $U\subset X^n$, the morphism $\Hecke_{\gpSch,D,X^n,\widetilde{I}_\bullet}\times_{X^n} U\isoto \Hecke_{\gpSch,D,X^n,I_\bullet}\times_{X^n} U$ from \eqref{EqHeckeChangeI} is an isomorphism.
\end{prop}

\begin{proof}
We define the divisors $\Delta_{\tilde j}:=\sum_{i\in \widetilde{I}_{\tilde j}}\Gamma_{x_i}$, which are pairwise disjoint over $U$. Then the inverse of the morphism \eqref{EqHeckeChangeI} is given by reconstructing the forgotten $\widetilde{\mcE}^{(\tilde j)}$ with $\ell_j<\tilde j<\ell_{j+1}$ from $(\widetilde{\mcE}^{(\ell_j)},\widetilde{\psi}^{(\ell_j)}):=(\mcE^{(j)}, \psi^{(j)})$. This is done via Lemma~\ref{LemmaBL} by succesively glueing $\mcE^{(\tilde j-1)}$ with $L^+_{\Delta_{\tilde j}}(\mcE^{(j+1)})$ via the isomorphism $L_{\Delta_{\tilde j}}(\varphi^{(j)})$ to obtain $\widetilde{\mcE}^{(\tilde j)}$ and $\widetilde{\varphi}^{(\tilde j-1)}$. Finally we set $\widetilde{\psi}^{(\tilde j)}:=\widetilde{\psi}^{(\tilde j-1)}\circ(\widetilde{\varphi}^{(\tilde j-1)})^{-1}|_{D_S}$.
\end{proof}

In the rest of this subsection, we fix a point $\infty\in X(\F_q)$. We are particularly interested in the Hecke stack with two legs, one fixed at $\infty$. For simplicity, we slightly change the notation and use $'$ and $''$ instead of numbers for upper-scripts.

\begin{defn}\label{DefHecke2legs}
Let $I=\{1,2\}$ and let $D\subset X\smallsetminus\{\infty\}$ be a proper closed subset. 

\smallskip\noindent
(a) For the finest partition $I_\bullet=(\{1\},\{2\})$, we define the \emph{Hecke stack with two legs, one fixed at $\infty$} as
\[
\Hecke_{\gpSch,D,X\times\infty} :=\Hecke_{\gpSch,D,X^2,(\{1\},\{2\})}\times_{X^2}(X\times_{\F_q}\Spec\F_\infty).
\]
Its $S$-valued points, for $S$ an $\F_q$-scheme, are tuples 
\[
\bigl(x,\,(\mcE,\psi),(\mcE',\psi'),(\mcE'',\psi''),\,\varphi,\varphi'\,\bigr) \\
\]
where
\begin{itemize}
\item $x \in (X\smallsetminus D)(S)$ is a section, called a \emph{leg}, 
\item $(\mcE,\psi),(\mcE',\psi'),(\mcE'',\psi'')$ are objects in $\Bun_{\gpSch,D}(S)$, and
\item the \emph{modifications} $\varphi\colon \mcE|_{{X_S}\smallsetminus\Gamma_{x}}\isoto\mcE'|_{{X_S}\smallsetminus\Gamma_{x}}$ and $\varphi'\colon \mcE'|_{(X\smallsetminus\{\infty\})_S}\isoto \mcE''|_{(X\smallsetminus\{\infty\})_S}$ are isomorphisms preserving the $D$-level structures, i.e.~$\psi'\circ\varphi|_{D_S}=\psi$ and $\psi''\circ\varphi'|_{D_S}=\psi'$.
\end{itemize}
We can visualize the above data as
\begin{equation}
\xymatrix @C+1pc {
(\mcE,\psi) \ar@{-->}[r]^{\varphi}_{x} & (\mcE',\psi') \ar@{-->}[r]^{\varphi'}_\infty & (\mcE'',\psi'') \,.
}
\end{equation}

\medskip\noindent
(b) We also define the stack for the coarsest partition $I_1:=I=\{1,2\}$
\[
\PHecke_{\gpSch,D,X\times\infty} :=\Hecke_{\gpSch,D,X^2,(\{1,2\})}\times_{X^2}(X\times_{\F_q}\Spec\F_\infty).
\]
classifying data $\bigl(x,\,(\mcE,\psi),(\mcE'',\psi''),\,\varphi\colon(\mcE,\psi)|_{X_S\smallsetminus(\Gamma_x\cup\infty)} \isoto (\mcE'',\psi'')|_{X_S\smallsetminus(\Gamma_x\cup\infty)} \,\bigr)$, visualized as
\begin{equation}
\xymatrix @C+1pc {
(\mcE,\psi) \ar@{-->}[r]^{\varphi}_{x,\infty} & (\mcE'',\psi'') \,.
}
\end{equation}

When $D=\varnothing$, we will drop it and the $\psi,\psi',\psi''$ from the notation. The projection map onto the leg $x$ defines morphisms $\Hecke_{\gpSch,D,X\times\infty}\to \PHecke_{\gpSch,D,X\times\infty}\to X$.
\end{defn}

``Creating'' an unnecessary additional leg at $\infty$ defines a morphism 
\begin{align}\label{EqHeckeCreateInfty}
\Hecke_{\gpSch,D,X} & \longrightarrow \PHecke_{\gpSch,D,X\times\infty} \\
\bigl(x,\,(\mcE,\psi)\xdashrightarrow[x]{\varphi}(\mcE',\psi')\bigr) & \longmapsto \bigl(x,\,(\mcE,\psi)\xdashrightarrow[x,\infty]{\varphi}(\mcE',\psi')\bigr) \nonumber
\end{align}
where we view the isomorphism $\varphi\colon (\mcE,\psi)|_{X_S\smallsetminus\Gamma_x} \isoto (\mcE',\psi')|_{X_S\smallsetminus\Gamma_x}$ outside $\Gamma_x$ as an isomorphism $\varphi\colon (\mcE,\psi)|_{X_S\smallsetminus(\Gamma_x\cup\infty)} \isoto (\mcE',\psi')|_{X_S\smallsetminus(\Gamma_x\cup\infty)}$ outside $\Gamma_x\cup (\infty\times_{\F_q}S)$. This is a morphism over $X\times_{\F_q} \Bun_{\gpSch,D}$ as in Proposition~\ref{PropHeckeArtin}.

\begin{lem}\label{LemmaHeckeCreateInfty}
The morphism \eqref{EqHeckeCreateInfty} is a monomorphism, and hence schematic.
\end{lem}

\begin{proof}
To prove that it is a monomorphism we must show that it is fully faithful as a functor. Let $(x,(\mcE,\psi),(\mcE',\psi'),\varphi)$ and $(x,(\widetilde{\mcE},\widetilde{\psi}),(\widetilde{\mcE}',\widetilde{\psi}'),\widetilde{\varphi})$ be two objects of $\Hecke_{\gpSch,D,X}(S)$. Then the isomorphism between these two objects in $\Hecke_{\gpSch,D,X}(S)$ and also in $\PHecke_{\gpSch,D,X\times\infty}(S)$ consist of isomorphisms $f\colon (\mcE,\psi)\isoto(\widetilde{\mcE},\widetilde{\psi})$ and $f'\colon (\mcE',\psi')\isoto(\widetilde{\mcE}',\widetilde{\psi}')$ which are compatible with $\varphi$ and $\widetilde{\varphi}$. This proves that \eqref{EqHeckeCreateInfty} is a monomorphism. It is schematic by \cite[Corollaire~8.1.3 and Th\'eor\`eme~A.2]{Laumon-Moret-Bailly}.
\end{proof}

\subsection{The Beilinson-Drinfeld Grassmannian}\label{subsec:Gr}
Let $n\in\N_0$, let $I=\{1,\ldots,n\}$, and let $I_\bullet=(I_1,\ldots,I_k)$ be a partition of $I$.

\begin{defn}\label{DefGr_nlegs}
The \emph{Beilinson-Drinfeld Grassmannian} $\Gr_{\gpSch,X^n,I_\bullet}$ is the $\fpqc$-sheaf of sets on $\F_q$, whose $S$-points, for an $\F_q$-scheme $S$, are tuples 
\[
\bigl(\underline x,\,(\mcE^{(j)})_{j=0\ldots k},\,(\varphi^{(j-1)})_{j=1\ldots k}\,\bigr)\in\Hecke_{\gpSch,\varnothing,X^n,I_\bullet}(S)
\]
as in Definition~\ref{DefHecke_nlegs} together with:
\begin{itemize}
\item a trivialization $\epsilon\colon \mcE^{(k)}\isoto \gpSch\times_X X_S$.
\end{itemize}
For $n=1$, the set $I=\{1\}$ only has the trivial partition $I_1:=I$. Thus we drop $I_\bullet$ from the notation and simply write $\Gr_{\gpSch,X}$. Thus $\Gr_{\gpSch,X^n,I_\bullet}$ is the fiber product
\begin{equation} \label{EqGrFiberProd}
\xymatrix @C+2pc {
\Gr_{\gpSch,X^n,I_\bullet} \ar[r] \ar[d] & \Hecke_{\gpSch,\varnothing,X^n,I_\bullet} \ar[d]^{\mcE^{(k)}} \\
\Spec\F_q \ar[r]^{\gpSch} & \Bun_\gpSch
}
\end{equation}
where the vertical map on the right was defined in Proposition~\ref{PropHeckeArtin} and the horizontal map on the bottom comes from the trivial $\gpSch$-bundle $\gpSch\in\Bun_\gpSch(\F_q)$.
\end{defn}

\begin{prop}\label{PropGrInd-Scheme}
$\Gr_{\gpSch,X^n,I_\bullet}$ is an ind-scheme of ind-finite type, which is ind-quasi-projective over $X^n$. It is even ind-projective if and only if the group scheme $\gpSch$ is parahoric.
\end{prop}

\begin{proof}
This was proven by Richarz~\cite[Lemma~2.8(ii)]{Richarz16} in case $n=1$. In general it follows from Proposition~\ref{PropHeckeArtin} by the base change diagram \eqref{EqGrFiberProd}. 
\end{proof}

We will need the following alternative description from \cite[(0.10)]{Lafforgue12} of the Beilinson-Drinfeld Grassmannian in terms of modifications of $L^+_{\underline x}\gpSch$-bundles. 

\begin{prop}\label{PropGrGlobLoc}
Via the functor $L_{\underline x}$ from Example~\ref{ExDivisors}(c), the Beilinson-Drinfeld Grassmannian $\Gr_{\gpSch,X^n,I_\bullet}$ is isomorphic to the $\fpqc$-sheaf on $\Spec\F_q$, whose $S$-valued points are tuples
\begin{equation}\label{EqPropGrGlobLoc}
\xymatrix @C+2pc {
\Bigl(\underline x,\,\mcL^{(0)} \ar@{-->}[r]^-{\widehat{\varphi}^{(0)}}_-{x_i\colon i\in I_1} & \mcL^{(1)} \ar@{-->}[r]^-{\widehat{\varphi}^{(1)}}_-{x_i\colon i\in I_2} & \ldots \ar@{-->}[r]^-{\widehat{\varphi}^{(k-1)}}_-{x_i\colon i\in I_k} & \mcL^{(k)} \ar[r]^-{\hat{\epsilon}}_-\sim & (L^+_{\underline x}\gpSch)_S\Bigr)},
\end{equation}
where the $\mcL^{(j)}$ are $L^+_{\underline x}\gpSch$-bundles on $S$, the $\widehat{\varphi}^{(j-1)}$ are isomorphism of $L_{\underline x}\gpSch$-bundles that above $\Gamma_{\underline x}\smallsetminus \bigcup_{i\in I_j}\Gamma_{x_i}$ are even isomorphisms of $L^+_{\underline x}\gpSch$-bundles, and $\hat{\epsilon}$ is an isomorphism of $L^+_{\underline x}\gpSch$-bundles.
\end{prop}

\begin{remark}\label{RemGrGlobLoc}
The condition on $\widehat{\varphi}^{(j-1)}$ means the following. Let $\Spec R\subset S$ be an affine open subset. For every $j=1,\ldots,k$ let $\mathscr{I}_j$ be the ideal sheaf of the effective relative Cartier divisor $\sum_{i\in I_j}\Gamma_{x_i}$. Then locally on $X_S$ the sheaf $\mathscr{I}_j$ is generated by an element $z_j\in\Oo_{X_S}$. For any $j$, we can define the partial loop group $L^j_{\underline x}\gpSch$ Zariski locally on $\Spec R$ as the $\fpqc$-sheaf of groups on $\Spec R$ given by $L^j_{\underline x}\gpSch(R'):=\gpSch(\widehat{\Oo}_{X_{R'},\Delta'}[z_j^{-1}])$, where $\Delta'$ denotes the pullback of $\Delta$ to $X_{R'}$. As in Definition~\ref{DefLoopGpAtDelta}, this is independent of the chosen generator $z_j$, and hence glues to an $\fpqc$-sheaf of groups on $S$. The inclusion of sheaves $L^+_{\underline x}\gpSch \subset L^j_{\underline x}\gpSch$ induces a functor $L_j\colon [S / L^+_{\underline x}\gpSch] \to [S / L^j_{\underline x}\gpSch]$. By the condition on $\widehat{\varphi}^{(j-1)}$ we mean that $\widehat{\varphi}^{(j-1)}$ is an isomorphism of the induced $L^j_{\underline x}\gpSch$-bundles $\widehat{\varphi}^{(j-1)}\colon L_j(\mcL^{(j-1)})\isoto L_j(\mcL^{(j)})$.
\end{remark}

\begin{proof}[Proof of Proposition~\ref{PropGrGlobLoc}.]
The functor $L^+_{\underline x}$ sends an object $\bigl(\underline x,\,(\mcE^{(j)})_{j=0\ldots k},\,(\varphi^{(j-1)})_{j=1\ldots k},\,\epsilon\,\bigr)\in\Gr_{\gpSch,X^n,I_\bullet}(S)$ to
\[
\xymatrix @C+2pc {
\Bigl(\underline x,\, L^+_{\underline x}(\mcE^{(0)}) \ar@{-->}[r]^-{L_{\underline x}(\varphi^{(0)})}_-{x_i\colon i\in I_1} & L^+_{\underline x}(\mcE^{(1)}) \ar@{-->}[r]^-{L_{\underline x}(\varphi^{(1)})}_-{x_i\colon i\in I_2} & \ldots \ar@{-->}[r]^-{L_{\underline x}(\varphi^{(k-1)})}_-{x_i\colon i\in I_k} & L_{\underline x}(\mcE^{(k)}) \ar[r]^-{L^+_{\underline x}(\epsilon)}_-\sim & L^+_{\underline x}\gpSch\Bigr)}
\]
That it is an equivalence follows from Lemma~\ref{LemmaBL}, by gluing $\gpSch\times_X (X_S\smallsetminus \Gamma_{\underline x})$ with $\mcL^{(j)}$ via the isomorphism $\hat{\epsilon} \circ \widehat{\varphi}^{(k-1)}\circ \ldots \circ \widehat{\varphi}^{(j)}$ to obtain $\mcE^{(j)}$ and $\varphi^{(j)}$. 
\end{proof}

\begin{defn}\label{DefLoopActionOnGr}
The global loop group $\mcL^+_{X^n}\gpSch$ from Definition~\ref{DefGlobalLoopGp} acts on $\Gr_{\gpSch,X^n,I_\bullet}$ by acting on $\hat{\epsilon}$. More precisely, the action is given by a morphism
\[
\mcL^+_{X^n}\gpSch \times_{X^n} \Gr_{\gpSch,X^n,I_\bullet} \longrightarrow \Gr_{\gpSch,X^n,I_\bullet}
\]
such that $(\underline x,g)\in \mcL^+_{X^n}\gpSch(S)$ sends an object \eqref{EqPropGrGlobLoc} of $\Gr_{\gpSch,X^n,I_\bullet}(S)$ to the same object but with $\hat{\epsilon}$ replaced by $g\cdot\hat{\epsilon}$.
\end{defn}

Let $n=1$ and recall from Definition~\ref{DefAffineFlagVarAtx} and Example~\ref{ExDivisors}(a) the local affine flag variety $\Fl_{\gpSch,x}$. For the global loop groups from Definition~\ref{DefGlobalLoopGp}, the set of $S$-valued points of the $\fppf$-quotient $\mcL_X\gpSch/\mcL^+_X\gpSch$ equals $\{\,(x,\overline{g})\colon x\in X(S),\, \overline{g}\in \Fl_{\gpSch,x}(S)\,\}$. Via the interpretation of $\Fl_{\gpSch,x}$ from Lemma~\ref{LemmaAffineFlagVar}, we consider the morphism
\begin{align}\label{GrGX-map}
\begin{split}
\Gr_{\gpSch,X} \qquad & \longto \qquad \mcL_X\gpSch/\mcL^+_X\gpSch\,,\\ 
\bigl(x,\,\mcE^{(0)}\xdashrightarrow[x]{\varphi^{(0)}}\mcE^{(1)}\xrightarrow[\sim]{\enspace\epsilon\enspace} \gpSch\times_X X_S\bigr) & \longmapsto \bigl(\,x, L^+_x(\mcE^{(0)}),\, L_x(\epsilon\circ\varphi^{(0)})\,\bigr)\,.
\end{split}
\end{align}

\begin{cor}\label{CorGrGlobLoc}
The map in \eqref{GrGX-map} is an isomorphism $\Gr_{\gpSch,X} \cong \mcL_X\gpSch/\mcL^+_X\gpSch$. In particular, for a fixed closed point $v\in X$, we obtain a canonical isomorphism of ind-schemes over $\Spf\Oo_v$
\begin{equation}\label{EqCorGrGlobLoc}
\Gr_{\gpSch,X}\times_X \Spf\Oo_v \isoto \Fl_{\gpSch,v}\widehattimes_{\F_v} \Spf\Oo_v
\end{equation}
where $\Fl_{\gpSch,v}$ is the local affine flag variety from Example~\ref{ExDivisors}(b) and $\widehattimes_{\F_v}$ denotes the fiber product of ind-schemes over $\F_v$.
\end{cor}

\begin{proof}
The first assertion follows from Proposition~\ref{PropGrGlobLoc}. To prove the second assertion, note that $\Spf\Oo_v= \dirlim\Spec \Oo_v/(z_v^n)$ where $z_v$ is a uniformizing parameter at $v$. Thus it suffices to prove that \eqref{EqCorGrGlobLoc} is an isomorphism modulo $z_v^n$ for every $n$. To this end we must show for every ring $R$ and morphism $x\colon \Spec R\to\Spec \Oo_v/(z_v^n)$ that $\Fl_{\gpSch,x}(R)=\Fl_{\gpSch,v}(R)$. Let $\zeta:=x^*(z_v)\in R$. Then $z_x:=z_v-\zeta$ is a generator of the ideal sheaf $\mathscr{I}_{\Gamma_x}$ defining the graph $\Gamma_x\subset X_R$ of $x$. Since $\zeta^n=0$ in $R$, we have $\widehat{\Gamma}_x = \Spf R\dbl z_x\dbr = \Spf R\dbl z_v\dbr$ and $R\dpl z_x\dpr = R\dpl z_v\dpr$. This proves the equality $\Fl_{\gpSch,x}(R)=\Fl_{\gpSch,v}(R)$.
\end{proof}

\begin{numberedparagraph}
Let $n=1$. We recall the definition of the affine Schubert varieties from \cite{Richarz16,Lafforgue12}. Let $\mu\colon \mathbb{G}_{m,\FcnFld^\sep}\to \gengpSch_{\FcnFld^\sep}$ be a cocharacter of $\gengpSch$, and let $K/\FcnFld$ be a finite separable extension over which $\mu$ is defined. Let $E_{\mu}$ be the reflex field of $\mu$, i.e.~$E_\mu$ is the field of definition of the conjugacy class of $\mu$. It is a finite separable field extension of $\FcnFld$, contained in $K$. Let $\widetilde{X}_K$ and $\widetilde{X}_\mu$ be the normalizations of $X$ in $K$ and $E_\mu$. The field extensions correspond to finite, flat, surjective morphisms $\widetilde{X}_K \to \widetilde{X}_\mu \to X$. We consider the canonical leg $x\colon \Spec K \into \widetilde{X}_K \twoheadrightarrow X$. The completion $\widehat{\Gamma}_x$ of its graph $\Gamma_x$ is of the form $\widehat{\Gamma}_x=\Spf K\dbl z_x\dbr$. If $z\in\FcnFld$ is an element for which $\FcnFld$ is a separable extension of $\F_q(z)$ and $\zeta$ denotes the image of $z$ in $K$, then $\widehat{\Gamma}_x=\Spf K\dbl z_x\dbr=\Spf K\dbl z-\zeta\dbr$ by \cite[Lemma~1.3]{HartlJuschka}. There is a ring homomorphism $\F_q(z)\into K\dbl z-\zeta\dbr$ sending $z$ to $z=\zeta+(z-\zeta)$, which modulo the maximal ideal $(z-\zeta)$ induces the inclusion $\F_q(z)\into K$ sending $z$ to $\zeta$. Since $K$ is finite separable over $\F_q(z)$ and $K\dbl z-\zeta\dbr$ is henselian, there is a unique $\F_q(z)$-homomorphism $K\into K\dbl z-\zeta\dbr$, which modulo the maximal ideal $(z-\zeta)$ induces the identity on $K$. Note that this $\F_q(z)$-homomorphism is \emph{not} the isomorphism $K\isoto K\cdot(z-\zeta)^0\subset K\dbl z-\zeta\dbr$, because it sends $z$ to $z=\zeta+(z-\zeta)\not\in K\cdot(z-\zeta)^0$. This $\F_q(z)$-homomorphism defines the upper row in the following diagram
\[
\xymatrix {
\widehat{\Gamma}_x = \Spf K\dbl z-\zeta\dbr \ar[r] & \Spec (K\otimes_{\F_q} K) \ar[r] \ar[d] & \Spec K \ar[r] \ar[d] & \Spec \FcnFld \ar[r] \ar[d] & X \ar[d] \\
& \Spec K \ar[r] & \Spec \F_q \ar@{=}[r] & \Spec \F_q \ar@{=}[r] & \Spec \F_q
}
\]
and realizes $\widehat{\Gamma}_x = \Spf K\dbl z-\zeta\dbr$ as the formal completion of $\Spec (K\otimes_{\F_q}K)$ along the diagonal embedding of $\Spec K$. Pulling back the group scheme $\gpSch$ from $X$ to $\Spec K$ under the upper row in the diagram, we see that the cocharacter $\mu\colon\G_{m,K}\to \gengpSch_K$ is defined over $\widehat{\Gamma}_x$.

We now define the $K$-valued point $(x,\bar{g})\in \Gr_{\gpSch,X}(K)=(\mcL_X\gpSch/\mcL^+_X\gpSch)(K)$ given as in Corollary~\ref{CorGrGlobLoc} by the canonical leg $x$ as above, and
\[
\bar{g} \;:=\; \mu(z_x)\cdot L^+_x\gpSch(K) \; \in \; \Fl_{\gpSch,x}(K).
\]
It is independent of the choice of $z_x$, because any other $\tilde{z}_x$ differs from $z_x$ by a unit $u\in K\dbl z_x\dbr^\times$, which is mapped under $\mu$ to $\mu(u)\in L^+_x\gpSch(K)$; see \cite[Remark~A.2]{Richarz16}. 
\end{numberedparagraph}

\begin{defn}\label{Def_SchubertVariety} 
The \emph{affine Schubert variety $\Gr_{\gpSch,X}^{\leq \mu}$ of $\mu$ in the Beilinson-Drinfeld Grassmannian} $\Gr_{\gpSch,X}$ is defined as the scheme-theoretic image of the morphism $(\mcL^+_X\gpSch\times_X \Spec K)\cdot (x,\bar{g})\to \Gr_{\gpSch,X}\times_X\widetilde{X}_K$. It is an $\mcL^+_X\gpSch\times_X  \widetilde{X}_K$-invariant, closed subscheme, which is quasi-compact, because $\mcL^+_X\gpSch$ is a quasi-compact scheme by Lemma~\ref{LemmaGlobalLoopIsqc}\ref{LemmaGlobalLoopIsqc_B}. It only depends on the conjugacy class of $\mu$, because if $\tilde{\mu}=\Int_h\circ \mu$ is conjugate to $\mu$ by an element $h\in \gengpSch(K)$, then $(x,h)\in \mcL^+_X\gpSch(K)$. Therefore, the affine Schubert variety descends to a closed subscheme $\Gr_{\gpSch,X}^{\leq \mu}$ of $\Gr_{\gpSch,X}\times_X \widetilde{X}_\mu$.
\end{defn}

\begin{prop}\label{PropGrProduct}
Let $U:=\{(x_i)_i\in X^n \colon x_i\neq x_j\text{ for }i\neq j\}\subseteq X^n$ be the complement of all diagonals. Over the open set $U\subset X^n$, the Beilinson-Drinfeld Grassmannian is a product
\begin{align}\label{product-Grassmannian}
\Gr_{\gpSch,X^n,I_\bullet} \times_{X^n} U & \isoto \bigl(\Gr_{\gpSch,X} \times_{\F_q} \ldots \times_{\F_q} \Gr_{\gpSch,X}\bigr) \times_{X^n} U\\
\bigl(\underline x,\,(\mcE^{(j)}),\,(\varphi^{(j-1)}),\epsilon\,\bigr)& \longmapsto \bigl(x_i,L_{x_i}^+\mcE^{(0)},L_{x_i}\mcE^{(0)}\xrightarrow{L_{x_i}(\epsilon\circ\varphi^{(k-1)}\circ\ldots\circ\varphi^{(0)})}L_{x_i}\gpSch\bigr)_{i=1,\ldots,n}. \nonumber
\end{align}
\end{prop}

\begin{proof}
Since the graphs $\Gamma_{x_i}$ are pairwise disjoint, we have $L^+_{\underline x}(\mcE^{(0)}) = \bigl(L^+_{x_i}(\mcE^{(0)})\bigr)_i$. Using the isomorphism \eqref{EqHeckeChangeI} from Proposition~\ref{PropHeckeChangeI} we may assume that $I_\bullet$ is the coarsest partition with $I_1=I$. Then the data $\mcE^{(0)}$ and $\varphi^{(0)}\colon \mcE^{(0)}|_{X_S\smallsetminus\Gamma_{\underline x}}\isoto \mcE^{(1)}|_{X_S\smallsetminus\Gamma_{\underline x}}$ can be recovered from $\mcE^{(1)}:= \gpSch\times_X X_S$ by gluing $\mcE^{(1)}|_{X_S\smallsetminus \Gamma_{\underline x}}$ with $L^+_{\underline x}(\mcE^{(0)}) = \bigl(L^+_{x_i}(\mcE^{(0)})\bigr)_i$ via $L_{\underline x}(\epsilon\circ\varphi^{(0)})$ as in Lemma \ref{LemmaBL}.
\end{proof}

This proposition allows us to define bounds on the left-hand side of \eqref{product-Grassmannian} using the bounds on the right-hand side of \eqref{product-Grassmannian} componentwise, see Construction \ref{Def_Bound_individual}.

\begin{defn}\label{DefGr2legs}
As in the introduction we fix a point $\infty\in X(\F_q)$.

\smallskip\noindent
(a) For $I:=\{1,2\}$ with the finest partition $I_\bullet=(\{1\},\{2\})$ we define the \emph{Beilinson-Drinfeld Grassmannian with two legs, one fixed at $\infty$} as 
\[
\Gr_{\gpSch,X\times\infty}:=\Gr_{\gpSch,X^2,(\{1\},\{2\})}\times_{X^2}(X\times_{\F_q}\Spec\F_\infty)
\] 
Its $S$-valued points, for an $\F_q$-scheme $S$, are tuples 
\[
\bigl(x,\mcE,\mcE',\mcE'',\,\varphi,\varphi'\,\bigr)\in\Hecke_{\gpSch,\varnothing,X\times\infty}(S)
\] 
as in Definition~\ref{DefHecke2legs}(a) together with 
\begin{itemize}
\item a trivialization $\epsilon\colon \mcE''\isoto \gpSch\times_X X_S$.
\end{itemize}

\medskip\noindent
(b) For the coarsest partition $I_1:=I=\{1,2\}$ we also define 
\[
\PGr_{\gpSch,X\times\infty} :=\Gr_{\gpSch,X^2,(\{1,2\})}\times_{X^2}(X\times_{\F_q}\Spec\F_\infty).
\]
classifying data $\bigl(x,\,\mcE,\mcE'',\,\varphi\colon \mcE|_{X_S\smallsetminus(\Gamma_x\cup\infty)} \isoto \mcE''|_{X_S\smallsetminus(\Gamma_x\cup\infty)} \,\bigr)$, visualized as
\begin{equation}
\xymatrix @C+1pc {
\mcE \ar@{-->}[r]^{\varphi}_{x,\infty} & \mcE'' \,.
}
\end{equation}
Again, the projection map onto the leg $x$ defines morphisms $\Gr_{\gpSch,X\times\infty}\to \PGr_{\gpSch,X\times\infty}\to X$.
\end{defn}

Recall the isomorphism $\Gr_{\gpSch,X}\times_X \Spec\F_\infty \isoto \Fl_{\gpSch,\infty}$ from \eqref{EqCorGrGlobLoc} in Corollary~\ref{CorGrGlobLoc}.

\begin{prop}\label{PropGrGeneric}
Above $X\smallsetminus\{\infty\}$ the base changes $\Gr_{\gpSch,X\times\infty} \times_X (X\smallsetminus\{\infty\})$ and $\PGr_{\gpSch,X\times\infty} \times_X (X\smallsetminus\{\infty\})$ are isomorphic to the product $\Gr_{\gpSch,X} \times_X (X\smallsetminus\{\infty\})\times_{\F_q} \Fl_{\gpSch,\infty}$.
\end{prop}
\begin{proof}
This follows from Proposition~\ref{PropGrProduct} and Corollary~\ref{CorGrGlobLoc}. As in that proposition the isomorphism is given by 
\begin{align*}
\begin{split}
\Gr_{\gpSch,X\times\infty}\times_X(X\smallsetminus \{\infty\})&\xlongrightarrow{\sim}\Gr_{\gpSch,X}\times_X (X\smallsetminus\{\infty\}) \times_{\F_q}\Fl_{\gpSch,\infty}\\
(x,\mcE\underset{x}{\xrightarrow{\varphi}}\mcE'\underset{\infty}{\xrightarrow{\varphi'}}\mcE''\underset{\sim}{\xrightarrow{\epsilon}}\gpSch\times_X X_S)&\longmapsto \bigl(x,L_x^+\mcE, L_x\mcE\xrightarrow{L_x(\epsilon\circ\varphi'\circ\varphi)}L_x\gpSch\bigr), \bigl(L_\infty^+\mcE', L_\infty\mcE'\xrightarrow{L_\infty(\epsilon\circ\varphi')}L_\infty\gpSch \bigr).
\end{split}
\end{align*}
\end{proof}

As in \eqref{EqHeckeCreateInfty}, ``creating'' an unnecessary additional leg at $\infty$ defines a morphism 
\begin{align}\label{EqGrCreateInfty}
\Gr_{\gpSch,X} & \longrightarrow \PGr_{\gpSch,X\times\infty} \\
\bigl(x,\,\mcE\xdashrightarrow[x]{\varphi}\mcE' \xrightarrow[\sim]{\enspace\epsilon\enspace} \gpSch\times_X X_S\bigr) & \longmapsto \bigl(x,\,\mcE\xdashrightarrow[x,\infty]{\varphi}\mcE'':=\mcE' \xrightarrow[\sim]{\enspace\epsilon\enspace} \gpSch\times_X X_S\bigr). \nonumber
\end{align}

\begin{lem}\label{LemmaGrCreateInfty}
The morphism \eqref{EqGrCreateInfty} is an ind-proper monomorphism of ind-schemes.
\end{lem}
\begin{proof}
We already saw in Lemma~\ref{LemmaHeckeCreateInfty} that it is a monomorphism. To prove that it is ind-proper, we may work on the $\fpqc$-covering $(X\smallsetminus\{\infty\})\coprod \Spec\Oo_\infty$ of $X$. Above $X\smallsetminus\{\infty\}$ the map is an ind-closed immersion, because there 
\[
\PGr_{\gpSch,X\times\infty}\times_X (X\smallsetminus\{\infty\})\cong \Gr_{\gpSch,X}\times_X (X\smallsetminus\{\infty\}) \times_{\F_q} \Fl_{\gpSch,\infty}
\]
by Proposition~\ref{PropGrGeneric}. The image of the morphism~\eqref{EqGrCreateInfty} equals $\Gr_{\gpSch,X}\times_X (X\smallsetminus\{\infty\}) \times_{\F_q} 1\cdot L^+_\infty \gpSch$, which is an ind-closed ind-subscheme. Over $\Oo_\infty$ the ind-scheme $\Gr_{\gpSch,X}\times_X \Spec\Oo_\infty$ is ind-proper by \cite[\S\,2.5]{Richarz16}, because we assumed that $\mcG_\infty$ is parahoric. Therefore also the morphism~\eqref{EqGrCreateInfty} is ind-proper over $\Oo_\infty$. This proves the lemma.
\end{proof}

\subsection{Bounds}\label{subsec:Bounds}

\newcommand{\BoundsGroundField}{F}

We recall the definition of bounds from Bieker~\cite[\S\,2]{Bieker} and generalize it slightly (to case \ref{DefThreeTypes_C} in Definition~\ref{DefThreeTypes} below). We need three cases in which we want to define bounds. 

\begin{defn}\label{DefThreeTypes}
For the \emph{ground field} $\BoundsGroundField=\FcnFld=\F_q(X)$ or $\BoundsGroundField=\FcnFld_v$ or $\BoundsGroundField=\F_v$ at a place $v\in X$, we fix an algebraic closure $\BoundsGroundField^{\sep}$ of the ground field $\BoundsGroundField$ and consider finite separable field extensions $\BoundsGroundField\subset K$ in $\BoundsGroundField^{\sep}$. For $K$ we define $\widetilde{X}_K$ as follows:
\begin{enumerate}
\item \label{DefThreeTypes_A}
\emph{global type:} For $\BoundsGroundField=\FcnFld=\F_q(X)$ let $\widetilde{X}_K$ be the normalization of $X$ in $K$. It is a smooth projective curve over $\F_q$. Then we obtain a finite morphism $\widetilde{X}_K\to X$.
\item \label{DefThreeTypes_B}
\emph{local type:} For $\BoundsGroundField=\FcnFld_v$ at a place $v\in X$ let $\widetilde{X}_K:=\Spec\Oo_K$, where $\Oo_K$ is the valuation ring of $K$. Then we obtain a finite morphism $\widetilde{X}_K\to \Spec\Oo_v$.
\item \label{DefThreeTypes_C}
\emph{finite type:} For $\BoundsGroundField=\F_v$ at a place $v\in X$ let $\widetilde{X}_K:=\Spec K$. Then we obtain a finite morphism $\widetilde{X}_K\to \Spec\F_v$.
\end{enumerate}
\end{defn}

\begin{remark}
Since we are mainly interested in bounds defined by cocharacters of the group $\gengpSch$ and bounds of the finite type \ref{DefThreeTypes_C}, we only consider separable field extensions $K$ of $\BoundsGroundField$ in this article.
\end{remark}

We slightly generalize the definition of bounds from Bieker~\cite[Definition~2.8]{Bieker}. 

\begin{defn}\label{DefBound}
\begin{enumerate}
\item \label{DefBound_A}
For $i=1,\ldots, n$ we fix one of the three types in Definition~\ref{DefThreeTypes} and let $K_i,K'_i$ be two fields of that type (for the same place $v$ in the local and finite case). We let $\widetilde{X}_{K_i}$ for $K_i$ and $\widetilde{X}_{K'_i}$ for $K'_i$ as in Definition~\ref{DefThreeTypes}. We write $\prod_i \widetilde{X}_{K_i} := \widetilde{X}_{K_1}\times_{\F_q}\ldots \times_{\F_q} \widetilde{X}_{K_n}$ and likewise $\prod_i \widetilde{X}_{K'_i}$. For any two quasi-compact closed subschemes $Z\subseteq \Gr_{\gpSch,X^n,I_\bullet}\times_{X^n} \prod_i \widetilde{X}_{K_i}$ and $Z'\subseteq \Gr_{\gpSch,X^n,I_\bullet}\times_{X^n} \prod_i \widetilde{X}_{K'_i}$, we say they are \emph{equivalent} if for each $i$ there is a finite extension $K''_i$ of both $K_i,K'_i$ such that
\[
Z\times_{\prod_i \widetilde{X}_{K_i}} \prod_i \widetilde{X}_{K''_i}=Z'\times_{\prod_i \widetilde{X}_{K'_i}}\prod_i \widetilde{X}_{K''_i}
\]
inside $\Gr_{\gpSch,X^n,I_\bullet}\times_{X^n} \prod_i \widetilde{X}_{K''_i}$. 
\item \label{DefBound_B}
Let $\mcZ$ be an equivalence class as in \ref{DefBound_A} and let $Z_{(K_i)_i}\subseteq \Gr_{\gpSch,X^n,I_\bullet}\times_{X^n} \prod_i \widetilde{X}_{K_i}$ be a quasi-compact closed subscheme for some tuple $(K_i)_i$ representating $\mcZ$. For every $i$ we let $\BoundsGroundField_i$ be the ground field of the corresponding type as in Definition~\ref{DefThreeTypes}, and $\widetilde{K}_i$ be the Galois closure of $K_i/\BoundsGroundField_i$. We define 
\[
\Aut_{\mcZ}({\textstyle\prod_i}\widetilde{K}_i):=\{(g_i)_i \in \prod_i \Gal(\widetilde{K}_i/\BoundsGroundField_i):(g_i)_i^*\mcZ=\mcZ\}. 
\]
The quotient
\[
\widetilde{X}_{\mcZ} := \bigl(\prod_i \widetilde{X}_{\widetilde{K}_i}\bigr) \big/ \Aut_{\mcZ}({\textstyle\prod_i}\widetilde{K}_i)
\]
is called the \emph{reflex scheme} of $\mcZ$. Since $\prod_i \Gal(\widetilde{K}_i/K_i)\subset \Aut_{\mcZ}({\textstyle\prod_i}\widetilde{K}_i)$, the reflex scheme is equiped with a finite and faithfully flat morphism $\prod_i \widetilde{X}_{K_i} \to \widetilde{X}_{\mcZ}$. 
\item \label{DefBound_C}
A \emph{bound in $\Gr_{\gpSch,X^n,I_\bullet}$} is an equivalence class $\mcZ$ as in \ref{DefBound_A}, such that all its representatives $Z_{(K_i)_i}\subseteq \Gr_{\gpSch,X^n,I_\bullet}\times_{X^n} \prod_i \widetilde{X}_{K_i}$ are stable under the left $\mcL^+_{X^n}\gpSch\times_{X^n} \prod_i \widetilde{X}_{K_i}$-action on $\Gr_{\gpSch,X^n,I_\bullet}\times_{X^n} \prod_i \widetilde{X}_{K_i}$ from Definition~\ref{DefLoopActionOnGr}, and such that $\mcZ$ has a representative over its reflex scheme $\widetilde{X}_{\mcZ}$.
\end{enumerate}
\end{defn}

\begin{remark}
The quotient $\widetilde{X}_{\mcZ}$ can be complicated and does not have to be a product. However, when $\mcZ$ is the product $\prod_i\mcZ^{(i)}$ as in Construction~\ref{Def_Bound_individual}, the reflex scheme $\widetilde{X}_{\mcZ}$ will be the product of the reflex schemes $\widetilde{X}_{\mcZ^{(i)}}$ of the individual factors $\mcZ^{(i)}$.
\end{remark}

\begin{lem}\label{LemmaLoopActionOnGr}
For every bound $\mcZ$ there is a tuple of non-negative integers $(c_i)_i$ such that the action of $\mcL^+_{X^n}\gpSch$ on all the representatives $Z$ of $\mcZ$ factors through the finite group scheme $\mcL^{+,(c_i)_i}_{X^n}\gpSch$ from \eqref{EqGlobalTruncLoopGp}.
\end{lem}

We generalize the definition of bounded Hecke data from \cite{Arasteh-Habibi} and \cite[Definition~3.3]{Bieker} as follows. 

\begin{defn} \label{Def_HeckeBounded}
Let $\mcZ$ be a bound in $\Gr_{\gpSch,X^n,I_\bullet}$ as in Definition~\ref{DefBound} with reflex scheme $\widetilde{X}_{\mcZ}$. Let 
\[
\underline{\mcE}=\Bigl(\underline x,(\mcE^{(0)},\psi^{(0)}) \xdashrightarrow{\varphi^{(0)}} (\mcE^{(1)},\psi^{(1)}) \xdashrightarrow{\varphi^{(1)}} \ldots \xdashrightarrow{\varphi^{(k-1)}} (\mcE^{(k)},\psi^{(k)})\Bigr)
\]
in $(\Hecke_{\gpSch,D,X^n,I_\bullet}\times_{X^n} \widetilde{X}_\mcZ)(S)$ be a Hecke datum over an $\widetilde{X}_\mcZ$-scheme $S$. 
By definition of $L^+_{\underline x}\gpSch$-bundles, there is an \'etale covering $S'\to S$ and a trivialization $\epsilon\colon L^+_{\underline x}(\mcE^{(k)})_{S'} \isoto (L^+_{\underline x}\gpSch)_{S'}$. As in Proposition \ref{PropGrGlobLoc}, the tuple
\[
\Bigl(L^+_{\underline x}\mcE^{(0)} \xdashrightarrow{L_{\underline x}\varphi^{(0)}} 
\ldots \xdashrightarrow{L_{\underline x}\varphi^{(k-1)}} L^+_{\underline x}\mcE^{(k)} \xrightarrow[\sim]{\enspace\epsilon\enspace} L^+_{\underline x}\gpSch\Bigr)\in (\Gr_{\gpSch,X^n,I_\bullet}\times_{X^n} \widetilde{X}_\mcZ)(S')
\]
defines a morphism $S'\to \Gr_{\gpSch,X^n,I_\bullet}\times_{X^n} \widetilde{X}_\mcZ$. We say that $\underline{\mcE}$ is \emph{bounded by $\mcZ$} if this morphism factors through $Z\subset \Gr_{\gpSch,X^n,I_\bullet}\times_{X^n} \widetilde{X}_\mcZ$ for the representative $Z$ of $\mcZ$ that is defined over the reflex scheme $\widetilde{X}_{\mcZ}$. By the invariance of $Z$ under the left multiplication by $L^+_{\underline x}\gpSch$, the definition is independent of the choice of $\epsilon$ and $S'\to S$.

We denote the stack of Hecke data bounded by $\mcZ$ by $\Hecke_{\gpSch,D,X^n,I_\bullet}^\mcZ$. Then diagram~\eqref{EqGrFiberProd} induces an isomorphism $\Hecke_{\gpSch,D,X^m,I_\bullet}^\mcZ \times_{\Bun_{\gpSch,D}} \Spec\F_q\cong Z\times_X (X\smallsetminus D)$ for the representative $Z$ of $\mcZ$ over $\widetilde{X}_\mcZ$.
\end{defn}

\begin{remark}\label{Rem_HeckeBounded}
For $n=1$, let $\mcZ$ be a bound in $\Gr_{\gpSch,X}$ with reflex scheme $\widetilde{X}_\mcZ$. Let $S$ be a scheme over $\widetilde{X}_\mcZ$ and let $\varphi\colon \mcE|_{X_S\smallsetminus\Gamma_x} \isoto \mcE'|_{X_S\smallsetminus\Gamma_x}$ be an isomorphism of $\gpSch$-bundles on $X_S$ outside the graph of a leg $x\in X(S)$. The definition allows to say that \emph{$\varphi$ is bounded by $\mcZ$}, by viewing $(x,\mcE,\mcE',\varphi)\in\Hecke_{\gpSch,\varnothing,X}$.
\end{remark}

\begin{thm}\label{Thm_HeckeBounded}
The stack $\Hecke_{\gpSch,D,X^n,I_\bullet}^\mcZ$ is a closed substack of $\Hecke_{\gpSch,D,X^n,I_\bullet}\times_{X^n} \widetilde{X}_\mcZ$. Moreover, it is an Artin-stack locally of finite type over $(X\smallsetminus D)^n$.
\end{thm}
\begin{proof}
For any test scheme $S$ over $\widetilde{X}_\mcZ$ and Hecke datum $\underline{\mcE}\in \Hecke_{\gpSch,D,X^n,I_\bullet}(S)$, we must show that in the cartesian diagram 
\[
\xymatrix @C+2pc {
\widetilde{S}\ar[r] \ar[d] & \Hecke_{\gpSch,D,X^n,I_\bullet}^\mcZ \ar[d] \\
S\ar[r]^-{\underline{\mcE}} & \Hecke_{\gpSch,D,X^n,I_\bullet}\times_{X^n} \widetilde{X}_\mcZ
}
\]
the map $\widetilde{S}\to S$ is a closed immersion of schemes. This can be tested after base change to an \'etale covering $S'\to S$, which we may choose as in Definition~\ref{Def_HeckeBounded}. Then the base change morphism $\widetilde{S}\times_S S' \to S'$ arises as the base change of $Z\into \Gr_{\gpSch,X^n,I_\bullet}\times_{X^n} \widetilde{X}_\mcZ$ which is a closed immersion.
\end{proof}

Next we give some examples of bounds. 

\begin{defn}[Bound in $\Gr_{\gpSch,X}$ given by $\mu$] \label{Def_BoundBy_mu}
Let $n=1$ and let $\mu\colon \mathbb{G}_{m,\FcnFld^\sep}\to \gengpSch_{\FcnFld^\sep}$ be a cocharacter of $\gengpSch$. Let $E_{\mu}$ be the reflex field of $\mu$, and $\widetilde{X}_\mu$ the normalization of $X$ in $E_\mu$ as in Definition~\ref{DefThreeTypes}. The affine Schubert variety $Z^{\leq \mu}:=\Gr_{\gpSch,X}^{\leq \mu}$ of $\mu$ from Definition~\ref{Def_SchubertVariety} defines a bound $\mcZ^{\leq\mu}$, which has the representative $Z^{\leq \mu}$ over the reflex scheme $\widetilde{X}_\mu$. We say that a Hecke datum $\underline{\mcE}\in (\Hecke_{\gpSch,D,X}\times_{X} \widetilde{X}_\mu)(S)$ is \emph{bounded by $\mu$} if it is bounded by $\mcZ^{\leq\mu}$. 

Since we will need it later, we introduce the following notation. We choose a point $\widetilde{\infty}$ of $\widetilde{X}_\mu$ that lies above the point $\infty\in X(\F_q)$ from the introduction. We write $\Oo_\mu$ for the complete local ring of $\widetilde{X}_\mu$ at $\widetilde{\infty}$ and $\kappa_\mu$ for its residue field. We let $\Breve{\Oo}_\mu$ be the completion of the maximal unramified ring extension of $\Oo_\mu$. Then $\Oo_\infty\subset \Oo_\mu$ and $\Breve{\Oo}_\infty\subset \Breve{\Oo}_\mu$ and $\F_\infty\subset \kappa_\mu$ are finite extensions.
\end{defn}

\begin{numberedparagraph}\label{Def_beta}
Let $\infty\in X(\F_q)$ and let $\beta\in L_\infty\gpSch(\overline{\F}_q)=\gengpSch(\Breve{\FcnFld}_\infty)$ with $\beta\cdot L^+_\infty\gpSch\cdot\beta^{-1}=L^+_\infty\gpSch$. Then there is a smallest finite field extension $\F_\beta$ of $\F_q$ with $\beta\in L_\infty\gpSch(\F_\beta)$. For the Frobenius $\tau_{\gengpSch_\infty}$ from \eqref{EqTau_G} we have $\tau_{\gengpSch_\infty}^j(\beta) \cdot L^+_\infty\gpSch = L^+_\infty\gpSch \cdot \tau_{\gengpSch_\infty}^j(\beta)$ for any $j\in\Z$, because $\gpSch_\infty$ is defined over $\F_q$. This equality implies that $\beta$ is small in the sense, that its orbit in $\Fl_{\gpSch,\infty}(\F_\beta)$ under the left action of $L^+_\infty\gpSch$ is the single point $\beta\cdot L^+_\infty\gpSch(\F_\beta)$; see Definition~\ref{Def_Zbeta}. 

That $\beta$ is small can also be interpreted in the following way. Assume that $\gpSch_\infty$ is an Iwahori group scheme corresponding to a $\tau$-stable alcove $\mathfrak{a}$ in the extended Bruhat-Tits building $\mathcal{B}(\gengpSch_\infty,\Breve{\FcnFld}_\infty)$ of $\gengpSch_\infty$ over $\Breve{\FcnFld}_\infty$. The condition $\beta\cdot L^+_\infty\gpSch\cdot\beta^{-1}=L^+_\infty\gpSch$ says that $\mathfrak{a}$ is a fixed point of $\beta$ under the action of $\gengpSch(\Breve{\FcnFld}_\infty)=L_\infty\gpSch(\overline{\F}_q)$ on $\mathcal{B}(\gengpSch_\infty,\Breve{\FcnFld}_\infty)$. Assume further, that $\beta$ also stabilizes an appartment containing $\mathfrak{a}$. If $A$ is the maximal split torus of $\gengpSch(\Breve{\FcnFld}_\infty)$ corresponding to that appartment, then $\beta$ normalizes $A$. In this case, $\beta$ induces an element in the \emph{Iwahori-Weyl group} $\widetilde{W}_\infty:=\widetilde{W}(\gengpSch_\infty,A,\Breve{\FcnFld}_\infty)$ of $A$ over $\Breve{\FcnFld}_\infty$; see Richarz~\cite{RicharzIWGp}. Then the condition $\beta\cdot L^+_\infty\gpSch\cdot\beta^{-1}=L^+_\infty\gpSch$ means that $\beta$ has \emph{length zero} in $\widetilde{W}_\infty$.
\end{numberedparagraph}

\begin{defn}[Bound in $\Gr_{\gpSch,X}\times_X \Spec\F_\beta$ given by $\beta$] \label{Def_Zbeta}
Let $n=1$. Let $\beta\in L_\infty\gpSch(\F_\beta)$ as in \S\,\ref{Def_beta}. We define the bound 
\begin{align}
\mcZ(\beta) & \; := \; (L^+_\infty\gpSch)_{\F_\beta} \cdot\beta\cdot (L^+_\infty\gpSch)_{\F_\beta}/(L^+_\infty\gpSch)_{\F_\beta} \;=\; \beta\cdot (L^+_\infty\gpSch)_{\F_\beta}/(L^+_\infty\gpSch)_{\F_\beta} \\
& \; \subset \; \Fl_{\gpSch,\infty}\times_{\F_q} \Spec\F_\beta \; = \; (L_\infty\gpSch/L^+_\infty\gpSch)_{\F_\beta}. \nonumber
\end{align}
Its reflex scheme $\widetilde{X}_{\mcZ(\beta)}$ is $\Spec\F_\beta$. This is a field of type \ref{DefThreeTypes_C} in Definition~\ref{DefThreeTypes}.
\end{defn}

\begin{construction}[Bounds in $\Gr_{\gpSch,X^n,I_\bullet}$ as products of bounds in $\Gr_{\gpSch,X}$] \label{Def_Bound_individual}
Let $(\mcZ^{(i)})_{i=1\ldots n}$ be an $n$-tuple of bounds in $\Gr_{\gpSch,X}$ with reflex schemes $\widetilde{X}_{\mcZ^{(i)}}$, and let $Z^{(i)}$ be the representative of $\mcZ^{(i)}$ over $\widetilde{X}_{\mcZ^{(i)}}$ for all $i=1,\ldots,n$. We allow each $\widetilde{X}_{\mcZ^{(i)}}$ to be of any of the three types in Definition~\ref{DefThreeTypes}. Let $U\subseteq X^n$ be the complement of all diagonals as in Proposition~\ref{PropGrProduct}. If $\bigl(\prod_i \widetilde{X}_{\mcZ^{(i)}}\bigr) \times_{X^n} U$ is non-empty, we define $Z$ to be the scheme-theoretic image of the composite morphism
\[
\xymatrix 
{
\Bigl(\prod\limits_{i=1}^n Z^{(i)}\Bigr) \times_{X^n} U \ar@{^{ (}->}[dr] \ar@{^{ (}->}[r] & \Bigl(\bigl(\Gr_{\gpSch,X}\times_X \widetilde{X}_{\mcZ^{(1)}}\bigr) \times_{\F_q} \ldots \times_{\F_q} \bigl(\Gr_{\gpSch,X}\times_X \widetilde{X}_{\mcZ^{(n)}}\bigr)\Bigr) \times_{X^n} U \ar@{^{ (}->}[d] \\
& \Gr_{\gpSch,X^n,I_\bullet}\times_{X^n} \prod_i\widetilde{X}_{\mcZ^{(i)}},
}
\]
where the right vertical map is given in Proposition \ref{PropGrProduct}. This $Z$ defines a bound $\mcZ$ in $\Gr_{\gpSch,X^n,I_\bullet}$, which we denote $\prod_i\mcZ^{(i)}$ and call the \emph{product bound} of the $\mcZ^{(i)}$. Its reflex scheme is $ \prod_i\widetilde{X}_{\mcZ^{(i)}}$, because the group $\Aut_{\mcZ}({\textstyle\prod_i}\widetilde{K}_i)$ from Definition~\ref{DefBound}\ref{DefBound_B} equals $\prod_i \Gal(\widetilde{K}_i/E_i)$ for the ``reflex fields'' $E_i$ for which $\widetilde{X}_{\mcZ^{(i)}}=\widetilde{X}_{E_i}$.

Note that $\bigl(\prod_i \widetilde{X}_{\mcZ^{(i)}}\bigr) \times_{X^n} U$ is non-empty unless two different $\widetilde{X}_{\mcZ^{(i)}}$ are of type~\ref{DefThreeTypes_C} in Definition~\ref{DefThreeTypes} for the same place $v\in X$.
\end{construction}

\begin{defn} \label{DefZmubeta}
Let $\beta\in L_\infty\gpSch(\F_\beta)$ as in \S\,\ref{Def_beta} and let $\mu\in X_*(T)$. We define the bounds $\mcZ(\mu,\beta)$ in $\Gr_{\gpSch,X\times\infty}$ and $\PmcZ(\mu,\beta)$ in $\PGr_{\gpSch,X\times\infty}$ as the product bound $\mcZ^{\leq\mu}\times \mcZ(\beta)$ from Construction~\ref{Def_Bound_individual}, such that the bound at the moving leg $x$ is $\mcZ^{\leq\mu}$ and the bound at the fixed leg $\infty$ equals $\mcZ(\beta)$ from Definition~\ref{Def_Zbeta}. The reflex scheme of these bounds is the product $\widetilde{X}_{\mu,\beta}:=\widetilde{X}_\mu\times_{\F_q} \Spec\F_\beta$.

Since we will need it later, we introduce the following notation. We choose a point $\widetilde{\infty}$ of $\widetilde{X}_{\mu,\beta}$ that lies above the point $\infty$. We write $\OReflZMuBeta$ for the complete local ring of $\widetilde{X}_{\mu,\beta}$ at $\widetilde{\infty}$ and $\KappaReflZMuBeta$ for its residue field. We let $\BreveOReflZMuBeta$ be the completion of the maximal unramified ring extension of $\OReflZMuBeta$. It is equal to the ring $\Breve{\Oo}_\mu$ from Definition~\ref{Def_BoundBy_mu} if the two points on $\widetilde{X}_\mu$ and $\widetilde{X}_{\mu,\beta}$ above $\infty\in X$ are chosen compatibly. Then $\Oo_\infty\subset \OReflZMuBeta$ and $\Breve{\Oo}_\infty\subset \BreveOReflZMuBeta$ and $\F_\infty\subset \KappaReflZMuBeta$ are finite extensions.
\end{defn}

\begin{prop}\label{PropDimZmuBeta}
The relative dimension of $\mcZ(\mu,\beta)$ over $\widetilde{X}_{\mu,\beta}$ equals $\langle \mu,2\check{\rho}\rangle$, where $2\check{\rho}$ is the sum of all positive coroots of $\gengpSch_{\FcnFld^\alg}$ with respect to some Borel subgroup for which $\mu$ is dominant.
\end{prop}

\begin{proof}
This was proven in \cite[Lemma 2.2 and the remarks thereafter]{Ngo-Polo}.
\end{proof}

\begin{lem}\label{LemmaZmu1}
Let $\beta=1$, and hence $\F_\beta=\F_q$. Let $\mcZ$ be a bound in $\Gr_{\gpSch,X}$ with reflex scheme $\widetilde{X}_\mcZ$ of global type in Definition~\ref{DefThreeTypes}\ref{DefThreeTypes_A}, such that its representative $Z$ over $\widetilde{X}_\mcZ$ is the scheme theoretic closure of its restriction $Z\times_X (X\smallsetminus\{\infty\})$. Let $\mcZ\times\mcZ(1)$ be the product bound in $\PGr_{\gpSch,X\times\infty}$ from Construction~\ref{Def_Bound_individual}.
\begin{enumerate}
\item\label{LemmaZmu1_A} 
Under the morphism $\Gr_{\gpSch,X} \to \PGr_{\gpSch,X\times\infty}$ from \eqref{EqGrCreateInfty}, the bound $\mcZ$ in $\Gr_{\gpSch,X}$ is mapped isomorphically to the bound $\mcZ\times\mcZ(1)$ in $\PGr_{\gpSch,X\times\infty}$.
\item\label{LemmaZmu1_B} 
The morphism~\eqref{EqHeckeCreateInfty} restricts to an isomorphism $\Hecke_{\gpSch,D,X}^{\mcZ}\isoto \PHecke_{\gpSch,D,X\times \infty}^{\mcZ\times\mcZ(1)}$.
\end{enumerate}
\end{lem}

\begin{proof}
\ref{LemmaZmu1_A} By Lemma~\ref{LemmaGrCreateInfty} the morphism $\mcZ\to\PGr_{\gpSch,X\times\infty}\times_X \widetilde{X}_\mcZ$ is a proper monomorphism, hence a closed immersion. We write $\widetilde{X}_\mcZ\smallsetminus\{\infty\}:=\widetilde{X}_\mcZ\times_X (X\smallsetminus\{\infty\})$. Identifying $\mcZ$ with its image, both $\mcZ$ and $\mcZ\times\mcZ(1)$ are defined as the scheme theoretic closure of their intersection with $\PGr_{\gpSch,X\times\infty}\times_X (\widetilde{X}_\mcZ\smallsetminus\{\infty\})$. The latter is isomorphic to the product $\Gr_{\gpSch,X}\times_X (\widetilde{X}_\mcZ\smallsetminus\{\infty\}) \times_{\F_q} \Fl_{\gpSch,\infty}$ by Proposition~\ref{PropGrGeneric}, and the intersections of $\mcZ$ and $\mcZ\times\mcZ(1)$ with that product are equal to $Z\times_{\widetilde{X}_\mcZ} (\widetilde{X}_\mcZ\smallsetminus\{\infty\})\times_{\F_q} (1\cdot L^+_\infty\gpSch)$. This proves \ref{LemmaZmu1_A}.

\medskip\noindent
\ref{LemmaZmu1_B} We must show that the two projection morphisms $\pr_1$ and $\pr_2$ in
\begin{equation}\label{EqLemmaZmu1}
\Hecke_{\gpSch,D,X}^{\mcZ} \; \xleftarrow{\;\pr_1\;} \;  \Hecke_{\gpSch,D,X}^{\mcZ} \underset{\;\PHecke_{\gpSch,D,X\times \infty}}{\times} \PHecke_{\gpSch,D,X\times \infty}^{\mcZ\times\mcZ(1)} \; \xrightarrow{\;\pr_2\;} \; \PHecke_{\gpSch,D,X\times \infty}^{\mcZ\times\mcZ(1)}
\end{equation}
are isomorphisms, where the fiber product is defined via the morphism \eqref{EqHeckeCreateInfty}. The projection $\pr_1$ (respectively $\pr_2$) is a closed immersion (respectively a monomorphism) by Theorem~\ref{Thm_HeckeBounded} (respectively Lemma~\ref{EqHeckeCreateInfty}). In particular, both projections are schematic by \cite[Corollaire~8.1.3 and Th\'eor\`eme~A.2]{Laumon-Moret-Bailly}. Let $(x,(\mcE,\psi),(\mcE',\psi'),\varphi)\in \Hecke_{\gpSch,D,X}^{\mcZ}(S)$ be a $\gpSch$-shtuka bounded by $\mcZ$ over a test scheme $S$. There is an \'etale covering $S'\to S$ over which a trivialization $\epsilon\colon L^+_\Delta(\mcE'_{S'})\isoto (L^+_\Delta\gpSch)_{S'}$ exists, where $\Delta:=\Gamma_x + (\infty\times_{\F_q}S')$. We let $S' \hookleftarrow \pr_1^*S'=:\widetilde{S}'$ be the base change of the closed immersion $\pr_1$ under the morphism $S'\to \Hecke_{\gpSch,D,X}^{\mcZ}$. The morphism from $\widetilde{S}'$ to the fiber product in \eqref{EqLemmaZmu1} corresponds to the object $(x,(\mcE,\psi),(\mcE',\psi'),\varphi)\in \Hecke_{\gpSch,D,X}^{\mcZ}(\widetilde{S}')$ and an object $(x,(\widetilde{\mcE},\widetilde{\psi}),(\widetilde{\mcE}',\widetilde{\psi}'),\widetilde{\varphi})\in \PHecke_{\gpSch,D,X\times \infty}^{\mcZ\times\mcZ(1)}(\widetilde{S}')$ together with an isomorphism $(f,f')\colon (x,(\widetilde{\mcE},\widetilde{\psi}),(\widetilde{\mcE}',\widetilde{\psi}'),\widetilde{\varphi})\isoto (x,(\mcE,\psi),(\mcE',\psi'),\varphi)$ in $\PHecke_{\gpSch,D,X\times \infty}(\widetilde{S}')$. In particular, $f'\colon\widetilde{\mcE}'\isoto \mcE'$ is an isomorphism of $\gpSch$-bundles. The trivialization $\epsilon$ induces the trivialization $\epsilon\circ L^+_\Delta(f')\colon L^+_\Delta(\widetilde{\mcE}')\isoto (L^+_\Delta\gpSch)_{\widetilde{S}'}$. Using these trivializations, the morphism $S'\hookleftarrow \widetilde{S}'$ is obtained as base change under $S'\to\mcZ$ from the projection $\pr_1$ in the following diagram
\begin{equation}\label{EqLemmaZmu2}
\mcZ \; \xleftarrow[\sim]{\;\pr_1\;} \; \mcZ \underset{\;\PGr_{\gpSch,X\times \infty}}{\times} (\mcZ\times\mcZ(1)) \; \xrightarrow[\sim]{\;\pr_2\;} \; (\mcZ\times\mcZ(1)).
\end{equation}
That diagram is obtained from \eqref{EqLemmaZmu1} under base change via $\Spec\F_q\to\Bun_{\gpSch,D}$ as in \eqref{EqGrFiberProd}. In diagram~\eqref{EqLemmaZmu2} the projections $\pr_1$ and $\pr_2$ are isomorphisms by \ref{LemmaZmu1_A}. This shows that $\pr_1$ in \eqref{EqLemmaZmu1} is an isomorphism. The analogous argument for the projection $\pr_2$ in \eqref{EqLemmaZmu1} starts with an object $(x,(\widetilde{\mcE},\widetilde{\psi}),(\widetilde{\mcE}',\widetilde{\psi}'),\widetilde{\varphi})\in \PHecke_{\gpSch,D,X\times \infty}^{\mcZ\times\mcZ(1)}(S)$ and proves that $\pr_2$ is an isomorphism.
\end{proof}

\subsection{Moduli spaces of shtukas}\label{Shtuka-subsection}
In this section, we recall the preliminaries on shtukas.
\begin{defn}\label{Def_Sht}
Let $D\subset X$ be a finite subscheme. The \emph{stack of global $\gpSch$-shtukas} $\Sht_{\gpSch,D,X^n,I_\bullet}$ \emph{with $n$ legs and $D$-level structure} is the stack fibered in groupoids over the category of $\F_q$-schemes, whose $S$-valued points, for an $\F_q$-scheme $S$, are tuples 
\[
\bigl((x_i)_{i=1\ldots n},\,(\mcE^{(i)},\psi^{(i)})_{i=0\ldots n},\,(\varphi^{(i-1)})_{i=1\ldots n}\,\bigr)\in\Hecke_{\gpSch,D,X^n,I_\bullet}(S)
\]
as in Definition~\ref{DefHecke_nlegs} together with:
\begin{itemize}
\item \emph{shtuka condition}: an isomorphism $\varphi^{(k)}\colon \mcE^{(k)}\isoto{}^{\tau\!}\mcE^{(0)}$ compatible with the $D$-level structure, i.e.~${}^{\tau\!}\psi^{(0)}\circ\varphi^{(k)}=\psi^{(k)}$.
\end{itemize}
Here the superscript ${}^{\tau\!}$ refers to the pullback under the absolute $q$-Frobenius $\tau=\Frob_{q,S}$ of $S$ as defined in \S \ref{subsec-notations}.

We usually drop $(\mcE^{(k)},\psi^{(k)})$ and $\varphi^{(k)}$ from the notation and simply identify $(\mcE^{(k)},\psi^{(k)})$ with ${}^{\tau\!}(\mcE^{(0)},\psi^{(0)})$. Then we can visualize the above data as
\begin{equation}\label{EqDef_Sht}
\xymatrix @C+1pc {
\underline{\mcE}:=\Bigl(\underline x,\, (\mcE^{(0)},\psi^{(0)}) \ar@{-->}[r]^-{\varphi^{(0)}}_-{x_i\colon i\in I_1} & (\mcE^{(1)},\psi^{(1)}) \ar@{-->}[r]^-{\varphi^{(1)}}_-{x_i\colon i\in I_2} & \ldots \ar@{-->}[r]^-{\varphi^{(k-1)}}_-{x_i\colon i\in I_k} & {}^{\tau\!}(\mcE^{(0)},\psi^{(0)})\Bigr)
}
\end{equation}
and call it a \emph{global $\gpSch$-shtuka with $D$-level structure over $S$}.

When $D=\varnothing$, we will drop it (and the $\psi^{(i)}$) from the notation.
\end{defn}

\begin{defn}\label{DefIsogGlobalSht}
Consider a scheme $S$ together with legs $x_i\colon S\to X\smallsetminus D$ for $i=1,\ldots,n$ and let $\underline{\mcE},\underline{\widetilde{\mcE}}\in\Sht_{\gpSch,D,X^n,I_\bullet}(S)$ be two global $\gpSch$-shtukas over $S$ with the same legs $x_i$. A \emph{quasi-isogeny} from $\underline{\mcE}$ to $\underline{\widetilde{\mcE}}$ is a tuple of isomorphisms $f^{(i)}\colon(\mcE^{(i)},\psi^{(i)})|_{X_S \setminus N_S}\isoto (\widetilde{\mcE}^{(i)},\widetilde{\psi}^{(i)})|_{X_S \setminus N_S}$ for $i=0,\ldots,n$ of $\gpSch$-bundles with $D$-level structure satisfying $\widetilde{\varphi}^{(i)}\circ f^{(i)}=f^{(i+1)}\circ\varphi^{(i)}$ for $i=0,\ldots,n-1$ and $\widetilde{\varphi}^{(k)}\circ f^{(k)}={}^{\tau\!}f^{(0)}\circ\varphi^{(k)}$, where $N\subset X$ is some proper closed subscheme. We denote the \emph{group of quasi-isogenies} from $\underline{\mcE}$ to itself by $\QIsog_S(\underline{\mcE})$.
\end{defn}

\begin{remark}\label{RemIsogGlobalSht}
If $S=\Spec\overline{\F}_q$ then we write $I_{\underline{\mcE}}(\FcnFld):=\mathrm{QIsog}_{\overline{{\F}}_\infty}(\underline{\mcE})$. We do not need the following result in this article. However, it justifies the notation. There is a linear algebraic group $I_{\underline{\mcE}}$ over $\FcnFld$ such that $I_{\underline{\mcE}}(\FcnFld)=\mathrm{QIsog}_{\overline{{\F}}_\infty}(\underline{\mcE})$. This can be proven as in \cite{AH_CMotives, AH_LRConj} by noting that $\underline{\mcE}$ gives rise to a tensor functor from $\Rep_\FcnFld \gengpSch$ to the category of $C$-motives (with $C=X$) as in \cite[(4.3)]{AH_CMotives}. Here $\Rep_\FcnFld \gengpSch$ is the neutral Tannakian category of algebraic representations of $\gengpSch$ in finite dimensional $\FcnFld$-vector spaces. By Tannakian duality the $\gpSch$-shtuka $\underline{\mcE}$ corresponds to a homomorphism $h=h_{\underline{\mcE}}$ from the Tannakian fundamental groupoid $\mathfrak{P}$ of the category of $C$-motives to the neutral groupoid of $\gengpSch$. Then $I_h:=\Aut(h)$ is the linear algebraic group over $\FcnFld$ defined by
\[
I_h(R) :=\Aut(h)(R) := \bigl\{\, g\in \gengpSch(\FcnFld^\alg \otimes_\FcnFld R)\colon \Int_g \circ h = h\,\bigr\}
\]
for $\FcnFld$-algebras $R$. Indeed, $I_h$ and $I_{\underline{\mcE}}$ are equal by Lemma \ref{lemma-ADLV-identied-Xphi}.
\end{remark}

By Definition~\ref{Def_Sht}, the stack $\Sht_{\gpSch,D,X^n,I_\bullet}$ is defined as the fiber product
\begin{equation} \label{Eq_DiagSht}
\xymatrix @C+4pc {
\Sht_{\gpSch,D,X^n,I_\bullet} \ar[r] \ar[d] &  \Hecke_{\gpSch,D,X^n,I_\bullet} \ar[d]^{(\mcE^{(0)}, \mcE^{(k)})} \\
\Bun_{\gpSch,D} \ar[r]^-{\id \times \Frob_q} & \Bun_{\gpSch,D} \times_{\F_q} \Bun_{\gpSch,D}
}
\end{equation}

There is a map $\Sht_{\gpSch,D,X^n,I_\bullet}\to (X\smallsetminus D)^n$ given by $\underline{\mcE}\mapsto (x_i)_{i=1,\ldots,n}$.
\begin{thm}\label{ThmSht}
\begin{enumerate}
\item \label{ThmSht_A} The stack $\Sht_{\gpSch,D,X^n,I_\bullet}$ is an ind-Deligne-Mumford stack, locally of ind-finite type and ind-separated over $(X\smallsetminus D)^n$. 
\item \label{ThmSht_B} If $D\subset D'\subset X$ are proper closed subschemes, then the natural morphism $\Sht_{\gpSch,D',X^n,I_\bullet} \to \Sht_{\gpSch,D,X^n,I_\bullet}\times_{(X\smallsetminus D)^n} (X\smallsetminus D')^n$ is finite, \'etale, surjective and a torsor for the finite abstract group $\ker\bigl(\gpSch(D')\to\gpSch(D)\bigr)$.
\end{enumerate}
\end{thm}

\begin{proof}
This was proven in \cite[Theorem~3.15]{AH_Unif} building on earlier work for constant split $\gpSch$ of Varshavsky~\cite{Varshavsky04} and Lafforgue~\cite{Lafforgue12}. 
\end{proof}

The stacks $\Sht_{\gpSch,D,X^n,I_\bullet}$ and the theorem generalize results on the moduli space of $F$-sheaves $FSh_{D,r}$ which were considered by Drinfeld~\cite{Drinfeld-moduli-Fsheaves} and Lafforgue~\cite{LafforgueL02} in their proof of the Langlands correspondence for $\gpSch=\GL_2$ (resp.\ $\gpSch=\GL_r$), and its generalization $FBun$ by Varshavsky~\cite{Varshavsky04}. It likewise generalizes the moduli stacks $\mcE\ell\ell_{X,\mathscr{D},I}$ of Laumon, Rapoport and Stuhler~\cite{Laumon-Rapoport-Stuhler}, their generalizations by L.~Lafforgue~\cite{Lafforgue-Ramanujan}, Lau~\cite{Lau07}, Ng\^o~\cite{Ngo06} and Spie{\ss}~\cite{Spiess10}, the spaces $\text{Cht}_{\underline\lambda}$ of Ng\^o and Ng\^o Dac~\cite{NgoNgo,NgoDac13}, and the spaces ${\rm AbSh}^{r,d}_H$ of the first author~\cite{HartlAbSh}.

\begin{cor}\label{CorAutFinite}
Let $\underline{\mcE}$ and $\underline{\mcE}'$ be global $\gpSch$-shtukas with the same legs in $\Sht_{\gpSch,D,X^n,I_\bullet}(S)$ over an $\F_q$-scheme $S$. Then the sheaf of sets on $S_\fpqc$ given by $\underline\Isom_S(\underline{\mcE},\underline{\mcE}')\colon T\mapsto\Isom_T(\underline{\mcE}_T,\underline{\mcE}'_T)$ is representable by a scheme, which is finite and unramified over $S$. In particular, the (abstract) group of automorphisms $\Aut_S(\underline{\mcE})$ of $\underline{\mcE}$ over $S$ is finite.
\end{cor}
\begin{proof}
Since $\Sht_{\gpSch,D,X^n,I_\bullet}$ is an ind-separated ind-Deligne-Mumford stack, its diagonal is unramified and proper. The base change of the diagonal under the morphism $(\underline{\mcE},\underline{\mcE}')\colon S \to \Sht_{\gpSch,D,X^n,I_\bullet}\times_{\F_q} \Sht_{\gpSch,D,X^n,I_\bullet}$ equals $\underline\Isom_S(\underline{\mcE},\underline{\mcE}')$, which is hence an algebraic space unramified and proper over $S$. In particular it is finite and affine over $S$ and hence a scheme; c.f.~\cite[Lemma~4.2]{Laumon-Moret-Bailly}.
\end{proof}

\begin{defn} \label{Def_ShtBounded}
Let $\mcZ$ be a bound in $\Gr_{\gpSch,X^n,I_\bullet}$ and let $\widetilde{X}_\mcZ$ be its reflex scheme. We define the closed substack $\Sht_{\gpSch,D,X^n,I_\bullet}^{\mcZ}$ of $\Sht_{\gpSch,D,X^n,I_\bullet}\times_{X^n} \widetilde{X}_\mcZ$ as the base change of $\Hecke_{\gpSch,D,X^n,I_\bullet}^{\mcZ}$ under the morphism $\Sht_{\gpSch,D,X^n,I_\bullet}\to \Hecke_{\gpSch,D,X^n,I_\bullet}$ from \eqref{Eq_DiagSht}. In terms of the data from \eqref{EqDef_Sht} this means that the boundedness is tested after trivializing $L^+_{\underline x}({}^{\tau\!}\mcE^{(0)})$. We say that a global $\gpSch$-shtuka $\underline{\mcE}\in(\Sht_{\gpSch,D,X^n,I_\bullet}\times_{X^n} \widetilde{X}_\mcZ)(S)$ is \emph{bounded by $\mcZ$} if it belongs to $\Sht_{\gpSch,D,X^n,I_\bullet}^{\mcZ}(S)$. 
\end{defn}

\begin{thm} \label{Thm_ShtBounded}
The stack $\Sht_{\gpSch,D,X^n,I_\bullet}^\mcZ$ is a Deligne-Mumford stack locally of finite type and separated over $(X\smallsetminus D)^n \times_{X^n} \widetilde{X}_\mcZ$. 
\end{thm}

In the following we recall that the Beilinson-Drinfeld Grassmannian (and a bound therein) is a local model for the moduli spaces of (bounded) shtukas.

\begin{defn}\label{DefLocalModel}
Let $D\subset X$ be a finite subscheme. We define $\widetilde{\Sht}_{\gpSch,D,X^n,I_\bullet}$ as the stack, whose $S$-valued points, for an $\F_q$-scheme $S$, are tuples 
\[
\bigl(\underline x=(x_i)_{i=1\ldots n},\,(\mcE^{(i)},\psi^{(i)})_{i=0\ldots n},\,(\varphi^{(i-1)})_{i=1\ldots n},\varphi^{(k)}\,\bigr)\in\Sht_{\gpSch,D,X^n,I_\bullet}(S)
\]
as in Definition~\ref{Def_Sht} together with:
\begin{itemize}
\item a trivialization $\hat{\epsilon}\colon L_{\underline x}\mcE^{(k)} \isoto (L^+_{\underline x}\gpSch)_S$.
\end{itemize}
Note that $\mcL^+_{X^n}\gpSch$ is the automorphism group of the trivial $L^+_{\underline x}\gpSch$-bundle.

If $\mcZ$ is a bound in $\Gr_{\gpSch,X^n,I_\bullet}$ and $Z$ is its representative over the reflex scheme $\widetilde{X}_{\mcZ}$, we define $\widetilde{\Sht}_{\gpSch,D,X^n,I_\bullet}^\mcZ:= \widetilde{\Sht}_{\gpSch,D,X^n,I_\bullet} \times_{\Sht_{\gpSch,D,X^n,I_\bullet}} \Sht_{\gpSch,D,X^n,I_\bullet}^\mcZ$.

Forgetting the trivialization $\hat{\epsilon}$ (respectively the isomorphism $\varphi^{(k)}\colon \mcE^{(k)} \isoto {}^{\tau\!}\mcE^{(0)}$) defines the $X^n$-morphisms $\pi_1$ (respectively $\pi_2$) in the following diagrams
\begin{equation}\label{EqLocalModel}
\xymatrix @C-2pc {
& \widetilde{\Sht}_{\gpSch,D,X^n,I_\bullet} \ar[dl]_{\pi_1} \ar[dr]^{\pi_2} & & & & \widetilde{\Sht}_{\gpSch,D,X^n,I_\bullet}^\mcZ \ar[dl]_{\pi_1} \ar[dr]^{\pi_2} \\
\Sht_{\gpSch,D,X^n,I_\bullet} \ar[dr]_{\overline{\pi}_2} & & \Gr_{\gpSch,X^n,I_\bullet} \ar[dl]^{\overline{\pi}_1} & \quad\text{and}\quad & \Sht_{\gpSch,D,X^n,I_\bullet}^\mcZ \ar[dr]_{\overline{\pi}_2} & & {\enspace Z \qquad} \ar[dl]^{\overline{\pi}_1} \\
& [\Gr_{\gpSch,X^n,I_\bullet} / \mcL^+_{X^n}\gpSch] & & & & {\quad [Z / \mcL^+_{X^n}\gpSch] \quad}
}
\end{equation}
in which the bottom entries are the stack quotients modulo the $\mcL^+_{X^n}\gpSch$-action. These diagrams are called the \emph{local model diagram} for $\Sht_{\gpSch,D,X^n,I_\bullet}$ (respectively for $\Sht_{\gpSch,D,X^n,I_\bullet}^\mcZ$). 
\end{defn}

\begin{thm}\label{ThmLocMod}
In the diagrams~\eqref{EqLocalModel} the $X^n$-morphisms $\pi_2$ and $\overline{\pi}_2$ are formally smooth and both diagrams are cartesian. The $X^n$-morphisms $\pi_1$ and $\overline{\pi}_1$ are $\mcL^+_{X^n}\gpSch$-torsors and have sections \'etale locally on the target. For any such section $s$ of $\pi_1$, the composition $\pi_2\circ s$ is \'etale. In particular, $\Gr_{\gpSch,X^n,I_\bullet}$ is a \emph{local model} for $\Sht_{\gpSch,D,X^n,I_\bullet}$, in the sense that both are isomorphic locally for the \'etale topology. Likewise, $Z$ is a local model for $\Sht_{\gpSch,D,X^n,I_\bullet}^\mcZ$.

If $(c_i)_i$ is a tuple of integers as in Lemma~\ref{LemmaLoopActionOnGr} for which the $\mcL^+_{X^n}\gpSch$-action on the representatives of the bound $\mcZ$ factors through $\mcL^{+,(c_i)_i}_{X^n}\gpSch$, then the morphism $\overline{\pi}_2$ factors through the morphism
\begin{equation}\label{EqThmLocMod}
\Sht_{\gpSch,D,X^n,I_\bullet}^\mcZ \to [Z / \mcL^{+,(c_i)_i}_{X^n}\gpSch].
\end{equation}
The latter is smooth of relative dimension $(\sum_i c_i)\cdot\dim\gengpSch$, which is also equal do the relative dimension of the group scheme $\mcL^{+,(c_i)_i}_{X^n}\gpSch$ over $X^n$; see Lemma~\ref{LemmaGlobalLoopIsqc}\ref{LemmaGlobalLoopIsqc_C}.
\end{thm}

\begin{proof}
This goes back to Varshavsky~\cite[Theorem~2.20]{Varshavsky04} for constant split reductive $\gpSch$ and was reproven by Lafforgue~\cite[Proposition~2.11]{Lafforgue12} and generalized by Arasteh Rad and Habibi~\cite[Theorem~3.2.1]{Arasteh-Habibi} to smooth affine group schemes $\gpSch$ over $X$. The proof relies on the observation that in the diagramm~\eqref{Eq_DiagSht} which defines $\Sht_{\gpSch,D,X^n,I_\bullet}$
\begin{equation*}
\xymatrix @C+4pc {
\Sht_{\gpSch,D,X^n,I_\bullet} \ar[r] \ar[d] &  \Hecke_{\gpSch,D,X^n,I_\bullet} \ar[d]^{(\mcE^{(0)}, \mcE^{(k)})} & \Gr_{\gpSch,X^n,I_\bullet} \ar[l] \ar[d] \\
\Bun_{\gpSch,D} \ar[r]^-{\id \times \Frob_q} & \Bun_{\gpSch,D} \times_{\F_q} \Bun_{\gpSch,D} & \Bun_{\gpSch,D} \ar[l]_-{\id \times \gpSch}
}
\end{equation*}
both horizontal morphisms in the bottom row have the same differential $(\id,0)$.
\end{proof}

\begin{defn}\label{Def_Sht2legs}
Let $\infty\in X(\F_q)$ and let $D\subset X\smallsetminus\{\infty\}$ be a proper closed subscheme. The \emph{stack of global $\gpSch$-shtukas} $\Sht_{\gpSch,D,X\times\infty}$ 
\emph{with two legs, one fixed at $\infty$}, is the stack, whose $S$-valued points, for $S$ an $\F_q$-scheme, are tuples $\bigl(x,\,(\mcE,\psi),(\mcE',\psi'),\,\varphi,\varphi'\,\bigr)$ where
\begin{itemize}
\item $x \in (X\smallsetminus D)(S)$ is a section, called a \emph{leg}, 
\item $(\mcE,\psi),(\mcE',\psi')$ are objects in $\Bun_{\gpSch,D}(S)$, and
\item $\varphi\colon \mcE|_{{X_S}\smallsetminus\Gamma_{x}}\isoto\mcE'|_{{X_S}\smallsetminus\Gamma_{x}}$ and $\varphi'\colon \mcE'|_{(X\smallsetminus\{\infty\})_S}\isoto {}^{\tau\!}\mcE|_{(X\smallsetminus\{\infty\})_S}$ are isomorphisms preserving the $D$-level structures, i.e.~$\psi'\circ\varphi|_{D_S}=\psi$ and ${}^{\tau\!}\psi\circ\varphi'|_{D_S}=\psi'$.
\end{itemize}
We can visualize the above data as
\begin{equation}
\xymatrix @C+1pc {
(\mcE,\psi) \ar@{-->}[r]^\varphi_{x} & (\mcE',\psi') \ar@{-->}[r]^{\varphi'}_\infty & {}^{\tau\!}(\mcE,\psi) \,.
}
\end{equation}
When $D=\varnothing$, we will drop it from the notation. The projection map 
\[\bigl(x,\,(\mcE,\psi),(\mcE',\psi'),\,\varphi,\varphi'\,\bigr)\mapsto x\] defines 
a morphism $\Sht_{\gpSch,D,X\times\infty}\to X\smallsetminus D$. 

For a bound $\mcZ$ in $\Gr_{\gpSch,X\times\infty}$, we also define the stack $\Sht_{\gpSch,D,X\times\infty}^{\mcZ}$ of global $\gpSch$-shtukas $\underline{\mcE}\in\Sht_{\gpSch,D,X\times\infty}$ that are bounded by $\mcZ$ as in Definition~\ref{Def_ShtBounded}.
\end{defn}

\begin{thm} \label{Thm_Sht2legsBounded}
The stack $\Sht_{\gpSch,D,X\times\infty}^\mcZ$ is a Deligne-Mumford stack locally of finite type and separated over $X\smallsetminus D$. Moreover, the stack $\Sht_{\gpSch,D,X\times\infty}$ is an ind-Deligne-Mumford stack, locally of ind-finite type and ind-separated over $X\smallsetminus D$. 
\end{thm}

\begin{prop}\label{PropLocMod2}
Let $\mcZ$ be a bound in $\Gr_{\gpSch,X\times\infty}$ and let $Z$ be its representative over the reflex scheme $\widetilde{X}_{\mu,\beta}$. Then $Z$ is a local model for $\Sht_{\gpSch,D,X\times\infty}^{\mcZ}$, i.e.~both are isomorphic locally for the \'etale topology. 
\end{prop}

\begin{proof}
This follows from Theoreom~\ref{ThmLocMod}.
\end{proof}

\subsection{Interpretation in terms of chains}\label{subsec:Chains}
In this section let $\infty\in X(\F_q)$ and let $\beta\in L_\infty\gpSch(\F_\beta)$ be the element from \S\,\ref{Def_beta}. Let $\mu\in X_*(T)$ and let $\mcZ(\mu,\beta)$ be the bound in $\Gr_{\gpSch,X\times\infty}$ from Definition~\ref{DefZmubeta}.
In Corollary~\ref{CorShtWithChains} we shall give an interpretation of $\Sht_{\gpSch,D,X\times\infty}^{\mcZ(\mu,\beta)}$ in terms of chains of $\gpSch$-bundles. This relates our global $\gpSch$-shtukas bounded by $\mcZ(\mu,\beta)$ to the objects defined by ``sequences'', namely the elliptic sheaves of Drinfeld~\cite{Drinfeld-commutative-subrings,Blum-Stuhler} (see Example~\ref{ExDrinfeld}), the $\mathscr{D}$-elliptic sheaves of Laumon, Rapoport and Stuhler~\cite{Laumon-Rapoport-Stuhler} (see Example~\ref{ExLRS}), and the abelian $\tau$-sheaves of the first author~\cite{HartlAbSh} (see Example~\ref{ExAbelianSheaves}). Recall the Frobenius map $\tau_{\gengpSch_\infty}$ of the group $\gengpSch_\infty$ over $\FcnFld_\infty$ from \eqref{EqTau_G}.

\begin{defn} \label{DefCBun}
Let $D\subset X\smallsetminus\{\infty\}$ be a closed subscheme. Let $\CBun_{\gpSch,D,\beta}$ 
be the stack over $\F_\beta$ classifying chains 
\begin{equation}\label{chain-eqn}
    \ldots \xdashleftarrow{\Pi_{-1}}(\mcE_{-1},\psi_{-1})\xdashleftarrow{\Pi_0}(\mcE_{0},\psi_{0})\xdashleftarrow{\Pi_1}(\mcE_1,\psi_1)\xdashleftarrow{\Pi_2}\ldots
\end{equation}
of $\gpSch$-bundles with $D$-level structure over $S$ such that for all $i\in\Z$ the modifications are isomorphisms of $\gpSch$-bundles $\Pi_i\colon (\mcE_i,\psi_i)|_{(X\smallsetminus\{\infty\})_S} \isoto (\mcE_{i-1},\psi_{i-1})|_{(X\smallsetminus\{\infty\})_S}$ bounded by $\mcZ(\tau_{\gengpSch_\infty}^{i-1}(\beta)^{-1})$, which preserve the $D$-level structures. 

Morphisms $(f_i)_i\colon(\mcE_i,\psi_i,\Pi_i)_i\to (\mcE'_i,\psi'_i,\Pi'_i)_i$ in $\CBun_{\gpSch,D,\beta}$ are tuples of isomorphisms $f_i\colon(\mcE_i,\psi_i)\isoto(\mcE'_i,\psi'_i)$ in $\Bun_{\gpSch,D}(S)$ satisfying $\Pi'_i\circ f_i = f_{i-1}\circ \Pi_i$ for all $i$.
\end{defn}

\begin{prop}\label{PropBunWithChains}
For any fixed $j\in\Z$, 
the functor 
\[
\mathfrak{pr}_j\colon \CBun_{\gpSch,D,\beta} \isoto \Bun_{\gpSch,D}\times_{\F_q} \Spec\F_\beta \,,\quad (\mcE_i,\psi_i,\Pi_i) \longmapsto (\mcE_j,\psi_j)
\]
is an isomorphism of stacks. In particular, $\CBun_{\gpSch,D,\beta}$ is a smooth Artin-stack locally of finite type over $\F_\beta$.
\end{prop}
\begin{proof}
Since the $D$-level structure is outside $\infty$, and the modifications $\Pi_i$ take place at $\infty$, the $D$-level structure is preserved. We will ignore the level structure in the rest of the proof. In the following, we prove the result for $j=0$. The same proof holds for arbitrary $j\in\Z$.

    (1) Full faithfulness of the functor $\mathfrak{pr}_0$: Let $(\mcE_i,\Pi_i)_i$ and $(\mcE_i',\Pi_i')_i$ be two chains over $S$, and let $f_0:\mcE_0\isoto\mcE_0'$ be an isomorphism in $\Bun_\gpSch(S)$. For any $i\geq 0$ (and similarly for $i\leq 0$), the map $f_0$ induces an isomorphism outside $\infty$
    \begin{equation}\label{Eq2.3.17}
       \mcE_i|_{(X\smallsetminus\{\infty\})_S}\cong \mcE_0|_{(X\smallsetminus\{\infty\})_S}\xrightarrow[f_0]{\;\sim\,\,}\mcE_0'|_{(X\smallsetminus\{\infty\})_S}\cong \mcE_i'|_{(X\smallsetminus\{\infty\})_S}.
    \end{equation}
We apply the functor $L_\infty^+$ from Definition~\ref{DefLoopGpAtDelta} and Example~\ref{ExDivisors}(b), and choose an \'etale covering $S'\to S$, over which trivializations exist that form the vertical maps in the following diagram 
\begin{equation}
\begin{tikzcd}[column sep=12ex]
    L_\infty(\mcE_i)_{S'}\arrow[]{r}{\Pi_1\circ\ldots\circ\Pi_i}[swap]{\sim}\arrow[]{d}{\alpha_i}[swap]{\sim} & L_\infty(\mcE_0)_{S'}\arrow[]{r}{L_\infty(f_0)}[swap]{\sim}\arrow[]{d}{\alpha_0}[swap]{\sim} & L_\infty(\mcE_0')_{S'}\arrow[]{r}{(\Pi_1'\circ\ldots\circ\Pi_i')^{-1}}[swap]{\sim}\arrow[]{d}{\alpha_0'}[swap]{\sim}&L_\infty(\mcE_i')_{S'}\arrow[]{d}{\alpha_i'}[swap]{\sim}\\
    (L_\infty\gengpSch)_{S'}\arrow[]{r}{\sim}&(L_\infty\gengpSch)_{S'}\arrow[]{r}{\sim}&(L_\infty\gengpSch)_{S'}\arrow[]{r}{\sim}&(L_\infty\gengpSch)_{S'}
\end{tikzcd}
\end{equation}
Since $\tau_{\gengpSch_\infty}^k(\beta)\cdot L^+_\infty\gpSch(S') = L^+_\infty\gpSch(S')\cdot\tau_{\gengpSch_\infty}^k(\beta)$ by \S\,\ref{Def_beta}, the three horizontal isomorphisms in the bottom row lie in $\tau_{\gengpSch_\infty}^{0}(\beta)^{-1}\cdot\ldots\cdot \tau_{\gengpSch_\infty}^{i-1}(\beta)^{-1} L^+_\infty\gpSch(S')$ and $L^+_\infty\gpSch(S')$ and $\tau_{\gengpSch_\infty}^{i-1}(\beta)\cdot\ldots\cdot \tau_{\gengpSch_\infty}^{0}(\beta) L^+_\infty\gpSch(S')$, respectively. Therefore, their composition lies in $L^+_\infty\gpSch(S')$ and yields an isomorphism 
    $L_\infty^+(\mcE_i)_{S'}\isoto L_\infty^+(\mcE_i')_{S'}$, 
which descends to $S$ and glues with \eqref{Eq2.3.17} to an isomorphism $f_i\colon\mcE_i \isoto \mcE'_i$ on all of $X_S$ by Lemma~\ref{LemmaBL}. The $f_i$ define an isomorphism of chains $(f_i)_i\colon (\mcE_i,\Pi_i)_i\isoto (\mcE'_i,\Pi'_i)_i$.

(2) Essential surjectivity of the functor $\mathfrak{pr}_0$: Fix any $\mcE\in \Bun_\gpSch(S)$. We want to construct a chain of the form \eqref{chain-eqn} with $\mcE_0=\mcE$. 
For every $i\in\Z_{\ge 0}$, we can construct $\mcE_{i+1}$ from $\mcE_i$ via Lemma~\ref{LemmaBL} by glueing 
$\overset{\circ}{\mcE}_{i+1}:=\mcE_i|_{(X\smallsetminus\{\infty\})_S}$ 
with $\mcL_{i+1}$, where $\mcL_{i+1}$ is constructed from $L_\infty^+(\mcE_i)$ as follows: choose an \'etale covering $S'\to S$ and a trivialization $\alpha_i: L_\infty^+(\mcE_i)_{S'}\isoto(L^+_\infty\gpSch)_{S'}$ over $S'$. Let $h_i\in L^+_\infty\gpSch(S'')$, for $S'':=S'\times_S S'$, be given by $h_i:=\pr_2^*\alpha_i\circ\pr_1^*\alpha_i^{-1}$, where $\pr_k:S''\to S'$ is the projection onto the $k$-th factor, for $k=1,2$. Let $(\mcL_{i+1})_{S'}:=(L^+_\infty\gpSch)_{S'}$, let $\alpha_{i+1}:=\id \colon (\mcL_{i+1})_{S'}\isoto(L^+_\infty\gpSch)_{S'}$, and let 
\[
h_{i+1}:=\tau_{\gengpSch_\infty}^{i}(\beta)\cdot h_i\cdot \tau_{\gengpSch_\infty}^{i}(\beta)^{-1}\in \tau_{\gengpSch_\infty}^{i}(\beta)\cdot L^+_\infty\gpSch(S'')\cdot \tau_{\gengpSch_\infty}^{i}(\beta)^{-1}=L^+_\infty\gpSch(S'').
\]
Then $h_{i+1}$ defines a descent datum on $(\mcL_{i+1})_{S'}$ and the latter descends to an $L^+_\infty\gpSch$-bundle $\mcL_{i+1}$ on $S$. Moreover, 
\begin{equation}\label{EqDefPi_i}
    \Pi_{i+1}:=\alpha_i^{-1}\circ\tau_{\gengpSch_\infty}^{i}(\beta)^{-1}\circ\alpha_{i+1}: L_\infty(\mcL_{i+1})_{S'}=(L_\infty\gpSch)_{S'}\isoto L_\infty(\mcE_i)_{S'}
\end{equation}
satisfies $\pr_2^*\Pi_{i+1}=\pr_1^*\alpha_i^{-1}\circ h_i^{-1}\tau_{\gengpSch_\infty}^i(\beta)^{-1} h_{i+1}\circ\pr_1^*\alpha_{i+1}=\pr_1^*\Pi_i$ and descends to an isomorphism 
\begin{equation}\label{PropBunWithChainsEqPi}
\Pi_{i+1}: L_\infty(\mcL_{i+1}) \isoto L_\infty(\mcE_i) 
\end{equation}
over $S$. By Lemma~\ref{LemmaBL} we glue $\mcE_{i+1}$ from $\overset{\circ}{\mcE}_{i+1}:=\mcE_i|_{(X\smallsetminus\{\infty\})_S}$ and $\mcL_{i+1}$ via the isomorphism
\[
\Pi_{i+1}^{-1} \colon L_\infty(\overset{\circ}{\mcE}_{i+1})=L_\infty(\mcE_i|_{(X\smallsetminus\{\infty\})_S})=L_\infty(\mcE_i)\xrightarrow{\Pi_{i+1}^{-1}}L_\infty(\mcL_{i+1}).
\]
The isomorphism $\Pi_{i+1}$ from \eqref{PropBunWithChainsEqPi} extends to the isomorphism 
\[
\Pi_{i+1}=\id\colon \mcE_{i+1}|_{(X\smallsetminus\{\infty\})_S} = \overset{\circ}{\mcE}_{i+1} = \mcE_i|_{(X\smallsetminus\{\infty\})_S} \isoto \mcE_i|_{(X\smallsetminus\{\infty\})_S}.
\]
We proceed analogously for all $i<0$. This construction defines a chain \eqref{chain-eqn} in $\CBun_{\gengpSch,\beta}(S)$. 
\end{proof}

Let $\mu\in X_*(T)$. Also the bound $\mcZ(\mu,\beta)$ from Definition~\ref{DefZmubeta} can be interpreted in terms of chains. Let $\widetilde{X}_{\mu,\beta}$ be its reflex scheme.

\begin{prop}\label{PropBoundWithChains}
The bound $\mcZ(\mu,\beta)$ is represented by the closed subscheme of $\Gr_{\gpSch,X\times\infty}\times_X \widetilde{X}_{\mu,\beta}$ whose $S$-valued points are 
tuples $\bigl(x,\,(\mcE_i,\Pi_i,\mcE_i'',\Pi_i'',\diagphi_i)_{i\in\Z},\,\epsilon\bigr)$ where
\begin{itemize}
    \item $x \colon S\to \widetilde{X}_{\mu,\beta}\to X$ is a leg,
    \item together with a commutative diagram,
\begin{equation}\label{EqPropBoundWithChains}
\begin{tikzcd}
    \ldots  & \mcE_{-1} \arrow[dashed]{l}[swap]{\Pi_{-1}} \arrow[dashed]{ld}[swap]{\diagphi_{-1}} & \mcE_0\arrow[dashed]{l}[swap]{\Pi_0}\arrow[dashed]{ld}[swap]{\diagphi_0} & \mcE_1\arrow[dashed]{l}[swap]{\Pi_1}\arrow[dashed]{ld}[swap]{\diagphi_1}& \arrow[dashed]{l}[swap]{\Pi_2} \arrow[dashed]{ld}[swap]{\diagphi_2}\ldots\\
    \ldots & \mcE''_{-1}\arrow[dashed]{l}{\Pi''_{-1}}&\mcE''_0\arrow[dashed]{l}{\Pi''_0}&\mcE''_1\arrow[dashed]{l}{\Pi''_1} & \arrow[dashed]{l}{\Pi''_2}\ldots
\end{tikzcd},
\end{equation}
where the $\diagphi_i$ are isomorphisms outside the graph $\Gamma_x$ bounded by $\mu$ in the sense of Definition~\ref{Def_BoundBy_mu} and Remark~\ref{Rem_HeckeBounded}, and the $\Pi_i$ and $\Pi_i''$ are isomorphisms outside $\infty$ with $\Pi_i$ bounded by $\mcZ(\tau_{\gengpSch_\infty}^{i-1}(\beta)^{-1})$ and $\Pi_i''$ bounded by $\mcZ(\tau_{\gengpSch_\infty}^{i}(\beta)^{-1})$,
\item and $\epsilon\colon\mcE''_0\isoto \gpSch\times_XX_S$ is an isomorphism of $\gpSch$-bundles.
\end{itemize}
\end{prop}

\begin{proof}
Let $Z(\mu,\beta)$ be the representative over $\widetilde{X}_{\mu,\beta}$ of the bound $\mcZ(\mu,\beta)$. We will work over the universal base scheme $S:=Z(\mu,\beta)$. It is defined as the scheme theoretic closure in $\Gr_{\gpSch,X\times\infty} \times_X \widetilde{X}_{\mu,\beta}$ of its restriction $\widetilde{S}:=Z(\mu,\beta)\times_X (X\smallsetminus\{\infty\})$ outside $\infty$. Therefore, $\widetilde{S}\subset S$ is schematically dense by \cite[I, Corollaire~9.5.11]{EGA1}. The restriction $\widetilde{S}$ classifies tuples $(x,\mcE_0,\mcE'_0,\mcE''_0,\varphi_0,\varphi'_0,\epsilon)\in\Gr_{\gpSch,X\times\infty}$ as in Definition~\ref{DefGr2legs}, where $\varphi_0$ is bounded by $\mcZ^{\leq\mu}$ and $\varphi'_0$ is bounded by $\mcZ(\beta)$. Under the isomorphism $\mathfrak{pr}_0$ from Proposition~\ref{PropBunWithChains}, we can uniquely extend $\mcE_0$, $\mcE'_0$, and $\mcE''_0$ to chains of $\gpSch$-bundles to obtain the rows of the following commutative diagram
\begin{equation} \label{EqPropBoundWithChains2}
\begin{tikzcd}
    \ldots & \mcE_{-1}\arrow[dotted]{d}{\varphi_{-1}} \arrow[dashed]{l}[swap]{\Pi_{-1}} & \mcE_0\arrow[dashed]{l}[swap]{\Pi_0}\arrow[dashed]{d}{\varphi_0}&\mcE_1\arrow[dashed]{l}[swap]{\Pi_1}\arrow[dotted]{d}{\varphi_1}&\arrow[dashed]{l}[swap]{\Pi_2}\ldots\\
    \ldots & \mcE'_{-1}\arrow[dotted]{d}{\varphi'_{-1}}\arrow[dashed]{l}[swap]{\Pi'_{-1}} & \mcE'_0\arrow[dashed]{l}[swap]{\Pi'_0}\arrow[dashed]{d}{\varphi'_0} & \mcE'_1\arrow[dashed]{l}[swap]{\Pi'_1}\arrow[dotted]{d}{\varphi'_1}&\arrow[dashed]{l}[swap]{\Pi'_2}\ldots\\
    \ldots & \mcE''_{-1}\arrow[dashed]{l}[swap]{\Pi''_{-1}}&\mcE''_0\arrow[dashed]{l}[swap]{\Pi''_0}&\mcE''_1\arrow[dashed]{l}[swap]{\Pi''_1}&\arrow[dashed]{l}[swap]{\Pi''_2} \ldots
\end{tikzcd}
\end{equation}
with $\Pi_i$ and $\Pi'_i$ bounded by $\mcZ(\tau_{\gengpSch_\infty}^{i-1}(\beta)^{-1})$, and $\Pi''_i$ bounded by $\mcZ(\tau_{\gengpSch_\infty}^{i}(\beta)^{-1})$. In this diagram the dotted maps $\varphi_i$ and $\varphi'_i$ for $i\ne 0$ are the induced isomorphisms outside $\Gamma_x\cup\{\infty\}$. We claim for all $i\in\Z$
\begin{enumerate}
\item \label{ProofPropBoundWithChains_A}
that $\Pi''_i\circ\varphi'_i\colon\mcE'_i|_{(X\smallsetminus\{\infty\})_S}\isoto\mcE''_{i-1}|_{(X\smallsetminus\{\infty\})_S}$ extends to an isomorphism over all of $X_S$, 
\item \label{ProofPropBoundWithChains_B}
and that $\varphi_i$, and hence also $\diagphi_i:=\Pi''_i\circ \varphi'_i \circ \varphi_i$ is bounded by $\mcZ^{\leq\mu}$. 
\end{enumerate} 
We prove the claim for a fixed $i>0$. For $i<0$ we can argue analogously. We consider the associated $L^+_\infty\gpSch$-bundles $L_\infty^+(\mcE^\Box_j)$ for $\Box\in\{\varnothing,\,',\,''\}$. Over some \'etale covering $S'\to S$ we choose trivializations
\[
\alpha^\Box_j\colon L_\infty^+(\mcE^\Box_j)_{S'}\isoto(L^+_\infty\gpSch)_{S'}
\]
for all $j\in\{0,\ldots,i\}$. 

Over $\widetilde{S}:=Z(\mu,\beta)\times_X (X\smallsetminus\{\infty\})$, the divisors $\Gamma_x$ and $\infty\times_{\F_q} \widetilde{S}$ are disjoint. There, the isomorphisms $\alpha^\Box_{j-1}\circ L_\infty(\Pi^\Box_j)\circ(\alpha^\Box_{j})^{-1}: (L_\infty\gengpSch)_{S'}\isoto (L_\infty\gengpSch)_{S'}$ for $\Box\in\{\varnothing,\,'\}$, respectively $\alpha''_{j-1}\circ L_\infty(\Pi''_j)\circ(\alpha''_{j})^{-1}$, respectively $\alpha''_{0}\circ L_\infty(\varphi'_0)\circ(\alpha'_{0})^{-1}$ are given by multiplication on the left with an element of $\tau_{\gengpSch_\infty}^{j-1}(\beta)^{-1}\cdot L^+_\infty\gpSch(S')$, respectively of $\tau_{\gengpSch_\infty}^{j}(\beta)^{-1}\cdot L^+_\infty\gpSch(S')$, respectively of $\beta\cdot L^+_\infty\gpSch(S')$; see \eqref{EqDefPi_i} in the proof of Proposition~\ref{PropBunWithChains}. We now use $\tau_{\gengpSch_\infty}^j(\beta)^{-1}\cdot L^+_\infty\gpSch(S')\cdot \tau_{\gengpSch_\infty}^j(\beta) = L^+_\infty\gpSch(S')$. 

To prove \ref{ProofPropBoundWithChains_A} we observe that over $\widetilde{S}$
\[
\alpha''_{i-1}\circ L_\infty\bigl(\Pi''_i\circ\varphi'_i)\circ(\alpha'_{i})^{-1}\;=\;\alpha''_{i-1}\circ L_\infty\bigl((\Pi''_1\circ\ldots\circ\Pi''_{i-1})^{-1} \circ\varphi'_0 \circ (\Pi'_{1}\circ\ldots\circ\Pi'_i)\bigr)\circ(\alpha'_{i})^{-1}
\]
is given by multiplication with an element of 
\[
\tau_{\gengpSch_\infty}^{i-1}(\beta)\cdot\ldots\cdot \tau_{\gengpSch_\infty}^{1}(\beta)\cdot\beta\cdot \tau_{\gengpSch_\infty}^{0}(\beta)^{-1} \cdot\ldots\cdot \tau_{\gengpSch_\infty}^{i-1}(\beta)^{-1} \cdot L^+_\infty\gpSch(S')=L^+_\infty\gpSch(S'), 
\]
and hence is an isomorphism at $\infty$. As $\Pi''_i\circ\varphi'_i$ is also an isomorphism outside $\infty$, it is an isomorphism on all of $X_{\widetilde{S}}$.

We consider the tuple $(\infty,\Pi''_i\circ\varphi'_i\colon \mcE'_i \xdashrightarrow[\infty]{} \mcE''_{i-1})$ as an object in $\Hecke_{\gpSch,\varnothing,X}(S)$. The condition that $\Pi''_i\circ\varphi'_i$ is an isomorphism over all of $X$ can be formulated as boundedness by the trivial bound $\mcZ(1):=1\cdot \mcL^+_X\gpSch\subset \Gr_{\gpSch,X}$ given as in Corollary~\ref{CorGrGlobLoc}. Since $\Pi''_i\circ\varphi'_i$ is an isomorphism on $X_{\widetilde{S}}$, the composition $\widetilde{S}\to S\to \Hecke_{\gpSch,\varnothing,X}$ factors through the closed substack $\Hecke_{\gpSch,\varnothing,X}^{\mcZ(1)}$ from Theorem~\ref{Thm_HeckeBounded}. Since $\widetilde{S}$ is schematically dense in $S$, already the morphism $S\to \Hecke_{\gpSch,\varnothing,X}$ factors through $\Hecke_{\gpSch,\varnothing,X}^{\mcZ(1)}$. We conclude that $\Pi''_i\circ\varphi'_i$ is an isomorphism on all of $X_S$.

To prove \ref{ProofPropBoundWithChains_B} we continue to work over $\widetilde{S}$, where the morphism $L^+_\infty(\varphi_0)\colon L^+_\infty\mcE_0\isoto L^+_\infty\mcE'_0$ is an isomorphism. Then
\[
\alpha'_{i}\circ L_\infty\bigl(\varphi_i)\circ(\alpha_{i})^{-1}\;=\;\alpha'_{i}\circ L_\infty\bigl((\Pi'_{1}\circ\ldots\circ\Pi'_i)^{-1} \circ\varphi_0 \circ (\Pi_{1}\circ\ldots\circ\Pi_i)\bigr)\circ(\alpha_{i})^{-1}
\]
is given by multiplication with an element of 
\[
\tau_{\gengpSch_\infty}^{i-1}(\beta)\cdot\ldots\cdot \tau_{\gengpSch_\infty}^{0}(\beta)\cdot\tau_{\gengpSch_\infty}^{0}(\beta)^{-1} \cdot\ldots\cdot \tau_{\gengpSch_\infty}^{i-1}(\beta)^{-1} \cdot L^+_\infty\gpSch(S')=L^+_\infty\gpSch(S').
\]
Therefore, over $\widetilde{S}$ the morphism $\varphi_i$ is an isomorphism at $\infty$. Since the $\Pi_i$ and $\Pi'_i$ are isomorphisms outside $\infty$, we conclude that with $\varphi_0$ also $\varphi_i$ and by \ref{ProofPropBoundWithChains_A} also $\diagphi_i$ are isomorphisms on $X_{\widetilde{S}}\smallsetminus\Gamma_x$ bounded by $\mcZ^{\leq\mu}$ at $\Gamma_x$.

We now consider the tuple $(x,\varphi_i\colon \mcE_i \xdashrightarrow[x,\infty]{} \mcE'_i)$ as an object in $\PHecke_{\gpSch,\varnothing,X\times\infty}(S)$. Our considerations show that the composition $\widetilde{S}\to S\to \PHecke_{\gpSch,\varnothing,X\times\infty}$ factors through the closed substack $\PHecke_{\gpSch,\varnothing,X\times \infty}^{\mcZ(\mu,1)}$ from Theorem~\ref{Thm_HeckeBounded} for the bound $\mcZ(\mu,1)$. Since $\widetilde{S}$ is schematically dense in $S$, already the morphism $S\to \PHecke_{\gpSch,\varnothing,X\times \infty}$ factors through $\PHecke_{\gpSch,\varnothing,X\times \infty}^{\mcZ(\mu,1)}$. But the latter is isomorphic to $\Hecke_{\gpSch,\varnothing,X}^{\mcZ^{\leq\mu}}$ by Lemma~\ref{LemmaZmu1}\ref{LemmaZmu1_B}. We conclude that $\varphi_i$ is bounded by $\mcZ^{\leq\mu}$. Thus diagram \eqref{EqPropBoundWithChains2} can be written in the form \eqref{EqPropBoundWithChains}, such that $\diagphi_i:=\Pi''_i\circ \varphi'_i\circ \varphi_i$ is bounded by $\mu$.
\end{proof}

\begin{cor}\label{CorShtWithChains}
The stack $\Sht_{\gpSch,D,X\times\infty}^{\mcZ(\mu,\beta)}$ is isomorphic to the stack over $\F_q$ whose points over $\F_q$-schemes $S$ 
are tuples $\bigl(x,(\mcE_i,\psi_i,\Pi_i,\diagphi_i)_{i\in\Z}\bigr)$ consisting of \begin{itemize}
    \item one leg $x \colon S\to \widetilde{X}_{\mu,\beta}\to X$ which factors through $X\setminus D$, 
    \item a commutative diagram
\begin{equation} \label{Eq_ShtukaDiagTildephi}
\begin{tikzcd}
    \ldots & (\mcE_{-1},\psi_{-1})\arrow[dashed]{ld}[swap]{\diagphi_{-1}}\arrow[dashed]{l}[swap]{\Pi_{-1}}&(\mcE_0,\psi_0)\arrow[dashed]{l}[swap]{\Pi_0}\arrow[dashed]{ld}[swap]{\diagphi_0}&(\mcE_1,\psi_1)\arrow[dashed]{l}[swap]{\Pi_1}\arrow[dashed]{ld}[swap]{\diagphi_1}&\arrow[dashed]{l}[swap]{\Pi_2}\arrow[dashed]{ld}[swap]{\diagphi_2}\ldots\\
    \ldots & {}^{\tau\!}(\mcE_{-1},\psi_{-1})\arrow[dashed]{l}{{}^{\tau\!}\Pi_{-1}}&{}^{\tau\!}(\mcE_0,\psi_0)\arrow[dashed]{l}{{}^{\tau\!}\Pi_0}&{}^{\tau\!}(\mcE_1,\psi_1)\arrow[dashed]{l}{{}^{\tau\!}\Pi_1}&\arrow[dashed]{l}{{}^{\tau\!}\Pi_2} \ldots
\end{tikzcd}
\end{equation}
\end{itemize}
of $\gpSch$-bundles with $D$-level structures $(\mcE_i,\psi_i)$ on $X_S$ where all the $\diagphi_i$ are isomorphisms outside the graph $\Gamma_x$ bounded by $\mu$, and the $\Pi_i$ are isomorphisms outside $\infty$ bounded by $\mcZ(\tau_{\gengpSch_\infty}^{i-1}(\beta)^{-1})$. 
\end{cor}
\begin{proof}
Let $\underline{\mcE}=(x,(\mcE_0,\psi_0),(\mcE_0',\psi_0'),(\mcE_0'',\psi_0''),\varphi_0,\varphi_0',\varphi_0'')\in \Sht_{\gpSch,D,X\times\infty}^{\mcZ(\mu,\beta)}(S)$ be a global $\gpSch$-shtuka over $S$ bounded by $\mcZ(\mu,\beta)$, where $\varphi_0''\colon (\mcE_0'',\psi_0'') \isoto {}^{\tau\!}(\mcE_0,\psi_0)$ is the isomorphism, which we often supress from the notation. Under the isomorphism $\mathfrak{pr}_0$ from Proposition~\ref{PropBunWithChains}, we can uniquely extend $(\mcE_0,\psi_0)$, $(\mcE'_0,\psi'_0)$, and $(\mcE''_0,\psi''_0)$ to chains of $\gpSch$-bundles to obtain the rows of the following commutative diagram
\begin{equation} \label{EqCorShtWithChains}
\begin{tikzcd}
    \ldots & (\mcE_{-1},\psi_{-1})\arrow[dotted]{d}{\varphi_{-1}}\arrow[dashed]{l}[swap]{\Pi_{-1}}&(\mcE_0,\psi_0)\arrow[dashed]{l}[swap]{\Pi_0}\arrow[dashed]{d}{\varphi_0}&(\mcE_1,\psi_1)\arrow[dashed]{l}[swap]{\Pi_1}\arrow[dotted]{d}{\varphi_1}&\arrow[dashed]{l}[swap]{\Pi_2}\ldots\\
    \ldots & (\mcE'_{-1},\psi'_{-1})\arrow[dotted]{d}{\varphi'_{-1}}\arrow[dashed]{l}[swap]{\Pi'_{-1}}&(\mcE'_0,\psi'_0)\arrow[dashed]{l}[swap]{\Pi'_0}\arrow[dashed]{d}{\varphi'_0}&(\mcE'_1,\psi'_1)\arrow[dashed]{l}[swap]{\Pi'_1}\arrow[dotted]{d}{\varphi'_1}&\arrow[dashed]{l}[swap]{\Pi'_2}\ldots\\
    \ldots & (\mcE''_{-1},\psi''_{-1})\arrow[dotted]{d}{\varphi''_{-1}} \arrow[dashed]{l}[swap]{\Pi''_{-1}}&(\mcE''_0,\psi''_0)\arrow[dashed]{l}[swap]{\Pi''_0}\arrow{d}{\varphi''_0}[swap]{\cong} &(\mcE''_1,\psi''_1)\arrow[dashed]{l}[swap]{\Pi''_1}\arrow[dotted]{d}{\varphi''_1}&\arrow[dashed]{l}[swap]{\Pi''_2}\ldots\\
    \ldots & {}^{\tau\!}(\mcE_{-1},\psi_{-1}) \arrow[dashed]{l}[swap]{{}^{\tau\!}\Pi_{-1}}&{}^{\tau\!}(\mcE_0,\psi_0)\arrow[dashed]{l}[swap]{{}^{\tau\!}\Pi_0}&{}^{\tau\!}(\mcE_1,\psi_1)\arrow[dashed]{l}[swap]{{}^{\tau\!}\Pi_1}&\arrow[dashed]{l}[swap]{{}^{\tau\!}\Pi_2}\ldots
\end{tikzcd}
\end{equation}
with $\Pi_i$ and $\Pi'_i$ bounded by $\mcZ(\tau_{\gengpSch_\infty}^{i-1}(\beta)^{-1})$, and $\Pi''_i$ bounded by $\mcZ(\tau_{\gengpSch_\infty}^{i}(\beta)^{-1})$. In this diagram the dotted maps $\varphi_i$, $\varphi'_i$, and $\varphi''_i$ for $i\ne 0$ are the induced isomorphisms outside $\Gamma_x\cup\{\infty\}$. 

We first show that the isomorphism $\varphi_0''\colon(\mcE''_0,\psi''_0)\isoto{}^{\tau\!}(\mcE_0,\psi_0)$ outside $\infty$ inductively induces isomorphisms $\varphi''_i:={}^{\tau\!}\Pi_{i}^{-1} \circ \varphi''_{i-1}\circ \Pi''_{i}\colon (\mcE''_i,\psi''_i)\isoto {}^{\tau\!}(\mcE_i,\psi_i)$ on all of $X_S$ for all $i>0$ and similarly for $i<0$. After choosing trivializations of $L^+_\infty(\mcE_i)$ and $L^+_\infty(\mcE''_i)$ over an \'etale covering of $S$, the $\Pi_{i}$ are given by multiplication with $\tau_{\gengpSch_\infty}^{i-1}(\beta)^{-1}$, and hence $\Pi''_{i}$ and ${}^{\tau\!}\Pi_{i}$ are given by multiplication with $\tau_{\gengpSch_\infty}^{i}(\beta)^{-1}$. 
Therefore, the $\varphi''_i$ are given by multiplication with an element of $\tau_{\gengpSch_\infty}^{i}(\beta)\cdot L^+_\infty\gpSch\cdot \tau_{\gengpSch_\infty}^{i}(\beta)^{-1}=L^+_\infty\gpSch$ for $i>0$, respectively $\tau_{\gengpSch_\infty}^{i+1}(\beta)^{-1}\cdot L^+_\infty\gpSch\cdot \tau_{\gengpSch_\infty}^{i+1}(\beta)=L^+_\infty\gpSch$ for $i<0$. This proves that all $\varphi''_i$ are isomorphisms on all of $X_S$.

Now we choose an \'etale covering $S'\to S$ and a trivialization $\epsilon\colon L^+_\Delta(\mcE''_0)\isoto (L^+_\Delta\gpSch)_{S'}$ for the divisor $\Delta:=\Gamma_x + (\infty\times_{\F_q}S')$. Then the boundedness of $\underline{\mcE}$ by $\mcZ(\mu,\beta)$ implies that we get a morphism $S'\to Z(\mu,\beta)$ for the representative $Z(\mu,\beta)$ of $\mcZ(\mu,\beta)$ over $\widetilde{X}_{\mu,\beta}$. By (the proof of) Proposition~\ref{PropBoundWithChains} the unique extension of $\underline{\mcE}$ to the diagram~\eqref{EqCorShtWithChains} corresponds to the unique extension of the morphism $S'\to Z(\mu,\beta)$ to data over $S'$ as in diagram~\eqref{EqPropBoundWithChains2}. There we saw for all $i\in\Z$, that $\Pi''_i\circ \varphi'_i$ is an isomorphism on all of $X$, and that $\Pi''_i\circ\varphi'_i\circ\varphi_i$ is an isomorphism outside $\Gamma_x$ bounded by $\mu$. We now take $\diagphi_i:=\varphi''_{i-1}\circ\Pi''_i\circ\varphi'_i\circ\varphi_i\colon (\mcE_i,\psi_i)|_{X_S\smallsetminus\Gamma_x} \isoto {}^{\tau\!}(\mcE_{i-1},\psi_{i-1})|_{X_S\smallsetminus\Gamma_x}$, which is bounded by $\mu$. Descending back to $S$ finishes the proof.
\end{proof}

\begin{cor} \label{CorIsog[n]}
We use the interpretation of $\Sht_{\gpSch,D,X\times\infty}^{\mcZ(\mu,\beta)}$ in terms of chains as in Corollary~\ref{CorShtWithChains}.
\begin{enumerate}
\item \label{CorIsog[n]_A}
There is an action of $\Z$ on the stack $\Sht_{\gpSch,D,X\times\infty}^{\mcZ(\mu,\beta)}$, which for an integer $n\in\Z$ is given by the index shift
\begin{equation}
\begin{split}
[n]\colon\qquad\quad \Sht_{\gpSch,D,X\times\infty}^{\mcZ(\mu,\beta)} & \longrightarrow \Sht_{\gpSch,D,X\times\infty}^{\mcZ(\mu,\beta)} \\
(x,\mcE_i,\psi_i,\Pi_i,\diagphi_i)_{i\in\Z} & \longmapsto (x,\mcE^{\mathrm{new}}_i,\psi^{\mathrm{new}}_i,\Pi^{\mathrm{new}}_i,\diagphi_i{}^{\mathrm{new}})_{i\in \Z}
\end{split}
\end{equation}
with $(\mcE^{\mathrm{new}}_i,\psi^{\mathrm{new}}_i,\Pi^{\mathrm{new}}_i,\diagphi_i{}^{\mathrm{new}}):=(\mcE_{i+n},\psi_{i+n},\Pi_{i+n},\diagphi_{i+n})$.
\item \label{CorIsog[n]_B}
For all integers $m,n$ and every global $\gpSch$-shtuka $\underline{\mcE}=(x,\mcE_i,\psi_i,\Pi_i,\diagphi_i)\in \Sht_{\gpSch,D,X\times\infty}^{\mcZ(\mu,\beta)}(S)$ over a scheme $S$, there is a quasi-isogeny 
\[
\Pi^n\colon [n+m](\underline{\mcE}) \to [m](\underline{\mcE}),
\]
which is given by the isomorphism $\Pi_{i+1}\circ\ldots\circ \Pi_{i+n}\colon \mcE_{i+n}\to \mcE_i$ for $n\ge 0$ (respectively $\Pi_i^{-1}\circ\ldots\circ \Pi_{i+n+1}^{-1}\colon \mcE_{i+n}\to \mcE_i$ for $n<0$) on $(X\smallsetminus\{\infty\})_S$. \qed
\end{enumerate}
\end{cor}

\subsection{Local shtukas and Rapoport-Zink spaces}\label{subsec-LocSht-and-RZ}

Let $v\in X$ be a place of $\FcnFld$ and let $q_v:=\#\F_v=q^{[\F_v:\F_q]}$. For a scheme $S\in\Nilp_{\Oo_v}$ let $\hat{\tau}_v:=\Frob_{q_v,S}$ be the $\F_v$-Frobenius of $S$ as defined in \S \ref{subsec-notations}.

\begin{defn}\label{Def_LocSht}
Let $S\in\Nilp_{\Oo_v}$. A \emph{local $\gpSch_v$-shtuka over $S\in \Nilp_{\Oo_v}$} is a pair $\underline{\mcL} = (\mcL,\widehat{\varphi})$ consisting of an $L^+_v\gpSch$-bundle $\mcL$ on $S$ and an isomorphism of the associated $L_v\gpSch$-bundles $\widehat{\varphi}\colon L_v\mcL \isoto {}^{\hat{\tau}_v\!} L_v\mcL$ from \eqref{EqLoopTorsorDelta} and Example~\ref{ExDivisors}(b). 
We denote the stack fibered in groupoids over $\Nilp_{\Oo_v}$ which classifies local $\gpSch_v$-shtukas by $\LocSht_{\gpSch_v}$.
\end{defn}

\begin{remark}\label{Rem_LocShtInversePhi}
Note that in the literature on local $\gpSch$-shtukas usually an isomorphism ${}^{\hat{\tau}_v\!} L_v\mcL \isoto L_v\mcL$ is considered. Our $\widehat{\varphi}$ is the inverse of that isomorphism. 
\end{remark}

\begin{remark}\label{Rem_LocShtTrivial}
By definition, an $L^+_v\gpSch$-bundle $\mcL$ over a scheme $S$ can be trivialized over a suitable \'etale covering $S'\to S$ via an isomorphism $\hat\alpha\colon \mcL \isoto (L^+_v\gpSch)_{S'}$. In this case we usually write $\underline{\mcL}\cong\bigl((L^+_v\gpSch)_{S'},b\bigr)$ meaning that the isomorphism ${}^{\hat{\tau}_v\!}\hat\alpha\circ\widehat{\varphi}\circ \hat\alpha^{-1}$ of the trivial $L_v\gpSch$-bundle is given as multiplication on the left with $b\in L_v\gpSch(S')$. Note that if $S$ is the spectrum of a strictly henselian local ring, then it has no non-trivial \'etale coverings and we can take $S'=S$.
\end{remark}

\begin{defn}
A \emph{quasi-isogeny} $f\colon\underline{\mcL}\to\underline{\mcL}'$ between two local $\gpSch_v$-shtukas $\underline{\mcL}=(\mcL,\widehat{\varphi})$ and $\underline{\mcL}'=(\mcL' ,\widehat{\varphi}')$ over $S\in\Nilp_{\Oo_v}$ is an isomorphism of the associated $L_v\gpSch$-bundles $f \colon  L_v\mcL \to L_v\mcL'$ satisfying ${}^{\hat{\tau}_v\!}f\circ\widehat{\varphi}=\widehat{\varphi}'\circ f$. We denote by $\QIsog_S(\underline{\mcL},\underline{\mcL}')$ the set of quasi-isogenies between $\underline{\mcL}$ and $\underline{\mcL}'$ over $S$, and we write $\QIsog_S(\underline{\mcL}):=\QIsog_S(\underline{\mcL},\underline{\mcL})$ for the quasi-isogeny group of $\underline{\mcL}$. 
\end{defn}

\begin{example}\label{ExFrobIsogLocSht}
Let $\underline{\mcL}=(\mcL,\widehat{\varphi})$ be a local $\gpSch$-shtuka over $S\in\Nilp_{\Oo_v}$. Then ${}^{\hat{\tau}_v\!}\underline{\mcL}=({}^{\hat{\tau}_v\!}\mcL,{}^{\hat{\tau}_v\!}\widehat{\varphi})$ is a local $\gpSch$-shtuka over $S$, too, and $f:=\widehat{\varphi}\colon L_v\mcL\isoto L_v{}^{\hat{\tau}_v\!}\mcL$ satisfies ${}^{\hat{\tau}_v\!}f \circ \widehat{\varphi}={}^{\hat{\tau}_v\!}\widehat{\varphi}\circ f$. Therefore, $\widehat{\varphi}\colon \underline{\mcL}\to {}^{\hat{\tau}_v\!}\underline{\mcL}$ is a quasi-isogeny, called the \emph{$q_v$-Frobenius isogeny} of $\underline{\mcL}$.
\end{example}

We recall the rigidity of quasi-isogenies of local $\gpSch_v$-shtukas from \cite[Proposition~2.11]{AH_Local}.

\begin{prop}[Rigidity of quasi-isogenies for local $\gpSch_v$-shtukas] \label{PropRigidityLocal}
Let $S$ be a scheme in $\Nilp_{\Oo_v}$ and let $j \colon  \bar{S}\to S$ be a closed immersion  defined by a sheaf of ideals which is locally nilpotent. Let $\underline{\mcL}=(\mcL,\widehat{\varphi})$ and $\underline{\mcL}'=(\mcL',\widehat{\varphi}')$ be two local $\gpSch_v$-shtukas over $S$. Then
\[
\QIsog_S(\underline{\mcL}, \underline{\mcL}') \longto \QIsog_{\bar{S}}(j^*\underline{\mcL}, j^*\underline{\mcL}') ,\quad f \mapsto j^*f
\]
is a bijection of sets.
\end{prop}

\begin{proof}
Let $\mathcal{I}$ be the ideal sheaf defining $j\colon \bar{S}\into S$. Arguing by induction over $\Oo_S/\mathcal{I}^{q_v^n}$ it suffices to treat the case where $\mathcal{I}^{q_v}=(0)$. In this case the morphism $\hat\tau_v=\Frob_{q_v,S}$ factors as $S\xrightarrow{\;i}\bar{S}\xrightarrow{\;j}S$ where $i$ is the identity on the underlying topological space $|\bar{S}|=|S|$ and on the structure sheaf this factorization is given by
\begin{eqnarray}\label{EqTauFactors}
\textstyle\Oo_S \enspace \xrightarrow{\enspace j^\ast\;} & \Oo_{\bar{S}} & \textstyle\xrightarrow{\enspace i^\ast\;} \enspace \Oo_S\\
\textstyle b\quad \mapsto\;\: & b \mod \mathcal{I}& \textstyle\;\:\mapsto\quad b^{q_v}\,. \nonumber
\end{eqnarray}
Therefore ${}^{\hat\tau_v\!}f=i^\ast(j^\ast f)$ for any $f\in\QIsog_S(\underline{\mcL},\underline{\mcL}')$. We obtain the diagram
\begin{equation}\label{EqRigidity}
\xymatrix @C=5pc @R+0.5pc {
L_\infty\mcL \ar[d]^\cong_{\textstyle\widehat{\varphi}} \ar[r]_\cong^{\textstyle f} & L_\infty\mcL' \ar[d]_\cong^{\textstyle\widehat{\varphi}'} \\
{}^{\hat\tau_v\!}L_\infty\mcL \ar[r]_\cong^{\textstyle i^\ast(j^\ast f)\enspace} & {}^{\hat\tau_v\!}L_\infty\mcL'
}
\end{equation}
from which the bijectivity is obvious.
\end{proof}

\begin{defn}\label{DefLocShtBounded}
Consider the local type \ref{DefThreeTypes_B} at the place $v$ in Definition~\ref{DefThreeTypes} and let $\mcZ$ be a bound in $\Gr_{\gpSch,X}\times_X \widetilde{X}_{\mcZ}$ with reflex scheme $\widetilde{X}_{\mcZ}$. Then $\widetilde{X}_{\mcZ}=\Spec\Oo_{\mcZ}$ for a finite ring extension $\Oo_{\mcZ}$ of $\Oo_v$. Let $Z$ be the representative of $\mcZ$ over $\Oo_\mcZ$. Recall from Corollary~\ref{CorGrGlobLoc} that $\Gr_{\gpSch,X}\times_X \Spf\Oo_{\mcZ} \cong \Fl_{\gpSch,v} \widehattimes_{\F_v} \Spf\Oo_{\mcZ}$. Let $S\in\Nilp_{\Oo_{\mcZ}}$. A local $\gpSch_v$-shtuka $\underline{\mcL}:=(\mcL,\widehat{\varphi})$ is \emph{bounded by $\mcZ$} if for every (some) \'etale covering $S'\to S$ and every (some) trivialization $\hat\alpha\colon \mcL \isoto (L^+_v\gpSch)_{S'}$ the element $b={}^{\hat{\tau}_v\!}\hat\alpha\circ\widehat{\varphi}\circ \hat\alpha^{-1}\in L_v\gpSch(S')$ factors through $Z\subset \Fl_{\gpSch,v}\widehattimes_{\F_v}\Spf \Oo_{\mcZ}$ when viewed as a morphism $b\colon S'\to \Fl_{\gpSch,v}\widehattimes_{\F_v}\Spf \Oo_{\mcZ}$. Note that the equivalences of ``every'' and ``some'' was proven in \cite[Remark~4.9]{AH_Local}. 

We write $\LocSht_{\gpSch_v}^{\mcZ}$ for the stack of local $\gpSch_v$-shtukas bounded by $\mcZ$. When $\mcZ=\mcZ^{\leq\mu}\times_{\widetilde{X}_\mu} \Spf\Oo_\mu$ for the bound $\mcZ^{\leq\mu}$ from Definition~\ref{Def_BoundBy_mu}, we write $\LocSht_{\gpSch_v}^{\leq \mu}$. 
\end{defn}

\begin{remark}\label{Rem_LocShtInversePhi2}
We continue with Remark~\ref{Rem_LocShtInversePhi}. Let $\underline{\mcL}=(\mcL,\widehat{\varphi})$ be a local $\gpSch_v$-shtuka as in our Definition~\ref{Def_LocSht}. In the literature on local $\gpSch$-shtukas like \cite{AH_Local, AH_Unif, HV1, HV2, HartlViehmann3} where the Frobenius of a local $\gpSch_v$-shtuka is an isomorphism ${}^{\hat{\tau}_v\!} L_v\mcL \isoto L_v\mcL$ one would have to consider $(\mcL,\widehat{\varphi}^{-1})$ as the local $\gpSch_v$-shtuka. In the language of that literature one says that ``$(\mcL,\widehat{\varphi}^{-1})$ is bounded by a bound $\widetilde{\mcZ}$'' if $\widehat{\varphi}^{-1}$ is bounded by $\widetilde{\mcZ}$. This coincides with our definition that ``$(\mcL,\widehat{\varphi})$ is bounded by $\mcZ$'' if one takes $\widetilde{\mcZ}=\mcZ^{-1}$ in the sense of \cite[Remark~2.3 and Lemma~2.12]{HartlViehmann3}.
\end{remark}

\begin{defn}\label{DefGlobLocG}
Let $S\in\Nilp_{\Breve{\Oo}_\infty}$. The \emph{global-local functor}
\begin{equation}\label{Eq_DefGlobLocG}
L^+_{\infty,\gpSch}\colon \Sht_{\gpSch,D,X\times\infty}(S) \longto \LocSht_{\gpSch_\infty}(S)
\end{equation}
is defined as follows. Let $\underline{\mcE}=\bigl(x,\,(\mcE,\psi,)(\mcE',\psi'),\varphi,\varphi'\bigr)\in \Sht_{\gpSch,D,X\times\infty}(S)$ be a global $\gpSch$-shtuka over $S$ as in Definition~\ref{Def_Sht2legs}. Then 
\begin{equation}\label{Eq_DefGlobLocG2}
L^+_{\infty,\gpSch}(\underline{\mcE}) := \bigl(L^+_{\infty,\gpSch}(\mcE), L_{\infty,\gpSch}(\varphi'\circ\varphi)\bigr).
\end{equation}
\end{defn}

\begin{remark}
Note that in the previous definition 
\[
\varphi'\circ\varphi\colon \mcE|_{(X\setminus\{\infty\})_S}\isoto {}^{\tau\!}\mcE|_{(X\setminus\{\infty\})_S}
\]
is an isomorphism of $\gpSch$-bundles, because $S\in\Nilp_{\Breve{\Oo}_\infty}$ implies that $X_S\setminus\Gamma_x = (X\setminus\{\infty\})_S$. So $L^+_{\infty,\gpSch}(\underline{\mcE})$ recovers the combined modification of $\mcE$ at the two legs $x$ and $\infty$.
\end{remark}

\begin{numberedparagraph}\label{ParInnerFormM}
Let $\beta\in L_\infty\gpSch(\F_\beta)$ as in \S\,\ref{Def_beta}. Recall the Frobenius $\tau_{\gengpSch_\infty}\colon L_\infty\gpSch(S)\to L_\infty\gpSch(S)$ from \eqref{EqTau_G}. Let $M:=M_{\beta^{-1}}$ be the inner form over $\FcnFld_\infty$ of $\gengpSch_\infty$ given by $\beta^{-1}$, i.e.~there is an isomorphism of linear algebraic groups over $\Breve{\FcnFld}_\infty$
\[
\iota\colon \gengpSch_\infty \times_{\FcnFld_\infty} \Breve{\FcnFld}_\infty \isoto M_{\beta^{-1}} \times_{\FcnFld_\infty} \Breve{\FcnFld}_\infty
\]
with $\FcnFld_\infty$-structure on $M_{\beta^{-1}}$ given by the Frobenius 
\begin{equation}\label{Eq_tau_M}
\tau_M(g)=\iota\bigl(\beta^{-1} \cdot \tau_{\gengpSch_\infty}(\iota^{-1}(g)) \cdot \beta\bigr)=\iota\bigl(\beta^{-1} \cdot {}^{\tau\!}\iota^{-1}(\tau_M(g)) \cdot \beta\bigr) 
\end{equation}
i.e.~given by $\id=\iota\circ\Int_{\beta^{-1}}\circ {}^{\tau\!} \iota^{-1}$ and ${}^{\tau\!} \iota=\iota\circ \Int_{\beta^{-1}}$. 

Since $\beta^{-1}\cdot \gpSch_\infty\cdot\beta = \gpSch_\infty$, there exists a smooth affine group scheme $\mcM=\mcM_{\beta^{-1}}$ over $\Oo_\infty$ such that $\iota$ restricts to an isomorphism of algebraic groups over $\Breve{\Oo}_\infty$
\begin{equation}\label{Eq_mcM}
\iota\colon \gpSch_\infty \times_{\Oo_\infty} \Breve{\Oo}_\infty \isoto \mcM_{\beta^{-1}} \times_{\Oo_\infty} \Breve{\Oo}_\infty\,. 
\end{equation}
We use $\iota$ to identify $L^+_\infty\gpSch(R)=\gpSch_\infty(R\dbl z \dbr) =\mcM(R\dbl z \dbr)=L^+_\infty\mcM(R)$ for $\Breve{\Oo}_\infty$-algebras $R$. 

The bound $\mcZ^{\leq\mu}\times_{\widetilde{X}_\mu} \Spf\Oo_\mu$ from Definition~\ref{DefLocShtBounded} induces a bound in $\Fl_{\mcM,\infty}$ which has a representative over $\OReflZMuBeta$, and which we will use to bound local $\mcM$-shtukas in Definition~\ref{DefRZforM}. That local bound depends on the choice of the map $\Spec \Oo_{\mu,\beta}\to\widetilde{X}_{\mu,\beta}$.
\end{numberedparagraph}

\begin{prop} \label{PropGvsM}
For any $S\in\Nilp_{\Breve{\Oo}_\infty}$ and any $\beta\in L_\infty\gpSch(\F_\beta)$ as in \S\,\ref{Def_beta}, there is an equivalence of categories given by left translation by $\beta^{-1}$
\begin{equation}\label{EqEquivG-M-Shtukas}
t_{\beta^{-1}}\colon \LocSht_{\gpSch_\infty}(S) \isoto \LocSht_{\mcM_{\beta^{-1}}}(S),
\end{equation}
such that the underlying $L^+_\infty\gpSch$-bundle $\mcL$ of a local $\gpSch_\infty$-shtuka $\underline{\mcL}=(\mcL,\widehat{\varphi})$ coincides with the underlying $L^+_\infty\mcM_{\beta^{-1}}$-bundle of $t_{\beta^{-1}}(\underline{\mcL})$. In terms of trivialized local shtukas, the functor $t_{\beta^{-1}}$ is given by 
\[
\bigl((L^+_\infty\gpSch)_S, b\bigr) \longmapsto \bigl((L^+_\infty\mcM_{\beta^{-1}})_S, \beta^{-1} b\bigr)
\]
for $b\in L_\infty\gpSch(S)=L_\infty\mcM_{\beta^{-1}}(S)$. The functor $t_{\beta^{-1}}$ also sends quasi-isogenies to quasi-isogenies.
\end{prop}
\begin{proof}
Let $\underline{\mcL}:=(\mcL,\widehat{\varphi})$ be a local $\gpSch_\infty$-shtuka over $S$. Let $S'\to S$ be an \'etale covering over which a trivialization $\alpha\colon \underline{\mcL}_{S'}\isoto \bigl((L^+_\infty\gpSch)_{S'}, b'\bigr)$ exists with $b':={}^{\tau\!}\alpha\circ\widehat{\varphi}\circ\alpha^{-1}\in L_\infty\gpSch(S')$. Over $S''=S'\times_S S'$, we obtain the descent datum $h:=\pr_2^*\alpha\circ \pr_1^* \alpha^{-1}\in L^+_\infty\gpSch(S'')$ for $\mcL$, where $\pr_j\colon S''\to S'$ is the projection onto the $j$-th factor for $j=1,2$. Then ${}^{\tau\!}\mcL$ is trivialized over $S'$ by ${}^{\tau\!}\alpha$ with descent datum $\tau_{\gengpSch_\infty}(h)=\pr_2^*({}^{\tau\!}\alpha)\circ \pr_1^*({}^{\tau\!}\alpha)^{-1}\in L^+_\infty\gpSch(S'')$. The fact that $\widehat{\varphi}$ is defined over $S$ is equivalent to $\pr_1^*\widehat{\varphi}=\pr_2^*\widehat{\varphi}$, which in turn is equivalent to the equation
\begin{equation}\label{EqPropGvsM_1}
\tau_{\gengpSch_\infty}(h)\cdot\pr_1^*b'=\pr_2^*({}^{\tau\!}\alpha)\circ \pr_1^*\widehat{\varphi}\circ \pr_1^*\alpha^{-1}=\pr_2^*({}^{\tau\!}\alpha)\circ \pr_2^*\widehat{\varphi}\circ \pr_1^*\alpha^{-1}=\pr_2^*(b') \cdot h.
\end{equation}
We now view $\mcL$ as an $L^+_\infty\mcM$-bundle via the isomorphism $\iota$ from \eqref{Eq_mcM}. Over $S'$, we obtain the trivialization $\iota\circ\alpha\colon \mcL_{S'}\isoto (L^+_\infty\mcM)_{S'}$, which gives rise to the descent datum $\iota(h)=\iota\circ\pr_2^*\alpha\circ \pr_1^* \alpha^{-1}\circ\iota^{-1}\in L^+_\infty\mcM(S'')$. However, ${}^{\tau\!}\mcL$ is now trivialized over $S'$ by ${}^{\tau\!}(\iota\circ\alpha)\colon {}^{\tau\!}\mcL_{S'}\isoto (L^+_\infty\mcM)_{S'}$ and has the descent datum $\pr_2^*{}^{\tau\!}(\iota\circ\alpha)\circ \pr_1^*{}^{\tau\!}(\iota\circ\alpha)^{-1}={}^{\tau\!}\iota \circ \pr_2^*({}^{\tau\!}\alpha)\circ \pr_1^*({}^{\tau\!}\alpha)^{-1} \circ {}^{\tau\!}\iota^{-1}$ which sends the neutral element $1\in L^+_\infty\mcM(S'')$ to ${}^{\tau\!}\iota (\tau_{\gengpSch_\infty}(h))=\iota\bigl(\beta^{-1}\cdot \tau_{\gengpSch_\infty}(h) \cdot\beta\bigr)=\tau_M(\iota(h))$ by definition of $\tau_M$. 

We equip $\mcL_{S'}$ with the new Frobenius $\widehat{\varphi}':={}^{\tau\!}(\iota\circ\alpha)^{-1}\circ \iota(\beta^{-1} b')\circ (\iota\circ\alpha)\colon L_\infty\mcL_{S'}\isoto {}^{\tau\!} L_\infty\mcL_{S'}$ such that $\iota\circ\alpha\colon (\mcL_{S'},\widehat{\varphi}')\isoto \bigl((L^+_\infty\mcM)_{S'},\iota(\beta^{-1}b')\bigr)$ is an isomorphism of local $\mcM$-shtukas over $S'$. From \eqref{EqPropGvsM_1}, we obtain the middle equality in the equation 
\[
\tau_M(\iota(h))\cdot \pr_1^*\iota(\beta^{-1}b') = \iota\bigl(\beta^{-1}\tau_{\gengpSch_\infty}(h)\beta\cdot \beta^{-1}\cdot \pr_1^*b'\bigr) = \iota\bigl(\beta^{-1}\cdot \pr_2^*b'\cdot h\bigr) = \pr_2^*\iota(\beta^{-1}b')\cdot \iota(h),
\]
where the first (resp.~last) equality follows because $\iota$ commutes with $\pr_1^*$ (resp.~$\pr_2^*$). Replacing $\tau_{\gengpSch_\infty}$ by $\tau_M\circ\iota$ and $b'$ by $\iota(\beta^{-1}b')$ in \eqref{EqPropGvsM_1}, this implies $\pr_1^*\widehat{\varphi}'=\pr_2^*\widehat{\varphi}'$, and thus $\widehat{\varphi}'$ descends to an isomorphism $L_\infty\mcL\isoto {}^{\tau\!} L_\infty\mcL$ over $S$. This defines the local $\mcM$-shtuka $t_{\beta^{-1}}(\underline{\mcL}):=(\mcL,\widehat{\varphi}')$ over $S$ and the functor $t_{\beta^{-1}}$. To visualize the construction, we have the diagram
\begin{equation*}
    \begin{tikzcd}
        (\mcL,\widehat{\varphi})_{S'}=\underline{\mcL}_{S'}\arrow[]{r}{\alpha}[swap]{\sim}\arrow[mapsto]{d}[swap]{t_{\beta^{-1}}}& ((L^+_\infty\gpSch)_{S'},b')\arrow[]{r}{\iota}[swap]{\sim}&((L^+_\infty\mcM_{\beta^{-1}})_{S'},\iota(b'))\arrow[mapsto]{d}{t_{\beta^{-1}}}\\
        (\mcL,\widehat{\varphi}')_{S'}=t_{\beta^{-1}}(\underline{\mcL})_{S'}\arrow[]{rr}{\iota\circ\alpha}[swap]{\sim}&&((L^+_\infty\mcM_{\beta^{-1}})_{S'},\iota(\beta^{-1}\cdot b'))
    \end{tikzcd}
\end{equation*}

We must show that $t_{\beta^{-1}}$ is independent of the choices of $S'$ and $\alpha$. If we choose a different $\widetilde{S}'$, we may as well replace it with a common refinement with the previous $S'$ and assume $\widetilde{S}'=S'$. Then the new $\widetilde{\alpha}$ differs from the previous $\alpha$ by left multiplication with an element $g:=\widetilde{\alpha}\circ \alpha^{-1}\in L^+_\infty\gpSch(S')$. This gives the new descent datum $\tilde h:=\pr_2^*\widetilde{\alpha} \circ \pr_1^*\widetilde{\alpha}^{-1}=\pr_2^*g\cdot h\cdot \pr_1^*g^{-1}$ of the $L^+_\infty\gpSch$-bundle $\mcL$ and changes $b'$ to $\tilde b'=\tau_{\gengpSch_\infty}(g)\cdot b' \cdot g^{-1}$. It also changes the trivialization of the $L^+_\infty\mcM$-bundle $\mcL$ to $\iota\circ\widetilde{\alpha}=\iota(g)\cdot \iota\circ\alpha$ and $\iota(\beta^{-1}b')$ to $\iota(\beta^{-1}\tilde b')=\iota(\beta^{-1}\tau_{\gengpSch_\infty}(g)\beta\cdot\beta^{-1} b' g^{-1})=\tau_M(\iota(g))\cdot \iota(\beta^{-1} b')\cdot \iota(g)^{-1}$. Therefore, $t_{\beta^{-1}}(g):=\iota(g)\colon \bigl((L^+_\infty\mcM)_{S'}, \iota(\beta^{-1} b')\bigr)\isoto \bigl((L^+_\infty\mcM)_{S'}, \iota(\beta^{-1}\tilde b')\bigr)$ is an isomorphism over $S'$. This shows that the change of the trivialization from $\alpha$ to $\widetilde{\alpha}$ does not change the $L^+_\infty\mcM$-bundle $\mcL$ and its new Frobenius $\widehat{\varphi}'$. Thus the functor $t_{\beta^{-1}}$ is well defined.

Clearly $t_{\beta^{-1}}$ is an equivalence, as its inverse functor is given by $t_{\iota(\beta)}$. Let $f:\underline{\mcL}\to \underline{\widetilde{\mcL}}$ be a quasi-isogeny. We have trivializations over a suitable \'etale covering $S'$ of $S$
\begin{equation}
    \alpha: \underline{\mcL}_{S'}\isoto((L^+_\infty\gpSch)_{S'},b')\quad\text{and}\quad \widetilde{\alpha}: \underline{\widetilde{\mcL}}_{S'}\isoto((L^+_\infty\gpSch)_{S'},\widetilde{b}').
\end{equation}
Let $\widetilde{\alpha}\circ f\circ\alpha^{-1}=:g\in L_\infty\gpSch(S')$. 
The same computations as in the preceeding paragraph with $g\in L_\infty\gpSch(S')$ instead of $L^+_\infty\gpSch(S')$ show that $t_{\beta^{-1}}(f)$ is a quasi-isogeny $t_{\beta^{-1}}(\underline{\mcL})\to t_{\beta^{-1}}(\underline{\widetilde{\mcL}})$. Thus the functor $t_{\beta^{-1}}$ is compatible with quasi-isogenies. 
\end{proof}

Let $\mcZ(\mu,\beta)$ be the bound in $\Gr_{\gpSch,X\times\infty}$ from Definition~\ref{DefZmubeta}. The following is a variant of the \emph{global-local functor} from \eqref{EqL^+_v} and \eqref{EqL_v}. It is the appropriate global-local functor for global shtukas in $\Sht_{\gpSch,D,X\times\infty}^{\mcZ(\mu,\beta)}$ when the two legs collide.

\begin{defn}\label{DefGlobLocM}
Let $S\in\Nilp_{\Breve{\Oo}_\infty}$. The \emph{$\beta^{-1}$-twisted global-local functor}
\begin{equation}\label{Eq_DefGlobLocM}
L^+_{\infty,\mcM_{\beta^{-1}}}:=t_{\beta^{-1}}\circ L^+_{\infty,\gpSch}\colon \Sht_{\gpSch,D,X\times\infty}(S) \longto \LocSht_{\mcM_{\beta^{-1}}}(S)
\end{equation}
is defined as follows. Let $\underline{\mcE}=\bigl(x,\,(\mcE,\psi),(\mcE',\psi'),\,\varphi,\varphi'\,\bigr)\in \Sht_{\gpSch,D,X\times\infty}(S)$ be a global $\gpSch$-shtuka over $S$ as in Definition~\ref{Def_Sht2legs}. Then 
\begin{equation}\label{Eq_DefGlobLocM2}
L^+_{\infty,\mcM_{\beta^{-1}}}(\underline{\mcE}) := t_{\beta^{-1}}\bigl(L^+_{\infty,\gpSch}(\mcE), L_{\infty,\gpSch}(\varphi'\circ\varphi)\bigr).
\end{equation}
\end{defn}

Note that when $\beta=1$, the $\beta^{-1}$-twisted global-local functor $L_{\infty,\mcM_{\beta^{-1}}}^+=L_{\infty,\gpSch}^+$ recovers the (usual) global-local functor given in \eqref{EqL^+_v} and \eqref{EqL_v}.

\begin{cor} \label{CorGlobalLocalFunctorWithChains}
Let $\underline{\mcE}\in \Sht_{\gpSch,D,X\times\infty}^{\mcZ(\mu,\beta)}(S)$ be given in terms of Corollary~\ref{CorShtWithChains} as the tuple $\underline{\mcE}=\bigl(x,(\mcE_i,\psi_i,\Pi_i,\diagphi_i)_{i\in\Z}\bigr)$. Then 
\begin{equation}\label{Eq_CorGlobalLocalFunctorWithChains}
L^+_{\infty,\mcM_{\beta^{-1}}}({}^{\tau\!}\mcE_{-1})\cong {}^{\tau\!} L^+_{\infty,\mcM_{\beta^{-1}}}(\mcE_0)
\end{equation}
and the $\beta^{-1}$-twisted global-local functor $L^+_{\infty,\mcM_{\beta^{-1}}}$ sends $\underline{\mcE}$ to 
\begin{equation}\label{Eq_CorGlobalLocalFunctorWithChains2}
L^+_{\infty,\mcM_{\beta^{-1}}}(\underline{\mcE})=\bigl(L^+_{\infty,\mcM_{\beta^{-1}}}(\mcE_0), L^+_{\infty,\mcM_{\beta^{-1}}}(\diagphi_0)\bigr).
\end{equation}
In particular, the local $\mcM$-shtuka $L^+_{\infty,\mcM_{\beta^{-1}}}(\underline{\mcE})$ is bounded by $\mu$.
\end{cor}

\begin{proof}
To prove \eqref{Eq_CorGlobalLocalFunctorWithChains}, let $S'\to S$ be an \'etale covering over which trivializations $\alpha_i\colon L^+_{\infty,\gpSch}(\mcE_i)\isoto (L^+_\infty\gpSch)_{S'}$ for $i=-1,0$ exist. After base change under $\tau=\Frob_{q,S'}$ this yields trivializations ${}^{\tau\!}\alpha_i\colon L^+_{\infty,\gpSch}({}^{\tau\!}\mcE_i)\isoto (L^+_\infty\gpSch)_{S'}$. We now view these $L^+_\infty\gpSch$-bundles as $L^+_\infty\mcM$-bundles via the isomorphism $\iota$ from \eqref{Eq_mcM}. As such they are trivialized by $\iota\circ\alpha_i\colon L^+_{\infty,\mcM}(\mcE_i)\isoto (L^+_\infty\mcM)_{S'}$ and $\iota\circ{}^{\tau\!}\alpha_i\colon L^+_{\infty,\mcM}({}^{\tau\!}\mcE_i)\isoto (L^+_\infty\mcM)_{S'}$ and ${}^{\tau\!}(\iota\circ\alpha_i)\colon {}^{\tau\!}L^+_{\infty,\mcM}(\mcE_i)\isoto (L^+_\infty\mcM)_{S'}$.

Over $S''=S'\times_S S'$ we obtain the descent data $h_i:=\pr_2^*\alpha_i \circ \pr_1^*\alpha_i^{-1}\in L^+_\infty\gpSch(S'')$ for $L^+_{\infty,\gpSch}(\mcE_i)$ and $\tau_{\gengpSch_\infty}(h_i)=\pr_2^*({}^{\tau\!}\alpha_i) \circ \pr_1^*({}^{\tau\!}\alpha_i)^{-1}\in L^+_\infty\gpSch(S'')$ for $L^+_{\infty,\gpSch}({}^{\tau\!}\mcE_i)$, as well as $\iota(h_i)=\iota\circ\pr_2^*\alpha_i \circ \pr_1^*\alpha_i^{-1}\circ\iota^{-1}\in L^+_\mcM(S'')$ for $L^+_{\infty,\mcM}(\mcE_i)$ and $\iota(\tau_{\gengpSch_\infty}(h_i))=\iota\circ\pr_2^*({}^{\tau\!}\alpha_i) \circ \pr_1^*({}^{\tau\!}\alpha_i)^{-1}\circ\iota^{-1}\in L^+_\infty\mcM(S'')$ for $L^+_{\infty,\mcM}({}^{\tau\!}\mcE_i)$ and $\tau_M(\iota(h_i))=\pr_2^*{}^{\tau\!}(\iota\circ\alpha_i) \circ \pr_1^*{}^{\tau\!}(\iota\circ\alpha_i)^{-1}\in L^+_\infty\mcM(S'')$ for ${}^{\tau\!}L^+_{\infty,\mcM}(\mcE_i)$; see the proof of Proposition~\ref{PropGvsM}. Now by construction of $\mcE_{-1}$ from $\mcE_0$ in Proposition~\ref{PropBunWithChains}, we have $h_{-1}=\tau_{\gengpSch_\infty}^{-1}(\beta)^{-1}\cdot h_0\cdot \tau_{\gengpSch_\infty}^{-1}(\beta)$. This implies $\tau_M(\iota(h_0))=\iota(\beta^{-1}\cdot\tau_{\gengpSch_\infty}(h_0)\cdot\beta)=\iota(\tau_{\gengpSch_\infty}(\tau_{\gengpSch_\infty}^{-1}(\beta)^{-1}\cdot h_0\cdot \tau_{\gengpSch_\infty}^{-1}(\beta)))=\iota(\tau_{\gengpSch_\infty}(h_{-1}))$, and hence  \eqref{Eq_CorGlobalLocalFunctorWithChains} follows. 

To prove \eqref{Eq_CorGlobalLocalFunctorWithChains2}, we recall from the construction of $t_{\beta^{-1}}$ in the proof of Proposition~\ref{PropGvsM}, that over $S'$ the Frobenius of $L^+_{\infty,\mcM_{\beta^{-1}}}(\underline{\mcE})_{S'}$ is given by the left hand side in the following equation
\[
\iota\bigl(\beta^{-1}\cdot {}^{\tau\!}\alpha_0\circ L_{\infty,\gpSch}(\varphi'_0\circ\varphi_0)\circ \alpha_0^{-1}\bigr)\;=\;\iota\bigl({}^{\tau\!}\alpha_1\circ L_{\infty,\gpSch}({}^{\tau\!}\Pi_0\circ\varphi'_0\circ\varphi_0)\circ \alpha_0^{-1}\bigr).
\]
This equation comes from the construction of $L_{\infty,\gpSch}(\Pi_0):=\alpha_{-1}^{-1}\circ \tau_{\gengpSch_\infty}^{-1}(\beta)^{-1}\circ\alpha_0$ in \eqref{EqDefPi_i} in Proposition~\ref{PropBunWithChains}, which implies $\beta^{-1}\cdot {}^{\tau\!}\alpha_0={}^{\tau\!}\alpha_1\circ L_{\infty,\gpSch}({}^{\tau\!}\Pi_0)$. Now the claim follows from the definition of $\diagphi_0:={}^{\tau\!}\Pi_0\circ\varphi'_0 \circ \varphi_0$ in the proof of Corollary~\ref{CorShtWithChains}.
\end{proof}

The following is an analogue of \cite[Proposition 8.1]{HartlAbSh}.
\begin{prop} \label{PropQIsogLocalGlobal}
Let $S\in \Nilp_{\Breve{\Oo}_\infty}$. 
\begin{enumerate}
    \item  \label{PropQIsogLocalGlobal_A}
    Given global $\gpSch$-shtukas $\underline{\mcE},\underline{\widetilde{\mcE}}\in \Sht_{\gpSch,D,X\times\infty}(S)$. Any quasi-isogeny $\delta:\underline{\widetilde{\mcE}}\to \underline{\mcE}$ induces a quasi-isogeny 
    \[
    \hat{\delta}:=L_{\infty,\mcM_{\beta^{-1}}}(\delta)\colon L_{\mcM,\infty
    }(\underline{\widetilde{\mcE}})\to L_{\mcM,\infty
    }(\underline{\mcE})
    \]
    of the corresponding local $\mcM$-shtukas. 
    \item \label{PropQIsogLocalGlobal_B}
    Conversely, fix a global $\gpSch$-shtuka $\underline{\mcE}\in \Sht_{\gpSch,D,X\times\infty}(S)$. Any quasi-isogeny $\hat{\delta}\colon \underline{\widetilde{\mcL}}\to L_{\infty,\mcM_{\beta^{-1}}}(\underline{\mcE})$ of local $\mcM$-shtukas over $S$ comes from a unique (up to a canonical isomorphism) global $\gpSch$-shtuka $\underline{\widetilde{\mcE}}$ and a quasi-isogeny $\delta: \underline{\widetilde{\mcE}}\to\underline{\mcE}$ which is an isomorphism outside $\infty$ such that $L_{\infty,\mcM_{\beta^{-1}}}(\underline{\widetilde{\mcE}})=\underline{\widetilde{\mcL}}$ and $L_{\infty,\mcM_{\beta^{-1}}}(\delta) = \hat{\delta}$.
    In this setting, we will write $\hat{\delta}^*\underline{\mcE}$ to denote $\underline{\widetilde{\mcE}}$. 
\end{enumerate}
\end{prop}

\begin{proof}
\ref{PropQIsogLocalGlobal_A} A quasi-isogeny $\delta\colon \underline{\widetilde{\mcE}}=(x,(\widetilde{\mcE},\widetilde{\psi}),(\widetilde{\mcE}',\widetilde{\psi}'),\widetilde{\varphi},\widetilde{\varphi}')\to \underline{\mcE}=(x,(\mcE,\psi),(\mcE',\psi'),\varphi,\varphi')$ is given by two isomorphisms $f\colon (\widetilde{\mcE},\widetilde{\psi})|_{X_S\smallsetminus N_S}\isoto (\mcE,\psi)|_{X_S\smallsetminus N_S}$ and $f'\colon (\widetilde{\mcE}',\widetilde{\psi}')|_{X_S\smallsetminus N_S}\isoto (\mcE',\psi')|_{X_S\smallsetminus N_S}$ for some proper closed subscheme $N\subset X$, such that $f'\circ\widetilde{\varphi}=\varphi\circ f$ and ${}^{\tau\!} f\circ\widetilde{\varphi}'=\varphi'\circ f'$. Then $\hat{\delta}:= L^+_{\infty,\mcM_{\beta^{-1}}}(f)$ satisfies $\tau_M(\hat{\delta})\cdot\bigl(\beta^{-1} \circ L^+_{\infty,\gpSch}(\widetilde{\varphi}'\circ\widetilde{\varphi})\bigr)=\beta^{-1}\cdot\tau_{\gengpSch_\infty}(\hat{\delta})\circ L^+_{\infty,\gpSch}(\widetilde{\varphi}'\circ\widetilde{\varphi})=\bigl(\beta^{-1}\cdot L^+_{\infty,\gpSch}(\varphi'\circ\varphi)\bigr)\circ \hat{\delta}$, and hence is a quasi-isogeny $L_{\infty,\mcM_{\beta^{-1}}}(\underline{\widetilde{\mcE}})\to L_{\infty,\mcM_{\beta^{-1}}}(\underline{\mcE})$. (Here, strictly speaking, the equality $\tau_M(\hat{\delta})\cdot\beta^{-1}=\beta^{-1}\cdot\tau_{\gengpSch_\infty}(\hat{\delta})$ only makes sense after trivializing the $L^+_\infty\mcM$-bundles $L^+_\infty(\widetilde{\mcE})$ and $L^+_\infty(\mcE)$ over an \'etale covering $S'\to S$ and viewing $\hat{\delta}$ as an element of $L^+_\infty\mcM(S')$, as in the proof of Proposition~\ref{PropGvsM}.)

\medskip\noindent
\ref{PropQIsogLocalGlobal_B} We construct $\underline{\widetilde{\mcE}}=(x,(\widetilde{\mcE},\widetilde{\psi}),(\widetilde{\mcE}',\widetilde{\psi}'),\widetilde{\varphi},\widetilde{\varphi}')$ with the Beauvill-Laszlo glueing in Lemma~\ref{LemmaBL} as follows. $\widetilde{\mcE}$ is obtained by glueing $\mcE|_{(X\smallsetminus\{\infty\})_S}$ with $\widetilde{\mcL}$ via the isomorphism 
\[
\hat{\delta}\colon L_\infty\widetilde{\mcL}\isoto L_\infty(L^+_\infty\mcE)=L_\infty(\mcE|_{(X\smallsetminus\{\infty\})_S}). 
\]
Moreover, $\widetilde{\mcE}'$ is obtained from ${}^{\tau\!}\widetilde{\mcE}$ as in the proof of Proposition~\ref{PropBunWithChains}, such that the isomorphism $\widetilde{\varphi}'\colon \widetilde{\mcE}'|_{(X\smallsetminus\{\infty\})_S} \isoto {}^{\tau\!}\widetilde{\mcE}|_{(X\smallsetminus\{\infty\})_S}$ is bounded by $\mcZ(\beta)$, or equivalently by glueing $\mcE'|_{(X\smallsetminus\{\infty\})_S}$ with $\widetilde{\mcL}$ via the isomorphism 
\[
\varphi\circ\beta^{-1}\cdot\hat{\delta}\colon L_\infty\widetilde{\mcL}\isoto L_\infty(L^+_\infty\mcE')=L_\infty(\mcE'|_{(X\smallsetminus\{\infty\})_S}). 
\]
Finally $\widetilde{\psi}=\psi$ and $\widetilde{\psi}'=\psi'$ and $\widetilde{\varphi}=\varphi$ and $\widetilde{\varphi}'=\varphi'$ and $\delta=(\id_{\mcE},\id_{\mcE'})$ as isomorphisms over $(X\smallsetminus\{\infty\})_S=X_S\smallsetminus\Gamma_x$. 
\end{proof}

When $\beta=1$, Proposition \ref{PropQIsogLocalGlobal} is analogous to the theory of abelian varieties and $p$-divisible groups.

Next we come to the analogue of the Serre-Tate-Theorem. 

\begin{defn}\label{DefDefo}
Let $S\in\Nilp_{\Breve{\Oo}_\infty}$ and let $j:\bar{S}\into S$ be a closed subscheme defined by a locally nilpotent sheaf $\mathcal{I}$ of ideals. Let $\underline{\bar{\mcE}}\in\Sht_{\gpSch,D,X\times\infty}(\bar{S})$ be a global $\gpSch$-shtuka over $\bar{S}$. The \emph{category $\Defo_S(\underline{\bar{\mcE}})$ of deformations of $\underline{\bar{\mcE}}$ to $S$} has 
\begin{description}
\item[as objects] all pairs $(\underline{\mcE}\,,\,\alpha\colon j^*\underline{\mcE}\isoto\underline{\bar{\mcE}})$ where $\underline{\mcE}$ is a global $\gpSch$-shtuka over $S$ and $\alpha$ an isomorphism of global $\gpSch$-shtukas over $\bar{S}$,
\item[as morphisms] isomorphisms between the $\underline{\mcE}$'s that are compatible with the $\alpha$'s.
\end{description}
If $\underline{\bar\mcL}:=L^+_{\infty,\mcM_{\beta^{-1}}} (\underline{\bar{\mcE}})$ is the associated local $\mcM_{\beta^{-1}}$-shtuka over $\bar{S}$, we similarly define the \emph{category $\Defo_S(\underline{\bar\mcL})$ of deformations of $\underline{\bar\mcL}$ to $S$}. By rigidity of quasi-isogenies (\cite[Prop 5.9]{AH_Local} and Proposition~\ref{PropRigidityLocal}), all $\Hom$-sets in these categories contain at most one element. 
\end{defn}

The following result generalizes \cite[Thm~8.4]{HartlAbSh}. It is the analogue of the classical Serre-Tate theorem for abelian varieties. 

\begin{prop}\label{PropSerre-Tate}
In the situation of Definition~\ref{DefDefo} the functor 
\[
\Defo_S(\underline{\bar{\mcE}}) \isoto \Defo_S(\underline{\bar\mcL})\,\qquad (\underline{\mcE}, \alpha)\longmapsto\bigl(L^+_{\infty,\mcM_{\beta^{-1}}}(\underline{\mcE}), L^+_{\infty,\mcM_{\beta^{-1}}}(\alpha)\bigr)
\]
induced from $L^+_{\infty,\mcM_{\beta^{-1}}}$ is an equivalence. 
\end{prop}
\begin{proof}

We construct the inverse of the above functor. Write $x\colon S\to \Spf \Breve{\Oo}_\infty$ for the structure morphism of $S\in\Nilp_{\Breve{\Oo}_\infty}$, and let $\bar{x}:=x\circ j\colon \bar{S}\to \Spf \Breve{\Oo}_\infty$. Write $\underline{\bar{\mcE}}\in\Sht_{\gpSch,D,X\times\infty}(\bar{S})$ as $\underline{\bar{\mcE}}=(\bar{x},(\bar{\mcE},\bar{\psi}),(\bar{\mcE}',\bar{\psi}'),\bar{\varphi},\bar{\varphi}')$. It suffices to treat the case where $\mathcal{I}^q = (0)$. In this case the morphism $\tau=\Frob_{q,S}$ factors as $\tau=j\circ i$ as in \eqref{Eq_TauFactors}. Let $(\underline\mcL,\hat{\alpha}\colon j^*\underline\mcL \isoto \underline{\bar\mcL})$ be an object of $\Defo_S(\underline{\bar\mcL})$. Since $X_S\smallsetminus\Gamma_x= (X\smallsetminus\{\infty\})_S$, we can define the global $\gpSch$-shtuka $\underline{\widetilde{\mcE}}:= [-1](x,i^*(\bar{\mcE},\bar{\psi}),i^*(\bar{\mcE}',\bar{\psi}'),i^*\bar{\varphi},i^*\bar{\varphi}')$ over $S$, where $[-1]$ denotes the index shift from Corollary~\ref{CorIsog[n]}\ref{CorIsog[n]_A}. It satisfies $[1](j^*\underline{\widetilde{\mcE}})= (\bar{x},({}^{\tau\!}\bar{\mcE},{}^{\tau\!}\bar{\psi}),({}^{\tau\!}\bar{\mcE}',{}^{\tau\!}\bar{\psi}'),{}^{\tau\!}\bar{\varphi},{}^{\tau\!}\bar{\varphi}')$. The isomorphisms $\bar{\varphi}'\circ\bar{\varphi}\colon (\bar{\mcE},\bar{\psi})|_{(X\smallsetminus\{\infty\})_S}\isoto ({}^{\tau\!}\bar{\mcE},{}^{\tau\!}\bar{\psi})|_{(X\smallsetminus\{\infty\})_S}$ and ${}^{\tau\!}\bar{\varphi}\circ\bar{\varphi}'\colon (\bar{\mcE}',\bar{\psi}')|_{(X\smallsetminus\{\infty\})_S}\isoto ({}^{\tau\!}\bar{\mcE}',{}^{\tau\!}\bar{\psi}')|_{(X\smallsetminus\{\infty\})_S}$ define a quasi-isogeny $\underline{\bar{\mcE}} \rightarrow [1](j^* \underline{\widetilde{\mcE}})$. We compose it with the quasi-isogeny $\Pi\colon [1](j^* \underline{\widetilde{\mcE}}) \to j^*\underline{\widetilde{\mcE}}$ from Corollary~\ref{CorIsog[n]}\ref{CorIsog[n]_B} to obtain a quasi-isogeny $\bar{\delta}\colon \underline{\bar{\mcE}} \rightarrow j^* \underline{\widetilde{\mcE}}$ which is an isomorphism outside $\infty$. 

We write $L^+_{\infty,\mcM_{\beta^{-1}}}(\underline{\widetilde{\mcE}})=\underline{\widetilde\mcL}$ and $\underline{\bar\mcL}=(\bar\mcL,\overline{\widehat{\varphi}})$. By Corollary~\ref{CorGlobalLocalFunctorWithChains} we have
\[
{}^{\tau\!}\underline{\bar{\mcL}}= L^+_{\infty,\mcM}({}^{\tau\!}\bar{\mcE}_{-1}) = L^+_{\infty,\mcM}([-1]{}^{\tau\!}\underline{\bar{\mcE}}) = L^+_{\infty,\mcM}(j^*\underline{\widetilde{\mcE}})= j^*\underline{\widetilde{\mcL}}.
\]
We view ${\overline{\widehat{\varphi}}}\colon \underline{\bar\mcL} \to {}^{\tau\!}\underline{\bar\mcL}$ as a quasi-isogeny as in Example~\ref{ExFrobIsogLocSht}, 
and compose it with $\hat{\alpha}$ to obtain the quasi-isogeny $\overline{\hat{\gamma}}:= \overline{\widehat{\varphi}} \circ \hat{\alpha} \colon j^*\underline\mcL \rightarrow {}^{\tau\!}\underline{\bar{\mcL}}=j^*\underline{\widetilde\mcL}$. By rigidity of quasi-isogenies (Proposition~\ref{PropRigidityLocal}) it lifts to a quasi-isogeny $\hat{\gamma}\colon \underline\mcL \rightarrow \underline{\widetilde\mcL}$ with $j^*\hat\gamma=\overline{\hat{\gamma}}$. As in Proposition~\ref{PropQIsogLocalGlobal}\ref{PropQIsogLocalGlobal_B}, we put $\underline{\mcE}:=\hat\gamma^*\underline{\widetilde{\mcE}}$ and recall that there is a quasi-isogeny $\gamma\colon \underline{\mcE} \rightarrow \underline{\widetilde{\mcE}}$ of global $\gpSch$-shtukas, which is an isomorphism outside $\infty$, and satisfies $L_{\infty,\mcM_{\beta^{-1}}}(\underline{\mcE})=\underline{\mcL}$ and $L_{\infty,\mcM_{\beta^{-1}}}(\gamma)=\hat\gamma$. We may now define the functor 
\[
Defo_S(\underline{\bar\mcL})\rightarrow Defo(\bar{\underline{\mcE}})
\]
by sending $(\underline\mcL,\hat{\alpha}\colon j^*\underline\mcL \isoto \underline{\bar\mcL})$ to $(\underline{\mcE},\bar{\delta}^{-1}\circ j^*\gamma)$. The quasi-isogeny $\alpha:=\bar{\delta}^{-1}\circ j^\ast \gamma\colon j^*\underline{\mcE}\to \underline{\bar{\mcE}}$ is an isomorphism outside $\infty$ by construction, and also at $\infty$ because $L_{\infty,\mcM_{\beta^{-1}}}(\alpha) = L_{\infty,\mcM_{\beta^{-1}}}(\Pi\circ\bar{\varphi}'\circ\bar{\varphi})^{-1}\circ\overline{\widehat{\varphi}} \circ \hat{\alpha} = \hat\alpha$ by Corollary~\ref{CorGlobalLocalFunctorWithChains}. It can easily be seen by the above construction that these functors are actually inverse to each other.
\end{proof}

Next we recall the definition of the Rapoport-Zink spaces for local shtukas bounded by a cocharacter $\mu\in X_*(T)$. Recall the bound $\mcZ^{\leq\mu}$ from Definition~\ref{Def_BoundBy_mu} and the local bound $\mcZ^{\leq\mu}\times_{\widetilde{X}_\mu} \Spec \Oo_\mu$ from Definition~\ref{DefLocShtBounded} used to bound local shtukas by $\mu$. That local bound depends on the choice of the map $\Spec \Oo_\mu\to\widetilde{X}_\mu$.

\begin{defn}\label{DefRZforM}
Fix an element $b\in LM(\overline{\F}_\infty)=L_\infty\gengpSch(\overline{\F}_\infty)$. Consider the local $\mcM$-shtuka $\ulLocalFramingObject=(L^+_\infty\mcM_{\overline{\F}_\infty},b)\in \LocSht_{\mcM}$ over $\overline{\F}_\infty$. 

(1) We define the \emph{Rapoport-Zink space} $\RZ_{\mcM,\ulLocalFramingObject}^{\leq \mu}$ \emph{(of the framing object $\underline{\mathbb{L}}$)} as the functor whose $S$-points, for $S\in\Nilp_{\Breve{\Oo}_\mu}$, are given by
\begin{align} \label{Eq_DefRZforM}
    \RZ_{\mcM,\ulLocalFramingObject}^{\leq \mu}(S):=\Bigl\{\,(\underline{\mcL},\hat{\delta}) & \text{ where $\underline{\mcL}$ is local $\mcM$-shtuka over $S$ bounded by $\mu$}, \nonumber 
    \\
    & \text{and }\hat{\delta}:\underline{\mcL}\to \ulLocalFramingObject_S \text{ is a quasi-isogeny}\,\Bigr\}. 
\end{align}
Here ``bounded by $\mu$'' means bounded by the bound $\mcZ^{\leq\mu}$ from Definition \ref{Def_BoundBy_mu}; compare \S\,\ref{ParInnerFormM}. 

(2) We define the \emph{affine Deligne--Lusztig variety} $X_{\mcM}^{\leq\mu}(b)$ by 
\begin{equation}
    X_{\mcM}^{\leq \mu}(b)(\overline{\F}_\infty):=\{\overline{g}\in (L_\infty\mcM/L^+_\infty\mcM)(\overline{\F}_\infty)\colon \tau_M(g)^{-1}b g\text{ is bounded by }\mu\}.
\end{equation}
\end{defn}

\begin{remark}\label{Rem_LocShtInversePhi3}
In the literature on affine Deligne-Lusztig varieties for an element $\tilde b\in M(\Breve{\FcnFld}_\infty)=LM(\overline{\F}_q)$ one usually requires that $g^{-1} \tilde{b}\,\tau_M(g)$ is bounded. However the $\tilde{b}$ in that literature is equal to our $b^{-1}$ as explained in Remarks~\ref{Rem_LocShtInversePhi} and \ref{Rem_LocShtInversePhi2}. Note that $\tau_M(g)^{-1} b g$ is bounded by $\mu$ if and only if $g^{-1}b^{-1}\tau_M(g)$ is bounded by $-\mu$. So our Rapoport-Zink spaces and affine Deligne-Lusztig varieties are the same as the ones usually considered in the literature after changing $\mu$ and $b$ to $-\mu$ and $b^{-1}$.
\end{remark}

\begin{remark} \label{RemRZIndProper}
Using the description of $\Fl_{\mcM,\infty}=L_\infty\mcM/L^+_\infty\mcM$ from Lemma~\ref{LemmaAffineFlagVar}, the space $\RZ_{\mcM,\ulLocalFramingObject}^{\leq \mu}$ can also be described as 
\begin{equation}
    \RZ_{\mcM,\ulLocalFramingObject}^{\leq \mu}(S):=\{\overline{g}\in (L_\infty\mcM/L^+_\infty\mcM)(S)\colon \tau_M(g)^{-1}b g\text{ is bounded by }\mu\},
\end{equation}
where $\overline{g}\in (L_\infty\mcM/L^+_\infty\mcM)(S)$ is represented by $g\in L_\infty\mcM(S')$ for some \'etale covering $S'\to S$. 

In particular, $\RZ_{\mcM,\ulLocalFramingObject}^{\leq \mu}$ is an ind-closed ind-subscheme of $\Fl_{\mcM,\infty}$, compare Remark~\ref{Rem_HeckeBounded} and Theorem~\ref{Thm_HeckeBounded}. 
\end{remark}

We recall the following 

\begin{thm}\label{ThmRRZSp} 
$\RZ_{\mcM,\ulLocalFramingObject}^{\leq \mu}$ is representable by a formal scheme over $\Spf\Breve{\Oo}_\mu$, which is locally formally of finite type and separated. Its underlying reduced subscheme is precisely $X_{\mcM}^{\leq \mu}(b)$. In particular, $X_{\mcM}^{\leq \mu}(b)$ is a reduced scheme locally of finite type and separated over $\overline{\F}_\infty$. 
\end{thm}
\begin{proof}
In view of Remarks~\ref{Rem_LocShtInversePhi2} and \ref{Rem_LocShtInversePhi3} this was proven in \cite[Theorem 4.18]{AH_Local}.
\end{proof}

\begin{remark}\label{RemMActsOnRZ}
The group $\QIsog_{\BaseFldInSectUnif}(\ulLocalFramingObject)$ of quasi-isogenies of $\ulLocalFramingObject:=(L^+_\infty\mcM_{\overline{\F}_\infty},b)$ acts on $\RZ_{\mcM,\ulLocalFramingObject}^{\leq \mu}$ via $j\cdot(\underline{\widehat{\mcL}},\hat\delta):=(\underline{\widehat{\mcL}},j\circ\hat\delta)$, for $j\in \QIsog_{\BaseFldInSectUnif}(\ulLocalFramingObject)$. There is a connected algebraic group $J_{\ulLocalFramingObject}$ over $\FcnFld_\infty$ whose $R$-points, for a $\FcnFld_\infty$-algebra $R$, is the group
\begin{equation}\label{EqGroupJ}
J_{\ulLocalFramingObject}(R):=\bigl\{\,j \in \mcM(R\otimes_{\FcnFld_\infty} \Breve{\FcnFld}_\infty)\colon \tau_M(j)^{-1}b j=b\,\bigr\}\,,
\end{equation}
see \cite[Remark~4.16]{AH_Local}.
In particular, $\QIsog_{\BaseFldInSectUnif}(\ulLocalFramingObject)=J_{\ulLocalFramingObject}(\FcnFld_\infty)$. This group also acts on $X_{\mcM}^{\leq \mu}(b)$ via multiplication $j\colon g\mapsto j\cdot g$ for $j\in J_{\ulLocalFramingObject}(\FcnFld_\infty)$.
\end{remark}

\subsection{Tate modules}
In this section, we consider the fiber products
\begin{align}
\Sht_{\gpSch,D,\widehat{\infty}\times\infty}&:=\Sht_{\gpSch,D,X\times\infty} \times_X \Spf \Oo_\infty \qquad\text{and} \label{EqShtHatInfty} \\
\Sht_{\gpSch,D,\infty\times\infty}&:=\Sht_{\gpSch,D,X\times\infty} \times_X \Spec \F_\infty \label{EqShtInfty} 
\end{align}
on which the moving leg $x\colon S\to X$ factors through $\widehat{\infty}:=\Spf \Oo_\infty$ or $\infty:=\Spec\F_\infty$, respectively. Recall that we have an additional fixed leg at $\infty$. This guarantees that we have a Tate module at every place $v\neq \infty$.

Let $\mathbb{O}^\infty:=\mathbb{O}_{\FcnFld}^\infty:=\prod_{v\neq\infty}\Oo_v$ be the ring of integral adeles of $X$ (i.e.~of the function field $\FcnFld=\F_q(X)$) outside $\infty$. Let $\mathbb{A}^\infty:=\mathbb{A}_{\FcnFld}^\infty:=\mathbb{O}^\infty \otimes_{\Oo_X} \FcnFld$ be the ring of adeles of $X$ outside $\infty$. The group $\gpSch(\mathbb{O}^\infty)$ acts through Hecke correspondences on the tower $\{\Sht_{\gpSch,D,\widehat{\infty}\times\infty}\}_D$.~We want to extend this to an action of $\gengpSch(\mathbb{A}^\infty)$; see Definition~\ref{Defn-Hecke-corr} below. For this purpose, we generalize the notion of level structures on global $\gpSch$-shtukas in this subsection. 

Let $S$ be a connected scheme in $\Nilp_{\Oo_\infty}$. We fix a geometric base point $\bar{s}$ of $S$. Let $\Rep_{\mathbb{O}^\infty}\gpSch$ be the category of pairs $(V,\rho)$, where $V$ is a finite free $\mathbb{O}^\infty$-module and $\rho\colon\gpSch\times_X\Spec\mathbb{O}^\infty\to\GL_{\mathbb{O}^\infty}(V)$ is an $\mathbb{O}^\infty\,$-morphism of algebraic groups. Let $\mathrm{Funct}^\otimes(\Rep_{\mathbb{O}^\infty}\gpSch,\; \mathfrak{M} od_{\mathbb{O}^\infty [\pi_1^\et(S,\bar{s})]})$ denote the category of tensor functors from $\Rep_{\mathbb{O}^\infty} \gpSch$ to the category $\mathfrak{M} od_{\mathbb{O}^\infty[\pi_1^\et(S,\bar{s})]}$ of $\mathbb{O}^\infty[\pi_1^\et(S,\bar{s})]$-modules. For a proper closed subscheme $D$ of $X\smallsetminus\{\infty\}$ the sheaf $\Oo_D$ is an $\mathbb{O}^\infty$-module and we can consider $\rho|_D\colon\gpSch|_D:=\gpSch\times_X D\to \GL_{\Oo_D}(V\otimes_{\mathbb{O}^\infty} \Oo_D)$.

Let $\underline{\mcE}=(x,\mcE,\mcE',\varphi,\varphi')\in\Sht_{\gpSch,\varnothing,\widehat{\infty}\times\infty}(S)$ be a global $\gpSch$-shtuka over $S$. Fix a proper closed subscheme $D$ of $X\smallsetminus\{\infty\}$ and let $\underline{\mcE}|_{D_S}:=\underline{\mcE}\times_{X_S}D_S$ denote the pullback of $\underline{\mcE}$ to $D_S$. Also fix $(V,\rho)\in \Rep_{\mathbb{O}^\infty}\gpSch$ and let 
\[
\mcF:=\rho_*\mcE|_{D_S}:=(\rho|_D)_*(\mcE|_{D_S}):=(\mcE|_{D_S})\overset{\gpSch|_D}{\times}(V\otimes_{\mathbb{O}^\infty} \Oo_D)
\]
denote the pushout vector bundle on $D_S$ of rank equal to $\dim V$. Since $D_S$ is disjoint from the graphs of the legs $x$ and $\infty$, the maps $\varphi|_{D_S}$ and $\varphi'|_{D_S}$ of $\underline{\mcE}|_{D_S}$ are isomorphisms. We equip $\mcF$ with the Frobenius isomorphism $\varphi_\mcF:=(\rho|_D)_*(\varphi'\circ\varphi)|_{D_S}\colon \mcF\isoto {}^{\tau\!} \mcF$. It is an isomorphism of $\Oo_{D_S}$-modules. For the fixed geometric base point $\bar s=\Spec k$ of $S$ let $\mcF_{\bar s}:= \mcF\otimes_{\Oo_S} k$.
Let 
\begin{equation}
(\rho_*\underline{\mcE}|_{D_{\bar{s}}})^\varphi:=\{m\in \mcF_{\bar s}\colon \varphi_\mcF(m)={}^{\tau\!} m\}
\end{equation}
be the $\varphi$-invariants. They form a free $\Oo_D$-module of rank equal to $\dim V$, equipped with a continuous action of the \'etale fundamental group $\pi_1^\et(S,\bar{s})$. This module $(\rho_*\underline{\mcE}|_{D_{\bar{s}}})^\varphi$ is independent of $\bar{s}$ up to a change of base point. 

Let $\check{\mathcal{T}}_{\mcE}\in \mathrm{Funct}^\otimes(\Rep_{\mathbb{O}^\infty}\gpSch,\; \mathfrak{M} od_{\mathbb{O}^\infty [\pi_1^\et(S,\bar{s})]})$ be the tensor functor defined as 
\begin{equation}
    \check{\mathcal{T}}_{\underline{\mcE}}(\rho):=\underset{D\subset X\setminus \infty}{\varprojlim}(\rho_*\underline\mcE|_{D_{\bar{s}}})^\varphi.
\end{equation}
Likewise, $\check{\mathcal{V}}_{\mcE}\in \mathrm{Funct}^\otimes(\Rep_{\mathbb{O}^\infty}\gpSch,\; \mathfrak{M} od_{\mathbb{A}^\infty [\pi_1^\et(S,\bar{s})]})$ is the tensor functor defined as 
\begin{equation}
    \check{\mathcal{V}}_{\underline{\mcE}}(\rho):=\mathbb{A}^\infty \otimes_{\mathbb{O}^\infty} \underset{D\subset X\setminus \infty}{\varprojlim}(\rho_*\underline{\mcE}|_{D_{\bar{s}}})^\varphi .
\end{equation}

\begin{defn}\label{DefTateModule}
We define the (\emph{dual})\footnote{In the case of abelian varieties $\mcA$, the Tate module $T(\mcA)$ is the homology $\prod\limits_{\ell}H_{1,\et}(\mcA,\Z_{\ell})$, while for $\gpSch$-shtukas $\underline{\mcE}$, the Tate module $\check{\mathcal{T}}_{\underline{\mcE}}(\rho)$ is the cohomology $\prod\limits_{v\neq\infty}H^1_{\et}(\rho_*\underline{\mcE},\Oo_v)$, see \cite[Definition~3.4.1]{HartlKim}, hence the ``dual'' terminology.} \emph{Tate-module functor} as follows.
\begin{align}\label{tatefunctor}
\begin{split}
\check{\mathcal{T}}_{-}\colon \Sht_{\gpSch,D,\widehat{\infty}\times\infty}(S)&\longrightarrow \mathrm{Funct}^\otimes (\Rep_{\mathbb{O}^\infty}\gpSch\,,\,\mathfrak{M} od_{\mathbb{O}^\infty[\pi_1^\et(S,\bar{s})]})\,\\
\underline{\mcE} &\longmapsto \check{\mathcal{T}}_{\underline\mcE},
\end{split}
\end{align}
We define the \emph{(dual) rational Tate-module functor} as follows. 
\begin{equation}
\begin{split}
\check{\mathcal{V}}_{-}\colon \Sht_{\gpSch,D,\widehat{\infty}\times\infty}(S) &\longrightarrow \mathrm{Funct}^\otimes (\Rep_{\mathbb{O}^\infty}\gpSch\,,\,\mathfrak{M} od_{\mathbb{A}^\infty[\pi_1^\et(S,\bar{s})]})\,\\ \label{rationaltatefunctor}
\underline{\mcE} &\longmapsto  \check{\mathcal{V}}_{\underline\mcE}.
\end{split}
\end{equation}
The functor $\check{\mathcal{V}}_{-}$ moreover transforms quasi-isogenies $\delta:\underline{\mcE}\to\underline{\mcE}'$ of global $\gpSch$-shtukas into isomorphisms $\check{\mathcal{V}}_{\delta}: \check{\mathcal{V}}_{\underline{\mcE}}\to\check{\mathcal{V}}_{\underline{\mcE}'}$ of their rational Tate modules. 
\end{defn}

Let $\omega_{\mathbb{O}^\infty}\colon \Rep_{\mathbb{O}^\infty}\gpSch \to \mathfrak{M} od_{\mathbb{O}^\infty}$ and $\omega:=\omega_{\mathbb{O}^\infty}\otimes_{\mathbb{O}^\infty}\mathbb{A}^\infty\colon \Rep_{\mathbb{O}^\infty}\gpSch \to \mathfrak{M} od_{\mathbb{A}^\infty}$ denote the forgetful functors sending $(V,\rho)$ to $V$, and $V\otimes_{\mathbb{O}^\infty} \mathbb{A}^\infty$, respectively. For a global $\gpSch$-shtuka $\underline{\mcE}$ over $S$, consider the sets of isomorphisms of tensor functors $\Isom^{\otimes}(\omega_{\mathbb{O}^\infty},\check{\mathcal{T}}_{\underline{\mcE}})$ and $\Isom^{\otimes}(\omega,\check{\mathcal{V}}_{\underline{\mcE}})$. These sets are non-empty by Lemma~\ref{LSisnonempty} below. Since $X\smallsetminus\{\infty\}$ is the spectrum of a Dedekind domain, by the generalized Tannakian formalism \cite[Corollary~5.20]{Wed}, we have $ \gengpSch(\mathbb{A}^\infty)=\Aut^\otimes(\omega)$ and $\gpSch(\mathbb{O}^\infty)=\Aut^\otimes(\omega_{\mathbb{O}^\infty})$. By the definition of $\check{\mathcal{T}}_{\underline{\mcE}}$ in \eqref{tatefunctor}, $\Isom^{\otimes}(\omega_{\mathbb{O}^\infty},\check{\mathcal{T}}_{\underline{\mcE}})$ admits an action of $\pi_1^\et(S,\bar{s})\times\gpSch(\mathbb{O}^\infty)$ where $\gpSch(\mathbb{O}^\infty)$ acts through $\omega_{\mathbb{O}^\infty}$ and $\pi_1^\et(S,\bar{s})$ acts through $\check{\mathcal{T}}_{\underline{\mcE}}$. Likewise, $\Isom^{\otimes}(\omega,\check{\mathcal{V}}_{\underline{\mcE}})$ admits an action of $\pi_1^\et(S,\bar{s})\times \gengpSch(\mathbb{A}^\infty)$.

\begin{lem}\label{LSisnonempty}
The sets $\Isom^{\otimes}(\omega_{\mathbb{O}^\infty},\check{\mathcal{T}}_{\underline{\mcE}})$ and $\Isom^{\otimes}(\omega,\check{\mathcal{V}}_{\underline{\mcE}})$ are non-empty. Moreover, $\Isom^{\otimes}(\omega_{\mathbb{O}^\infty},\check{\mathcal{T}}_{\underline{\mcE}})$ (resp.~$\Isom^{\otimes}(\omega,\check{\mathcal{V}}_{\underline{\mcE}})$) is a principal homogeneous space under the group $\gpSch(\mathbb{O}^\infty)$ (resp.~$\gengpSch(\mathbb{A}^\infty)$).
\end{lem}
\begin{proof}
This is \cite[Lemma~6.2]{AH_Unif}.
\end{proof}

\begin{defn}\label{DeFfrobIsog}
Let $m\in\N_0$.  
For every global $\gpSch$-shtuka $\underline{\mcE}:=(x,(\mcE,\psi),(\mcE',\psi'),\varphi,\varphi')\in \Sht_{\gpSch,D,\infty\times\infty}(S)$, the \emph{$q^m$-Frobenius quasi-isogeny} $\Phi_{\underline{\mcE}}^m\colon\underline{\mcE}\longto{}^{\tau^m\!} \underline{\mcE}
$ is defined as the tuple $\Phi_{\underline{\mcE}}^m = (f,f')$ with
\begin{alignat*}{4}
f & :={}^{\tau^{m-1}}(\varphi'\circ\varphi)\circ\ldots\circ{}^{\tau\!} (\varphi'\circ\varphi)\circ (\varphi'\circ\varphi) && \colon\;(\mcE,\psi)|_{(X\setminus\infty)_S} && \longto{}^{\tau^m\!}(\mcE,\psi)|_{(X\setminus\infty)_S} \\
f' & :={}^{\tau^m\!}\varphi\circ {}^{\tau^{m-1}}(\varphi'\circ\varphi)\circ\ldots\circ{}^{\tau\!} (\varphi'\circ\varphi)\circ \varphi' && \colon\;(\mcE',\psi')|_{(X\setminus\infty)_S} && \longto{}^{\tau^m\!}(\mcE',\psi')|_{(X\setminus\infty)_S}
\end{alignat*}
satisfying ${}^{\tau^m\!}\varphi\circ f = f'\circ \varphi$ and ${}^{\tau^m\!}\varphi'\circ f'={}^{\tau\!}f \circ \varphi'$; see Definition~\ref{DefIsogGlobalSht}.

Here we observe that the global $\gpSch$-shtuka ${}^{\tau^m\!}\underline{\mcE}$ is obtained by pulling back $\underline{\mcE}$ under the absolute $q^m$-Frobenius $\tau^m=\Frob_{q^m,S}\colon S\to S$. The leg $x$ of $\underline{\mcE}$ satisfies $x\circ\Frob_{q^m,S}=\Frob_{q^m,\F_\infty}\circ x=x$, because $x\colon S\to X$ factors through $\{\infty\}=\Spec\F_\infty\in X$ and $\Frob_{q^m,\F_\infty}=\id_{\F_\infty}$ by our assumption $\F_\infty=\F_q$. So $\underline{\mcE}$ and ${}^{\tau^m\!}\underline{\mcE}$ have the leg.
\end{defn}

\begin{cor}\label{CorFrobIsog}
Keep the situation of Definition~\ref{DeFfrobIsog}.

If $\gamma\in\Isom^{\otimes}(\omega,\check{\mathcal{V}}_{\underline{\mcE}})$ then we have an equality ${}^{\tau^m\!} (\gamma)=\check{\mathcal{V}}_{\Phi_{\underline{\mcE}}^m}\circ\gamma$ inside $\Isom^{\otimes}(\omega,\check{\mathcal{V}}_{{}^{\tau^m\!} \underline{\mcE}})$.
\end{cor}

\begin{proof}
This 
follows by the same argument as in \cite[Corollary~6.3]{AH_Unif}.
\end{proof}

\begin{example}\label{ExIsogTildephi}
For $m=1$ the composition of the $q$-Frobenius quasi-isogeny $\Phi_{\underline{\mcE}}\colon\underline{\mcE}\longto{}^{\tau\!} \underline{\mcE}$ with the isogeny $\Pi\colon \underline{\mcE}\longto [-1](\underline{\mcE})$ from Corollary~\ref{CorIsog[n]}\ref{CorIsog[n]_B} is given by the maps $\diagphi_i\colon \mcE_i \to {}^{\tau\!}\mcE_{i+1}$ from diagram~\eqref{Eq_ShtukaDiagTildephi}, which are isomorphisms on $(X\smallsetminus\{\infty\})_S$.
\end{example}

\begin{defn}\label{DefRatLevelStr}
Let $S\in\Nilp_{\Oo_\infty}$ be a connected scheme. 
For a compact open subgroup $H\subseteq  \gengpSch(\mathbb{A}^\infty)$, we define a \emph{rational $H$-level structure} $\bar\gamma$ on a global $\gpSch$-shtuka $\underline{\mcE}$ over $S$ as a $\pi_1^\et(S,\bar{s})$-invariant $H$-orbit $\bar\gamma=\gamma H$ in $\Isom^{\otimes}(\omega,\check{\mathcal{V}}_{\underline{\mcE}})$. For a non-connected scheme $S$, we make a similar definition by choosing a base point on each connected component and a rational $H$-level structure on the restriction to each connected component separately.

Let $\mcZ$ be a bound in $\Gr_{\gpSch,X\times\infty}\times_X \Spf\Oo_\infty$ which in terms of Definition~\ref{DefThreeTypes} is of local type \ref{DefThreeTypes_B} at the moving leg and of finite type \ref{DefThreeTypes_C} at the fixed leg $\infty$. Its reflex scheme is $\Spec\Oo_\mcZ$ for a finite ring extension $\Oo_\infty\subset \Oo_\mcZ$. We write $\kappa_\mcZ$ for the residue field of $\Oo_\mcZ$. We denote by $\Sht_{\gpSch,H,\widehat{\infty}\times\infty}$ (respectively $\Sht_{\gpSch,H,\widehat{\infty}\times\infty}^{\mcZ}
$) the category fibered in groupoids over $\Nilp_{\Oo_\infty}$ (respectively $\Nilp_{\Oo_\mcZ}$) whose $S$-valued points $\Sht_{\gpSch,H,\widehat{\infty}\times\infty}(S)$ (respectively $\Sht_{\gpSch,H,\widehat{\infty}\times\infty}^{\mcZ}(S)$) is the category whose 
\begin{itemize}
    \item objects are tuples $(\underline{\mcE},\gamma H)$ consisting of a global $\gpSch$-shtuka  $\underline{\mcE}$ over $S$ (respectively, which is bounded by $\mcZ$) together with a rational $H$-level structure $\gamma H$;
    \item morphisms are quasi-isogenies of global $\gpSch$-shtukas that are isomorphisms above $\infty$ and are compatible with the $H$-level structures.
\end{itemize}
We also set
\begin{align}
\Sht_{\gpSch,H,\infty\times\infty} & := \Sht_{\gpSch,H,\widehat{\infty}\times\infty}\times_{\Oo_\infty} \Spec \F_\infty \qquad \text{and} \\
\Sht_{\gpSch,H,\infty\times\infty}^{\mcZ} & := \Sht_{\gpSch,H,\widehat{\infty}\times\infty}^{\mcZ} \times_{\Oo_\mcZ} \Spec \kappa_\mcZ
\end{align}
\end{defn}

This definition of level structures generalizes our initial Definition~\ref{DefD-LevelStr} as follows. 
\begin{prop}\label{H_DL-Str} 
Let $D\subset X$ be a proper closed subscheme disjoint from $\infty$. Consider the compact open subgroup $H_D:=\ker\bigl(\gpSch(\mathbb{O}^\infty)\to\gpSch(\Oo_D)\bigr)\subset \gengpSch(\mathbb{A}^\infty)$. There is a canonical isomorphism of stacks
$$
\Sht_{\gpSch,D,\widehat{\infty}\times\infty} \enspace \isoto \enspace\Sht_{\gpSch,H_D,\widehat{\infty}\times\infty}.
$$
In particular, it induces an isomorphism of the moduli stacks of bounded $\gpSch$-shtukas
$$
\Sht_{\gpSch,D,\widehat{\infty}\times\infty}^\mcZ \enspace \isoto \enspace\Sht_{\gpSch,H_D,\widehat{\infty}\times\infty}^\mcZ.
$$
\end{prop}

\begin{proof}
This follows as in \cite[Theorem 6.5]{AH_Unif}. 
\end{proof}

\begin{remark}\label{RemIndAlgStack}
(a) Let $(\underline{\mcE},\gamma H)\in\Sht_{\gpSch,H,\widehat{\infty}\times\infty}(S)$. Then every choice of a representative $\gamma\in\Isom^{\otimes}(\omega,\check{\mathcal{V}}_{\underline{\mcE}})$ of the $H$-level structure $\gamma H$ induces a representation of the \'etale fundamental group given by
\begin{equation}\label{EqRepOfpi1}
\rho_{\underline{\mcE},\gamma}\colon \pi_1^\et(S,\bar s)\;\longto\; H\,,\quad g\;\longmapsto\;\gamma^{-1}\circ g(\gamma)\;=:\;\rho_{\underline{\mcE},\gamma}(g)\,.
\end{equation}
It is a group homomorphism because $$\rho_{\underline{\mcE},\gamma}(gg')=\gamma^{-1}\circ g(\gamma)\circ g\bigl(\gamma^{-1}\circ g'(\gamma)\bigr)=\rho_{\underline{\mcE},\gamma}(g)\cdot\rho_{\underline{\mcE},\gamma}(g'),$$ as $\gamma^{-1}\circ g'(\gamma)$ lies in $H$ on which $\pi_1^\et(S,\bar s)$ acts trivially. Replacing $\gamma$ by $\gamma h$, for $h\in H$, gives $\rho_{\underline{\mcE},\gamma}=\Int_h\circ\rho_{\underline{\mcE},\gamma h}$.

\medskip\noindent
(b) For any compact open subgroup $H\subseteq \gengpSch(\mathbb{A}^\infty)$ and any element $h\in \gengpSch(\mathbb{A}^\infty)$, there is an isomorphism $\Sht_{\gpSch,H,\widehat{\infty}\times\infty}\;\isoto\;\Sht_{\gpSch,h^{-1}Hh,\widehat{\infty}\times\infty}$ of stacks, given by $(\underline{\mcE},\gamma H)\mapsto\bigl(\underline{\mcE},\gamma h(h^{-1}Hh)\bigr)$. 
\end{remark}

\begin{prop}\label{PropLSGGsht1}
\begin{enumerate}
\item \label{PropLSGGsht1_A}
For any compact open subgroup $H\subseteq \gengpSch(\mathbb{A}^\infty)$, the stack 
$\Sht_{\gpSch,H,\widehat{\infty}\times\infty}$ is an ind-Deligne-Mumford stack, ind-separated and locally of ind-finite type over $\Spf \Oo_\infty$. 
If $\mathfrak{m}_\mcZ$ denotes the maximal ideal of $\Oo_\mcZ$, then for every $e\in\N$ the fiber product 
\[
\Sht_{\gpSch,H,\widehat{\infty}\times\infty}^\mcZ \widehattimes_{\Oo_\mcZ} \Spec \Oo_\mcZ/\mathfrak{m}_\mcZ^e
\]
is an (algebraic) Deligne-Mumford stack separated and locally of finite type over $\Spec\Oo_\mcZ/\mathfrak{m}_\mcZ^e$, and $\Sht_{\gpSch,H,\widehat{\infty}\times\infty}^\mcZ$ is a locally noetherian, adic formal algebraic Deligne-Mumford stack, separated and locally of finite type over $\Spf\Oo_\mcZ$ in the sense of \cite[Appendix~A]{HartlAbSh}. 
\item \label{PropLSGGsht1_B}
If $\widetilde H\subset H\subseteq \gengpSch(\mathbb{A}^\infty)$ are compact open subgroups then the forgetful morphism
\begin{equation}\label{forgetting-level-map}
\Sht_{\gpSch,\widetilde H,\widehat{\infty}\times\infty}\;\longto\;\Sht_{\gpSch,H,\widehat{\infty}\times\infty},\quad (\underline{\mcE},\gamma\widetilde H)\;\longmapsto\;(\underline{\mcE},\gamma H)
\end{equation}
is finite \'etale and surjective. Moreover, the same is true for the stack $\Sht_{\gpSch,\widetilde H,\widehat{\infty}\times\infty}^\mcZ$. 
\item \label{PropLSGGsht1_C}
Furthermore, if $\widetilde H$ is a normal subgroup of $H$, then the group $H/\widetilde H$ acts on $\Sht_{\gpSch,\widetilde H,\widehat{\infty}\times\infty}$ from the right via $h\widetilde H\colon(\underline{\mcE},\gamma\widetilde H)\mapsto(\underline{\mcE},\gamma h\widetilde H)$ for $h\widetilde H\in H/\widetilde H$. 

Moreover, the stack $\Sht_{\gpSch,H,\widehat{\infty}\times\infty}$ is canonically isomorphic to the stack quotient $\bigl[\Sht_{\gpSch,\widetilde H,\widehat{\infty}\times\infty}\big/(H/\widetilde H)\bigr]$ and $\Sht_{\gpSch,\widetilde H,\widehat{\infty}\times\infty}$ is a right $H \slash \widetilde H$-torsor over $\Sht_{\gpSch,H,\widehat{\infty}\times\infty}$ under the forgetful morphism \eqref{forgetting-level-map}. The same is true for the bounded shtukas. 
\end{enumerate}
\end{prop}
\begin{proof}
Assertions~\ref{PropLSGGsht1_B} and \ref{PropLSGGsht1_C} are standard. See for example \cite[Theorem 6.7]{AH_Unif} for a detailed proof.      

\smallskip\noindent
\ref{PropLSGGsht1_A} The intersection $H_1:=H\cap\gpSch(\mathbb{O}^\infty)$ has finite index in $H$, because it is open and $H$ is compact. Thus the intersection $H_2:=\bigcap_{h\in H/H_1}h H_1 h^{-1}\subset H_1\subset \gpSch(\mathbb{O}^\infty)$ is compact open, normal in $H$, and of finite index in $\gpSch(\mathbb{O}^\infty)$. There is a proper closed subscheme $D\subset X$ with $H_D\subset H_2$, and this is a normal subgroup because $H_D$ is normal in $\gpSch(\mathbb{O}^\infty)$. Therefore, statement \ref{PropLSGGsht1_A} holds for $\Sht_{\gpSch,H_D,\widehat{\infty}\times\infty}$ by Theorems~\ref{H_DL-Str} and \ref{Thm_Sht2legsBounded}, see the definition in \eqref{EqShtHatInfty}. Consequently also $\Sht_{\gpSch,H_2,\widehat{\infty}\times\infty}$ and $\Sht_{\gpSch,H,\widehat{\infty}\times\infty}$ are ind-Deligne-Mumford stacks locally of ind-finite type over $\Spf \Oo_\infty$ by \ref{PropLSGGsht1_C} because they are obtained as stack quotients by finite groups. They are (ind-)separated over $\Spf \Oo_\infty$, because the forgetful morphisms in \ref{PropLSGGsht1_C} are finite surjective with (ind-)separated source.
\end{proof}

\subsection{Examples and relation to the previous literature}

We explain how our Theorems~\ref{Uniformization1} and \ref{Uniformization2} generalize uniformization results in the literature.

\begin{example}\label{ExDrinfeld}
Drinfeld~\cite{Drinfeld-elliptic-modules} defined ``elliptic modules'' (which are today called \emph{Drinfeld modules}), and constructed moduli spaces for them. In \cite{Drinfeld-commutative-subrings} he also defined the equivalent notion of \emph{elliptic sheaves}; see also \cite[Chapter~3]{Blum-Stuhler}. An elliptic sheaf over an $\F_q$-scheme $S$ is by definition a global $\gpSch$-shtuka $\underline{\mcE}=\bigl(x,(\mcE_i,\psi_i,\Pi_i,\diagphi_i)_{i\in\Z}\bigr)$ as in Corollary~\ref{CorShtWithChains} for $\gengpSch=\GL_r$, $\mu=(0,\ldots,0,-1)$. The group scheme $\gpSch\times_X (X\smallsetminus\{\infty\})$ equals $\GL_r$, but $\gpSch_\infty$ is an Iwahori subgroup of $\GL_r$ (e.g.~whose reduction is the Borel of lower triangular matrices), and 
\begin{equation}\label{EqExDrinfeld}
\beta = \left( \raisebox{4.6ex}{$
\xymatrix @C=0.2pc @R=0pc {
0 \ar@{.}[rrr]\ar@{.}[drdrdrdr] & & & 0 & z_\infty\\
1\ar@{.}[drdrdr]  & & &  & 0 \ar@{.}[ddd]\\
0 \ar@{.}[dd]\ar@{.}[drdr] & & & &  \\
& & & & \\
0 \ar@{.}[rr] & & 0 & 1 & 0\\
}$}
\right),
\end{equation}
where $z_\infty$ is a uniformizer at $\infty$. Its reflex ring $\BreveOReflZMuBeta=\Breve{\Oo}_\infty$. The elliptic sheaf $\underline{\mcE}$ is written in \cite{Drinfeld-commutative-subrings,Blum-Stuhler} in terms of chains of vector bundles $\mcF_i$ of rank $r$. These vector bundles are obtained from our right $\gpSch$-bundles $\mcE_i$ as $\mcF_i:=\mcE_i\times^{\gpSch} \Oo_{X_S}^r$ via the action of $\gpSch$ on $\Oo_{X_S}^r$ by left multiplication. The boundedness of $\diagphi_i$ by $\mu=(0,\ldots,0,-1)$ and of $\Pi_i$ by $\tau_{\gengpSch_\infty}^{i-1}(\beta)^{-1}=\beta^{-1}$ implies that $(\diagphi_i)^{-1}\colon {}^{\tau\!}\mcF_{i-1}\to\mcF_i$ and $\Pi_i^{-1}\colon \mcF_{i-1}\to\mcF_i$ are morphisms of vector bundles on $X_S$ whose cokernels are locally free of rank $1$ over $\Oo_S$ supported on $\infty$, and $\Gamma_x$, respectively. Moreover, for any $i$ the composition $\Pi_{i+r}^{-1}\circ\ldots\circ\Pi_{i+1}^{-1}\colon\mcF_i\to\mcF_{i+r}$ is bounded by $\beta^r= z_\infty\in \FcnFld_\infty^\times\subset \GL_r(\FcnFld_\infty)$. This is equivalent to the periodicity condition formulated in \cite[ii) on page~146]{Blum-Stuhler}. In this case $\mcM_{\beta^{-1}}$ is the group of units in the maximal order of the central division algebra over $\FcnFld_\infty$ of Hasse-invariant $1/r$. 

Drinfeld~\cite[Proposition~3]{Drinfeld-commutative-subrings} proved that the category of elliptic sheaves of rank $r$ over $S$ is equivalent to the category of Drinfeld $A$-modules over $S$ of rank $r$, where $A=\Gamma(X\smallsetminus\{\infty\},\Oo_X)$. In \cite{Drinfeld-elliptic-modules} he showed that the moduli spaces $M^r_{A,D}:=\Sht_{\gpSch,D,X\times\infty}^{\mcZ(\mu,\beta)}$ of Drinfeld $A$-modules of rank $r$ with $D$-level structure are uniformized over $\C_\infty$ by his $(r-1)$-dimensional upper halfspace $\Omega^r_{\FcnFld_\infty}$. This uniformization was worked out by Blum and Stuhler~\cite[Chapter~4]{Blum-Stuhler} in terms of elliptic sheaves. Both uniformization results are implied by our Theorem~\ref{Uniformization1Intro} as follows. As the global framing object $\FramingObject$  
we take the global $\gpSch$-shtuka over $\F_q$ given by $\mcE_0=\gpSch$ and $\diagphi_0={}^{\tau\!}\Pi_0$. Its quasi-isogeny group $I_\FramingObject$ equals $\gengpSch$. Then $\ulLocalFramingObject=L^+_{\infty,\mcM}(\FramingObject)=\bigl((L^+_\infty\mcM_{\beta^{-1}})_{\overline{\F}_q}, \beta^{-1})$, i.e.~$b=\beta^{-1}$. In view of Remark~\ref{Rem_LocShtInversePhi3}, the Rapoport-Zink space $\RZ_{\mcM,\ulLocalFramingObject}^{\leq\mu}$ is the disjoint union indexed by $\Z$ of Deligne's formal model $\widehat{\Omega}^r_{\FcnFld_\infty}$ of $\Omega^r_{\FcnFld_\infty}$; see for example~\cite[Example~6.14]{HartlViehmann3}. By the ``Lemma of the critical index'' \cite[Lemma~3.3.1]{Blum-Stuhler}, the isogeny class $\mathcal{X}_\FramingObject$ of $\FramingObject$ from Theorem~\ref{Uniformization1Intro} is the whole fiber $\Sht_{\gpSch,X\times\infty}^{\mcZ(\mu,\beta)}\times_X \Spec\overline{\F}_\infty$. This also follows from the fact that $B(M_{\beta^{-1}},-\mu)$ is a set of one element by \cite[\S\,6.11]{KottwitzII}. Thus this whole fiber is the \emph{Newton stratum} of $\ulLocalFramingObject:=L^+_{\infty,\mcM}(\FramingObject)$, i.e. for every global $\gpSch$-shtukas $\underline{\mcE}$ in this fiber the associated local $\mcM$-shtuka $L^+_{\infty,\mcM}(\underline{\mcE})$ is isogenous to $\ulLocalFramingObject$. By \cite[Proposition~11.4]{HartlAbSh} this implies that already $\underline{\mcE}$ and $\FramingObject$ are isogenous. As a consequence $\mathcal{X}_\FramingObject=\Sht_{\gpSch,X\times\infty}^{\mcZ(\mu,\beta)}\times_X \Spec\overline{\F}_\infty$. The adelic uniformization isomorphism described in \cite[Theorem~5.6]{Deligne-Husemoller} follows from our Theorem~\ref{Uniformization1Intro}
\begin{equation*}\label{Eq_ExLRS}
\Theta_{\FramingObject,\mathcal{X}}\colon \GL_r(\FcnFld)\backslash \widehat{\Omega}^r_{\FcnFld_\infty} \times \gengpSch(\mathbb{A}^\infty)/H_D \; \isoto \; M^r_{A,D} \times_X \Spf\Breve{\Oo}_\infty\,.
\end{equation*}
Note that we can use $\widehat{\Omega}^r_{\FcnFld_\infty}$ instead of $\RZ_{\mcM,\ulLocalFramingObject}^{\leq\mu}=\coprod_\Z \widehat{\Omega}^r_{\FcnFld_\infty}$, because the existence of elements in $\FcnFld^\times$ with vanishing order $1$ at $\infty$ implies $\FcnFld^\times\backslash\Z\times (\mathbb{A}^\infty)^\times=\FcnFld^\times\backslash\{1\}\times(\mathbb{A}^\infty)^\times$.
\end{example}

\begin{example}\label{ExLRS}
As a variant of the previous example, Laumon, Rapoport and Stuhler~\cite{Laumon-Rapoport-Stuhler} defined $\mathscr{D}$-elliptic sheaves, where $\mathscr{D}$ is a sheaf on $X$ of maximal orders in a central division algebra over $\FcnFld$ of dimension $d^2$. The point $\infty$ is chosen such that $\mathscr{D}\otimes_{\Oo_X} \Oo_\infty\cong M_d(\Oo_\infty)$. The group scheme $\gpSch\times_X (X\smallsetminus\{\infty\})$ equals the group $\mathscr{D}^\times$ of units in $\mathscr{D}$, and $\gpSch_\infty$ is an Iwahori subgroup of $\GL_r$ (e.g.~whose reduction is the Borel of lower triangular matrices). We take $\mu=(0,\ldots,0,-1)$, and $\beta$ is given by the matrix from \eqref{EqExDrinfeld}. Its reflex ring $\BreveOReflZMuBeta=\Breve{\Oo}_\infty$. Then by definition, a $\mathscr{D}$-elliptic sheaf over an $\F_q$-scheme $S$ is a global $\gpSch$-shtuka $\underline{\mcE}=\bigl(x,(\mcE_i,\psi_i,\Pi_i,\diagphi_i)_{i\in\Z}\bigr)$ as in Corollary~\ref{CorShtWithChains}. 
It is written in \cite[Definition~2.2]{Laumon-Rapoport-Stuhler} in terms of right modules $\mcF_i$ over $\mathscr{D}\otimes_{\F_q} \Oo_S$ which are locally free of rank one. These modules are obtained from our right $\gpSch$-bundles $\mcE_i$ as $\mcF_i:=\mcE_i\times^{\mathscr{D}^\times} \mathscr{D}$ via the action of $\gpSch=\mathscr{D}^\times$ on $\mathscr{D}$ by left multiplication. As in the previous example, $\mcM_{\beta^{-1}}$ is the group of units in the maximal order of the central division algebra over $\FcnFld_\infty$ of Hasse-invariant $1/r$.

The uniformization at $\infty$ of the moduli space $\mcE\ell\ell_{X,\mathscr{D},D}:=\Sht_{\gpSch,H_D,X\times\infty}^{\mcZ(\mu,\beta)}$ of $\mathscr{D}$-elliptic sheaves was mentioned by Laumon, Rapoport and Stuhler~\cite[(14.19)]{Laumon-Rapoport-Stuhler}, but not proven. We prove it in our Theorem~\ref{Uniformization1Intro}. As the global framing object $\FramingObject$ we take the global $\gpSch$-shtuka over $\F_q$ given by $\mcE_0=\gpSch$ and $\diagphi_0={}^{\tau\!}\Pi_0$. Then $\ulLocalFramingObject=L^+_{\infty,\mcM}(\FramingObject)=\bigl((L^+_\infty\mcM_{\beta^{-1}})_{\overline{\F}_q}, \beta^{-1})$, i.e.~$b=\beta^{-1}$. Its quasi-isogeny group $I_\FramingObject$ equals $\gengpSch$. As in the previous example, the Rapoport-Zink space $\RZ_{\mcM,\ulLocalFramingObject}^{\leq\mu}$ is the disjoint union indexed by $\Z$ of Deligne's formal model $\widehat{\Omega}^d_{\FcnFld_\infty}$ of $\Omega^d_{\FcnFld_\infty}$, and the isogeny class $\mathcal{X}_\FramingObject$ of $\FramingObject$ from Theorem~\ref{Uniformization1Intro} is the whole fiber $\Sht_{\gpSch,X\times\infty}^{\mcZ(\mu,\beta)}\times_X \Spec\overline{\F}_\infty$. We obtain the uniformization isomorphism
\begin{equation*}\label{Eq_ExLRS}
\Theta_{\FramingObject,\mathcal{X}}\colon \gengpSch(\FcnFld)\backslash \widehat{\Omega}^d_{\FcnFld_\infty} \times \gengpSch(\mathbb{A}^\infty)/H_D \; \isoto \; \mcE\ell\ell_{X,\mathscr{D},D} \times_X \Spf\Breve{\Oo}_\infty\,.
\end{equation*}
As in the previous example note that we can use $\widehat{\Omega}^r_{\FcnFld_\infty}$ instead of $\RZ_{\mcM,\ulLocalFramingObject}^{\leq\mu}=\coprod_\Z \widehat{\Omega}^r_{\FcnFld_\infty}$, because $\FcnFld^\times\backslash\Z\times (\mathbb{A}^\infty)^\times=\FcnFld^\times\backslash\{1\}\times(\mathbb{A}^\infty)^\times$.
\end{example}

\section{Uniformization}\label{ChaptUnif}

\subsection{Source and target of the uniformization map}\label{subsec:SourceTarget}

\begin{numberedparagraph}\label{SetupUniformization}
Throughout Chapter~\ref{ChaptUnif} we fix a global framing object $\FramingObject\in \Sht_{\gpSch,\varnothing,X\times\infty}^{\mcZ(\mu,\beta)}(\overline{\F}_\infty)$ and let $\ulLocalFramingObject:=(\LocalFramingObject,\widehat{\varphi}_{\LocalFramingObject}):=L^+_{\infty,\mcM_{\beta^{-1}}}(\FramingObject)$ denote the associated local $\mcM$-shtuka over $\BaseFldInSectUnif$, where $L^+_{\infty,\mcM_{\beta^{-1}}}$ is the $\beta^{-1}$-twisted global-local functor from Definition~\ref{DefGlobLocM}. We fix a trivialization $\ulLocalFramingObject\cong(L^+_\infty\mcM_{\overline{{\F}}_\infty},b)$ for $[b]\in B(M,\mu)$ by Proposition \ref{PropBoundWithChains}. Recall from Definition~\ref{DefIsogGlobalSht} and Remark~\ref{RemIsogGlobalSht} the group $I_{\FramingObject} (\FcnFld):=\mathrm{QIsog}_{\overline{\F}_\infty}(\FramingObject)$ of quasi-isogenies of the framing object.
\end{numberedparagraph}

\begin{prop}\label{Prop7.1}
Let $S$ be a connected non-empty scheme in $\Nilp_{\BreveOReflZMuBeta}$. 
Then the natural group homomorphism $\QIsog_\BaseFldInSectUnif(\FramingObject)\to\QIsog_{S}(\FramingObject_S)$, $g\mapsto g \times_{\BaseFldInSectUnif} S=:g_S$ is an isomorphism. 
\end{prop}

\begin{proof}
This follows as in \cite[Proposition~7.2]{AH_Unif}.
\end{proof}

We next describe the source of the morphism $\Theta_{\FramingObject}$ from \eqref{EqUnifIntro}. 
Let $\RZ_{\mcM,\ulLocalFramingObject}^{\leq \mu}$ be the Rapoport-Zink space of $\ulLocalFramingObject$ with underlying topological space $X_{\mcM}^{\leq \mu}(b)$, see Theorem~\ref{ThmRRZSp}. 

\begin{numberedparagraph}\label{Point7.6}
Recall that the group $J_{\ulLocalFramingObject}(\FcnFld_\infty)=\QIsog_\BaseFldInSectUnif(\ulLocalFramingObject)$ of quasi-isogenies of $\ulLocalFramingObject$ over $\BaseFldInSectUnif$ acts naturally on $\RZ_{\mcM,\ulLocalFramingObject}^{\leq \mu}$ and on $X_{\mcM}^{\leq \mu}(b)$; see Remark~\ref{RemMActsOnRZ}. In particular, we see that the group $I_{\FramingObject}(\FcnFld)$ acts on $\RZ_{\mcM,\ulLocalFramingObject}^{\leq \mu}$ and $X_{\mcM}^{\leq \mu}(b)$ via the natural group homomorphism 
\begin{equation}\label{EqI(\FcnFld)}
L_{\infty,\mcM_{\beta^{-1}}}\colon I_{\FramingObject}(\FcnFld) \;\longto\; J_{\ulLocalFramingObject}(\FcnFld_\infty) \,,
\end{equation}
which sends a quasi-isogeny $\eta\in I_{\FramingObject}(\FcnFld)$ of $\FramingObject$ to the quasi-isogeny $L_{\mcM_{\beta^{-1}}}(\eta)$ of $\ulLocalFramingObject=L_{\mcM_{\beta^{-1}}}(\FramingObject)$.

The group $I_{\FramingObject}(\FcnFld)$ also acts naturally on $\check{{\mathcal{V}}}_{\FramingObject}$ and $\Isom^{\otimes}(\omega,\check{{\mathcal{V}}}_{\FramingObject})$ by sending $\gamma\in\Isom^{\otimes}(\omega,\check{{\mathcal{V}}}_{\FramingObject})$ to $\check{\mathcal{V}}_\eta\circ\gamma$ for $\eta\in I_{\FramingObject}(\FcnFld)$; see Definition~\ref{DefTateModule} and Lemma~\ref{LSisnonempty}. Upon choosing an element $\gamma_0\in\Isom^{\otimes}(\omega_{\mathbb{O}^\infty} ,\check{{\mathcal{T}}}_{\FramingObject})$, this defines an injective morphism 
\begin{equation}\label{EqEpsilon}
\zeta\colon  I_{\FramingObject}(\FcnFld) \;\longto\; \Aut^\otimes(\omega )\;\cong\;\gengpSch(\mathbb{A}^\infty),\quad\eta\mapsto\gamma_0^{-1}\circ\check{{\mathcal{V}}}_\eta\circ\gamma_0\,.
\end{equation}
The map $\zeta$ depends on $\gamma_0$ only up to $\gengpSch(\A^\infty)$-conjugation. 
\begin{lem}\label{injective-L-and-zeta}
The group homomorphisms $L_{\infty,\mcM_{\beta^{-1}}}$ and $\zeta$ from \eqref{EqI(\FcnFld)} and \eqref{EqEpsilon} are injective.
\end{lem}
\begin{proof}
Consider a faithful representation $\rho\colon\gpSch\into\GL({\mathcal{V}})$ as in \cite[Proposition~2.2(a)]{AH_Unif} for a vector bundle $\mathcal{V}$ on $X$. Then $\rho_*\eta$ induces a quasi-isogeny of the vector bundle $M:=\rho_*\FramingObjectComp$ on $X_S$ associated with $\rho_*\FramingObject$. If $\eta$ lies in the kernel of $\zeta$ then its restriction to $X_\BaseFldInSectUnif\setminus\{\infty\}$ is the identity on $M$ because of \cite[Proposition~3.4]{AH_Local}. Therefore $\eta$ must be the identity. This proves that $\zeta$ is injective. On the other hand, $L^+_{\infty,\mcM_{\beta^{-1}}}(M)$ is the completion of $M$ at the graph $\Gamma_{x}$ of $x\colon\Spec\BaseFldInSectUnif\to X$. Since this completion functor is faithful, also $L_{\infty,\mcM_{\beta^{-1}}}\colon I_{\FramingObject}(\FcnFld) \to J_{\ulLocalFramingObject}(\FcnFld_\infty)$ is injective.
\end{proof}
By Lemma \ref{injective-L-and-zeta}, we have 
the following injective morphism
\begin{equation}
(L_{\infty,\mcM_{\beta^{-1}}},\zeta)\colon I_{\FramingObject}(\FcnFld)\hookrightarrow J_{\ulLocalFramingObject}(\FcnFld_\infty)\times \gengpSch(\mathbb{A}^\infty)
\end{equation}
and we identify $I_{\FramingObject}(\FcnFld)$ with its image. 
\begin{lem}\label{LemmaIisDiscrete}
$I_{\FramingObject}(\FcnFld)$ is a discrete subgroup of $J_{\ulLocalFramingObject}(\FcnFld_\infty)\times \gengpSch(\mathbb{A}^\infty)$.
\end{lem}

\begin{proof}   
To show this, we take an open subgroup $U\subset\Aut_\BaseFldInSectUnif(\ulLocalFramingObject)$ and consider the open subgroup $U\times\gpSch(\mathbb{O}^\infty)\subset J_{\ulLocalFramingObject}(\FcnFld_\infty)\times \gengpSch(\mathbb{A}^\infty)$. Since $\gpSch(\mathbb{O}^\infty)=\gamma_0\Aut^{\otimes}(\check{{\mathcal{T}}}_{\FramingObject})\gamma_0^{-1}$, the elements of $I_{\FramingObject}(\FcnFld)\cap \bigl(U\times\gpSch(\mathbb{O}^\infty)\bigr)$ give automorphisms of the global $\gpSch$-shtuka $\FramingObject$. Then the finiteness of $I_{\FramingObject}(\FcnFld)\cap \bigl(U\times\gpSch(\mathbb{O}^\infty)\bigr)$ follows from Corollary~\ref{CorAutFinite}. 
\end{proof}

To state the properties of the source of $\Theta_{\FramingObject}$, we say that a formal algebraic Deligne-Mumford stack $\mathcal{Y}$ over $\Spf\BreveOReflZMuBeta$ (see \cite[Definition~A.5]{HartlAbSh}) is \emph{$\mathcal{I}$-adic} for a sheaf of ideals $\mathcal{I}\subset\Oo_\mathcal{Y}$, if for some (any) presentation $Y\to\mathcal{Y}$, the formal scheme $Y$ is $\mathcal{I}\Oo_Y$-adic, i.e.~$\mathcal{I}^r\Oo_Y$ is an ideal of definition of $Y$ for all $r\in \N_{>0}$. We then call $\mathcal{I}$ an \emph{ideal of definition of $\mathcal{Y}$}. We say that $\mathcal{Y}$ is \emph{locally formally of finite type} over $\Spf\BreveOReflZMuBeta$ if $\mathcal{Y}$ is locally noetherian, adic, and if the closed substack defined by the largest ideal of definition (see \cite[A.7]{HartlAbSh}) is an algebraic stack locally of finite type over $\Spec\BaseFldInSectUnif$.
\end{numberedparagraph}
\noindent Let $H\subset \gengpSch(\mathbb{A}^\infty)$ be a compact open subgroup. Consider the countable double coset 
\begin{equation}\label{double-coset-IF}
I_{\FramingObject}(\FcnFld) \backslash\Isom^{\otimes}(\omega ,\check{{\mathcal{V}}}_{\FramingObject})/H\cong I_{\FramingObject}(\FcnFld) \backslash \gengpSch(\mathbb{A}^\infty)/H.
\end{equation}
For $\bar\gamma:=\gamma H\in \Isom^{\otimes}(\omega,\check{\mathcal{V}}_{\FramingObject})/H$ the coset $h_\gamma H\subset \gengpSch(\mathbb{A}^\infty)/H$ of $h_\gamma:=\gamma_0^{-1}\gamma \in \gengpSch(\mathbb{A}^\infty)$ is well defined. We set  
\begin{equation}
\Gamma_{\bar\gamma}\;:=\; I_{\FramingObject}(\FcnFld)\cap \bigl(J_{\ulLocalFramingObject}(\FcnFld_\infty)\times h_\gamma H h_\gamma^{-1}\bigr)\;\subset\; J_{\ulLocalFramingObject}(\FcnFld_\infty).
\end{equation}
This is a discrete subgroup for the $\infty$-adic topology by Lemma~\ref{LemmaIisDiscrete}. Moreover, $\Gamma_{\overline{\gamma}}$ is separated in the profinite topology, i.e.~for every $1\ne \eta \in \Gamma_{\bar\gamma}$, there is a normal subgroup of finite index in $\Gamma_{\bar\gamma}$ that does not contain $g$. Indeed, there is a normal subgroup $\widetilde H\subset H$ of finite index such that $h_\gamma\widetilde H h_\gamma^{-1}$ does not contain the element $\zeta(\eta)\ne1$. 

\begin{prop}\label{PropQuotientByI}
(a) The quotient of $\RZ_{\mcM,\ulLocalFramingObject}^{\leq \mu}\times \Isom^{\otimes}(\omega ,\check{{\mathcal{V}}}_{\FramingObject})/H$ by the abstract group $I_{\FramingObject}(\FcnFld)$ exists as a locally noetherian, adic, formal algebraic Deligne-Mumford stack locally formally of finite type over $\Spf\BreveOReflZMuBeta$ and is given by the following disjoint union 
\begin{equation}\label{EqSourceOfTheta}
I_{\FramingObject}(\FcnFld) \big{\backslash}\bigl(\RZ_{\mcM,\ulLocalFramingObject}^{\leq \mu}\times \Isom^{\otimes}(\omega ,\check{{\mathcal{V}}}_{\FramingObject})/H\bigr) \enspace\cong\enspace \coprod_{\bar\gamma} \Gamma_{\bar\gamma}\big{\backslash}\RZ_{\mcM,\ulLocalFramingObject}^{\leq \mu}\,.
\end{equation}
Here $\bar\gamma:=\gamma H\in\Isom^{\otimes}(\omega ,\check{{\mathcal{V}}}_{\FramingObject})/H$ runs through a set of representatives for the double coset \eqref{double-coset-IF}.

(b) The quotient morphisms 
\begin{align}\label{EqPropQuotientByIEtale}
\RZ_{\mcM,\ulLocalFramingObject}^{\leq \mu} & \enspace \onto \enspace \Gamma_{\bar\gamma}\big{\backslash}\RZ_{\mcM,\ulLocalFramingObject}^{\leq \mu} \qquad \text{and} \nonumber \\
\bigl(\RZ_{\mcM,\ulLocalFramingObject}^{\leq \mu}\times \Isom^{\otimes}(\omega ,\check{{\mathcal{V}}}_{\FramingObject})/H\bigr) & \enspace \onto \enspace I_{\FramingObject}(\FcnFld) \big{\backslash}\bigl(\RZ_{\mcM,\ulLocalFramingObject}^{\leq \mu}\times \Isom^{\otimes}(\omega ,\check{{\mathcal{V}}}_{\FramingObject})/H\bigr) 
\end{align}
are adic and \'etale.

(c) In particular, the closed substack of \eqref{EqSourceOfTheta} defined by the largest ideal of definition is the Deligne-Mumford stack locally of finite type over $\Spec\BaseFldInSectUnif$ given by
\begin{equation}\label{EqReducedSourceOfTheta}
I_{\FramingObject}(\FcnFld) \big{\backslash}X_{\mcM}^{\leq \mu}(b)\times \Isom^{\otimes}(\omega ,\check{{\mathcal{V}}}_{\FramingObject})/H \enspace\cong\enspace \coprod_{\bar\gamma} \Gamma_{\bar\gamma}\big{\backslash}X_{\mcM}^{\leq \mu}(b)\,.
\end{equation}
\end{prop}

\begin{remark}
Note that all unipotent subgroups of $J_{\ulLocalFramingObject}(\FcnFld_\infty)$ are torsion, so they might not act fixed-point-freely on $\RZ_{\mcM,\ulLocalFramingObject}^{\leq\mu}$ and $X_{\mcM}^{\leq\mu}(b)$, thus the quotients \eqref{EqSourceOfTheta} and \eqref{EqReducedSourceOfTheta} may not be (formal) schemes. 
\end{remark}

\begin{proof}[Proof of Proposition~\ref{PropQuotientByI}]
The quotients $\Gamma_{\bar\gamma}\big{\backslash}\RZ_{\mcM,\ulLocalFramingObject}^{\leq \mu}$ and $I_{\FramingObject}(\FcnFld) \big{\backslash}\bigl(\RZ_{\mcM,\ulLocalFramingObject}^{\leq \mu}\times \Isom^{\otimes}(\omega ,\check{{\mathcal{V}}}_{\FramingObject})/H\bigr)$ are formal algebraic Deligne-Mumford stacks by \cite[Proposition~4.27]{AH_Local}. That they are Deligne-Mumford and the last assertion about \eqref{EqReducedSourceOfTheta} follow from the proof of \cite[Proposition~4.27]{AH_Local} where it is shown that \eqref{EqReducedSourceOfTheta} (respectively \eqref{EqSourceOfTheta}) are locally the stack quotient of a (formal) scheme by a finite group. 
\end{proof}

\subsection{Various actions on the source and target}

We keep the situation of Section~\ref{subsec:SourceTarget}. There is an action of $\gengpSch(\mathbb{A}^\infty)$ by Hecke correspondences, which is explicitly given as follows. 
\begin{defn}\label{Defn-Hecke-corr}
Let $H,H'\subset \gengpSch(\mathbb{A}^\infty)$ be compact open subgroups and let $h\in \gengpSch(\mathbb{A}^\infty)$. We define the Hecke correspondence $\pi(h)_{H'\!,H}$ by the two diagrams: the local case
\begin{equation}\label{EqHeckeSource}
\xymatrix @C=-4pc {
& \RZ_{\mcM,\ulLocalFramingObject}^{\leq \mu}\times \Isom^{\otimes}(\omega ,\check{{\mathcal{V}}}_{\FramingObject})/(hHh^{-1}\cap H') \ar[ld] \ar[rd]& \\
\RZ_{\mcM,\ulLocalFramingObject}^{\leq \mu}\times \Isom^{\otimes}(\omega ,\check{{\mathcal{V}}}_{\FramingObject})/H& & \RZ_{\mcM,\ulLocalFramingObject}^{\leq \mu}\times \Isom^{\otimes}(\omega ,\check{{\mathcal{V}}}_{\FramingObject})/H'\ar@{-->}[ll]_{\pi(h)_{H',H}}\\
}
\end{equation}
On the $S$-valued points, it is given by
\begin{equation*}
\begin{tikzcd}
& (\underline{\mcL},\hat{\delta}) \times \gamma(hHh^{-1}\cap H')\arrow[mapsto]{ld}{}
\arrow[]{rd}{}
\\
(\underline{\mcL},\hat{\delta}) \times \gamma hH & & (\underline{\mcL},\hat{\delta}) \times \gamma H' \arrow[dotted,mapsto]{ll}{}
\end{tikzcd}
\end{equation*}
We define the Hecke correspondence in the global case by 
\begin{equation}\label{EqHeckeTarget}
\begin{tikzcd}
& \Sht_{\gpSch,(hHh^{-1}\cap H'),\widehat{\infty}\times\infty}^{\mcZ(\mu,\beta)} \arrow[]{ld}{} \arrow[]{rd}{} & \\
\Sht_{\gpSch,H,\widehat{\infty}\times\infty}^{\mcZ(\mu,\beta)} & & \Sht_{\gpSch,H'\!,\widehat{\infty}\times\infty}^{\mcZ(\mu,\beta)} \arrow[dotted]{ll}{\pi(h)_{H',H}}
\end{tikzcd}
\end{equation}
On the level of $S$-points, it is given by
\begin{equation*}
\begin{tikzcd}
& (\underline{\mcE},\gamma(hHh^{-1}\cap H')) \arrow[mapsto]{ld}{} 
\arrow[mapsto]{rd}{}
\\
(\underline{\mcE},\gamma hH) & & (\underline{\mcE},\gamma H') \arrow[dotted,mapsto]{ll}{}
\end{tikzcd}
\end{equation*}
\end{defn}

\begin{example}\label{ExLSGGsht1}
A special case for $H'\subset H$ and $h=1$ are the forgetful morphisms 
\[
\pi(1)_{H'\!,H}\colon\RZ_{\mcM,\ulLocalFramingObject}^{\leq \mu}\times \Isom^{\otimes}(\omega ,\check{{\mathcal{V}}}_{\FramingObject})/H'\to\RZ_{\mcM,\ulLocalFramingObject}^{\leq \mu}\times \Isom^{\otimes}(\omega ,\check{{\mathcal{V}}}_{\FramingObject})/H
\]
and $\pi(1)_{H'\!,H}\colon\Sht_{\gpSch,H'\!,\widehat{\infty}\times\infty}^{\mcZ(\mu,\beta)}\to\Sht_{\gpSch,H,\widehat{\infty}\times\infty}^{\mcZ(\mu,\beta)}$, which are finite \'etale and surjective; see Proposition~\ref{PropLSGGsht1}\ref{PropLSGGsht1_B}.
\end{example}

\begin{numberedparagraph}\label{center-action-paragraph}
There is a another group acting on the source and the target of the morphism $\Theta_{\FramingObject}$ from \eqref{EqUnifIntro}, namely the group $Z(\FcnFld_\infty)$ where $Z\subset \gpSch\times_X\Spec \FcnFld =\gengpSch$ is the center. Since the center of $M$ and that of $\gengpSch_\infty$ coincide, $Z(\FcnFld_\infty)$ lies in the center of $L_\infty\mcM(\overline{\F}_\infty)$. 
Recall the maps $\tau_{\gengpSch_\infty}$ and $\tau_M$ from \eqref{Eq_tau_M}. Writing $\ulLocalFramingObject\cong\bigl((L^+_\infty\mcM)_\BaseFldInSectUnif,b\bigr)$ with $b\in L_\infty\mcM(\BaseFldInSectUnif)=L_\infty\gpSch(\BaseFldInSectUnif)$, there is an inclusion
\begin{align}
\label{EqActionCenter1}
\begin{split}
Z(\FcnFld_\infty) \enspace\longinto\enspace & J_{\ulLocalFramingObject}(\FcnFld_\infty)\;=\;\bigl\{\,j\in L_\infty\mcM(\BaseFldInSectUnif)\colon \tau_M(j)\, b=b\,j\,\bigr\}\;=\;\QIsog_\BaseFldInSectUnif(\ulLocalFramingObject)\,, \\
c \quad\enspace\longmapsto\enspace & \enspace\quad c_\ulLocalFramingObject \enspace\; := \quad c=\tau_{\gengpSch_\infty}(c) = \beta^{-1}\tau_{\gengpSch_\infty}(c) \beta = \tau_M(c) = b^{-1}\tau_M(c)b 
\end{split}
\end{align}
through which $c\in Z(\FcnFld_\infty)$ acts on $(\underline{\mcL},\hat{\delta})\in\RZ_{\mcM,\ulLocalFramingObject}^{\leq \mu}$ via $c\colon (\underline{\mcL},\hat{\delta})\mapsto (\underline{\mcL},c_{\ulLocalFramingObject}\circ\hat{\delta})$, where we denote the image of $c$ inside $\QIsog_{\overline{\F}_\infty}(\ulLocalFramingObject)$ by $c_{\ulLocalFramingObject}$ whenever confusion may arise.

This action can also be described in a different way as follows. 

For any local $\mcM$-shtuka $\underline{\mcL}$ over a scheme $S\in\Nilp_{\breve\Oo_\infty}$, we claim that $c$ induces an element $c_{\underline{\mcL}}\in\QIsog_S(\underline{\mcL})$ of the quasi-isogeny group of $\underline{\mcL}$ such that $\hat{\delta}\circ c_{\underline{\mcL}}=c_{\ulLocalFramingObject}\circ\hat{\delta}$. Over an \'etale covering $S'\to S$, we can choose a trivialization $\alpha\colon\underline{\mcL}_{S'}\isoto\bigl((L^+_\infty\mcM)_{S'},b'\bigr)$ for $b'\in L_\infty\mcM(S')$. Then $c=\tau_M(c)$ implies $b'c= \tau_M(c)b'$, i.e.~$\alpha^{-1}\circ c\circ\alpha\in\QIsog_{S'}(\underline{\mcL})$. 

Next we show that this quasi-isogeny descends to $S$. Let $\pr_1,\pr_2\colon S'':=S'\times_S S' \to S'$ be the projections onto the first and second factor. Set $h:=\pr_2^*\alpha\circ \pr_1^*\alpha^{-1}\in L^+_\infty\mcM(S'')$. Then $h \circ \pr_1^*c\circ h^{-1}=\pr_1^*c = \pr_2^*c$ implies
\[
\pr_1^*(\alpha^{-1}\circ c\circ\alpha)\;=\;\pr_2^*\alpha^{-1}\circ h \circ \pr_1^*c\circ h^{-1}\circ \pr_2^*\alpha\;=\;\pr_2^*(\alpha^{-1}\circ c\circ\alpha)\,.
\]
Therefore, $\alpha^{-1}\circ c\circ\alpha$ descends to a quasi-isogeny of $\underline{\mcL}$ over $S$ which we denote by $c_{\underline{\mcL}}\in\QIsog_S(\underline{\mcL})$. If moreover we are given a quasi-isogeny $\hat{\delta}\colon\underline{\mcL}\to \ulLocalFramingObject_{S}$, i.e.~if $(\underline{\mcL},\hat{\delta})\in\RZ_{\mcM,\ulLocalFramingObject}^{\leq \mu}(S)$, then $\hat{\delta}$ is automatically compatible with the quasi-isogenies $c_{\underline{\mcL}}\in\QIsog_S(\underline{\mcL})$ and $c_{\ulLocalFramingObject}\in\QIsog_\BaseFldInSectUnif(\ulLocalFramingObject)$, i.e.~$\hat{\delta}\circ c_{\underline{\mcL}}=c_{\ulLocalFramingObject}\circ\hat{\delta}$.~To see this:~using the trivialization $\alpha$ over $S'$ from above, $\hat{\delta}$ corresponds to $g:=\hat{\delta}\circ\alpha^{-1}\in L_\infty\gpSch(S')$. Thus
\begin{equation}
\hat{\delta}\circ c_{\underline{\mcL}}\;:=\;\hat{\delta}\circ(\alpha^{-1}\circ c\circ\alpha)\;=\;g\circ c\circ\alpha\;=\;c\circ g\circ\alpha\;=\;c_{\ulLocalFramingObject}\circ\hat{\delta}\,.
\end{equation}
This shows that the action of $c\in Z(\FcnFld_\infty)$ on $\RZ_{\mcM,\ulLocalFramingObject}^{\leq \mu}$ is given as
\begin{equation}\label{EqActionCenter3}
c\colon\RZ_{\mcM,\ulLocalFramingObject}^{\leq \mu}\;\longto\;\RZ_{\mcM,\ulLocalFramingObject}^{\leq \mu}\,,\quad (\underline{\mcL},\hat{\delta})\;\longmapsto\; (\underline{\mcL},\hat{\delta}\circ c_{\underline{\mcL}})\;=\;(\underline{\mcL},c_{\ulLocalFramingObject}\circ\hat{\delta})\,. 
\end{equation}
As such, it does not matter which quasi-isogeny group we realize $Z(\FcnFld_\infty)$ in. 
The action of $Z(\FcnFld_\infty)$ on $\RZ_{\mcM,\ulLocalFramingObject}^{\leq \mu}$ commutes with the action of $\eta\in I_{\FramingObject}(\FcnFld)$, as $\eta$ and $c$ act on $\RZ_{\mcM,\ulLocalFramingObject}^{\leq \mu}\times \Isom^{\otimes}(\omega ,\check{{\mathcal{V}}}_{\FramingObject})/H$ by 
\begin{equation}
(\underline{\mcL},\hat{\delta}) \times \gamma H\;\longmapsto\;(\underline{\mcL},L_{\infty,\mcM_{\beta^{-1}}}(\eta)\circ\hat{\delta}\circ c) \times \check{{\mathcal{V}}}_\eta\gamma H\,.
\end{equation}

On the other hand, $c\in Z(\FcnFld_\infty)$ also acts on the target $\Sht_{\gpSch,H,\widehat{\infty}\times\infty}^{\mcZ(\mu,\beta)}$ of the map $\Theta_{\FramingObject}$ as follows. Let $(\underline{\mcE},\gamma H)$ be in $\Sht_{\gpSch,H,\widehat{\infty}\times\infty}^{\mcZ(\mu,\beta)}(S)$ for some $S\in\Nilp_\BreveOReflZMuBeta$. Consider the associated local $\mcM$-shtuka $\underline{\mcL}:=L^+_{\infty,\mcM_{\beta^{-1}}}(\underline{\mcE})$. We have seen above that $c$ induces an element $c_{\underline{\mcL}}\in\QIsog_S(\underline{\mcL})$ of the quasi-isogeny group of $\underline{\mcL}$. By Proposition~\ref{PropQIsogLocalGlobal}\ref{PropQIsogLocalGlobal_B}, there is a uniquely determined global $\gpSch$-shtuka $c^*\,\underline{\mcE}$ and a quasi-isogeny $c_{\underline{\mcE}}\colon c^*\,\underline{\mcE}\to\underline{\mcE}$, which is an isomorphism outside $\infty$ and satisfies $L_{\infty,\mcM_{\beta^{-1}}}(c_{\underline{\mcE}})=c_{\underline{\mcL}}$. We can now define the action of $c\in Z(\FcnFld_\infty)$ on $(\underline{\mcE},\gamma H)\in\Sht_{\gpSch,H,\widehat{\infty}\times\infty}^{\mcZ(\mu,\beta)}(S)$ as
\begin{equation}\label{EqActionCenter2}
c\colon\;(\underline{\mcE},\gamma H)\;\longmapsto\;(c^*\,\underline{\mcE}\,,\,\check{\mathcal{V}}_{c_{\underline{\mcE}}}^{-1}\gamma H)\,.
\end{equation}
\end{numberedparagraph}

\begin{numberedparagraph}\label{Weil-descent-paragraph}
The source and target of $\Theta_{\FramingObject}$ carry a \emph{Weil-descent datum} for the ring extension $\OReflZMuBeta\subset\BreveOReflZMuBeta$, compare~\cite[Definition~3.45]{RZ}. We explain what this means. Recall that $\KappaReflZMuBeta$ is the residue field of $\OReflZMuBeta$. Consider the $\OReflZMuBeta$-automorphism $\lambda$ of $\BreveOReflZMuBeta$ inducing
\begin{equation}\label{EqWeilDescent_lambda}
\lambda|_\BaseFldInSectUnif\;:=\;\Frob_{\#\KappaReflZMuBeta,\BaseFldInSectUnif}\colon x\;\longmapsto\;x^{\#\KappaReflZMuBeta}\quad\text{for }x\in\BaseFldInSectUnif
\end{equation}
on the residue field $\overline{\F}_\infty$ of $\BreveOReflZMuBeta$.
For a scheme $(S,\theta)\in\Nilp_{\BreveOReflZMuBeta}$, where $\theta\colon S\to\Spf\BreveOReflZMuBeta$ denotes the structure morphism of the scheme $S$, we denote by $S_{[\lambda]}\in\Nilp_{\BreveOReflZMuBeta}$ the pair ($S,\lambda\circ\theta)$. For a stack $\mathcal{H}$ over $\Spf\BreveOReflZMuBeta$, we define the stack ${}^\lambda\mathcal{H}$ by 
\[
{}^\lambda\mathcal{H}(S)\;:=\;\mathcal{H}(S_{[\lambda]})\,.
\]
\begin{defn}
A \emph{Weil-descent datum} on $\mathcal{H}$ is an isomorphism of stacks $\mathcal{H}\isoto{}^\lambda\mathcal{H}$, i.e.~an equivalence $\mathcal{H}(S)\isoto\mathcal{H}(S_{[\lambda]})$ for every $S\in\Nilp_{\BreveOReflZMuBeta}$ compatible with morphisms in $\Nilp_{\BreveOReflZMuBeta}$.
\end{defn}
Let $S\in\Nilp_{\BreveOReflZMuBeta}$. Under the inclusion $\Nilp_{\BreveOReflZMuBeta}\into\Nilp_{\OReflZMuBeta}$, we have $S=S_{[\lambda]}$ in $\Nilp_{\OReflZMuBeta}$.~Therefore, 
on $\Sht_{\gpSch,H,\widehat{\infty}\times\infty}^{\mcZ(\mu,\beta)}\widehattimes_{\OReflZMuBeta} \Spf\BreveOReflZMuBeta$, the canonical Weil-descent datum is given by the identity
\begin{equation}\label{EqDescentDatumOnNablaH}
\id\colon\Sht_{\gpSch,H,\widehat{\infty}\times\infty}^{\mcZ(\mu,\beta)}(S)\;\isoto\;\Sht_{\gpSch,H,\widehat{\infty}\times\infty}^{\mcZ(\mu,\beta)}(S_{[\lambda]})\,,\quad(\underline{\mcE},\gamma H)\;\longmapsto\;(\underline{\mcE},\gamma H).
\end{equation}
On $\RZ_{\mcM,\ulLocalFramingObject}^{\leq \mu}$, we consider the Weil descent datum given by
\begin{align}\label{EqDescentDatumSource}
\begin{split}
\RZ_{\mcM,\ulLocalFramingObject}^{\leq \mu}(S)\;\isoto\; & \RZ_{\mcM,\ulLocalFramingObject}^{\leq \mu}(S_{[\lambda]}),\\
(\underline{\mcL},\hat{\delta}\colon\underline{\mcL}\to\ulLocalFramingObject_{S})\;\longmapsto\; & (\underline{\mcL},\theta^*(\widehat{\varphi}_{\LocalFramingObject}^{[\KappaReflZMuBeta\colon\F_q]})\circ\hat{\delta}\colon\underline{\mcL}\to\ulLocalFramingObject_{S_{[\lambda]}})\,,
\end{split}
\end{align}
where $\widehat{\varphi}_{\LocalFramingObject}$ is the Frobenius of the local shtuka $\ulLocalFramingObject=(\ulLocalFramingObject,\widehat{\varphi}_{\LocalFramingObject})$. 
Here we observe that $\ulLocalFramingObject_{S}:=\theta^*\ulLocalFramingObject$ and $\ulLocalFramingObject_{S_{[\lambda]}}:=(\lambda\circ\theta)^*\ulLocalFramingObject=\theta^*\lambda^*\ulLocalFramingObject=\theta^*({}^{\tau^{[\KappaReflZMuBeta\colon\F_q]}}\ulLocalFramingObject)$, and that $\widehat{\varphi}_{\LocalFramingObject}^{[\KappaReflZMuBeta\colon\F_q]}\colon\ulLocalFramingObject\to{}^{\tau^{[\KappaReflZMuBeta\colon\F_q]}}\ulLocalFramingObject$ is a quasi-isogeny. 

\begin{remark}
On on $\RZ_{\mcM,\ulLocalFramingObject}^{\leq \mu}$ there is even a Weil descent datum for the ring extension $\Oo_\mu\subset \Breve{\Oo}_\mu$. It is defined analogously by replacing $\KappaReflZMuBeta$ in \eqref{EqWeilDescent_lambda} and \eqref{EqDescentDatumSource} by $\Oo_\mu$. We do not discuss the question whether this Weil descent datum on $\RZ_{\mcM,\ulLocalFramingObject}^{\leq \mu}$ is effective. In the analogous situation for $p$-divisible groups, this is true and proven by Rapoport and Zink in \cite[Theorem~3.49]{RZ}. Their argument uses a morphism $\mathbb{G}_m\to\gpSch$, which might not exist in our setup. 
\end{remark}
Moreover, on $\RZ_{\mcM,\ulLocalFramingObject}^{\leq \mu}\times \Isom^{\otimes}(\omega ,\check{{\mathcal{V}}}_{\FramingObject})/H$ we consider the product of the Weil Descent datum \eqref{EqDescentDatumSource} with the identity on $\Isom^{\otimes}(\omega ,\check{{\mathcal{V}}}_{\FramingObject})/H$. Let $\eta\in I_{\FramingObject}(\FcnFld)$. Since 
\begin{align*}
\theta^*(\widehat{\varphi}_{\LocalFramingObject}^{[\KappaReflZMuBeta\colon\F_q]}\circ L_{\infty,\mcM_{\beta^{-1}}}(\eta)) & =\theta^*\Bigl({}^{\tau^{[\KappaReflZMuBeta\colon\F_q]}}(L_{\infty,\mcM_{\beta^{-1}}}(\eta))\circ\widehat{\varphi}_{\LocalFramingObject}^{[\KappaReflZMuBeta\colon\F_q]}\Bigr) \\
& =(\lambda\circ\theta)^*(L_{\infty,\mcM_{\beta^{-1}}}(\eta))\circ\theta^*(\widehat{\varphi}_{\LocalFramingObject}^{[\KappaReflZMuBeta\colon\F_q]}),
\end{align*}
this product Weil descent datum commutes with the action of $I_{\FramingObject}(\FcnFld)$ via the following diagram

\begin{equation}
\xymatrix @C+2pc @R=1pc {
(\underline{\mcL},\hat{\delta}) \ar@{|->}[r]^-{\textstyle\eta} \ar@{|->}[dd]^(0.45){\textstyle\text{Weil descent}} & (\underline{\mcL},\theta^*(L_{\infty,\mcM_{\beta^{-1}}}(\eta))\circ\hat{\delta}) \ar@{|->}[dd]^(0.45){\textstyle\text{Weil descent}} \\ \\
(\underline{\mcL},\theta^*(\widehat{\varphi}_{\LocalFramingObject}^{[\KappaReflZMuBeta\colon\F_q]})\circ\hat{\delta}) \ar@{|->}[rd]^-{\textstyle\eta} & (\underline{\mcL},\theta^*(\widehat{\varphi}_{\LocalFramingObject}^{[\KappaReflZMuBeta\colon\F_q]})\circ\theta^*(L_{\infty,\mcM_{\beta^{-1}}}(\eta))\circ\hat{\delta}) \ar@{=}[d]\\
& \;(\underline{\mcL},(\lambda\circ\theta)^*(L_{\infty,\mcM_{\beta^{-1}}}(\eta))\circ\theta^*(\widehat{\varphi}_{\LocalFramingObject}^{[\KappaReflZMuBeta\colon\F_q]})\circ\hat{\delta})\,.
}
\end{equation}
This defines a Weil descent datum on $\mathcal{Y}:=I_{\FramingObject}(\FcnFld) \big{\backslash}\RZ_{\mcM,\ulLocalFramingObject}^{\leq \mu}\times \Isom^{\otimes}(\omega ,\check{{\mathcal{V}}}_{\FramingObject})/H$ by
\begin{align}\label{EqDescentDatumSourceModI}
\mathcal{Y}(S) \;\isoto\; & {}^\lambda\mathcal{Y}(S)\;=\;\mathcal{Y}(S_{[\lambda]}) \\
(\underline{\mcL},\hat{\delta}, \gamma H)\;\longmapsto\; & (\underline{\mcL},\theta^*(\widehat{\varphi}_{\LocalFramingObject}^{[\KappaReflZMuBeta\colon\F_q]})\circ\hat{\delta}, \gamma H)\,. \nonumber
\end{align}
\end{numberedparagraph}

\begin{numberedparagraph}\label{Point7.10}
Now we define the Frobenius endomorphism on the source and target of $\Theta_{\FramingObject}$. For every multiple $m\in\N_0$ of $[\Oo_\mu:\F_q]$, the special fiber $\RZ_{\mcM,\ulLocalFramingObject}^{\leq \mu}\widehattimes_{\BreveOReflZMuBeta}\Spec\BaseFldInSectUnif$ of $\RZ_{\mcM,\ulLocalFramingObject}^{\leq \mu}$ carries a Frobenius endomorphism structure $\Phi_m$ defined as follows. Consider the absolute $q^m$-Frobenius $\tau^m:=\Frob_{q^m,S}\colon S\to S$ on an $\BaseFldInSectUnif$-scheme $S$. Consider a pair $(\underline{\mcL},\hat{\delta})\in\RZ_{\mcM,\ulLocalFramingObject}^{\leq \mu}(S)$, which induces the left horizontal morphisms in the diagram
\begin{equation}\label{qm-Frob-diagram}
\xymatrix @C+1pc {
S \ar[d]_{\textstyle\Frob_{q^m,S}} \ar[rrrr]^-{\textstyle(\underline{\mcL},\hat{\delta})} \ar[drrrr]_(.35){\textstyle({}^{\tau^m\!}\underline{\mcL},{}^{\tau^m\!}\hat{\delta})\qquad} & & & & \RZ_{\mcM,\ulLocalFramingObject}^{\leq \mu}\widehattimes_{\BreveOReflZMuBeta}\Spec\BaseFldInSectUnif \ar[d]_{\textstyle\Frob_{q^m}} \ar[r] & \Spec\BaseFldInSectUnif \ar[d]_{\textstyle\Frob_{q^m,\BaseFldInSectUnif}} \\
S \ar[rrrr]_-{\textstyle(\underline{\mcL},\hat{\delta})} & & & &\RZ_{\mcM,\ulLocalFramingObject}^{\leq \mu}\widehattimes_{\BreveOReflZMuBeta}\Spec\BaseFldInSectUnif \ar[r] & \Spec\BaseFldInSectUnif 
}
\end{equation}
Let $\theta\colon S\to\Spec\BaseFldInSectUnif$ be the structure morphism of $S$. Then the upper-left $S$ in diagram \eqref{qm-Frob-diagram}, viewed as a scheme over the lower-right $\Spec\BaseFldInSectUnif$, has structure morphism $\Frob_{q^m,\BaseFldInSectUnif}\circ\theta$. Thus in terms of \S\ref{Weil-descent-paragraph}, taking $\lambda:=\lambda_m:=\Frob_{q^m,\BaseFldInSectUnif}$, the upper-left $S$ becomes $S_{[\lambda_m]}$ over the lower-right $\Spec\BaseFldInSectUnif$. 

Let $\mcZ(\mu):=\mcZ^{\leq\mu}$ be the bound from Definition \ref{Def_BoundBy_mu}, which we used as the bound on $\RZ_{\mcM,\ulLocalFramingObject}^{\leq\mu}$ from Definition \ref{DefRZforM}. Its special fiber $\mcZ(\mu)_\infty:=\mcZ(\mu)\times_{\widetilde{X}_\mu} \Spec \kappa_\mu$ is defined over $\kappa_\mu$. Since $m$ is a multiple of $[\kappa_\mu\colon\F_q]$, the Frobenius $\widehat{\varphi}_{{}^{\tau^m\!}\mcL}={}^{\tau^m\!}\widehat{\varphi}_{\mcL}$ lies in ${}^{\tau^m\!}\mcZ(\mu)_\infty=\mcZ(\mu)_\infty$. Thus ${}^{\tau^m\!}\underline{\mcL}$ is also bounded by $\mcZ(\mu)$, and therefore the diagonal arrow $({}^{\tau^m\!}\underline{\mcL},{}^{\tau^m\!}\hat{\delta})$ lies in $\RZ_{\mcM,\ulLocalFramingObject}^{\leq \mu}(S_{[\lambda_m]})={}^{\lambda_m}\bigl(\RZ_{\mcM,\ulLocalFramingObject}^{\leq \mu}\bigr)(S)$. 

Via the Weil descent datum \eqref{EqDescentDatumSource}, $({}^{\tau^m\!}\underline{\mcL},{}^{\tau^m\!}\hat{\delta})$ is mapped to $({}^{\tau^m\!}\underline{\mcL}, \theta^*(\widehat{\varphi}_{\LocalFramingObject}^{\;-m})\circ{}^{\tau^m\!}\hat{\delta})$, which lies in $\RZ_{\mcM,\ulLocalFramingObject}^{\leq \mu}(S)$. We therefore define the $q^m$-Frobenius endomorphism of $\RZ_{\mcM,\ulLocalFramingObject}^{\leq \mu}\widehattimes_{R}\Spec\BaseFldInSectUnif$ as
\begin{align}\label{EqFrobOnRZ}
\Phi_m\colon\;\RZ_{\mcM,\ulLocalFramingObject}^{\leq \mu}\widehattimes_{R}\Spec\BaseFldInSectUnif \;\longto\; & \bigl(\RZ_{\mcM,\ulLocalFramingObject}^{\leq \mu}\widehattimes_{R}\Spec\BaseFldInSectUnif\bigr) \\
(\underline{\mcL},\hat{\delta}) \;\longmapsto\; & ({}^{\tau^m\!}\underline{\mcL}, \widehat{\varphi}_{\LocalFramingObject}^{\;-m}\circ{}^{\tau^m\!}\hat{\delta})\;=\;({}^{\tau^m\!}\underline{\mcL}, \hat{\delta}\circ\widehat{\varphi}_{\mcL}^{\;-m})\,. \nonumber
\end{align}
The product of the $q^m$-Frobenius endomorphism $\Phi_m$ from \eqref{EqFrobOnRZ} with the identity on $\Isom^{\otimes}(\omega ,\check{{\mathcal{V}}}_{\FramingObject})/H$ gives a $q^m$-Frobenius morphism (which we again denote by $\Phi_m$)
\begin{align}\label{EqFrobOnSource}
\Phi_m\colon\;\bigl(\RZ_{\mcM,\ulLocalFramingObject}^{\leq \mu}\widehattimes_{R}\Spec\BaseFldInSectUnif\bigr) & \times \Isom^{\otimes}(\omega ,\check{{\mathcal{V}}}_{\FramingObject})/H\;\longto\; \\
& \longto\; \bigl(\RZ_{\mcM,\ulLocalFramingObject}^{\leq \mu}\widehattimes_{R}\Spec\BaseFldInSectUnif\bigr)\times \Isom^{\otimes}(\omega ,\check{{\mathcal{V}}}_{\FramingObject})/H\,.\nonumber
\end{align}
Since the composition of $\eta\in I_{\FramingObject}(\FcnFld)$ and $\Phi_m$ is given by 
\[
(\underline{\mcL},\hat{\delta}) \times \gamma H\;\longmapsto\;(\underline{\mcL},L_{\infty,\mcM_{\beta^{-1}}}(\eta)\circ\hat{\delta}\circ\widehat{\varphi}_{\mcL}^{\;-m}) \times \check{{\mathcal{V}}}_\eta\gamma H,
\]
it follows that $\Phi_m$ commutes with the action of $ I_{\FramingObject}(\FcnFld)$. 
This defines the $q^m$-Frobenius endomorphism $\Phi_m$ of the source 
$I_{\FramingObject}(\FcnFld) \big{\backslash}\RZ_{\mcM,\ulLocalFramingObject}^{\leq \mu}\times \Isom^{\otimes}(\omega ,\check{{\mathcal{V}}}_{\FramingObject})/H$ of $\Theta_{\FramingObject}$
.

Now we discuss the Frobenius endomorphisms on the target. Let $m\in\N_0$ be a multiple of $[\KappaReflZMuBeta\colon\F_q]$. Since 
$\Sht_{\gpSch,H,\widehat{\infty}\times\infty}^{\mcZ(\mu,\beta)}\times_{\OReflZMuBeta} \Spec\BaseFldInSectUnif=(\Sht_{\gpSch,H,\widehat{\infty}\times\infty}^{\mcZ(\mu,\beta)}\times_{\OReflZMuBeta} \KappaReflZMuBeta)\times_{\KappaReflZMuBeta}\Spec\BaseFldInSectUnif$, the map $\id\times\tau^m$ defines the relative $q^m$-Frobenius of the left-hand side, which is an endomorphism, because this stack arises by base change from $\Spec\KappaReflZMuBeta$. 
Explicitly, this $q^m$-Frobenius endomorphism is given by
\begin{align}\label{EqFrobOnTarget}
\Phi_m\colon \Sht_{\gpSch,H,\widehat{\infty}\times\infty}^{\mcZ(\mu,\beta)}\widehattimes_{\OReflZMuBeta} \Spec\BaseFldInSectUnif\;\longto\; & \Sht_{\gpSch,H,\widehat{\infty}\times\infty}^{\mcZ(\mu,\beta)}\widehattimes_{\OReflZMuBeta} \Spec\BaseFldInSectUnif \nonumber \\
(\underline{\mcE},\gamma H) \;\longmapsto\; & ({}^{\tau^m\!}\underline{\mcE},{}^{\tau^m\!}(\gamma)H)\,.
\end{align}
We observe that ${}^{\tau^m\!}\underline{\mcE}$ is indeed bounded by $\mcZ(\mu,\beta)$: Since $m$ is a multiple of $[\KappaReflZMuBeta\colon\F_q]$ and the special fiber $\mcZ(\mu,\beta)_\infty:={\mcZ(\mu,\beta)}\times_{\widetilde{X}_{\mu,\beta}} \Spec\KappaReflZMuBeta$ of the bound ${\mcZ(\mu,\beta)}$ 
is defined over $\KappaReflZMuBeta$, it follows that ${}^{\tau^m\!}\underline{\mcE}$ is bounded by ${}^{\tau^m\!}\mcZ(\mu,\beta)_\infty=\mcZ(\mu,\beta)_\infty$.
\end{numberedparagraph}

\subsection{Statement of the uniformization theorem}
In this subsection, we define the uniformization morphism and state our main results, Theorems~\ref{Uniformization1} and \ref{Uniformization2}, which will be proven in Section~\ref{subsec:ProofMainThms}. Recall from \S\ref{SetupUniformization} that we fix a global $\gpSch$-shtuka $\FramingObject\in\Sht_{\gpSch,\varnothing,X\times\infty}^{\mcZ(\mu,\beta)}(\overline{\F}_\infty)$ and its associated local $\mcM$-shtuka $\ulLocalFramingObject:=L^+_{\infty,\mcM_{\beta^{-1}}}(\FramingObject)$. We fix a trivialization $\ulLocalFramingObject\cong\left((L^+_\infty\mcM)_{\overline{\F}_\infty},b\right)$. 
\begin{numberedparagraph}\label{PointDefTheta}
To define the uniformization map $\Theta_{\FramingObject}$, let
\[
(\underline{\mcL},\hat{\delta}:\underline{\mcL}\to \ulLocalFramingObject_S)\in \RZ_{\mcM,\ulLocalFramingObject}^{\leq\mu}(S), \quad\text{for }S\in\Nilp_{\BreveOReflZMuBeta},
\]
where $\underline{\mcL}\in \mathrm{LocSht}_{\mcM}^{\leq \mu}(S)$ is a local $\mcM$-shtuka bounded by $\mu$. By Proposition~\ref{PropQIsogLocalGlobal}\ref{PropQIsogLocalGlobal_B}, there is a global $\gpSch$-shtuka $\hat{\delta}^*\FramingObject$ in $\Sht_{\gpSch,\varnothing,\widehat{\infty}\times\infty}^{\mcZ(\mu,\beta)}(S)$ and a quasi-isogeny $\delta\colon \hat{\delta}^*\FramingObject\to \FramingObject$ with $L_{\infty,\mcM_{\beta^{-1}}}(\delta) = \hat{\delta}$. Since $\delta$ is defined over $S$, the isomorphism $\check{\mathcal{V}}_\delta$ is equivariant for the action of $\pi_1^\et(S,\bar s)$ which acts on $\check{\mathcal{V}}_\FramingObject$ through the map $\pi_1^\et(S,\bar s)\to \pi_1^\et(\Spec\overline{\F}_q,\bar s)=(1)$, that is trivially. Let $H\subset \gengpSch(\A^\infty)$ be an arbitrary compact open subgroup. In particular, the $H$-orbit $\check{\mathcal{V}}_\delta^{-1}\circ\gamma H$ of the tensor isomorphism $\check{\mathcal{V}}_\delta^{-1}\circ\gamma\colon\omega \isoto\check{\mathcal{V}}_{\hat{\delta}^*\FramingObject}$ is invariant under $\pi_1^\et(S,\bar s)$. This defines the following morphism 
\begin{equation}\label{EqTheta}
\begin{split}
\widetilde{\Theta}_{\FramingObject}: \RZ_{\mcM,\ulLocalFramingObject}^{\leq \mu}\times \Isom^{\otimes}(\omega,\check{{\mathcal{V}}}_{\FramingObject})/H & \longrightarrow \Sht_{\gpSch,H,\widehat{\infty}\times\infty}^{\mcZ(\mu,\beta)}\widehattimes_{\OReflZMuBeta} \Spf\BreveOReflZMuBeta \\
(\underline{\mcL},\hat{\delta},\gamma H)&\longmapsto (\hat{\delta}^*\FramingObject,\check{\mathcal{V}}_{\delta}^{-1}\circ \gamma H),
\end{split}
\end{equation}
which is obviously equivariant for the action of the center $Z(\FcnFld_\infty)$ given in \eqref{EqActionCenter3} and \eqref{EqActionCenter2}, and the action of $\gengpSch(\mathbb{A}^\infty)$ through Hecke correspondences given in \eqref{EqHeckeSource} and \eqref{EqHeckeTarget}.\end{numberedparagraph}

\begin{thm}
\label{Uniformization1} 

Consider a compact open subgroup $H\subset \gengpSch(\mathbb{A}^\infty)$. 
The morphism $\widetilde{\Theta}_{\FramingObject}$ from \eqref{EqTheta} is $I_{\FramingObject}(\FcnFld)$-invariant, where $I_{\FramingObject}(\FcnFld)$ acts trivially on the target and diagonally on the source as described in \S\ref{Point7.6}. Furthermore, this morphism factors through a morphism 
\begin{equation}\label{EqUnifMorph}
\Theta_{\FramingObject}\colon  I_{\FramingObject}(\FcnFld) \big{\backslash}\RZ_{\mcM,\ulLocalFramingObject}^{\leq \mu}\times \Isom^{\otimes}(\omega ,\check{{\mathcal{V}}}_{\FramingObject})/H \;\longto\; \Sht_{\gpSch,H,\widehat{\infty}\times\infty}^{\mcZ(\mu,\beta)}\widehattimes_{\OReflZMuBeta} \Spf\BreveOReflZMuBeta
\end{equation}
of formal algebraic Deligne-Mumford stacks over $\Spf\BreveOReflZMuBeta$ that is a monomorphism, i.e.~the functor $\Theta_{\FramingObject}$ is fully faithful, or equivalently its diagonal is an isomorphism. Both $\widetilde{\Theta}_{\FramingObject}$ and $\Theta_{\FramingObject}$ are ind-proper and formally \'etale.
\end{thm}

\begin{proof}
An elelemt $\eta\in I_{\FramingObject}(\FcnFld)$ acts on the source of the morphism $\widetilde{\Theta}_{\FramingObject}$ by sending an $S$-valued point $(\underline{\mcL},\hat{\delta}) \times \gamma H$ to $(\underline{\mcL},L_{\infty,\mcM_{\beta^{-1}}}(\eta)\hat{\delta})\times \check{\mathcal{V}}_\eta\circ\gamma H$. These two $S$-valued points are mapped under $\widetilde{\Theta}_{\FramingObject}$ to global $\gpSch$-shtukas with $H$-level structure $(\underline{\mcE},\check{\mathcal{V}}_\delta^{-1}\gamma H)$ and $(\underline{\widetilde{\mcE}},\check{\mathcal{V}}_{\tilde\delta}^{-1}\check{\mathcal{V}}_\eta\gamma H)$ over $S$, where $\delta\colon\underline{\mcE}:=\hat{\delta}^*\FramingObject\to\FramingObject$ is the isogeny satisfying $L_{\infty,\mcM_{\beta^{-1}}}(\delta)=\hat{\delta}$ and $\tilde\delta\colon\underline{\widetilde{\mcE}}:=(L_{\infty,\mcM_{\beta^{-1}}}(\eta)\hat{\delta})^*\,\FramingObject\to\FramingObject$ is the isogeny satisfying $L_{\infty,\mcM_{\beta^{-1}}}(\tilde\delta)=L_{\infty,\mcM_{\beta^{-1}}}(\eta)\hat{\delta}$. Since $\check{\mathcal{V}}_{\tilde\delta^{-1}\eta\delta}\circ\check{\mathcal{V}}_\delta^{-1}\gamma H=\check{\mathcal{V}}_{\tilde\delta}^{-1}\check{\mathcal{V}}_\eta\gamma H$, these two global $\gpSch$-shtukas with $H$-level structure are isomorphic via the quasi-isogeny $\tilde\delta^{-1}\eta\delta\colon\underline{\mcE}\to\underline{\widetilde{\mcE}}$, which is an isomorphism at $\infty$, because $L_{\infty,\mcM_{\beta^{-1}}}(\tilde\delta^{-1}\eta\delta)=(L_{\infty,\mcM_{\beta^{-1}}}(\eta)\hat{\delta})^{-1}\circ L_{\infty,\mcM_{\beta^{-1}}}(\eta)\hat{\delta}=\id$. In other words, $\widetilde{\Theta}_{\FramingObject}$ is invariant under the action of $I_{\FramingObject}(\FcnFld)$ and factors through the morphism $\Theta_{\FramingObject}$ from \eqref{EqUnifMorph} of formal algebraic Deligne-Mumford stacks. 

The remaining statements will be proven in Lemmas \ref{LemmaThetaEtaleProper} through \ref{LemmaThetaisadic}. 
\end{proof}
Recall that the scheme $X_{\mcM}^{\leq \mu}(b)\times \Isom^{\otimes}(\omega ,\check{{\mathcal{V}}}_{\FramingObject})/H$  is locally of finite type over $\BaseFldInSectUnif$. 
Let $\{T_j\}$ be a set of representatives of $I_{\FramingObject}(\FcnFld)$-orbits of its irreducible components. 
\begin{lem}\label{LemmaImageOfTheta}
\begin{enumerate}
\item\label{LemmaImageOfTheta_A}
The image $\widetilde{\Theta}_{\FramingObject}(T_j)$ of $T_j$ under $\widetilde{\Theta}_{\FramingObject}$ is a closed substack with the reduced substack structure, and each $\widetilde{\Theta}_{\FramingObject}(T_j)$ intersects only finitely many $\widetilde{\Theta}_{\FramingObject}(T_j')$ for $j'\neq j$. 
\item\label{LemmaImageOfTheta_B}
Let $\mathcal{X}_\FramingObject$ be the union of all the $\widetilde{\Theta}_{\FramingObject}(T_j)$. Then $\mathcal{X}_\FramingObject$ (with its reduced structure) is a separated Deligne-Mumford stack over $\overline{\F}_q$. Its underlying set $|\mathcal{X}_\FramingObject|\subset\Sht_{\gpSch,H,\infty\times\infty}^{\mcZ(\mu,\beta)}\times_{\KappaReflZMuBeta} \Spec\overline{\F}_q$ is the isogeny class of $\FramingObject$, i.e.~the set of all $(\underline{\mcE},\gamma H)$ for which $\underline{\mcE}$ is isogenous to $\FramingObject$.

Moreover, the morphism $\widetilde{\Theta}_{\FramingObject}$ restricted to $X_{\mcM}^{\leq \mu}(b)$ factors through a map $\iota_{\mathcal{X}}:\mathcal{X}_\FramingObject\to \Sht_{\gpSch,H,\widehat{\infty}\times\infty}^{\mcZ(\mu,\beta)}\widehattimes_{\OReflZMuBeta} \Spf\BreveOReflZMuBeta$. 
\end{enumerate}
\end{lem}
\begin{proof}
\ref{LemmaImageOfTheta_A} Note that the $T_j$ correspond bijectively to the irreducible components of the Deligne-Mumford stack $I_{\FramingObject}(\FcnFld) \big{\backslash}X_{\mcM}^{\leq \mu}(b)\times \Isom^{\otimes}(\omega,\check{{\mathcal{V}}}_{\FramingObject})/H $ from \eqref{EqReducedSourceOfTheta}, which is locally of finite type over $\Spec\BaseFldInSectUnif$. Since $\Theta_{\FramingObject}$ is a monomorphism by Theorem~\ref{Uniformization1}, each $\widetilde{\Theta}_{\FramingObject}(T_j)$ intersects only finitely many $\widetilde{\Theta}_{\FramingObject}(T_{j'})$ for $j'\neq j$. Since $T_j$ is quasi-compact by \cite[Corollary~4.26]{AH_Local} and $\widetilde{\Theta}_{\FramingObject}$ is ind-proper by Theorem~\ref{Uniformization1}, the restriction of $\widetilde{\Theta}_{\FramingObject}$ to $T_j$ is proper. Thus $\widetilde{\Theta}_{\FramingObject}(T_j)$ are closed substacks.

\noindent
\ref{LemmaImageOfTheta_B} 
For a field $K$, every $K$-valued point $(\underline{\mcE},\gamma H)$ of $\mathcal{X}_\FramingObject$ lies in the image of $\widetilde{\Theta}_{\FramingObject}$, and hence is of the form $\underline{\mcE}=\hat{\delta}^*\FramingObject$ with an isogeny $\delta\colon\underline{\mcE}\to\FramingObject$. This shows that $\mathcal{X}_\FramingObject$ is contained in the (quasi-)isogeny class of $\FramingObject$. 

Conversely, let $(\underline{\mcE},\gamma H)$ be a $K$-valued point of $\Sht_{\gpSch,H,\infty\times\infty}^{\mcZ(\mu,\beta)}$ in the isogeny class of $\FramingObject$, and let $\delta\colon\underline{\mcE}\to\FramingObject_K$ be a quasi-isogeny. Let $\underline{\mcL}:=L_{\infty,\mcM_{\beta^{-1}}}(\underline{\mcE})$ and $\hat{\delta}:=L_{\infty,\mcM_{\beta^{-1}}}(\delta)\colon\underline{\mcL}\to\ulLocalFramingObject_K$ and $(\check{\mathcal{V}}_{\delta}\circ\gamma) H\in\Isom^{\otimes}(\omega,\check{\mathcal{V}}_{\FramingObject})/H$. Then $(\underline{\mcL},\hat{\delta}, (\check{\mathcal{V}}_{\delta}\circ\gamma) H)$ is a $K$-valued point of the source of $\widetilde{\Theta}_{\FramingObject}$ which is mapped under $\widetilde{\Theta}_{\FramingObject}$ to $(\underline{\mcE}',(\check{\mathcal{V}}_{\delta'}^{-1}\circ(\check{\mathcal{V}}_{\delta}\circ\gamma) H)$, where $\underline{\mcE}':=\hat{\delta}^*\,\FramingObject$ and $\delta'\colon\underline{\mcE}'\to\FramingObject$ is the isogeny with $L_{\infty,\mcM_{\beta^{-1}}}(\delta')=\hat{\delta}$ which is an isomorphism outside $\infty$. The isogeny $\delta^{-1}\circ\delta'\colon\underline{\mcE}'\to\underline{\mcE}$ satisfies $L_{\infty,\mcM_{\beta^{-1}}}(\delta^{-1}\delta')=\id$ and $\check{\mathcal{V}}_{\delta^{-1}\delta'}\circ \check{\mathcal{V}}_{\delta'}^{-1}\circ\check{\mathcal{V}}_{\delta}\circ\gamma H=\gamma H$, and so $(\underline{\mcE}',\check{\mathcal{V}}_{\delta'}^{-1}\circ\check{\mathcal{V}}_{\delta}\circ \gamma H)\cong(\underline{\mcE},\gamma H)$ in $\Sht_{\gpSch,H,\infty\times\infty}^{\mcZ(\mu,\beta)}(K)$. The point $(\underline{\mcL},\hat{\delta}, \gamma H)$ lies on an irreducible component of $X_{\mcM}^{\leq \mu}(b)\times \Isom^{\otimes}(\omega,\check{\mathcal{V}}_{\FramingObject})/H$ belonging to the $I_{\FramingObject}(\FcnFld)$-orbit of some irreducible component $T_j$. By the $I_{\FramingObject}(\FcnFld)$-equivariance of $\widetilde{\Theta}_{\FramingObject}$ we can move the point $(\underline{\mcL},\hat{\delta}, \gamma H)$ to $T_j$ and then its image $(\underline{\mcE},\gamma H)$ under $\widetilde{\Theta}_{\FramingObject}$ lies in $\widetilde{\Theta}_{\FramingObject}(T_j)\subset\mathcal{X}_\FramingObject$ as desired.

To prove that $\mathcal{X}_\FramingObject$ is separated over $\Spec\BaseFldInSectUnif$ we use the valuative criterion \cite[Proposition~7.8]{Laumon-Moret-Bailly}. Let $R$ be a valuation ring containing $\BaseFldInSectUnif$ with fraction field $K$. Consider two morphisms $f_1,f_2\colon\Spec R\to \mathcal{X}_\FramingObject$ whose restrictions $f_{i,K}$ to $K$ are isomorphic in $\mathcal{X}_\FramingObject(K)$. We must show that $f_1\cong f_2$ in $\mathcal{X}_\FramingObject(R)$. The $K$-valued point $f_{1,K}\cong f_{2,K}$ lies in $\widetilde{\Theta}_{\FramingObject}(T_j)(K)\subset\mathcal{X}_\FramingObject(K)$ for some $j$. Since $\widetilde{\Theta}_{\FramingObject}(T_j)$ is a closed substack of $\Sht_{\gpSch,H,\infty\times\infty}^{\mcZ(\mu,\beta)}$, the two morphisms $f_1,f_2$ factor through $\widetilde{\Theta}_{\FramingObject}(T_j)$. Since $\Sht_{\gpSch,H,\infty\times\infty}^{\mcZ(\mu,\beta)}$ is separated over $\BaseFldInSectUnif$, also $\widetilde{\Theta}_{\FramingObject}(T_j)$ is separated over $\BaseFldInSectUnif$, and so $f_1\cong f_2$ in $\widetilde{\Theta}_{\FramingObject}(T_j)(R)$. Thus $\mathcal{X}_\FramingObject$ is separated over $\BaseFldInSectUnif$.
\end{proof}

Reasoning as in \cite[6.22]{RZ}, we may form the formal completion $\Sht_{\gpSch,H,\widehat{\infty}\times\infty}^{\mcZ(\mu,\beta)}{}_{/\mathcal{X}}$ of $\Sht_{\gpSch,H,\widehat{\infty}\times\infty}^{\mcZ(\mu,\beta)}\widehattimes_{\OReflZMuBeta} \Spf\BreveOReflZMuBeta$ along the set $\mathcal{X}:=\mathcal{X}_\FramingObject$ From Lemma~\ref{LemmaImageOfTheta}\ref{LemmaImageOfTheta_B}. It is the category fiberd over $\Nilp_{\BreveOReflZMuBeta}$ whose  $S$-valued points give the full subcategory 
\[
\Sht_{\gpSch,H,\widehat{\infty}\times\infty}^{\mcZ(\mu,\beta)}{}_{/\mathcal{X}}(S):=\bigl\{\,f\colon S\to\Sht_{\gpSch,H,\widehat{\infty}\times\infty}^{\mcZ(\mu,\beta)}\widehattimes_{\OReflZMuBeta} \Spf\BreveOReflZMuBeta, \enspace\text{such that } f|_{S_\red}\text{ factors through }\mathcal{X}_\FramingObject\,\bigr\}\,,
\]
where $S_\red$ is the underlying reduced closed subscheme of $S$. Note that it follows immediately that the natural morphism
\begin{equation}\label{EqShtCompletion}
\Sht_{\gpSch,H,\widehat{\infty}\times\infty}^{\mcZ(\mu,\beta)}{}_{/\mathcal{X}}\;\longto\;\Sht_{\gpSch,H,\widehat{\infty}\times\infty}^{\mcZ(\mu,\beta)}\widehattimes_{\OReflZMuBeta} \Spf\BreveOReflZMuBeta
\end{equation}
is an ind-proper monomorphism and formally \'etale, because for an affine scheme $S=\Spec B\in\Nilp_{\BreveOReflZMuBeta}$ and an ideal $I\subset B$ with $I^2=(0)$ one has $S_\red=(\Spec B/I)_\red$.

\begin{thm}\label{Uniformization2}
Let $\Sht_{\gpSch,H,\widehat{\infty}\times\infty}^{\mcZ(\mu,\beta)}{}{}_{/\mathcal{X}}$ be the formal completion of $\Sht_{\gpSch,H,\widehat{\infty}\times\infty}^{\mcZ(\mu,\beta)}\widehattimes_{\OReflZMuBeta} \Spf\BreveOReflZMuBeta$ along the set $\mathcal{X}_\FramingObject$ from Lemma~\ref{LemmaImageOfTheta}\ref{LemmaImageOfTheta_B}. 
\begin{enumerate}
\item \label{Uniformization2_B} Then $\Theta_{\FramingObject}$ induces an isomorphism of locally noetherian, adic formal algebraic Deligne-Mumford stacks locally formally of finite type over $\Spf\BreveOReflZMuBeta$
$$
\Theta_{\FramingObject,\mathcal{X}}\colon  I_{\FramingObject}(\FcnFld) \big{\backslash}\bigl(\RZ_{\mcM,\ulLocalFramingObject}^{\leq \mu}\times \Isom^{\otimes}(\omega ,\check{{\mathcal{V}}}_{\FramingObject})/H\bigr)\;\isoto\; \Sht_{\gpSch,H,\widehat{\infty}\times\infty}^{\mcZ(\mu,\beta)}{}_{/\mathcal{X}}\,,
$$ 
and in particular of the underlying Deligne-Mumford stacks
$$
\Theta_{\FramingObject,\mathcal{X}}\colon  I_{\FramingObject}(\FcnFld) \big{\backslash}X_{\mcM}^{\leq \mu}(b)\times \Isom^{\otimes}(\omega ,\check{{\mathcal{V}}}_{\FramingObject})/H\;\isoto\; \mathcal{X}_\FramingObject
$$
which are locally of finite type and separated over $\Spec\BaseFldInSectUnif$.

\item \label{Uniformization2_C}
The morphisms $\widetilde{\Theta}_{\FramingObject}$, $\Theta_{\FramingObject}$ and $\Theta_{\FramingObject,\mathcal{X}}$ are compatible with the following actions on source and target: the action of $Z(\FcnFld_\infty)$ described in \ref{center-action-paragraph}, the action of $\gengpSch(\mathbb{A}^\infty)$ through Hecke-corres\-ponden\-ces described in Definition~\ref{Defn-Hecke-corr}, and the Weil descent data described in \ref{Weil-descent-paragraph}. 

For every multiple $m\in\N_0$ of $[\KappaReflZMuBeta\colon\F_q]$, the base changes of $\widetilde{\Theta}_{\FramingObject}$, $\Theta_{\FramingObject}$ and $\Theta_{\FramingObject,\mathcal{X}}$ to $\Spec\BaseFldInSectUnif$ are compatible with the $q^m$-Frobenius endomorphisms $\Phi_m$ from \eqref{EqFrobOnSource} and \eqref{EqFrobOnTarget}. 
\end{enumerate}
\end{thm}
The proof will be given in the next section. 
\begin{remark}
Isogeny classes on $\Sht_{\gpSch,H,\widehat{\infty}\times\infty}^{\mcZ(\mu,\beta)}\widehattimes_{\OReflZMuBeta} \Spf\BreveOReflZMuBeta$ have the structure of quotients of affine Deligne--Lusztig varieties by some $\FcnFld$-rational group $I_{\FramingObject}(\FcnFld)$. 
\end{remark}
Theorem \ref{Uniformization2} has consequences for point-counting in the Langlands-Rapoport conjectures. 

\begin{remark} \label{RemManyPairsOfLegs}
Conbining our techniques with the ones from \cite[Theorem~7.11]{AH_Unif}, one can extend Theorems~\ref{Uniformization1} and \ref{Uniformization2} to the case of $n$ disjoint pairs $(x_i,\infty_i)_{i=1\ldots n}$ of two colliding legs, such that in each pair the leg $x_i$ varies and the other leg is fixed at a place $\infty_i$ and bounded by some element $\beta_i\in L_{\infty_i}\gpSch(\overline{\F}_q)$ with $\beta_i\cdot L^+_{\infty_i}\gpSch \cdot \beta_i^{-1} = L^+_{\infty_i}\gpSch$ for all $i$. Here disjoint means that $\infty_i\ne\infty_j$ for $i\ne j$. For each $i$ one considers the $\beta_i^{-1}$-twisted global-local functor $L^+_{\infty_i,\mcM_i}$ from global $\gpSch$-shtukas to local $\mcM_i$-shtukas, where $\mcM_i$ is the inner form of $\gpSch_{\infty_i}$ given by $\beta_i^{-1}$. The target space of the uniformization will be the stack $\Sht_{\gpSch,H,(\widehat{\infty}_i\times\infty_i)_i}^{\mcZ((\mu_i,\beta_i)_i)}$ of global $\gpSch$-shtukas with $n$ varying legs $x_i\colon S\to\Spf\Oo_{\infty_i}$ (respectively $n$ fixed legs $\infty_i$) at which the modification is bounded by a cocharacter $\mu_i$ (respectively by $\beta_i$), and with a $H$-level structure for a compact open subgroups $H\subset \gengpSch(\mathbb{A}^{\underline{\infty}})$ for $\underline{\infty}=(\infty_1,\ldots,\infty_n)$. As a global framing object one fixes a global $\gpSch$-shtuka $\FramingObject\in \Sht_{\gpSch,\varnothing,(\widehat{\infty}_i\times\infty_i)_i}^{\mcZ((\mu_i,\beta_i)_i)}(\overline{\F}_q)$ over $\overline{\F}_q$. For every $i$ the associated local $\mcM_i$-shtuka is $\ulLocalFramingObject_i:=L^+_{\infty_i,\mcM_i}(\FramingObject)\cong\bigl((L^+\mcM_i)_{\overline{\F}_q},b_i)$ with $b_i\in L_{\infty_i}\mcM_i(\overline{\F}_q)$. One obtains the Rapoport-Zink space $\RZ_{\mcM_i,\ulLocalFramingObject_i}^{\leq \mu_i}$ over $\Breve{\Oo}_{\mu_i}=:\overline{\F}_q\dbl \xi_i\dbr$ with the affine Deligne-Lusztig varietie $X_{\mcM_i}^{\leq \mu_i}(b_i)$ as its underlying topological space. The uniformization is then given by an isomorphism
\begin{equation}\label{EqRemManyPairsOfLegs}
\Theta_{\FramingObject,\mathcal{X}} \colon I_\FramingObject(\FcnFld) \big\backslash \bigl( \prod_{i=1}^n X_{\mcM_i}^{\leq \mu_i}(b_i) \times \Isom^\otimes(\omega,\check{\mathcal{V}}_\FramingObject)/H \bigr) \; \isoto \; \mathcal{X}_\FramingObject
\end{equation}
of Deligne-Mumford stacks locally of finite type and separated over $\overline{\F}_q$, obtained as the restriction of an isomorphism
\[
\Theta_{\FramingObject,\mathcal{X}} \colon I_\FramingObject(\FcnFld) \big\backslash \bigl( \prod_{i=1}^n \RZ_{\mcM_i,\ulLocalFramingObject_i}^{\leq \mu_i} \times \Isom^\otimes(\omega,\check{\mathcal{V}}_\FramingObject)/H \bigr) \; \isoto \; \Sht_{\gpSch,H,(\widehat{\infty}_i\times\infty_i)_i}^{\mcZ((\mu_i,\beta_i)_i)}{}_{/\mathcal{X}}
\]
of locally noetherian, adic formal algebraic Deligne-Mumford stacks locally formally of finite type over $\Breve{\Oo}_{(\mu_i,\beta_i)_i}:=\overline{\F}_q\dbl \xi_1,\ldots,\xi_n\dbr$, where the target is the formal completion of $\Sht_{\gpSch,H,(\widehat{\infty}_i\times\infty_i)_i}^{\mcZ((\mu_i,\beta_i)_i)} \times \Spf\Breve{\Oo}_{(\mu_i,\beta_i)_i}$ along the underlying set $\mathcal{X}_\FramingObject$ which is the image of the morphism \eqref{EqRemManyPairsOfLegs} and the isogeny class of $\FramingObject$.
\end{remark}

\subsection{Proof of the Main Theorems}\label{subsec:ProofMainThms}
\begin{lem}\label{LemmaThetaEtaleProper}
The maps $\widetilde{\Theta}_{\FramingObject}$, $\Theta_{\FramingObject}$, and $\Theta_{\FramingObject,\mathcal{X}}$ are ind-proper and formally \'etale.
\end{lem}
\begin{proof}
The claim that $\widetilde{\Theta}_{\FramingObject}$ is formally \'etale follows from the rigidity of quasi-isogenies. We give more details. Let $S\in\Nilp_{\BreveOReflZMuBeta}$, and let $(\underline{\mcE},\gamma H)\in \Sht_{\gpSch,H,\widehat{\infty}\times\infty}^{\mcZ(\mu,\beta)}(S)$ with $\gamma\in\Isom^{\otimes}(\omega ,\check{{\mathcal{V}}}_{\underline{\mcE}})$. Let $S'$ be a closed subscheme of $S$ defined by a locally nilpotent sheaf of ideals. Suppose that the base change $(\underline{\mcE}_{S'},\gamma H)$ to $S'$ lies in the image of $\widetilde{\Theta}_{\FramingObject}$, i.e.~there is a tuple 
\begin{equation}\label{tuple-in-RZtimesIsom}
(\underline{\mcL}',\hat{\delta}'\colon \underline{\mcL}'\to \ulLocalFramingObject_{S'},\gamma'H) \in \RZ_{\mcM,\ulLocalFramingObject}^{\leq \mu}(S')\times \Isom^{\otimes}(\omega ,\check{{\mathcal{V}}}_{\FramingObject})/H
\end{equation}
such that $\widetilde{\Theta}_{\FramingObject}(\underline{\mcL}',\hat{\delta}',\gamma'H)=(\underline{\mcE}_{S'},\gamma H)$ in $\Sht_{\gpSch,H,\widehat{\infty}\times\infty}^{\mcZ(\mu,\beta)}(S')$. The last equality means that there is a quasi-isogeny $\tilde{\delta}'\colon \underline{\mcE}_{S'} \to (\hat{\delta}')^*\FramingObject_{S'}$ of global $\gpSch$-shtukas over $S'$ that is an isomorphism over $\infty$ with $\check{\mathcal{V}}_{\tilde{\delta}'}^{-1}\circ\check{\mathcal{V}}_{\delta'}^{-1}\circ\gamma' H=\gamma H$ in $\Isom^{\otimes}(\omega ,\check{{\mathcal{V}}}_{\underline{\mcE}})/H$, where $\delta'\colon(\hat{\delta}')^*\FramingObject_{S'}\to \FramingObject_{S'}$ is the quasi-isogeny from Proposition~\ref{PropQIsogLocalGlobal}\ref{PropQIsogLocalGlobal_B} with $L_{\infty,\mcM_{\beta^{-1}}}((\hat{\delta}')^*\FramingObject_{S'})=\underline{\mcL}'$ and $L_{\infty,\mcM_{\beta^{-1}}}(\delta')=\hat{\delta}'$. In particular, 
\begin{equation}\label{loop-isom-equation}
L_{\infty,\mcM_{\beta^{-1}}}(\tilde{\delta}')\colon L_{\infty,\mcM_{\beta^{-1}}}(\underline{\mcE})_{S'} \isoto L_{\infty,\mcM_{\beta^{-1}}}\bigl((\hat{\delta}')^*\FramingObject_{S'}\bigr) = \underline{\mcL}'
\end{equation}
is an isomorphism of local $\mcM$-shtukas. Thus $(\underline{\mcL}',\hat{\delta}')$ equals $(L_{\infty,\mcM_{\beta^{-1}}}(\underline{\mcE})_{S'},\hat{\delta}'\circ L_{\infty,\mcM_{\beta^{-1}}}(\tilde{\delta}'))$ in $\RZ_{\mcM,\ulLocalFramingObject}^{\leq \mu}(S')$. We may replace the former by the latter in \eqref{tuple-in-RZtimesIsom} and thus assume $L_{\infty,\mcM_{\beta^{-1}}}((\hat{\delta}')^*\FramingObject_{S'})=L_{\infty,\mcM_{\beta^{-1}}}(\underline{\mcE})_{S'}$ and $L_{\infty,\mcM_{\beta^{-1}}}(\tilde{\delta}')=\id_{L_{\infty,\mcM_{\beta^{-1}}}(\underline{\mcE})_{S'}}$ in \eqref{loop-isom-equation}

Now the quasi-isogeny $\hat{\delta}'=\hat{\delta}'\circ L_{\infty,\mcM_{\beta^{-1}}}(\tilde{\delta}')\colon L_{\infty,\mcM_{\beta^{-1}}}(\underline{\mcE})_{S'}\to\ulLocalFramingObject_{S'}$ 
lifts uniquely to a quasi-isogeny $\hat{\delta}\colon L_{\infty,\mcM_{\beta^{-1}}}(\underline{\mcE})\to \ulLocalFramingObject_S$ 
over $S$ by the rigidity of quasi-isogenies for local $\mcM$-shtukas; see Proposition~\ref{PropRigidityLocal}. 
Therefore, $(L_{\infty,\mcM_{\beta^{-1}}}(\underline{\mcE}), \hat{\delta},\gamma' H)$ is an $S$-valued point of $\RZ_{\mcM,\ulLocalFramingObject}^{\leq \mu}(S)\times \Isom^{\otimes}(\omega ,\check{{\mathcal{V}}}_{\FramingObject})/H$. Its image under $\widetilde{\Theta}_{\FramingObject}$ is $(\hat{\delta}^*\FramingObject_S,\check{\mathcal{V}}_\delta^{-1}\circ \gamma'H)$, where $\delta\colon\hat{\delta}^*\FramingObject_S\to \FramingObject_S$ is the quasi-isogeny from Proposition~\ref{PropQIsogLocalGlobal}\ref{PropQIsogLocalGlobal_B} with $L_{\infty,\mcM_{\beta^{-1}}}(\hat{\delta}^*\FramingObject_S)=L_{\infty,\mcM_{\beta^{-1}}}(\underline{\mcE})$ and $L_{\infty,\mcM_{\beta^{-1}}}(\delta)=\hat{\delta}$. Since $(\hat{\delta}^*\FramingObject_S)_{S'}=(\hat{\delta}')^*\FramingObject_{S'}$, the quasi-isogeny $\tilde{\delta'}$ over $S'$ lifts uniquely to a quasi-isogeny $\tilde{\delta}\colon\underline{\mcE}\to \hat{\delta}^*\FramingObject_S$ over $S$ by rigidity of quasi-isogenies for global $\gpSch$-shtukas; see \cite[Proposition~5.9]{AH_Local}. It satisfies $L_{\infty,\mcM_{\beta^{-1}}}(\tilde{\delta})=\id_{L_{\infty,\mcM_{\beta^{-1}}}(\underline{\mcE})}$ by the uniqueness of the quasi-isogeny lifting $L_{\infty,\mcM_{\beta^{-1}}}(\tilde{\delta}')=\id$ to $S$. This shows that $\tilde{\delta}$ is a quasi-isogeny which is an isomorphism over $\infty$ and identifies $\widetilde{\Theta}_{\FramingObject}(L_{\infty,\mcM_{\beta^{-1}}}(\underline{\mcE}), \hat{\delta},\gamma'H):=(\hat{\delta}^*\FramingObject_S,\check{\mathcal{V}}_\delta^{-1}\circ \gamma'H)$ with $(\underline{\mcE},\gamma H)$ in $\Sht_{\gpSch,H,\widehat{\infty}\times\infty}^{\mcZ(\mu,\beta)}(S')$. This finishes the proof that $\widetilde{\Theta}_{\FramingObject}$ is formally \'etale.

Since the quotient morphism
\[
\bigl(\RZ_{\mcM,\ulLocalFramingObject}^{\leq \mu}\times \Isom^{\otimes}(\omega ,\check{{\mathcal{V}}}_{\FramingObject})/H\bigr) \enspace \onto \enspace I_{\FramingObject}(\FcnFld) \big{\backslash}\bigl(\RZ_{\mcM,\ulLocalFramingObject}^{\leq \mu}\times \Isom^{\otimes}(\omega ,\check{{\mathcal{V}}}_{\FramingObject})/H\bigr) 
\]
is \'etale by Proposition~\ref{PropQuotientByI}(b) and $\widetilde{\Theta}_{\FramingObject}$ is formally \'etale, also $\Theta_{\FramingObject}$ is formally \'etale. And since the morphism \eqref{EqShtCompletion} is formally \'etale, also $\Theta_{\FramingObject,\mathcal{X}}$ is formally \'etale.

Since $\mcM$ is parahoric, $\RZ_{\mcM,\ulLocalFramingObject}^{\leq \mu}$ is ind-proper over $\BreveOReflZMuBeta$ by Remark~\ref{RemRZIndProper}. Therefore, $\widetilde{\Theta}_{\FramingObject}$ is ind-proper. 
The morphism $\Theta_{\FramingObject}$ is ind-proper, because $\widetilde{\Theta}_{\FramingObject}$ is ind-proper and 
\[
\bigl(\RZ_{\mcM,\ulLocalFramingObject}^{\leq \mu}\times \Isom^{\otimes}(\omega ,\check{{\mathcal{V}}}_{\FramingObject})/H\bigr) \onto I_{\FramingObject}(\FcnFld) \big{\backslash}\bigl(\RZ_{\mcM,\ulLocalFramingObject}^{\leq \mu}\times \Isom^{\otimes}(\omega ,\check{{\mathcal{V}}}_{\FramingObject})/H\bigr)
\]
is surjective. Finally $\Theta_{\FramingObject,\mathcal{X}}$ is ind-proper, because the morphism \eqref{EqShtCompletion} is ind-proper.
\end{proof}

\begin{lem}\label{LemmaThetaismono}
We use the abbreviations $Y_1:=\RZ_{\mcM,\ulLocalFramingObject}^{\leq\mu}\times \Isom^{\otimes}(\omega,\check{\mathcal{V}}_{\FramingObject})/H$ and $Y_2:=X_{\mcM}^{\leq \mu}(b)\times \Isom^{\otimes}(\omega,\check{\mathcal{V}}_{\FramingObject})/H$. Then for $j=1$ or $2$, the action of $I_{\FramingObject}(\FcnFld)$ on $Y_j$ induces an isomorphism of stacks 
\begin{equation}\label{EqDiagIsIsom}
I_{\FramingObject}(\FcnFld) \times Y_j \enspace := \enspace \coprod_{I_{\FramingObject}(\FcnFld)} Y_j \enspace \isoto \enspace Y_j \underset{\Sht_{\gpSch,H,\widehat{\infty}\times\infty}^{\mcZ(\mu,\beta)}\widehattimes_{\OReflZMuBeta} \BreveOReflZMuBeta}{\times} Y_j\,,
\end{equation}
where the map to the first copy of $Y_j$ is the identity and the map to the second copy is given by the action of $I_{\FramingObject}(\FcnFld)$ on $Y_j$. In particular, $\Theta_{\FramingObject}$ is a monomorphism in the sense stated in Theorem~\ref{Uniformization1}. 
\end{lem}
\begin{proof}
By \cite[Tag~04Z7]{stacks-project}, the two definitions of a monomorphism given in Theorem~\ref{Uniformization1} are equivalent.
By the $I_{\FramingObject}(\FcnFld)$-equivariance of $\widetilde{\Theta}_{\FramingObject}$, the morphism \eqref{EqDiagIsIsom} is well defined. To describe its inverse, consider a connected scheme $S\in\Nilp_{\BreveOReflZMuBeta}$ and two $S$-valued points of $Y_j$ 
\[
y:=\bigl((\underline{\mcL},\hat{\delta}),\gamma H \bigr)~~~~ \text{and}~~~~ y':=\bigl((\underline{\mcL}',\hat{\delta}'),\gamma'H \bigr). 
\]
Under $\widetilde{\Theta}_{\FramingObject}$, they are mapped to global $\gpSch$-shtukas $(\underline{\mcE},\check{\mathcal{V}}_\delta^{-1} \gamma H)$ and $(\underline{\mcE}',\check{\mathcal{V}}_{\delta'}^{-1} \gamma'H)$ with $H$-level structures, where $\delta\colon\underline{\mcE} \to \FramingObject_S$ and $\delta'\colon  \underline{\mcE}'\to \FramingObject_S$ are the canonical quasi-isogenies which are isomorphisms outside $\infty$ with $L_{\infty,\mcM_{\beta^{-1}}}(\delta)=\hat{\delta}$ and $L_{\infty,\mcM_{\beta^{-1}}}(\delta')=\hat{\delta}'$. Suppose that $(\underline{\mcE},\check{\mathcal{V}}_\delta^{-1} \gamma H)$ and $(\underline{\mcE}',\check{\mathcal{V}}_{\delta'}^{-1} \gamma'H)$ are isomorphic in $\Sht_{\gpSch,H,\widehat{\infty}\times\infty}^{\mcZ(\mu,\beta)}(S)$ via a quasi-isogeny $\psi\colon\underline{\mcE}\to \underline{\mcE}'$ which is an isomorphism above $\infty$ and compatible with the $H$-level structures, i.e.~$\check{\mathcal{V}}_\psi\circ\check{\mathcal{V}}_\delta^{-1} \gamma H=\check{\mathcal{V}}_{\delta'}^{-1} \gamma'H$ (see Definition~\ref{DefRatLevelStr}). Consider the quasi-isogeny $\eta:=\delta'\psi \delta^{-1}$ from $\FramingObject_S$ to itself. By Proposition~\ref{Prop7.1} we may view $\eta$ as an element of $I_{\FramingObject}(\FcnFld)$.

Consider the corresponding quasi-isogenies between the associated local $\mcM$-shtukas
\begin{equation}
\xymatrix @C+2pc {
\underline{\mcL} \ar[r]^{L_{\infty,\mcM_{\beta^{-1}}}(\psi)} \ar[d]_{\hat{\delta}} & \underline{\mcL}' \ar[d]_{\hat{\delta}'} \\
\ulLocalFramingObject_S \ar[r]^{L_{\infty,\mcM_{\beta^{-1}}}(\eta)} & \ulLocalFramingObject_S
}
\end{equation}
Since $\psi\colon\underline{\mcE}\to \underline{\mcE}'$ is an isomorphism above $\infty$, the quasi-isogeny $L_{\infty,\mcM_{\beta^{-1}}}(\psi)$ is an isomorphism. Therefore, $\eta\cdot(\underline{\mcL},\hat{\delta}):=(\underline{\mcL},L_{\infty,\mcM_{\beta^{-1}}}(\eta)\circ\hat{\delta})\cong(\underline{\mcL}',\hat{\delta}')$ in $\RZ_{\mcM,\ulLocalFramingObject}^{\leq\mu}(S)$. Moreover, $\eta$ sends $\gamma H\in\Isom^{\otimes}(\omega,\check{\mathcal{V}}_{\FramingObject})/H$ to $\check{\mathcal{V}}_\eta\circ\gamma H=\check{\mathcal{V}}_{\delta'}\circ\check{\mathcal{V}}_\psi\circ \check{\mathcal{V}}_\delta^{-1}\circ\gamma H=\gamma'H$. This proves that $\eta\cdot y=y'$, i.e.~$(\eta,y)$ maps to $(y,y')$ under \eqref{EqDiagIsIsom}. Thus $\Theta_{\FramingObject}$ is a monomorphism.
\end{proof}

\newcommand{\SourceTheta}{\mathcal{R}}
\newcommand{\TargetTheta}{\mathcal{S}}

Next we turn to the proof of Theorem~\ref{Uniformization2}\ref{Uniformization2_B}. To shorten notations, we write
\begin{align*}
\SourceTheta & := I_{\FramingObject}(\FcnFld) \big{\backslash}\bigl(\RZ_{\mcM,\ulLocalFramingObject}^{\leq \mu}\times \Isom^{\otimes}(\omega ,\check{{\mathcal{V}}}_{\FramingObject})/H\bigr) \qquad\text{and} \\
\TargetTheta & := \Sht_{\gpSch,H,\widehat{\infty}\times\infty}^{\mcZ(\mu,\beta)}{}_{/\mathcal{X}}
\end{align*}
for the source and target of the morphism $\Theta_{\FramingObject,\mathcal{X}}$. These are locally noetherian, adic formal algebraic Deligne-Mumford stacks. For $\SourceTheta$, this was proven in Proposition~\ref{PropQuotientByI}, and for $\TargetTheta$, this follows from Proposition~\ref{PropLSGGsht1} and \cite[Proposition~A.14]{HartlAbSh}. Thus both $\SourceTheta$ and $\TargetTheta$ have unique largest ideals of definition $\mathcal{I}_{\SourceTheta}\subset\Oo_\SourceTheta$ and $\mathcal{I}_{\TargetTheta}\subset\Oo_\TargetTheta$ containing the maximal ideal $\breve{\mathfrak{m}}_{\mu,\beta}$ of $\BreveOReflZMuBeta$. For positive integers $m$, the closed substacks $\SourceTheta_m:=\Var(\mathcal{I}_{\SourceTheta}^m)\subset\SourceTheta$ and $\TargetTheta_m:=\Var(\mathcal{I}_{\TargetTheta}^m)\subset\TargetTheta$ are algebraic by \cite[Proposition~A.8]{HartlAbSh}. 
\begin{lem}\label{new-lemma}
 (a) $\SourceTheta_m$ and $\TargetTheta_m$ are Deligne-Mumford stacks locally of finite type over $\Spec\BreveOReflZMuBeta/\breve{\mathfrak{m}}_{\mu,\beta}^m$. We have $\SourceTheta=\varinjlim\limits_m\SourceTheta_m$ and $\TargetTheta=\varinjlim\limits_m\TargetTheta_m$. 
 Moreover, $\TargetTheta_1=\mathcal{X}_\FramingObject$ and $\SourceTheta_1=I_{\FramingObject}(\FcnFld) \big{\backslash}X_{\mcM}^{\leq \mu}(b)\times \Isom^{\otimes}(\omega ,\check{{\mathcal{V}}}_{\FramingObject})$.

 (b) The morphism $\Theta_{\FramingObject,\mathcal{X}}$ induces a morphism $\SourceTheta_m\to\TargetTheta_m$ for every $m$, which is locally of finite presentation as a morphism between Deligne-Mumford stacks locally of finite type over $\Spec\BreveOReflZMuBeta/\breve{\mathfrak{m}}_{\mu,\beta}^m$.
\end{lem}
\begin{proof}
For $\SourceTheta_m$, this follows since $\SourceTheta$ is a formal algebraic Deligne-Mumford stack locally formally of finite type over $\Spec\BreveOReflZMuBeta$. For $\TargetTheta_m$, it follows because $\TargetTheta_m$ is a closed substack of the Deligne-Mumford stack $\Sht_{\gpSch,H,\widehat{\infty}\times\infty}^{\mcZ(\mu,\beta)}\times_{\BreveOReflZMuBeta}\Spec\BreveOReflZMuBeta/\breve{\mathfrak{m}}_{\mu,\beta}^m$ which is locally of finite type over $\Spec\BreveOReflZMuBeta/\breve{\mathfrak{m}}_{\mu,\beta}^m$ by Proposition~\ref{PropLSGGsht1}. 

By definitions of $\mathcal{I}_{\TargetTheta}$ and $\mathcal{I}_{\SourceTheta}$, for $m=1$, the stacks $\TargetTheta_1=\TargetTheta_{\red}$ and $\SourceTheta_1=\SourceTheta_{\red}$ are reduced. Recall that $\mathcal{X}_\FramingObject$ is reduced by Lemma \ref{LemmaImageOfTheta}\ref{LemmaImageOfTheta_B}, we have $\mathcal{X}_\FramingObject=\TargetTheta_1$. Moreover, it is clear that $\SourceTheta_1=I_{\FramingObject}(\FcnFld) \big{\backslash}X_{\mcM}^{\leq \mu}(b)\times \Isom^{\otimes}(\omega ,\check{{\mathcal{V}}}_{\FramingObject})$.

Moreover, $\Theta_{\FramingObject,\mathcal{X}}$ induces a morphism $\Theta_{\FramingObject,\mathcal{X}}\colon \SourceTheta_1 \to \TargetTheta_1$, because if $\mathcal{P}\onto\SourceTheta_1$ is a presentation, then $\mathcal{P}$ is a reduced scheme, and hence $\Theta_{\FramingObject,\mathcal{X}}$ induces a morphism $\mathcal{P}\to \mathcal{X}_\FramingObject=\TargetTheta_1$ which descends to a morphism $\Theta_{\FramingObject,\mathcal{X}}\colon\SourceTheta_1\to\TargetTheta_1$. In particular, $\Theta_{\FramingObject,\mathcal{X}}^*(\mathcal{I}_{\TargetTheta})\subset\mathcal{I}_{\SourceTheta}$. 
\end{proof}

\begin{lem}\label{LemmaThetaRedQC}
For every $m$, the induced morphism $\Theta_{\FramingObject,\mathcal{X}}\colon \SourceTheta_m \to \TargetTheta_m$ is quasi-compact and surjective.
\end{lem}

\begin{proof}
The assertion only depends on the underlying topological spaces $|\TargetTheta_m|=|\TargetTheta_1|$ and $|\SourceTheta_m|=|\SourceTheta_1|$ (see \cite[Chapitre~5]{Laumon-Moret-Bailly}), so we may assume that $m=1$. The morphism $\Theta_{\FramingObject,\mathcal{X}}\colon \SourceTheta_1 \to \TargetTheta_1$ is surjective by the definition of $\mathcal{X}_\FramingObject$ as the image of $\widetilde{\Theta}_{\FramingObject}$.

We show that the morphism is quasi-compact. Let $S$ be a quasi-compact scheme and let $f\colon S\to\TargetTheta_1$ be a morphism. We must show that $S\times_{\TargetTheta_1}\SourceTheta_1$ is quasi-compact. Consider the topological space $|\mathcal{X}_\FramingObject|$ underlying $\mathcal{X}_\FramingObject$, and the set $\{T_j\}_{j\in N}$ of representatives of $I_{\FramingObject}(\FcnFld)$-orbits of the irreducible components of the scheme $X_{\mcM}^{\leq \mu}(b)\times \Isom^{\otimes}(\omega,\check{\mathcal{V}}_{\FramingObject})/H$ from Lemma~\ref{LemmaImageOfTheta}. By  Lemma~\ref{LemmaImageOfTheta}\ref{LemmaImageOfTheta_A}, every point $z\in|\mathcal{X}_\FramingObject|$ lies only in finitely many $\widetilde{\Theta}_{\FramingObject}(T_j)$. Let $N(z)\subset N$ be the set of $j\in N(z)$ such that $z\in\widetilde{\Theta}_{\FramingObject}(T_j)$. The open substack
\renewcommand{\ell}{l}
\[
U_z\;:=\;\bigcup_{j\in N(z)}\widetilde{\Theta}_{\FramingObject}(T_j)\;\setminus\bigcup_{j\notin N(z)}\widetilde{\Theta}_{\FramingObject}(T_j)\;\subset\;\mathcal{X}_\FramingObject
\]
contains $z$, and hence $\mathcal{X}_\FramingObject$ is covered by the $U_z$ for all $z\in|\mathcal{X}_\FramingObject|$. Let $s\in S$ be a point and set $z=f(s)\in|\mathcal{X}_\FramingObject|$. The preimage $f^{-1}(U_{f(s)})$ of $U_{f(s)}$ in $S$ contains $s$. We choose an affine open neighborhood $S_s$ of $s$ in $S$ which is contained in $f^{-1}(U_{f(s)})$. Then the $S_s$ cover $S$, and since $S$ is quasi-compact, we have $S=S_{s_1}\cup\ldots\cup S_{s_r}$ for finitely many points $s_k\in S$, $k=1,\ldots,r$. Since $\widetilde{\Theta}_{\FramingObject}\colon T_j\to\mathcal{X}_\FramingObject$ is proper by Lemma~\ref{LemmaThetaEtaleProper}, $S_{s_k} \times_{\mathcal{X}_\FramingObject} T_j$ is proper over the affine scheme $S_{s_k}$. The scheme $S'$ defined as the finite disjoint union
\[
S'\;:=\;\coprod_{k=1}^r\;\coprod_{j\in N(f(s_k))} S_{s_k} \underset{f,\mathcal{X}_\FramingObject,\widetilde{\Theta}_{\FramingObject}}{\times} T_j
\]
is quasi-compact. Since every point $s\in S$ lies in one $S_{s_k}$, and then $f(s)\in U_{f(s_k)}\subset\bigcup_{j\in N(f(s_k))}\widetilde{\Theta}_{\FramingObject}(T_j)$ has a preimage in one of the $T_j$, we conclude that the projection $S'\to S$ is surjective. Therefore, the projection $S'\to \coprod\limits_{k,j} T_j$ defines the upper horizontal morphism in the following commutative diagram
\begin{equation}
\xymatrix {
S' \ar@{->>}[d] \ar[rr] & & X_{\mcM}^{\leq \mu}(b)\times \Isom^{\otimes}(\omega,\check{\mathcal{V}}_{\FramingObject})/H \ar[d]{\widetilde{\Theta}_{\FramingObject}}\ar@{-->}[ld] \\
S \ar[r]_-{f} & \TargetTheta_1 \ar[r]_-{\iota_{\mathcal{X}}} & \Sht_{\gpSch,H,\widehat{\infty}\times\infty}^{\mcZ(\mu,\beta)}\widehattimes_{\OReflZMuBeta} \Spf\BreveOReflZMuBeta
}
\end{equation}
where $\iota_{\mathcal{X}}$ is as defined in Lemma \ref{LemmaImageOfTheta}\ref{LemmaImageOfTheta_B}, through which $\widetilde{\Theta}_{\FramingObject}$ factors.

Then we can consider the surjective morphisms 
\begin{equation}
\xymatrix @C=+8pc @R+1pc {
**{!L(0.45) !U(0.5)} \objectbox{\; S' \underset{\Sht_{\gpSch,H,\widehat{\infty}\times\infty}^{\mcZ(\mu,\beta)}\widehattimes_{\OReflZMuBeta} \Spf\BreveOReflZMuBeta}{\times} X_{\mcM}^{\leq \mu}(b)\times \Isom^{\otimes}(\omega,\check{\mathcal{V}}_{\FramingObject})/H} \ar@{->>}[d] & I_{\FramingObject}(\FcnFld) \times S' \ar[l]_-{\textstyle\sim} \ar@{-->}[d] \\
\quad\; S'\times_{\TargetTheta_1}\SourceTheta_1 \ar@{->>}[r] & S\times_{\TargetTheta_1}\SourceTheta_1 
}
\end{equation}
in which the isomorphism in the upper row comes from Lemma~\ref{LemmaThetaismono}. The left downward map is defined by taking the identity on $S'$, multiplied by the surjective quotient map by $I_{\FramingObject}(\FcnFld)$, and observing that $\Id_{S'}$ and the quotient map form a fiber product over $\TargetTheta_1$. Hence the left downward map is surjective. 
By the $I_{\FramingObject}(\FcnFld)$-equivariance of $\widetilde{\Theta}_{\FramingObject}$, the composite surjective map $I_{\FramingObject}(\FcnFld)\times S'\twoheadrightarrow S\times_{\TargetTheta_1}\SourceTheta_1$ gives 
a surjective map $S'\twoheadrightarrow S\times_{\TargetTheta_1}\SourceTheta_1$, and hence $S\times_{\TargetTheta_1}\SourceTheta_1$ is quasi-compact by \cite[Tag~04YC]{stacks-project}. 
\end{proof}

\begin{lem}\label{LemmaThetaisadic}
We have $\Theta_{\FramingObject,\mathcal{X}}^*(\mathcal{I}_{\TargetTheta})=\mathcal{I}_{\SourceTheta}$. In particular, the morphism $\Theta_{\FramingObject,\mathcal{X}}\colon\SourceTheta\to \TargetTheta$ of formal algebraic stacks is adic. 
\end{lem}

\begin{proof} 
Let $\mathcal{P}' \twoheadrightarrow \TargetTheta_1=\Var(\mathcal{I}_{\TargetTheta})$ be an atlas. By Lemmas~\ref{LemmaThetaEtaleProper}, \ref{LemmaThetaismono} and \ref{new-lemma}, we see that $\Theta_{\mathcal{X},m}\colon\SourceTheta_m\times_{\TargetTheta}\mathcal{P}'=\SourceTheta_m\times_{\TargetTheta_m}\mathcal{P}' \to \mathcal{P}'$ is a monomorphism locally of finite presentation satisfying the valuative criterion for properness. Since $\Theta_{\mathcal{X},m}$ is quasi-compact by Lemma~\ref{LemmaThetaRedQC}, it is a proper monomorphism, hence a closed immersion of schemes by \cite[Corollaire~A.2.2]{Laumon-Moret-Bailly}. Since $\Theta_{\mathcal{X},m}$ is surjective by Lemma~\ref{LemmaThetaRedQC} and $\mathcal{P}'$ is reduced, it must be an isomorphism for all $m$. Therefore $\SourceTheta_m\times_{\TargetTheta} \mathcal{P}'=\mathcal{P}'=\SourceTheta_1\times_{\TargetTheta} \mathcal{P}'$, and thus $\SourceTheta\times_{\TargetTheta} \mathcal{P}'=\dirlim\SourceTheta_m\times_{\TargetTheta} \mathcal{P}'=\mathcal{P}'=\SourceTheta_1\times_{\TargetTheta} \mathcal{P}'$. This shows that $\Var(\Theta_{\FramingObject,\mathcal{X}}^*\mathcal{I}_{\TargetTheta})=\SourceTheta\times_{\TargetTheta}\TargetTheta_1=\SourceTheta_1\times_{\TargetTheta}\TargetTheta_1\subset\SourceTheta_1=\Var(\mathcal{I}_{\SourceTheta})$ as closed substacks of $\SourceTheta$. Therefore $\mathcal{I}_{\SourceTheta}\subset\Theta_{\FramingObject,\mathcal{X}}^*(\mathcal{I}_{\TargetTheta})$ as the corresponding ideals. With the opposite inclusion established in Lemma \ref{new-lemma}, we have $\mathcal{I}_{\SourceTheta}=\Theta_{\FramingObject,\mathcal{X}}^*(\mathcal{I}_{\TargetTheta})$, and hence $\Theta_{\FramingObject,\mathcal{X}}$ is adic. 
\end{proof}

\begin{proof}[Proof of Theorem~\ref{Uniformization2}\ref{Uniformization2_B}]
By Lemma~\ref{LemmaThetaisadic}, we have $\Theta_{\FramingObject,\mathcal{X}}^*(\mathcal{I}_{\TargetTheta})=\mathcal{I}_{\SourceTheta}$. Then $\SourceTheta_m\to\TargetTheta_m$ is obtained from $\SourceTheta\to\TargetTheta$ by base change via $\TargetTheta_m\to \TargetTheta$, and hence is a formally \'etale morphism locally of finite presentation of algebraic stacks by Lemma~\ref{LemmaThetaEtaleProper}. Since $\Theta_{\FramingObject}$ is a monomorphism by Lemma~\ref{LemmaThetaismono} and \eqref{EqShtCompletion} is a monomorphism, thus  $\SourceTheta_m\to\TargetTheta_m$ is a monomorphism. In particular, $\SourceTheta_m\to\TargetTheta_m$ is relatively representable by an \'etale monomorphism of schemes; see \cite[Corollaire~8.1.3 and Th\'eor\`eme~A.2]{Laumon-Moret-Bailly}. In addition, $\SourceTheta_m\to\TargetTheta_m$ is surjective by Lemma~\ref{LemmaThetaRedQC}, hence an isomorphism by 
[EGA IV$_4$, Th\'eor\`eme~17.9.1]
. As this holds for all $m$, we conclude that $\Theta_{\FramingObject,\mathcal{X}}:\SourceTheta\to\TargetTheta$ is an isomorphism of stacks.
\end{proof}

\begin{proof}[Proof of  Theorem~\ref{Uniformization2}\ref{Uniformization2_C}]
Recall from earlier that $\widetilde{\Theta}_{\FramingObject}$ is equivariant for the action of the center $Z(\FcnFld_\infty)$ given in \eqref{EqActionCenter3} and \eqref{EqActionCenter2}, and the action of $\gengpSch(\mathbb{A}^\infty)$ through Hecke correspondences given in \eqref{EqHeckeSource} and \eqref{EqHeckeTarget}. Thus it suffices to check that $\widetilde{\Theta}_{\FramingObject}$ is compatible with Weil descent data and Frobenius endomorphism structure.

First we show that the morphism $\widetilde{\Theta}_{\FramingObject}$ is compatible with the Weil descent data \eqref{EqDescentDatumOnNablaH} and \eqref{EqDescentDatumSource}. Let $(S,\theta)\in\Nilp_{\BreveOReflZMuBeta}$ and $S_{[\lambda]}=(S,\lambda\circ\theta)\in\Nilp_{\BreveOReflZMuBeta}$ where $\lambda$ is defined in \eqref{EqWeilDescent_lambda}. The $S$-valued point $(\underline{\mcL},\hat{\delta}, \gamma H)$, respectively the $S_{[\lambda]}$-valued point $(\underline{\mcL},\theta^*(\widehat{\varphi}_{\LocalFramingObject}^{[\KappaReflZMuBeta\colon\F_q]})\circ\hat{\delta}, \gamma H)$, of the source of $\widetilde{\Theta}_{\FramingObject}$ are sent to $(\underline{\mcE},\check{\mathcal{V}}_\delta^{-1}\theta^*(\gamma) H)$ and $(\underline{\mcE}',\check{\mathcal{V}}_{\delta'}^{-1}(\lambda\circ\theta)^*(\gamma) H)$, respectively, where $\underline{\mcE}:=\hat{\delta}^*\FramingObject_S$ and $\underline{\mcE}':=(\theta^*\widehat{\varphi}_{\LocalFramingObject}^{[\KappaReflZMuBeta\colon\F_q]}\circ\hat{\delta})^*\FramingObject_{S_{[\lambda]}}$, and $\delta\colon\underline{\mcE}\to\FramingObject_S:=\theta^*\FramingObject$ and $\delta'\colon\underline{\mcE}'\to\FramingObject_{S_{[\lambda]}}:=(\lambda\circ\theta)^*\FramingObject=\theta^*\lambda^*\FramingObject=\theta^*({}^{\tau^{[\KappaReflZMuBeta\colon\F_q]}\!}\FramingObject)$ are the quasi-isogenies with $L_{\infty,\mcM_{\beta^{-1}}}(\delta)=\hat{\delta}$ and $L_{\infty,\mcM_{\beta^{-1}}}(\delta')=\theta^*(\widehat{\varphi}_{\LocalFramingObject}^{[\KappaReflZMuBeta\colon\F_q]})\circ\hat{\delta}$, which are isomorphisms outside $\infty$. Then $\psi:=\delta^{-1}\circ\theta^*(\Phi_{\FramingObject}^{-[\KappaReflZMuBeta\colon\F_q]})\circ\delta'\colon\underline{\mcE}'\to\underline{\mcE}$ is a quasi-isogeny by Definition~\ref{DeFfrobIsog}, which by Corollary~\ref{CorFrobIsog} satisfies $L_{\infty,\mcM_{\beta^{-1}}}(\psi)=\id_{\underline{\mcL}}$ and $\check{\mathcal{V}}_\psi\circ\check{\mathcal{V}}_{\delta'}^{-1}\circ\theta^*\lambda^*(\gamma) H=\check{\mathcal{V}}_\delta^{-1}\theta^*(\gamma) H$, because $\lambda^*(\gamma)={}^{\tau^{[\KappaReflZMuBeta\colon\F_q]}\!}\gamma=\check{\mathcal{V}}_{\Phi_{\FramingObject}^{[\KappaReflZMuBeta\colon\F_q]}}\circ\gamma$. Therefore, $\psi$ identifies $(\underline{\mcE}',\check{\mathcal{V}}_{\delta'}^{-1}(\lambda\circ\theta)^*(\gamma) H)$ with $(\underline{\mcE},\check{\mathcal{V}}_\delta^{-1}\theta^*(\gamma) H)$, and this establishes the compatibility of $\widetilde{\Theta}_{\FramingObject}$ with the Weil descent data \eqref{EqDescentDatumOnNablaH} and \eqref{EqDescentDatumSource}. 

Since the Weil descent datum on the source of $\widetilde{\Theta}_{\FramingObject}$ commutes with the action of $I_{\FramingObject}(\FcnFld)$, this also proves the compatibility of $\Theta_{\FramingObject}$ with the Weil descent data \eqref{EqDescentDatumOnNablaH} and \eqref{EqDescentDatumSourceModI}. 

Finally, the target $\Sht_{\gpSch,H,\widehat{\infty}\times\infty}^{\mcZ(\mu,\beta)}{}_{/\mathcal{X}}$ carries the Weil descent datum \eqref{EqDescentDatumOnNablaH} induced from $\Sht_{\gpSch,H,\widehat{\infty}\times\infty}^{\mcZ(\mu,\beta)} \widehattimes_{\OReflZMuBeta} \Spf\BreveOReflZMuBeta$. To see this: if a morphism $S_\red\to\Sht_{\gpSch,H,\widehat{\infty}\times\infty}^{\mcZ(\mu,\beta)} \widehattimes_{\OReflZMuBeta} \Spf\BreveOReflZMuBeta$ given by $(\underline{\mcE},\gamma H)$ factors through $\mathcal{X}_\FramingObject=\Image(\widetilde{\Theta}_{\FramingObject})$, then the morphism $(S_{[\lambda]})_\red\to\Sht_{\gpSch,H,\widehat{\infty}\times\infty}^{\mcZ(\mu,\beta)} \widehattimes_{\OReflZMuBeta} \Spf\BreveOReflZMuBeta$ given by $(\underline{\mcE},\gamma H)$ also factors through $\mathcal{X}_\FramingObject=\Image(\widetilde{\Theta}_{\FramingObject})$, because $\widetilde{\Theta}_{\FramingObject}$ commutes with the Weil descent data. This shows that $\Theta_{\FramingObject,\mathcal{X}}$ is also compatible with the Weil descent data \eqref{EqDescentDatumOnNablaH} and \eqref{EqDescentDatumSourceModI}.

\bigskip

We also prove that $\widetilde{\Theta}_{\FramingObject}$ commutes with the $q^m$-Frobenius endomorphisms $\Phi_m$ from \eqref{EqFrobOnSource} and \eqref{EqFrobOnTarget}. Let $y:=(\underline{\mcL},\hat{\delta}, \gamma H)$ be an $S$-valued point of 
\[
\bigl(\RZ_{\mcM,\ulLocalFramingObject}^{\leq \mu}\widehattimes_{\BreveOReflZMuBeta}\Spec\BaseFldInSectUnif\bigr) \times \Isom^{\otimes}(\omega,\check{\mathcal{V}}_{\FramingObject})/H.
\]
The images of this point and of $\Phi_m(y)=({}^{\tau^m\!}\underline{\mcL}, \widehat{\varphi}_{\LocalFramingObject}^{\;-m}\circ{}^{\tau^m\!}\hat{\delta},\gamma H)$ in $\Sht_{\gpSch,H,\widehat{\infty}\times\infty}^{\mcZ(\mu,\beta)}$ are given by $\widetilde{\Theta}_{\FramingObject}(y)=(\underline{\mcE},\check{\mathcal{V}}_\delta^{-1}\gamma H)$ and $\widetilde{\Theta}_{\FramingObject}\circ\Phi_m(y)=(\underline{\mcE}',\check{\mathcal{V}}_{\delta'}^{-1}\gamma H)$, respectively, where $\underline{\mcE}:=\hat{\delta}^*\FramingObject_S$ and $\underline{\mcE}':=(\widehat{\varphi}_{\LocalFramingObject}^{\;-m}\circ {}^{\tau^m\!}\hat{\delta})^*\FramingObject_S$, and $\delta\colon\underline{\mcE}\to\FramingObject_S$ and $\delta'\colon\underline{\mcE}'\to\FramingObject_S$ are the quasi-isogenies with $L_{\infty,\mcM_{\beta^{-1}}}(\delta)=\hat{\delta}$ and $L_{\infty,\mcM_{\beta^{-1}}}(\delta')=\widehat{\varphi}_{\LocalFramingObject}^{\;-m}\circ{}^{\tau^m\!}\hat{\delta}$, which are isomorphisms outside $\infty$. We obtain the image $\Phi_m\circ\widetilde{\Theta}_{\FramingObject}(y)=({}^{\tau^m\!}\underline{\mcE},{}^{\tau^m\!}(\check{\mathcal{V}}_\delta^{-1}\gamma) H)$, which comes with the quasi-isogeny ${}^{\tau^m\!}\delta\colon{}^{\tau^m\!}\underline{\mcE}\to{}^{\tau^m\!}\FramingObject_S$. Then $\psi:={}^{\tau^m\!}\delta^{-1}\circ\Phi_{\FramingObject}^{m}\circ\delta'\colon\underline{\mcE}'\to{}^{\tau^m\!}\underline{\mcE}$ is a quasi-isogeny by Definition~\ref{DeFfrobIsog} which by Corollary~\ref{CorFrobIsog} satisfies $L_{\infty,\mcM_{\beta^{-1}}}(\psi)=\id_{{}^{\tau^m\!}\underline{\mcL}}$ and $\check{\mathcal{V}}_\psi\circ\check{\mathcal{V}}_{\delta'}^{-1}\gamma H=\check{\mathcal{V}}_{{}^{\tau^m\!}\delta}^{-1}\circ\check{\mathcal{V}}_{\Phi_{\FramingObject}^m}\circ\gamma H={}^{\tau^m\!}(\check{\mathcal{V}}_\delta^{-1}\gamma) H$, because $\check{\mathcal{V}}_{\Phi_{\FramingObject}^{m}}\circ\gamma={}^{\tau^m\!}(\gamma)$. Therefore, $\psi$ identifies $(\underline{\mcE}',\check{\mathcal{V}}_{\delta'}^{-1}\gamma H)$ with $({}^{\tau^m\!}\underline{\mcE},{}^{\tau^m\!}(\check{\mathcal{V}}_\delta^{-1}\gamma) H)$, and this proves $\widetilde{\Theta}_{\FramingObject}\circ\Phi_m=\Phi_m\circ\widetilde{\Theta}_{\FramingObject}$. 

Since $\Phi_m$ commutes with the action of $I_{\FramingObject}(\FcnFld)$, this also proves that $\Theta_{\FramingObject}$ commutes with the $q^m$-Frobenius endomorphisms $\Phi_m$. Finally, the target $\Sht_{\gpSch,H,\widehat{\infty}\times\infty}^{\mcZ(\mu,\beta)}{}_{/\mathcal{X}}$ carries the $q^m$-Frobenius endomorphism $\Phi_{\underline{\mcE}}^m$ induced from \eqref{EqFrobOnTarget} on $\Sht_{\gpSch,H,\widehat{\infty}\times\infty}^{\mcZ(\mu,\beta)} \widehattimes_{\OReflZMuBeta} \Spf\BreveOReflZMuBeta$. The reason is that if a morphism $S_\red\to\Sht_{\gpSch,H,\widehat{\infty}\times\infty}^{\mcZ(\mu,\beta)} \widehattimes_{\OReflZMuBeta} \Spf\BreveOReflZMuBeta$ given by $(\underline{\mcE},\gamma H)$ factors through $\mathcal{X}_\FramingObject=\Image(\widetilde{\Theta}_{\FramingObject})$, then the morphism $S_\red\to\Sht_{\gpSch,H,\widehat{\infty}\times\infty}^{\mcZ(\mu,\beta)} \widehattimes_{\OReflZMuBeta} \Spf\BreveOReflZMuBeta$ given by $({}^{\tau^m\!}\underline{\mcE},{}^{\tau^m\!}(\lambda)H)$ also factors through $\mathcal{X}_\FramingObject=\Image(\widetilde{\Theta}_{\FramingObject})$, because $\widetilde{\Theta}_{\FramingObject}$ commutes with the $\Phi_m$. This shows that also $\Theta_{\FramingObject,\mathcal{X}}$ is compatible with the $q^m$-Frobenius endomorphisms $\Phi_m$ from \eqref{EqFrobOnTarget} and \eqref{EqFrobOnSource}. This completes the proof of Theorem~\ref{Uniformization2}\ref{Uniformization2_C}.
\end{proof}

\subsection{Application to the Langlands-Rapoport Conjecture}
We finish this section with an application of our main theorem \ref{Uniformization2} to the Langlands-Rapoport Conjecture over function fields. Consider the following category $\mathcal{M}ot_X^{\infty}(\overline{\F}_q)$, which generalizes Anderson's $t$-motives \cite{Anderson-tmotives}. Write $\Breve{\FcnFld}:=\FcnFld\otimes_{\F_q}\overline{\F}_q$.
\begin{defn}\label{defn-category-Xmotives}
The category of $X$-motives $\mathcal{M}ot_X^{\infty}$ is defined as 
\begin{equation}
\begin{split}
    \mathcal{M}ot_X^{\infty}(\overline{\F}_q):=\{\, & (\mcF,\varphi)\colon \mcF\text{ is a vector bundle on $X_{\overline{\F}_q}$, and } \\
    & \varphi: \mcF|_{(X\setminus\{\infty\})_{\overline{\F}_q}}\xrightarrow{\sim}{}^{\tau\!}\mcF|_{(X\setminus\{\infty\})_{\overline{\F}_q}} \text{ is an isomorphism of vector bundles} \},
\end{split}
\end{equation}
with morphisms in this category given by 
\begin{equation}
    \mathrm{Hom}_{\mathcal{M}ot_X^{\infty}(\overline{\F}_q)}((\mcF,\varphi),(\widetilde{\mcF},\widetilde{\varphi})):=\{f\in \mathrm{Hom}_{\Breve{\FcnFld}}(\mcF\otimes_{\mathcal{O}_X}\FcnFld,\widetilde{\mcF}\otimes_{\mathcal{O}_X}\FcnFld) \colon \widetilde{\varphi}\circ f={}^{\tau}f\circ\varphi\}.
\end{equation}
The \emph{motivic Galois gerbe} is the Tannakian fundamental group 
\begin{equation}
    \mathfrak{P}:=\Aut^{\otimes}(\underline{\omega}|\mathcal{M}ot_X^{\infty}(\overline{\F}_q))
\end{equation}
corresponding to the fiber functor $\underline{\omega}$. 
\end{defn}

\begin{remark}\label{RemXMotives} A $X$-motive $\underline{\mathcal{F}}=(\mcF,\varphi)$ of rank $r$ is just the same as a global $\GL_r$-shtuka in $\Sht_{\GL_r,\varnothing,\infty\times\infty}(\overline{\F}_\infty)$ over $\overline{\F}_\infty$. But note that the category $\Sht_{\GL_r,\varnothing,\infty\times\infty}(\overline{\F}_\infty)$ is a groupoid with all its morphisms being isomorphisms, while $\mathcal{M}ot_X^{\infty}(\overline{\F}_q)$ is abelian and contains morphisms which may fail to be isomorphisms, and even morphisms between $X$-motives of different rank are allowed.
\end{remark}

In \cite[Theorem~1.5]{AH_CMotives}, Arasteh Rad and the first author proved the following result.

\begin{lem}
$\mathcal{M}ot_X^{\infty}(\overline{\F}_q)$ is a semi-simple $\FcnFld$-linear Tannakian category, with a canonical fiber functor $\underline{\omega}$ over $\Breve{\FcnFld}$ given by $\underline{\omega}: (\mcF,\varphi)\mapsto \mcF\otimes_{\Oo_X}\FcnFld=\mcF\otimes_{\Oo_{X_{\overline{\F}_q}}}\Breve{\FcnFld}$. The motivic Galois gerbe is an extension
\[
1\to \mathfrak{P}^{\Delta}\to\mathfrak{P}\to \Gal(\Breve{\FcnFld}/\FcnFld)\to 1,
\]
where the kernel group $\mathfrak{P}^{\Delta}$ is a pro-reductive group over $\Breve{\FcnFld}$.
\end{lem}

Note that $\underline{\omega}$ is not the only fiber functor for 
$\mathcal{M}ot_X^{\infty}(\overline{\F}_q)$.

\begin{defn}
For any closed point $v\in X\setminus\{\infty\}$, there is a fiber functor given by the $v$-adic dual Tate module (also called $v$-adic \'etale cohomology) of $\underline{\mathcal{F}}$
\[
{\rm H}^1_{v,\et}:\mathcal{M}ot_X^{\infty}(\overline{\F}_q)\longrightarrow Q_v\text{-Vect},\qquad \underline{\mathcal{F}} \mapsto {\rm H}^1_{v,\et}(\underline{\mathcal{F}},\Oo_v):= \{m\in \mathcal{F}\otimes_{\Oo_{X_{\overline{\F}_q}}} \Breve{\Oo}_v\colon \varphi(m) = {}^{\tau\!}m\}\otimes_{\Oo_v} \FcnFld_v.
\]
The fiber functor ${\rm H}^1_{v,\et}$ corresponds to a homomorphism of Galois gerbs $\xi_v\colon \mathfrak{H}_v\to\mathfrak{P}$, where $\mathfrak{H}_v=\Gal(\Breve{\FcnFld}/\FcnFld)$ is the trivial Galois gerbe, i.e.~the Tannakian fundamental group of the category of $\FcnFld_v$-vector spaces. 
At $\infty$ there is a fiber functor given by the ``crystalline cohomology'' of $\underline{\mathcal{F}}$
\[
{\rm H}^1_{\infty,\rm crys}:\mathcal{M}ot_X^{\infty}(\overline{\F}_q)\longrightarrow \Breve{Q}_{\infty}\text{-Vect},\qquad \underline{\mathcal{F}} \mapsto \underline{\mathcal{F}} \otimes_{\Oo_{X_{\overline{\F}_q}}} \Breve{\FcnFld}_\infty\,.
\]
The fiber functor ${\rm H}^1_{\infty,\rm crys}$ corresponds to a homomorphism of Galois gerbs $\xi_\infty\colon \mathfrak{H}_\infty\to\mathfrak{P}$, where $\mathfrak{H}_\infty$ is the ``Dieudonne gerbe'', which is the Tannakian fundamental group of the category of $F$-crystals for $\FcnFld_\infty$. 
\end{defn}

\begin{defn}\label{Def_GMotives}
For a smooth affine group scheme $\gpSch$ with connected fibers over $X$ and reductive generic fiber $\gengpSch$, a \emph{$\gengpSch$-motive} is a tensor functor $\underline{\mathcal{M}}_\gengpSch:\Rep_\FcnFld(\gengpSch)\to \mathcal{M}ot_X^{\infty}(\overline{\F}_q)$. Equivalently, a $\gengpSch$-motive is a homomorphism of Galois gerbes $h\colon \mathfrak{P}\to\mathfrak{G}_\gengpSch$, where $\mathfrak{G}_\gengpSch:=\gengpSch(\Breve{\FcnFld})\rtimes \Gal(\Breve{\FcnFld}/\FcnFld)$ is the neutral Galois gerbe of $\gengpSch$. We denote by $\gengpSch\mathcal{M}ot_X^{\infty}(\overline{\F}_q)$ the category of $\gengpSch$-motives. 
\end{defn}

For any morphism $h: \mathfrak{P}\to \mathfrak{G}_\gengpSch$ of Galois gerbes, we let $h_v:=h\circ\xi_v: \mathfrak{H}_v\to \mathfrak{G}_\gengpSch$ and $h_\infty:=h\circ\xi_\infty: \mathfrak{H}_{\infty}\to \mathfrak{G}_\gengpSch$ be the compositum. The morphism $h_\infty$ defines a ``local $\gengpSch$-isoshtuka'' $\bigl((L_\infty\mcM)_{\overline{\F}_q},b)\bigr)$ over $\overline{\F}_q$; see \cite{AH_LRConj}. We define
\[
X^{\infty}(h):=\{(g_v)_{v\neq \infty}\in \gengpSch(\A^{\infty}):\Int(g_v)\circ\xi_v=\xi_v \}.
\]
\[
X_{\infty}(h):=\{\overline{g}\in (L_\infty\mcM/L^+_\infty\mcM)(\overline{\F}_\infty)\colon \tau_M(g)^{-1}b g\text{ is bounded by }\mu\}.
\]

There is a functor $\Sht_{\gpSch,\infty\times\infty}(\overline{\F}_q)\to \gengpSch\mathcal{M}ot_X^{\infty}(\overline{\F}_q)$ given by sending a global $\gpSch$-shtuka $\underline{\mcE}=(\mcE,\mcE',\varphi,\varphi')$ to the tensor functor 
\begin{equation}
\underline{\mathcal{M}}_\gengpSch^{\underline{\mcE}}:(V,\rho)\mapsto \bigl(\mcE\overset{\gengpSch}{\times}V,(\varphi'\circ\varphi)\times 1\bigr),
\end{equation}
where $\rho:\gengpSch\to \GL(V)$. We denote its associated homomorphism from Definition~\ref{Def_GMotives} of Galois gerbes $\mathfrak{P}\to\mathfrak{G}_\gengpSch$ by $h_{\underline{\mcE}}$.

\begin{lem}\label{lemma-ADLV-identied-Xphi}
We have an equality of the isogeny groups $I_{\underline{\mcE}}=I_{h_{\underline{\mcE}}}$. We have 
    $X_{\infty}(h_{\underline{\mcE}})=X_{\gpSch}^{\leq\mu}(b)$ and $X^{\infty}(h_{\underline{\mcE}})=\Isom^{\otimes}(\omega,\breve{\mathcal{V}}_{\underline{\mcE}})$.
\end{lem}
\begin{proof}
This follows directly from the definitions.
\end{proof}

\begin{cor}
The $\overline{\F}_{\infty}$-points of the Shtuka space $\Sht_{\gpSch,H,\widehat{\infty}\times\infty}^{\mathcal{Z}(\mu,\beta)}$ has the form predicted by the Langlands-Rapoport conjecture, i.e. $\Sht_{\gpSch,H,\widehat{\infty}\times\infty}^{\mathcal{Z}(\mu,\beta)}(\overline{\F}_{\infty})=\coprod\limits_hI_h(\FcnFld)\backslash X_{\infty}(h)\times X^{\infty}(h)/H$ compatible with Hecke correspondences, Frobenius, action of the center. 
\end{cor}
\begin{proof}
    This follows by combing Lemma \ref{lemma-ADLV-identied-Xphi} and Theorem \ref{Uniformization2}. 
\end{proof}

\bibliographystyle{amsalpha}
\bibliography{bibfile}

\end{document}